\documentclass[10pt]{article}

\usepackage{amsmath,amsthm,amsfonts,amssymb}
\usepackage[usenames]{color}
\usepackage[hypertex]{hyperref}
\usepackage{graphicx}
\usepackage{psfig}
\usepackage{pb-diagram}

\title{Actions, length functions, and non-archemedian words}
\author{\textsf{Olga Kharlampovich}\thanks{Hunter College, CUNY, Supported by NSF grant DMS-1201379 and PSC-CUNY enhanced award}
\and \textsf{Alexei Myasnikov}\thanks{Stevens Institute of Technology, Supported by NSF grant DMS-1201379}
\and \textsf{Denis Serbin}\thanks{Stevens Institute of Technology}}
\date{}

\newtheorem{theorem}{Theorem}
\newtheorem{definition}{Definition}
\newtheorem{example}{Example}
\newtheorem{lemma}{Lemma}
\newtheorem{cor}{Corollary}
\newtheorem{remark}{Remark}

\newtheorem{prop}{Proposition}

\def\N{{\mathbb{N}}}
\def\Z{{\mathbb{Z}}}
\def\R{{\mathbb{R}}}

\def\Zt{{\mathbb{Z}[t]}}
\def\FZt{{F^{\mathbb{Z}[t]}}}

\begin{document}
\maketitle

\begin{abstract}
In this paper we survey recent developments in the theory of groups acting on $\Lambda$-trees.
We are trying to unify all significant methods and techniques, both classical and recently developed, 
in an attempt to present various faces of the theory and to show how these methods can be used
to solve major problems about finitely presented $\Lambda$-free groups. Besides surveying results
known up to date we draw many new corollaries concerning structural and algorithmic properties 
of such groups.
\end{abstract}

\tableofcontents

\section{Introduction}
\label{sec:intro}

In 2013 the theory of group actions on $\Lambda$-trees will be half a century old. There are three
stages in the development of the theory: the initial period when the basic concepts, methods and
open problem were laid down; the period of concentration on actions on $\R$-trees; and the recent
phase when the focus is mostly on non-Archimedean actions.

The initial stage takes its roots in several, seemingly independent, areas of group theory and
topology: the study of abstract length functions in groups, Bass-Serre theory of actions on simplicial
trees, Stallings pregroups and free constructions, and the theory of $\R$-trees with its connections
with Thurston's Geometrisation Theorem.

In 1963 Lyndon introduced groups $G$ equipped with an abstract length function $l: G \to \Z$ as a 
tool to carry over Nielsen cancellation theory from free groups and free products, to a much more 
general setting \cite{Lyndon:1963}. The main idea was to measure the amount of cancellation that
occurs in passing to the reduced form of a product of two reduced forms in  a free group (or a free 
product of groups) and describe it in an abstract axiomatic way. To this end Lyndon introduced a 
new quantity $c(g,h) = \frac{1}{2}(l(g) + l(h) - l(g^{-1}h)$ for elements $g,h \in G$ (to be precise, 
he considered $(c(h^{-1},g^{-1})$, which gives equivalent axioms), thus anticipating the Gromov product
in metric spaces, and gave some rather natural and simple axioms on the length function $l$, which
now could be interpreted as axioms of $0$-hyperbolic $\Lambda$-metric spaces. He completely
described groups with {\em free} $\Z$-valued length functions (such length functions correspond to
actions on simplicial trees without fixed points and edge inversion), as subgroups of ambient free 
groups with induced length. Some initial  results for groups with length functions in $\Z$ or $\R$ 
were obtained in 1970s (see \cite{Harrison:1972, Chiswell:1976, Hoare:1979, Hoare:1988}), which 
led to  the general notion of a Lyndon length function $l:G \to \Lambda$  with values in an ordered 
abelian group $\Lambda$. We refer to the book \cite{Lyndon_Schupp:2001} on the early history of 
the subject.

At about the same time Serre laid down fundamentals of the theory of groups acting on simplicial
trees. Serre and Bass  described  the structure of such groups via fundamental groups of graphs of groups. In the
following decade their approach unified several geometric and combinatorial methods of group
theory into a unique tool, known today as Bass-Serre theory. This was one of the major achievements
of combinatorial group theory in 1970s. We refer here to Serre's seminal book \cite{Serre:1980} and
more recent books (for example, in \cite{DunwoodyDicks:1989}).

On the other hand, in 1960s Stallings introduced a notion of a {\em pregroup} $P$ and its {\em
universal group} $U(P)$ \cite{Stallings:1971}. Pregroups $P$ provide a very convenient tool to
describe reduced forms of elements of their universal groups $U(P)$. Precise connections between
pregroups and free constructions (the fundamental groups of graphs of groups)  were
established by Rimlinger in \cite{Rimlinger:1987, Rimlinger:1987(2)}. Recently this technique was
invaluable in dealing with infinite non-Archimedean words (see below).

The first definition of an $\R$-tree appeared in the work of Tits \cite{Tits:1977} in connection with
Bruhat-Tits buildings. One year earlier Chiswell \cite{Chiswell:1976} came up with a crucial
construction that shows that a group $G$ with a Lyndon length function $l : G \to \R$  has a natural
action by isometries on an $\R$-tree (even though he did not refer to the space as an $\R$-tree),
and vice versa. This was an important result which showed that free (no fixed points) group actions
and free Lyndon length functions are just two equivalent languages describing essentially the same
objects. We refer to the book \cite{Chiswell:2001} for a detailed discussion on the subject. In 1984,
in their very influential paper \cite{Morgan_Shalen:1984}, Morgan and Shalen linked group actions
on $\R$-trees with Thurston's Geometrization Theorem. In the same paper they introduced
$\Lambda$-trees for an arbitrary ordered abelian group $\Lambda$ and established a general form of
Chiswell's construction. Thus, it became clear that abstract length functions with values in $\Lambda$
and group actions on $\Lambda$-trees are just two equivalent approaches to the same realm of group
theory questions. In the subsequent papers Morgan and Shalen, and Culler and Morgan further developed
the theory of $\R$-trees and group actions \cite{Morgan_Shalen:1984, Culler_Morgan:1987,
Morgan_Shalen:1988, Morgan_Shalen:1988b}, we refer to surveys \cite{Shalen:1987, Shalen:1991,
Morgan:1992} for more details and a discussion on these developments.

Alperin and Bass \cite{Alperin_Bass:1987} developed much of the initial framework of the  theory of
group actions on $\Lambda$-trees and stated the fundamental research goal (see also \cite{Bass:1991}):

\medskip

{\bf Fundamental Problem}: {\em Find the group-theoretic information carried by a $\Lambda$-tree
action, analogous to the Bass-Serre theory for the case $\Lambda = \Z$.}

\medskip

It is not surprising, from the view-point of Bass-Serre theory, that the following problem (from
\cite{Alperin_Bass:1987}) became of crucial importance:

\medskip

{\bf The Main Problem:} {\em  Describe groups acting freely on $\Lambda$-trees}.

\medskip

Following Bass \cite{Bass:1991} we refer to such groups as $\Lambda$-free groups. They are the main
object of this survey.

Recall that Lyndon himself completely characterized arbitrary $\Z$-free groups \cite{Lyndon:1963}.
He showed also that free products of subgroups of $\R$ are $\R$-free and conjectured (in the
language of length functions) that the converse is also true. It was shown by Aperin and Moss
\cite{Alperin_Moss:1985}, and by Promislow \cite{Promislow:1985} that there are infinitely
generated $\R$-free groups whose length functions are not induced by the natural ones on free products of subgroups of $\R$. Furthermore, examples by Dunwoody\cite{Dunwoody:1997}  and Zastrow \cite{Zastrow:1998} show that there are infinitely generated  $\mathbb{R}$-free groups that  cannot be decomposed  as free products of surface groups and subgroups of
$\mathbb{R}$ (see Section \ref{subs:rips_th}).

It became clear that
infinitely generated $\Lambda$-free groups are so diverse that any reasonable description of them
was (and still is) out of reach. The problem was modified by imposing some natural restrictions:
\medskip

{\bf The Main Problem (Restricted):} {\em  Describe finitely generated (finitely presented) groups
acting freely on $\Lambda$-trees}.

\medskip

This ends the first and starts the second stage of the development of the theory of group actions,
where the main effort is focused on group actions on $\R$-trees.

\medskip

In 1991 Morgan and Shalen \cite{Morgan_Shalen:1991} proved that all surface groups (with exception
of non-orientable ones of genus $1, 2$ and $3$) are $\R$-free. In the same paper they conjectured
the following, thus enlarging the original Lyndon's statement.

\medskip

{\bf The Morgan-Shalen Conjecture:} {\em  Show that finitely generated $\R$-free groups are free
products of free abelian groups of finite rank and non-exceptional surface groups}.

\medskip

This conjecture turned out to be extremely influential and in the years after much of the research
in this area was devoted to proving it. In 1991 Rips came out with the idea of a solution of the
conjecture in the affirmative. Gaboriau, Levitt and Paulin \cite{GLP:1994}, and independently
Bestvina and Feighn \cite{Bestvina_Feighn:1995} published the solution. The final result that
completely characterizes finitely generated $\R$-free groups is now known as Rips Theorem. Notice,
that free actions on $\R$-trees cover all free Archimedean actions, since every group acting freely
on a $\Lambda$-tree for an Archimedean ordered abelian group $\Lambda$ also acts freely on an
$\R$-tree.

In fact, Bestvina and Feighn proved much more, they showed that if a finitely presented group $G$
admits a nontrivial, stable and minimal action on an $\R$-tree then it splits into a free construction
of a special type. The key ingredient of  their proof is what they called a  "Rips machine", which is a
geometric version of the Makanin process (algorithm) developed for solving equations in free groups
\cite{Makanin:1982, Razborov:1985}. The machine (either Rips or Makanin) takes a band complex or a
generalized equation related to the action as an input and then outputs again a band complex or a
generalized equation but in a very specific form, which allows one to see the splitting of the group.
Nowadays, as we will see below, there are many different variations of the Makanin-Razborov processes
(besides the Rips machine), we will refer to them in all their incarnations as {\em Machines} or
{\em Elimination Processes}.  In fact, there is no any published description of the original Rips 
machine, and the machine described by Bestvina and Feighn in \cite{Bestvina_Feighn:1995} (we refer 
to it as Bestvina-Feighn machine) seems much more powerful then was envisioned by Rips. The range 
of applications of Bestvina-Feighn machines is quite broad due to a particular construction introduced 
by Bestvina \cite{Bestvina:1988} and Paulin \cite{Paulin:1988}.
This construction shows how  a sequence of isometric actions of a group $G$ on some ``negatively
curved spaces'' gives rise to an isometric action of $G$ on the ``limit space'' (in the Gromov-Hausdorff
topology) of the sequence, which is an $\R$-tree, so the machine applies. We refer to Bestvina's
preprint \cite{Bestvina:1999} for various applications of Rips-Bestvina-Feighn machine.

There are other interesting results on non-free actions on $\R$-trees which we do not discuss here,
since our main focus is on $\Lambda$-free groups, but there is one which we  would like to mention
now. In \cite{Gromov:1987} Gromov introduced a notion of an asymptotic cone of a metric space. He
showed that an asymptotic cone of a hyperbolic group $G$ is an $\R$-tree, on which $G$ acts
non-trivially. This gives a powerful tool to study hyperbolic groups.

In 1991 Bass in his pioneering work \cite{Bass:1991} made inroads into the realm of non-Archimedean
trees ($\Lambda$-trees with non-Archimedean $\Lambda$) and obtained some structural results on
$\lambda$-free groups for specific $\Lambda$, thus began the third stage in study of $\Lambda$-trees.
To explain, let $\Lambda_0$ be an ordered abelian group and $\Lambda = \Lambda_0 \oplus \Z$ with
the right lexicographical order. It was shown in \cite{Bass:1991} that if $G$ is $\Lambda$-free then
$G$ is the fundamental group of a graph of groups where the vertex groups are $\Lambda_0$-free, the
edge groups are either trivial or maximal abelian in the adjacent vertex groups, and some other
``compatibility conditions'' are satisfied.  In particular, since $\Z^n = \Z^{n-1} \oplus \Z$ one
gets by induction on $n$ structural results on $\Z^n$-free groups. This result gave a standard on
how to describe algebraic structure of groups acting freely on $\Lambda$-trees.

In 1995 the first two authors of this survey together with Remeslennikov began a detailed study
of $\Z^n$-free groups. The starting point was to show that every finitely generated fully residually
free group $G$ is $\Z^n$-free for a suitable $n$. These groups play an important part in several
areas of group theory, where they occur under different names: freely discriminated groups (see,
for example, \cite{Newmann:1967, Myasnikov_Remeslennikov:2000, Baumslag_Miasnikov_Remeslennikov:1999})
$E$-free groups (that is, groups with the same existential theory as free non-abelian groups
\cite{Remeslennikov:1989}), limit groups in terms of Sela \cite{Sela:2001}, and limits of free groups in
the Gromov-Hausdorff topology (see \cite{Champetier_Guirardel:2005}). Though the first two names (for the
same definition) occurred much earlier, we use here the term ``limit'' as a much shorter one. We
describe the main points of our approach to limit groups here since these techniques  became the
corner stones for the further development of the general theory of groups acting on
$\Lambda$-trees and hyperbolic $\Lambda$-metric spaces.

In 1960 Lyndon introduced a free $\Z[t]$-exponential group $F^{\Z[t]}$, where $F$ is a free
non-abelian group, and showed that $F^{\Z[t]}$ is freely discriminated by $F$, so all finitely
generated subgroups of $F^{\Z[t]}$ are limit groups \cite{Lyndon:1960}.  In 1995 Myasnikov and
Remeslennikov constructed a free Lyndon length function on $F^{\Z[t]}$ with values in $\Z^\omega$
(the direct sum of countably many copies of $\Z$ with the lexicographical order)
\cite{Myasnikov_Remeslennikov:1995}. To do this we  proved first that $F^{\Z[t]}$ is the union of a
countable chain of subgroups $F = G_1 < G_2 \cdots < G_n < \cdots$, where each group $G_{n+1}$ is
obtained from $G_n$ by an extension of a centralizer \cite{Myasnikov_Remeslennikov:1994,
Myasnikov_Remeslennikov:1996}; and then showed that if $G_n$ is a $\Z^{n}$-free group then $G_{n+1}$
is $\Z^{n+1}$-free \cite{Myasnikov_Remeslennikov:1995}. Two years later Kharlampovich and Myasnikov
showed that every finitely generated fully residually free group embeds into $G_n$ (from the sequence
above) for a suitable $n$, hence every limit groups is a $\Z^n$-free group for a suitable $n$
\cite{Kharlampovich_Myasnikov:1998(1), Kharlampovich_Myasnikov:1998(2)}. These results became
crucial in the study of limit groups and their solution to the Tarski problems
\cite{Kharlampovich_Myasnikov:2006}. A direct application of Bass-Serre theory to subgroups of the
groups $G_n$ gives an algebraic structure of limits groups as the fundamental groups of graphs of
groups induced from the corresponding graphs for $G_n$. In \cite{Kharlampovich_Myasnikov:1998(1),
Kharlampovich_Myasnikov:1998(2)} a new version of Makanin-Razborov process was introduced,
termed the {\em Elimination process}, a.k.a. a general $\Z$-machine (it works over arbitrary finite
sets of usual words in a finite alphabet, that is, $\Z$-words). Notice that we will further denote
systems of equations by $S = 1$ meaning that $S$ is a set of words, and each of these words is equal
to the identity. Application of such a machine to a finite system $S = 1$ of equations (with constants)
over $F$ results in finitely many systems $U_1 = 1, \ldots, U_n = 1$ over $F$ such that the solution
set of $S = 1$ is the union (up to a change of coordinates) of the solution sets of $U_i, i \in [1,n]$
and each $U_i$ is given in the standard NTQ (non-degenerate quasi-quadratic) form
\cite{Kharlampovich_Myasnikov:1998(2)}. This is a precise analog of the elimination and parametrization
theorems from the classical algebraic geometry, thus $\Z$-machines play the role of the classical
resolution method in the case of algebraic geometry over free groups \cite{Myasnikov_Remeslennikov:2000,
Baumslag_Miasnikov_Remeslennikov:1999, Kharlampovich_Myasnikov:1998(1), Kharlampovich_Myasnikov:1998(2)}.
At the level of groups (here limit groups occur as the coordinate groups of systems of equations in
$F$) $\Z$-machines give embedings of limit groups into the coordinate groups of NTQs, which were
later also called fully residually free towers \cite{Sela:2001}. There are very interesting corollaries
from this result along for limit groups (see Section \ref{subs:alg_prob_limit_gps} for details).
Observe, that these $\Z$-machines (accordingly modified) were successfully used for other classes
of groups in similar situations. For example, Kazachkov and Casals-Ruiz used it to solve equations
in free product of groups \cite{CK1} and right-angled Artin groups \cite{CK2}.

To study algorithmic properties of limit groups Kharlampovich and Myasnikov developed an
elimination process (a $\Z^n$-machine) which works for arbitrary free Lyndon length functions with
values in $\Z^n$ (see \cite{Kharlampovich_Myasnikov:2005(2)}), thus giving a further generalization
of the original $\Z$-machines mentioned above. These $\Z^n$-machines can be described in several
(equivalent) forms: algorithmic, geometric, group-theoretic, or as dynamical systems. In the
group-theoretic language the machines provide various splittings of a group as the fundamental group
of a graph of groups with abelian edge groups (abelian splittings). Since $\Z^n$-machines (unlike
$\R$-machines mentioned above) are algorithms - one can use them in solving algorithmic problems.
A direct application of such machines (precisely like in the case of free groups)  allows one to
solve arbitrary equations and describe their solution sets in arbitrary limit groups
\cite{Kharlampovich_Myasnikov:2005(2)}. In \cite{Kharlampovich_Myasnikov:2005(2)} Kharlampovich and
Myasnikov used these machines to effectively construct JSJ-decompositions of limit groups given by
their finite presentations (or almost any other effective way), which allowed them together with
Bumagin to solve the isomorphism problem for limit groups (see \cite{Bumagin_Kharlampovich_Myasnikov:2007}).
Later these results (following basically the same line of argument) were generalized to arbitrary toral relatively
hyperbolic groups \cite{Dahmani_Groves:2008}.

$\Z^n$-machines, while working with $\Z^n$-free groups in fact manipulate with the so-called {\em
infinite non-Archimedean words}, in this case $\Z^n$-words (so the ordinary words are just $\Z$-words).
Discovery of non-Archimedean words turned out to be one of the major recent developments in the
theory of group actions, it is worthwhile to say a few words about it here. In
\cite{Myasnikov_Remeslennikov_Serbin:2005} Myasnikov, Remeslennikov and Serbin introduced infinite
$\Lambda$-words for arbitrary $\Lambda$ and showed that groups admitting faithful representations
by $\Lambda$-words act freely on $\Lambda$-trees, while Chiswell proved the converse \cite{Chiswell:2005}.
This gives another, equivalent, approach to free actions on $\Lambda$-trees, so now one can replace the
axiomatic viewpoint of length functions along with many geometric arguments coming from
$\Lambda$-trees by the standard combinatorics on $\Lambda$-words. In particular, this approach allows
one to naturally generalize powerful techniques of  Nielsen's method, Stallings' graph approach to
subgroups, and Makanin-Razborov type of elimination processes (the machines) from free groups to
$\Lambda$-free groups (see \cite{Myasnikov_Remeslennikov_Serbin:2005, Myasnikov_Remeslennikov_Serbin:2006,
Kharlampovich_Myasnikov:2005(2), Kharlampovich_Myasnikov:2006, Kharlampovich_Myasnikov:2010,
KMRS:2012, DM, KMS:2011, KMS:2011(3), Nikolaev_Serbin:2011, Nikolaev_Serbin:2011(2),
Serbin_Ushakov:2011(1)}). In the case when $\Lambda$ is equal to either $\Z^n$ or $\Z^\infty$ all
these techniques are effective, so, for example, one can solve many algorithmic problems for limit
groups or $\Z^n$-free groups using these methods.

The technique of infinite $Z^n$-words allowed us to solve many algorithmic problems for limit groups
precisely in the same manner as it was done for free groups by the standard Stallings graph techniques
(we refer to \cite{Kapovich_Myasnikov:2002} for a survey on the related results in free groups). We
discuss these results in Section \ref{subs:alg_prob_limit_gps}.

In our approach to limit groups (or arbitrary $\Z^n$-groups) one more important concept transpired.
To carry over Nielsen cancellation argument or apply the machines  the length function (or the action)
has to satisfy some natural ``completeness'' conditions. To this end,  given a group $G$ acting
on a $\Lambda$-tree $\Gamma$, we say that the action is {\em regular with respect to $x \in \Gamma$}
(see \cite{KMRS:2012} for details) if for any $g,h \in G$ there exists $f \in G$ such that $[x, f x]
= [x, g x] \cap [x, h x]$. In fact, the definition above does not depend on $x$ and there exist
equivalent formulations for length functions and $\Lambda$-words (see \cite{Promislow:1985,
Myasnikov_Remeslennikov_Serbin:2005}). Roughly speaking, regularity of action implies that all
branch-points of $\Gamma$ belong to the same $G$-orbit and it tells a lot about the structure of
$G$ in the case of free actions (see \cite{KMS:2011(3), KMRS:2012} or Section \ref{subse:regact}).
In the language of $\Lambda$-words this condition means that for any given two infinite words
representing elements in the group $G$ their common longest initial segment (which always exists)
represents an element in $G$. The importance of the regularity condition was not recognized earlier
simply because in the case when $\Lambda$ is $\Z$, or $\R$ every $\Lambda$-free group has a regular
action. The regularity axiom appeared first in \cite{Myasnikov_Remeslennikov:1995,
Myasnikov_Remeslennikov_Serbin:2005} as a tool to deal with length functions in $\mathbb{Z}^n$ (with
respect to the lexicographic order). The outcome of this research is that if a finitely generated
group $G$ has a regular free action on a $\Z^n$-tree, then the Nielsen method and $\Z^n$-machines
work in $G$. This, as expected, implies a lot of interesting results for $\Z^n$-free groups.

In the paper \cite{KMRS:2012} we described finitely generated groups which admit free regular actions
on $\Z^n$-trees, we call such groups finitely generated {\em complete $\Z^n$-free groups}.

Suppose $G_i$ is an $\Z^{n_i}$-free group, $i = 1, 2$  with a  maximal abelian subgroups  $A
\leqslant G_1,\ B \leqslant G_2$ such that
\begin{enumerate}
\item[(a)] $A$ and $B$ are cyclically reduced with respect to the corresponding embeddings of $G_1$
and $G_2$ into infinite words,
\item[(b)] there exists an isomorphism $\phi : A \rightarrow B$ such that $|\phi(a)| = |a|$ for any
$a \in A$.
\end{enumerate}
Then we call the amalgamated free product
$$\langle G_1, G_2 \mid A \stackrel{\phi}{=} B \rangle$$
the {\em length-preserving amalgam} of $G_1$ and $G_2$.

Given a $\Z^n$-free group $H$ and non-conjugate maximal abelian subgroups $A, B \leqslant H$ such
that
\begin{enumerate}
\item[(a)] $A$ and $B$ are cyclically reduced with respect to the embedding of $H$ into infinite words,
\item[(b)] there exists an isomorphism $\phi : A \rightarrow B$ such that $|\phi(a)| = |a|$ and
$a$ is not conjugate to $\phi(a)^{-1}$ in $H$ for any $a \in A$.
\end{enumerate}
then we call the HNN extension
$$\langle H, t \mid t^{-1} A t = B \rangle$$
the {\em length-preserving separated HNN extension} of $H$.

Given a $\Z^n$-free group $H$ and a maximal abelian subgroup $A$, we call the HNN extension
$$\langle H, t \mid t^{-1} A t = A \rangle$$ (here the isomorphism $\phi:A\rightarrow A$ is identical)
the {\em centralizer extension} of $H$.

The following result describes complete $\Z^n$-free groups.

\medskip

{\bf Theorem \ref{th:complete_desrc}.}
{\em A finitely generated group $G$ is complete $\Z^n$-free if and only if it can be obtained from
free groups by finitely many length-preserving separated HNN extensions and centralizer extensions.}

\medskip

Notice that this is "if and only if" characterization of finitely generated complete $\Z^n$-free groups.

The following principle theorem  allows one to transfer results from complete $\Z^n$-free groups to
arbitrary $\Z^n$-free groups.

\medskip{\bf Theorem \ref{th:embed_complete}.} \cite{KMS:2011(2)}
{\em Every finitely generated $\Z^n$-free  group $G$ has a length-preserving embedding into a
finitely generated complete $\Z^n$-free group $H$.}

\medskip

As in the case of limit groups,  Bass-Serre theory, applied to subgroups of complete $\Z^n$-free
groups, immediately gives algebraic structure of arbitrary  finitely generated $\Z^n$-free groups:

\medskip

{\bf Theorem \ref{th:Z^n_desrc}.}
{\em A finitely generated group $G$ is $\Z^n$-free if and only if it can be obtained from free groups
by a finite sequence of length-preserving amalgams, length-preserving separated HNN extensions,
and centralizer extensions.}

\medskip

Again, we would like to emphasize here that the description above is ``if and only if''.

Using these techniques we generalized many algorithmic results from limit groups to arbitrary finitely
generated $\Z^n$-free groups. We refer to the Section \ref{subs:alg_prob_limit_gps} for details.

In 2004 Guirardel described the structure of finitely generated $\R^n$-free groups in a similar
fashion \cite{Guirardel:2004} (see Theorem \ref{th:Guirardel} in Section \ref{se:R^n-free}). Since
$\Z^n$-free groups are also $\R^n$-free this, of course, generalizes the Bass' result on $\Z^n$-free
groups. Observe that given description of the algebraic structure of finitely generated  $\R^n$-free
groups does not ``characterize'' such groups completely, that is, the converse of the theorem does
not hold. Nevertheless, the result is strong, it implies several important corollaries: firstly, it
shows that finitely generated $\R^n$-free groups are finitely presented; and secondly, taken together
with Dahmani's combination theorem \cite{Dahmani:2003}, it implies that every finitely generated
$\R^n$-free group is hyperbolic relative to its non-cyclic abelian subgroups.

{\bf New techniques}. The recent developments in the theory of groups acting freely on
$\Lambda$-trees are based on several new techniques that occurred after 1995: infinite
$\Lambda$-words as an equivalent language for free $\Lambda$-actions
\cite{Myasnikov_Remeslennikov_Serbin:2005}; regular actions and regular completions; general
Machines (Elimination Processes) for regular $\Lambda$-actions that generalize Rips and
Bestvina-Feighn Machines.

The elimination process techniques developed in Section \ref{sec:proc} allow one to prove the
following  theorems.

\medskip

\noindent
{\bf Theorem \ref{th:main1}.} [The Main Structure Theorem \cite{KMS:2011(3)}]
{\em Any finitely presented group $G$ with a regular free length function in an ordered abelian
group $\Lambda$ can be represented as a union of a finite series of groups
$$G_1 < G_2 < \cdots < G_n = G,$$
where
\begin{enumerate}
\item $G_1$ is a free group,
\item $G_{i+1}$ is obtained from $G_i$ by finitely many HNN-extensions in which associated subgroups
are maximal abelian, finitely generated, and length isomorphic as subgroups of $\Lambda$.
\end{enumerate}}

\medskip

\noindent
{\bf Theorem \ref{th:main3}.} \cite{KMS:2011(3)}
{\em Any finitely presented $\Lambda$-free group is  $\R^n$-free.}

In his book \cite{Chiswell:2001} Chiswell (see also \cite{Remeslennikov:1989})  asked the following
very important question (Question 1, p. 250): If $G$ is a finitely generated $\Lambda$-free group,
is $G$ $\Lambda_0$-free for some finitely generated abelian ordered group $\Lambda_0$? The
following result answers this question in the affirmative in the strongest form. It comes from the
proof of Theorem \ref{th:main3} (not the statement of the theorem itself).

\medskip

\noindent
{\bf Theorem \ref{co:main1}}
{\em Let $G$ be a finitely presented group with a free Lyndon length function $l : G \to \Lambda$.
Then the subgroup $\Lambda_0$ generated by $l(G)$ in $\Lambda$ is finitely generated.}

\medskip

\noindent
{\bf Theorem \ref{th:main4}.} \cite{KMS:2011(3)}
{\em Any finitely presented group $\widetilde G$ with a free length function in an ordered abelian
group $\Lambda$ can be isometrically embedded into a finitely presented group $G$ that has a free
regular length function in $\Lambda$.}

The following result automatically follows from Theorem \ref{th:main1} and Theorem \ref{th:main4}
by simple application of Bass-Serre Theory.

\medskip

\noindent
{\bf Theorem \ref{co:main5}.}
{\em Any finitely presented $\Lambda$-free group $G$ can be obtained from free groups by a finite
sequence of amalgamated free products and HNN extensions with maximal abelian associated
subgroups, which are free abelain groups of finite rank.}

The following result is about abelian subgroups of $\Lambda$-free groups. For $\Lambda = \Z^n$ it
follows from the main structural result for $\Z^n$-free groups and \cite{KMRS:2008}, for $\Lambda =
\R^n$ it was proved in \cite{Guirardel:2004}. The statement (1) below answers Question 2, p. 250
in \cite{Chiswell:2001} in the affirmative for finitely presented $\Lambda$-free groups.

\medskip

\noindent
{\bf Theorem \ref{co:main1b}.}
{\em Let $G$ be a finitely presented $\Lambda$-free group. Then
\begin{itemize}
\item[(1)] every abelian subgroup of $G$ is a free abelian group of finite rank uniformly bounded
from above by the rank of the abelianization of $G$.
\item[(2)] $G$ has only finitely many conjugacy classes of maximal non-cyclic abelian subgroups,
\item[(3)] $G$  has a finite classifying space and the cohomological dimension of $G$  is at most
$\max \{2, r\}$ where $r$  is the maximal rank of an abelian subgroup of $G$.
\end{itemize}}

\medskip

\noindent
{\bf Theorem \ref{co:Lambda_Guirardel_2}.}
{\em Every finitely presented  $\Lambda$-free group is hyperbolic relative to its non-cyclic abelian
subgroups.}

The following results answers affirmatively in the strongest form to the Problem (GO3) from the
Magnus list of open problems \cite{BaumMyasShpil:2002} in the case of finitely presented groups.

\medskip

\noindent
{\bf Corollary \ref{co:Lambda_Guirardel_3}.}
{\em Every finitely presented  $\Lambda$-free group  is biautomatic.}

\medskip

\noindent
{\bf Theorem \ref{th:Lambda_quasi-convex_hierarchy}.}
{\em Every finitely presented $\Lambda$-free group $G$ has a quasi-convex hierarchy.}

As a corollary one gets the following result.

\medskip
\noindent
{\bf Theorem \ref{co:Lambda_quasi-convex_hierarchy_1}.}
{\em Every finitely presented  $\Lambda$-free group $G$  is locally undistorted, that is, every 
finitely generated subgroup of $G$ is quasi-isometrically embedded into $G$.}

Since a finitely generated $\R^n$-free group $G$ is hyperbolic relative to to its non-cyclic abelian
subgroups and $G$ admits a quasi-convex hierarchy then recent results of Wise \cite{Wise:2011}
imply the following.

\medskip

\noindent
{\bf Corollary \ref{co:Lambda_special}.}
{\em Every finitely presented $\Lambda$-free group $G$ is  virtually special, that is, some subgroup
of finite index in $G$ embeds into a right-angled Artin group.}

In his book \cite{Chiswell:2001} Chiswell posted Question 3 (p. 250): Is every $\Lambda$-free group
orderable, or at least right-orderable? The following result answers this question in the affirmative
in the case of finitely presented groups.

\medskip

\noindent
{\bf Theorem \ref{co:Lambda_special_0}.}
{\em Every finitely presented $\Lambda$-free group is right orderable.}

The following addresses Chiswell's question whether $\Lambda$-free groups are orderable or not.

\medskip

\noindent
{\bf Theorem \ref{co:Lambda_special_1}.}
{\em Every finitely presented $\Lambda$-free group is virtually orderable, that is, it contains an
orderable subgroup of finite index.}

Since right-angled Artin groups are linear (see \cite{Humphries:1994, Hsu_Wise:1999, DJ:2000} and
the class of linear groups is closed under finite extension we get the following

\medskip

\noindent
{\bf Theorem \ref{co:Lambda_quasi-convex_hierarchy_2}.}
{\em Every finitely presented $\Lambda$-free group is linear.}

Since every linear group is residually finite we get the following.

\medskip

\noindent
{\bf Corollary \ref{co:Lambda_quasi-convex_hierarchy_3}.}
{\em Every finitely presented  $\Lambda$-free group is residually finite.}

It is known that linear groups are equationally Noetherian (see \cite{Baumslag_Miasnikov_Remeslennikov:1999}
for discussion on equationally Noetherian groups), therefore the following result holds.

\medskip

\noindent
{\bf Corollary \ref{co:Lambda_quasi-convex_hierarchy_4}.}
{\em Every finitely presented $\Lambda$-free group is equationally Noetherian.}

The structural results of the previous section give solution to many algorithmic problems on
finitely presented $\Lambda$-free groups.

\medskip

\noindent
{\bf Theorem \ref{th:word_conjugacy_lambda}.} \cite{KMS:2011(3)}
{\em Let $G$ be a finitely presented $\Lambda$-free group. Then the following algorithmic problems
are decidable in $G$:
\begin{itemize}
\item the Word Problem;
\item the Conjugacy Problems;
\item the Diophantine Problem (decidability of arbitrary equations in $G$).
\end{itemize}
}

Theorem \ref{co:Lambda_Guirardel_2} combined with results of Dahmani and Groves
\cite{Dahmani_Groves:2008} immediately  implies the following two corollaries.

\medskip

\noindent
{\bf Corollary \ref{co:algorithm_lambda_1}.}
{\em Let $G$ be a finitely presented  $\Lambda$-free group. Then:
\begin{itemize}
\item $G$ has a non-trivial abelian splitting and one can find such a splitting effectively,
\item $G$ has a non-trivial abelian JSJ-decomposition and one can find such a decomposition
effectively.
\end{itemize}
}

\medskip

\noindent
{\bf Corollary \ref{co:algorithm_lambda_2}.}
{\em The Isomorphism Problem is decidable in the class of finitely presented groups that act freely
on some $\Lambda$-tree.}

\medskip

\noindent
{\bf Theorem \ref{co:membership_lambda}.}
{\em The Subgroup Membership Problem is decidable in every finitely presented $\Lambda$-free group.}

\section{Bass-Serre Theory}
\label{sec:bass-serre}

In his seminal book \cite{Serre:1980} J.~P. Serre laid down fundamentals of the theory of groups acting
on simplicial trees. In the following decade Serre's novel approach unified several geometric,
algebraic, and combinatorial methods of group theory into a unique powerful tool, known today as
Bass-Serre Theory. This tool allows one to obtain a lot of structural information about the group
from its action on a simplicial tree in terms of stabilizers of vertices and edges. One of the most
important consequences of this approach is the structure of subgroups of the initial group which
inherit the action on the tree and hence can be dealt with in the same manner as the ambient group.

\smallskip

In this section, following \cite{DunwoodyDicks:1989} we give basic treatment of Bass-Serre Theory.
The original ideas and results can be found in \cite{Serre:1980}, a more topological treatment of the theory - in \cite{Scott_Wall:1979}.

\subsection{$G$-sets and $G$-graphs}
\label{subs:G-set}

Let $G$ be a group. $X$ is called a {\em $G$-set} if there is a function $G \times X \rightarrow X$
({\em left action of $G$ on $X$}) given by $(g,x) \rightarrow gx$ such that $1 x = x$ for all $x
\in X$ and $f(gx) = (fg)x$ for any $f,g \in G,\ x \in X$.

A function $\alpha : X \rightarrow Y$ between $G$-sets is a {\em $G$-map} if $\alpha(gx) = g
\alpha(x)$ for any $g \in G,\ x \in X$. $X$ and $Y$ are {\em $G$-isomorphic} if $\alpha$ is a
bijection.

If $X$ is a $G$-set, then the {\em $G$-stabilizer} of $x \in X$ is a subgroup $G_x = \{g \in G \mid
gx = x\}$ of $G$. $X$ is {\em $G$-free} if $G_x = 1$ for any $x \in X$. For $x \in X$, the {\em
$G$-orbit} of $x$ is $Gx = \{gx \mid g \in G\}$, a $G$-subset of $X$ which is $G$-isomorphic to $G /
G_x$ with $gx \in Gx$ corresponding to $gG_x \in G / G_x$. The {\em quotient set} for the $G$-set
$X$ is defined as $G \backslash X = \{Gx \mid x \in X\}$, the set of $G$-orbits. A {\em
$G$-transversal} in $X$ is a subset $S$ of $X$ which meets each $G$-orbit exactly once, so $S
\rightarrow G \backslash X$ is bijective.

A {\em $G$-graph} $(X,V(X),E(X),\sigma,\tau)$ is a non-empty $G$-set $X$ with a non-empty $G$-subset
$V(X)$, its complement $E(X) = X - V(X)$, and three maps
$$\sigma : E(X) \rightarrow V(X), \ \ \tau : E(X) \rightarrow V(X), \ \ ^- : E(X) \rightarrow E(X),$$
which satisfy the following conditions:
$$\sigma(\bar{e}) = \tau(e),\ \tau(\bar{e}) = \sigma(e),\ \bar{\bar{e}}= e,\ \bar{e} \neq e.$$
$\sigma$ and $\tau$ are called {\em incidence maps}.

For $G$-graphs $X,Y$, a {\em $G$-graph map} $\alpha : X \rightarrow Y$ is a $G$-map such that
$\alpha(V(X)) \subseteq V(Y),\ \alpha(E(X)) \subseteq E(Y)$, and $\alpha(\sigma(e)) =
\sigma(\alpha(e)),\ \alpha(\tau(e)) = \tau(\alpha(e))$ for any $e \in E(X)$.

A {\em path} $p$ in a $G$-graph $X$ is a sequence $e_1 \cdots e_n$ of edges such that $\tau(e_i) =
\sigma(e_{i+1}),\ i \in [1,n-1]$. In this case, $\sigma(e_1)$ is the origin of $p$ and $\tau(e_n)$
is its terminus. $p$ is {\em closed} if $\sigma(e_1) = \tau(e_n)$.

A $G$-graph $X$ is a {\em $G$-tree} if for any $x,y \in V(X)$ there exists a unique path from $x$
to $y$, this path is called in this case the {\em $X$-geodesic} from $x$ to $y$.

\begin{prop} [\cite{DunwoodyDicks:1989}, Prop 2.6]
\label{pr:transversal}
If $X$ is a $G$-graph and $G \backslash X$ is connected then there exist subsets $Y_0 \subseteq Y
\subseteq X$ such that $Y$ is a $G$-transversal in $X$, $Y_0$ is a subtree of $X$, $V(Y) = V(Y_0)$,
and for each $e \in E(Y)$, $\sigma(e) \in V(Y) = V(Y_0)$.
\end{prop}

The subset $Y$ from Proposition \ref{pr:transversal} is called a {\em fundamental $G$-transversal
in $X$, with subtree $Y_0$}.

\subsection{Graphs of groups}
\label{subs:graphs}

Let ${\cal G}$ be a class of groups. A {\em graph of groups} $({\cal G}, X)$ consists of a connected
graph $X$ and an assignment of $G(x) \in {\cal G}$ to every $x \in V(X) \cup E(X)$, such that for
every $e \in E(X)$, $G(e) = G(\bar{e})$, and there exists a boundary monomorphism $i_e : G(e)
\rightarrow G(\sigma(e))$. $G(v),\ v \in V(X)$ and $G(e),\ e \in E(X)$ are called respectively
{\em vertex} and {\em edge groups}.

Let $({\cal G}, X)$ be a graph of groups with a maximal subtree $Y_0$. The {\em fundamental group
$\pi({\cal G},X, Y_0)$ of the graph of groups $({\cal G}, X)$ with respect to $Y_0$} is the group
with the following presentation:
$$\langle G(v)\ (v \in V(X)),\ t_e\ (e \in E(X)) \mid rel(G(v)),\ t_e i_e(g) t^{-1}_e =
i_{\bar{e}}(g)\ (g \in G(e)),\ $$
$$t_e t_{\bar{e}} = 1,\ (e \in E(X)),\ t_e = 1\ (e \in Y_0)\rangle.$$

Let $X$ be a $G$-graph such that $G \backslash X$ is connected, and let $Y$ be a fundamental
$G$-transversal for $X$ with subtree $Y_0$. For each $e \in E(Y)$ there are unique
$\tilde{\sigma}(e), \tilde{\tau}(e) \in V(Y)$ which belong to the same $G$-orbits as $\sigma(e)$
and $\tau(e)$ respectively, and we can assume $\tilde{\sigma}(e) = \sigma(e)$. $Y$ equipped with
incidence functions $\tilde{\sigma}, \tilde{\tau} : E(Y) \rightarrow V(Y)$ becomes a graph
$G$-isomorphic to $G \backslash X$, and $Y_0$ is its maximal subtree. Observe that for each $e \in
E(Y)$, $\tau(e)$ and $\tilde{\tau}(e)$ are in the same $G$-orbit, so we can choose $t_e \in G$ such
that $t_e \tilde{\tau}(e) = \tau(e)$. It is easy to see that $t_e = 1$ if $e \in E(Y_0)$ since $Y_0$
is a subtree of $X$ and $\tilde{\tau}(e) = \tau(e)$. The set $\{t_e \mid e \in E(Y)\}$ is called a
{\em family of connecting elements}. Now, $G_e \subseteq G_{\sigma(e)}$ and $G_e \subseteq
G_{\tau(e)} = t_e G_{\tilde{\tau}(e)} t_e^{-1}$, so there is an embedding $i_e : G_e \rightarrow
G_{\tilde{\tau}(e)}$ defined by $g \rightarrow t_e g t_e^{-1}$. This data defines the graph of
groups {\em associated to $X$} with respect to the fundamental $G$-transversal $Y$, the maximal
subtree $Y_0$, and the family of connecting elements $t_e$. Denote ${\cal G} = \{ G_v \mid v \in
V(Y)\} \cup \{ G_e \mid e \in E(Y)\}$.

\begin{theorem}\cite[Theorem I.13]{Serre:1980}
\label{pr:structure}
If $X$ is a $G$-tree then $G$ is naturally isomorphic to $\pi({\cal G}, Y, Y_0)$.
\end{theorem}

\begin{remark}
\label{co:free}
From Theorem \ref{pr:structure} it follows that if $X$ is $G$-free then $G$ is isomorphic to a free
group.
\end{remark}

On the other hand, given a graph of groups $({\cal G}, X)$ with the fundamental group $G = \pi({\cal G},
X, X_0)$, one can construct a $G$-tree $Y$ which is a universal cover of $X$.

\begin{example}
\label{ex:amalgam}
Let $G = A \ast_C B$ be a free product of groups $A$ and $B$ with amalgamation along a
subgroup $C$. Observe, that $G$ is isomorphic to the fundamental group $\pi({\cal G},X, X_0)$ of the
graph of groups $X$ (see Figure \ref{amalgam}), where $X_0 = X$.
\begin{figure}[h]
\centering{\mbox{\psfig{figure=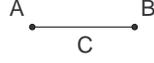,height=0.6in}}}
\caption{The graph of groups for $G = A \ast_C B$}
\label{amalgam}
\end{figure}
Define $Y$ as follows: $V(Y)$ consists of all cosets $gA$ and $gB$ ($g \in G$), $E(Y)$ consists of
all cosets $gC$ ($g \in G$), the maps $\sigma$ and $\tau$ which give the endpoints of the edge are
defined as $\sigma(gC) = gA, \ \tau(gC) = gB$. It is easy to check that $Y$ is a tree and that $G$
acts on $Y$ without inversions by the left multiplication (see Figure \ref{amalgam_tree}).
\begin{figure}[h]
\centering{\mbox{\psfig{figure=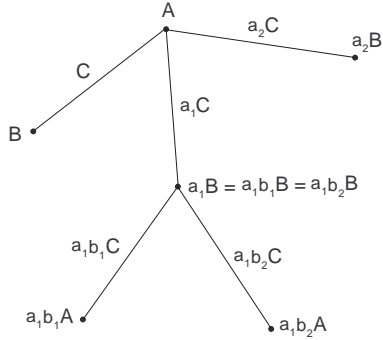,height=2in}}}
\caption{The tree $Y$ for $A \ast_C B:\ a_1, a_2 \in A \smallsetminus B,\ b_1, b_2 \in B \smallsetminus A$}
\label{amalgam_tree}
\end{figure}
All vertices $gA,\ g \in G$ are in the same orbit, the same is true about all vertices $gB,\ g \in G$
and edges $gC,\ g \in G$, and $G \backslash Y = X$.
\end{example}

\begin{example}
\label{ex:hnn}
Let $G = A \ast_C = \langle A, t | t^{-1} c t = \phi(c) \rangle$ be an HNN extension of
a group $A$ with associated subgroups $C$ and $\phi(C)$. Here $G$ is isomorphic to to the fundamental
group $\pi({\cal G},X, X_0)$ of the graph of groups $X$ (see Figure \ref{hnn}), where $X_0$ consists
of a single vertex with no edges.
\begin{figure}[h]
\centering{\mbox{\psfig{figure=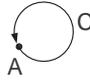,height=0.8in}}}
\caption{The graph of groups for $G = A \ast_C$}
\label{hnn}
\end{figure}
Define $Y$ as follows: $V(Y) = \{gA | g \in G\},\ E(Y) = \{ gC | g \in G\}$, and $\sigma(gC) = gA, \
\tau(gC) = (gt)A$. Again, it is easy to check that $Y$ is a tree, $G$ acts on $Y$ without inversions
by the left multiplication, and $G \backslash Y = X$ (see Figure \ref{hnn_tree}).
\begin{figure}[h]
\centering{\mbox{\psfig{figure=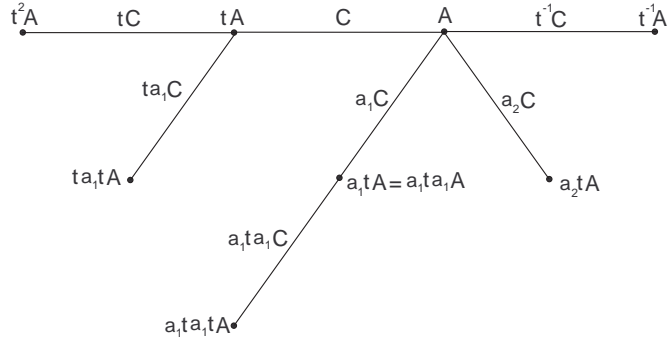,height=2in}}}
\caption{The tree $Y$ for $A \ast_C:\ a_1, a_2 \in A \smallsetminus C$}
\label{hnn_tree}
\end{figure}
\end{example}

The idea of constructing a covering tree for a given graph of groups given in the examples above can
be generalized as follows.

Given a graph of groups $({\cal G}, X)$, A {\em $({\cal G}, X)$-path of length $k \geqslant 0$} from
$v \in V(X)$ to $v' \in V(X)$ is a sequence
$$p = g_0, e_1, g_1, \ldots, e_k, g_k,$$
where $k \geqslant 0$ is an integer, $e_1 \cdots e_k$ is a path in $X$ from $v \in V(X)$ to $v' \in V(X)$,
$g_0 \in G(v), g_k \in G(v')$ and $g_i \in G(\tau(e_i)) = G(\sigma(e_{i+1})),\ i \in [1,k-1]$. If $p$
is a $({\cal G}, X)$-path from $v$ to $v'$ and $q$ is a $({\cal G}, X)$-path from $v'$ to $v''$
then one can define the concatenation $p q$ of $p$ and $q$ in the obvious way.

One can introduce the equivalence relation on the set of all $({\cal G}, X)$-paths generated by
$$g, e, i_{\bar{e}}(c), \bar{e}, f\ \sim\ g\ i_e(c)\ f,$$
where $e \in E(X),\ c \in G(e),\ g,f \in G(\sigma(e))$. Observe that if $p \sim q$ then $p$ and $q$
have the same initial and terminal vertices in $V(X)$.

\begin{lemma} \cite{KMW:2005}
\label{le:path_fund}
Let $({\cal G}, X)$ be a graph of groups and let $v_0 \in V(X)$. Then
\begin{enumerate}
\item the set $P({\cal G},X, v_0)$ of ``$\sim$''-equivalence classes of $({\cal G}, X)$-loops at $v_0$
is a group with respect to concatenation of paths,
\item for any spanning tree $T$ of $X$, $P({\cal G}, X, v_0)$ is naturally isomorphic to
$\pi({\cal G}, X, T)$.
\end{enumerate}
\end{lemma}

Let $({\cal G}, X)$ be a graph of groups and let $v_0 \in V(X)$. For $({\cal G}, X)$-paths $p,q$
originating at $v_0$ we write $p \approx q$ if
\begin{enumerate}
\item $t(p) = t(q)$,
\item $p \sim q a$ for some $a \in G(t(p))$.
\end{enumerate}

For $({\cal G}, X)$-path $p$ from $v_0$ to $v$ we denote the ``$\approx$''-equivalence class of $p$ by
$\overline{p}\ G(v)$.

Now one can define the {\em universal Bass-Serre covering tree} $T_X$ associated with $({\cal G},
X)$ as follows. The vertices of $Y$ are ``$\approx$''-equivalence classes of $({\cal G}, X)$-paths
originating at $v_0$. Two vertices $x,x' \in T_X$ are connected by an edge if and only if $x =
\overline{p}\ G(v),\ x' = \overline{p a e}\ G(v')$, where $p$ is a $({\cal G}, X)$-path from $v_0$
to $v$ and $a \in G(v),\ e \in E(X)$ with $\sigma(e) = v,\ \tau(e) = v'$.

It is easy to see that $T_X$ is indeed a tree with a natural base vertex $x_0 = \overline{1}\ G(v_0)$
and $G = P({\cal G}, X, v_0)$ has a natural simplicial action on $T_X$ defined as follows: if $g =
\overline{q}\ \in G$, where $q$ is a $({\cal G}, X)$-loop at $v_0$ and $u = \overline{p}\ G(v)$,
where $p$ is a $({\cal G}, X)$-path from $v_0$ to $v$, then $g \cdot u = \overline{q p}\ G(v)$.

\subsection{Induced splittings}
\label{subs:effective}

Let $({\cal G}, X)$ be a graph of groups with a base-vertex $v_0 \in V(X)$. Let $G = P({\cal G}, X,
v_0)$ and $T_X$ be the universal Bass-Serre covering tree associated with $({\cal G}, X)$. If $H
\leqslant G$ then the action of $G$ on $T_X$ induces an action of $H$ on $T_X$ and $H$ can be
represented as the fundamental group of a graph of groups by Theorem \ref{pr:structure}. More
precisely, the following result holds.

\begin{theorem} \cite{Serre:1980}
\label{induced_split}
Let $x_0$ be the base-vertex of $T_X$ mapping to $v_0$ under the natural quotient map and let $T_H
\subset T_X$ be an $H$-invariant subtree containing $x_0$. Then $H$ is isomorphic to the fundamental
group of a graph of groups $({\cal H}, Y)$, where $Y = H \backslash T_H$ and ${\cal H} = \{ K \cap H
\mid K \in {\cal G} \}$.
\end{theorem}
The graph of groups $({\cal H}, Y)$ from the theorem above is called the {\em induced splitting of
$H$ with respect to $T_H$}.

\begin{remark}
If $T_X$ is $G$-free and $H \leqslant G$, then $H$ also acts on $T_X$ and $T_xX$ is $H$-free. Now
it follows that any subgroup of a free group is free, which is a well-known result (see
\cite{Nielsen:1921, Schreier:1928})
\end{remark}

\begin{example}
\label{kurosh}
Let $G = A \ast B$ be a free product of $A$ and $B$. The amalgamated subgroup is trivial and $G$ is
isomorphic to the fundamental graph of groups $X$ (see Figure \ref{free_prod}).
\begin{figure}[h]
\centering{\mbox{\psfig{figure=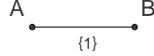,height=0.6in}}}
\caption{The graph of groups for $G = A \ast B$}
\label{free_prod}
\end{figure}
As in Example \ref{ex:amalgam}, one can construct a tree $Y$ on which $G$ acts so that $X = G
\backslash Y$. Since the amalgamated subgroup is trivial, $G(e)$ is trivial for every $e \in E(Y)$
and for each $v \in V(Y),\ G(v)$ is either $g^{-1} A g$, or $g^{-1} B g$ for some $g \in G$. Now,
if $H \leqslant G$ then by Theorem \ref{induced_split}, $H$ is isomorphic to the fundamental
group of a graph of groups $({\cal H}, X')$, where $X' = H \backslash Y_H$, $Y_H$ is an $H$-invariant
subtree of $Y$, and ${\cal H} = \{ K \cap H \mid K \in {\cal G} \}$. In other words, for each
$e \in E(Y_H),\ H(e) = H \cap G(e)$ is trivial and for each $v \in V(Y_H),\ H(v) = H \cap G(v)$ is
either $H \cap g^{-1} A g$, or $H \cap g^{-1} B g$ for some $g \in G$. It follows that if $H$ is
finitely generated then
$$H \simeq H_1 \ast \cdots \ast H_k \ast F,$$
where each $H_i$ is conjugate into either $A$, or $B$, and $F$ is a finitely generated free group.
This result is known as the Kurosh subgroup theorem (see \cite{Kurosh:1934}.
\end{example}

In some situations it is possible to construct an induced splitting of $H$ effectively. Below is an
algorithmic version of Theorem \ref{induced_split} which can be applied when the splitting of the
ambient group is ``nice''.

\begin{theorem} \cite[Theorem 1.1]{KMW:2005}
\label{th:eff}
Let $({\cal G}, X)$ be a graph of groups such that
\begin{enumerate}
\item $G(v)$ is either locally quasiconvex word-hyperbolic or polycyclic-by-finite for every $v \in
V(X)$,
\item $G(e)$ is polycyclic-by-finite for every $e \in E(X)$.
\end{enumerate}
Then there is an algorithm which, given a finite subset $S \subset G$, constructs the induced
splitting and a finite presentation for $H = \langle S \rangle \leqslant G = \pi({\cal G}, X, T)$, where
$T$ is a maximal subtree of $X$.
\end{theorem}

Recall that the {\em Uniform Membership Problem} if solvable in a group $G$ with a finite generating
set $S$ if there is an algorithm which, for any finite family of words $u, w_1, w_2, \ldots, w_n$
over $S^{\pm1 }$ decides whether or not the element of $G$ represented by $u$ belongs to the subgroup
of $G$ generated by the elements of $G$ corresponding to $w_1, w_2, \ldots, w_n$. The definition does
not depend on the choice of a finite generating set for $G$.

In particular, Theorem \ref{th:eff} implies the following result.

\begin{cor} \cite[Theorem 1.1]{KMW:2005}
\label{le:member}
Let $({\cal G}, X)$ be a graph of groups with the properties listed in Theorem \ref{th:eff}, and
let $G = \pi({\cal G}, X, T)$, where $T$ is a maximal subtree of $X$. Then the uniform membership
problem for $G$ is solvable.
\end{cor}

The proof of Theorem \ref{th:eff} given in \cite{KMW:2005} involves the notion of {\em folding},
which are particular transformations of graphs of groups.

\section{Stallings' pregroups and their universal groups}
\label{se:pregroups}

The notions of a {\em pregroup} and its {\em universal group} were first introduced by J. Stallings
in \cite{Stallings:1971}, but the ideas behind these notions go back to B.~L. van der Waerden
\cite{VanDerWaerden:1948} and R. Baer \cite{Baer:1950}. A pregroup $P$ provides a convenient tool to
introduce reduced forms for elements of $U(P)$ and in some cases gives rise to an integer-valued
length function on $U(P)$ which can be connected with an action of $U(P)$ on a simplicial tree.
Connections between pregroups and free constructions in groups were established by
F. Rimlinger in \cite{Rimlinger:1987, Rimlinger:1987(2)}. Some generalizations of pregroups were
obtained in \cite{Kushner_Lipschutz:1988, Lipschutz:1989, Kushner_Lipschutz:1993}.

\subsection{Definitions and examples}
\label{subse:pregr_defn}

A {\em pregroup} $P$ is a set $P$, with a distinguished element $1$, equipped with a partial
multiplication, that is, a function $D \rightarrow P$, $(x,y) \rightarrow xy$, where $D \subset P
\times P$, and an inversion, that is, a function $^{-1} : P \rightarrow P$, $x \rightarrow x^{-1}$,
satisfying the following axioms (below $xy$ is {\em defined} if $(x,y) \in D$):
\begin{enumerate}
\item[(P1)] for all $u \in P$, the products $u 1$ and $1 u$ are defined and $u 1 = 1 u = u$,
\item[(P2)] for all $u \in P$, the products $u^{-1} u$ and $u u^{-1}$ are defined and $u^{-1} u =
u u^{-1} = 1$,
\item[(P3)] for all $u, v \in P$, if $u v$ is defined, then so is $v^{-1} u^{-1}$ and $(u  v)^{-1}
= v^{-1} u^{-1}$,
\item[(P4)] for all $u, v, w \in P$, if $u v$ and $v w$ are defined, then $(u v) w$ is defined if
and only if $u (v w)$ is defined, in which case
$$ (u v) w =  u (v w),$$
\item[(P5)] for all $u, v, w, z \in P$, if $u v, v w,$ and $w z$ are all defined then either $u v w$,
or $v w z$ is defined.
\end{enumerate}

It was noticed (see \cite{Hoare:1988}) that (P3) follows from (P1), (P2), and (P4), hence, it can
be omitted.

A finite sequence $u_1, \ldots, u_n$ of elements from $P$ is termed a {\em $P$-product} and it is
denoted by $u_1 \cdots u_n$ (one may view it as a word in the alphabet $P$). The {\em $P$-length}
of $u_1 \cdots u_n$ is equal to $n$. A $P$-product $u_1 \cdots u_n$ is called {\em reduced} if for
every $i \in [1,n-1]$ the product $u_i u_{i+1}$ is not defined in $P$.

The following lemma lists some simple implications from the axioms (P1) -- (P5).
\begin{lemma}\cite{Stallings:1971}
\label{le:pregr_properties}
Let $P$ be a pregroup. Then
\begin{enumerate}
\item[(1)] $(x^{-1})^{-1} = x$ for every $x \in P$,
\item[(2)] if $ax$ is defined, then $a^{-1} (ax)$ is defined and $a^{-1} (ax) = x$,
\item[(3)] if $xa$ is defined, then $(xa) a^{-1}$ is defined and $(xa) a^{-1} = x$,
\item[(4)] if $ax$ and $a^{-1} y$ are defined, then $xy$ is defined if and only if $(xa) (a^{-1}y)$
is defined, in which case $xy = (xa)(a^{-1}y)$,
\item[(5)] if $xa$ and $a^{-1}y$ are defined, then $xyz$ is a reduced $P$-product if and only if
$(xa)(a^{-1}y)z$ is reduced; similarly, $zxy$ is reduced if and only if $z(xa)(a^{-1}y)$ is reduced,
\item[(6)] if $xy$ is a reduced $P$-product and if $xa, a^{-1}y, yb$ are defined then $(a^{-1}y)b$
is defined,
\item[(7)] if $xy$ is a reduced $P$-product and $xa, a^{-1}y, (xa)b$ and $b^{-1} (a^{-1}y)$ are
defined then $ab$ is defined.
\end{enumerate}
\end{lemma}

Now, one can define the universal group $U(P)$ of a pregroup $P$ as follows. Consider all reduced
$P$-products. Observe that if a product $u_1 \cdots u_n$ is not reduced, that is, the product $u_i
u_{i+1}$ is defined in $P$ for some $i$ then $u_i u_{i+1} = v \in P$ and one can {\em reduce} $u_1
\cdots u_n$ by replacing the pair $u_i u_{i+1}$ by $v$. Now, given two reduced $P$-products $u_1
\cdots u_n$ and $v_1 \cdots v_m$, we write
$$u_1 \cdots u_n \sim v_1 \cdots v_m$$
if and only if $m = n$ and there exist elements $a_1, \ldots, a_{n-1} \in P$ such that $a_{i-1}^{-1}
u_i a_i$ are defined and $v_i = a_{i-1}^{-1} u_i a_i$ for $i \in [1,n]$ (here $a_0 = a_n = 1$). In
this case we also say that $v_1 \cdots v_m$ can be obtained from $u_1 \cdots u_n$ by {\em interleaving}.
From Lemma \ref{le:pregr_properties} it follows that ``$\sim$'' is an equivalence relation on the set
of all reduced $P$-products. Now, the group $U(P)$ can be described as the set $U(P)$ of
``$\sim$''-equivalence classes of reduced $P$-products, where multiplication is given by concatenation
of representatives and consecutive reduction of the resulting product. Obviously, $P$ embeds into
$U(P)$ via the canonical map $u \rightarrow u$.

\smallskip

Recall that a mapping $\phi: P \rightarrow Q$ of pregroups is a {\em morphism} if for any $x, y \in
P$ whenever $x y$ is defined in $P$, $\phi(x) \phi(y)$ is defined in $Q$ and equal to $\phi(xy)$.
Now the group $U(P)$ can be characterized by the following universal property: there is a morphism
of pregroups $\lambda: P \rightarrow U(P)$, such that, for any morphism  $\phi: P \rightarrow G$ of
$P$ into a group $G$, there is a unique group homomorphism $\psi: U(P) \rightarrow G$ for which
$\psi \lambda  = \phi$. This shows that $U(P)$ is a group with a  generating set $P$ and a set of
relations $x y = z$, where $x, y \in P$, $xy$ is defined  in $P$, and equal to $z$. If the map
$\psi: U(P) \rightarrow G$ above is an isomorphism then we say that $P$ is a {\em pregroup structure}
for $G$.

\smallskip

Given a group $G$ one can try to find a pregroup structure for $G$ as a subset of $G$. In this case
the following lemma helps (this result was used implicitly in \cite{Myasnikov_Remeslennikov_Serbin:2005,
KMRS:2012} to find certain pregroup structures).

\begin{lemma}
\label{le:pregr_P5}
Let $G$ be a group and let $P \subseteq G$ be a generating set for $G$ such that $P^{-1} = P$. Let
$D \subseteq (P, P)$ be such that $(x,y) \in D$ implies $xy \in P$, and assume that multiplication
and inversion are induced on $P$ from $G$. Then $P$ is a pregroup structure for $G$ if $P$ satisfies
(P5).
\end{lemma}
\begin{proof}
The result follows immediately from the fact that $P$ is a subset of $G$, because in this case the
axioms (P1) - (P4) are satisfied automatically for $P$.
\end{proof}

There is another way to check if $P \subseteq G$ is a pregroup structure. Again, we assume that $P$
is a generating set for $G$ such that $P^{-1} = P$, $D \subseteq (P, P)$ is such that $(x,y) \in D$
implies $xy \in P$, and the multiplication and inversion are induced on $P$ from $G$. If all reduced
$P$-products representing the same group element have equal $P$-length then we say that $(P,D)$ is
a {\em reduced word structure} for $G$.

\begin{theorem} \cite{Rimlinger:1987(2)}
\label{th:red_word_str}
If $P$ is a reduced word structure for $G$ then $P$ is a pregroup structure for $G$.
\end{theorem}

The principal examples of pregroups and their universal groups are shown below.

\begin{example}
\label{ex:pregr_0}
For any group $G$ define $P = G,\ D = (P, P)$, and let the multiplication and inversion on $P$ be
induced from $G$. Then $U(P) \simeq G$.
\end{example}

\begin{example}
\label{ex:pregr_1}
Let $X$ be a set. Define $P = X \cup \overline{X} \cup \{1\}$, where $\overline{X} = \{\overline{x}
| x \in X\}$ and $1 \notin X$. Define the inversion function $^{-1}: P \to P$ as follows: $x^{-1} =
\overline{x},\ \overline{x}^{-1} = x$ for every $x \in X$ and the corresponding $\overline{x} \in
\overline{X}$, and $1^{-1} = 1$. Without loss of generality we identify $\overline{X}$ with $X^{-1}$,
the image of $X$ under $^{-1}$. Next, $(x, y) \in D$ if either $y = x^{-1}$ (hence, $xy = 1 \in P$),
or either $x = 1$ (hence, $xy = y \in P$), or $y = 1$ (so, $xy = x \in P$). It is easy to check that
$P$ with the inversion and the set $D$ is a pregroup, and $U(P) \simeq F(X)$, a free group on $X$.
The pregroup $P$ is called the {\em free pregroup on $X$}.
\end{example}

\begin{example}
\label{ex:pregr_2} \cite{Stallings:1971}
Let $A, B$ and $C$ be groups, and $\phi : C \to A,\ \psi : C \to B$ monomorphisms. Identify $\phi(C)$
with $\psi(C)$, then $A \cap B = C$. Let $P = A \cup B$. The identity $1$ and inversion function are
obvious. For, $x,y \in P$, the product $xy$ is defined only if $x$ and $y$ both belong either to $A$,
or to $B$. One can verify that $P$ is a pregroup and $U(P) \simeq A \ast_C B$.
\end{example}

\begin{example}
\label{ex:pregr_3} \cite{Stallings:1971}
Consider $A \ast_C B$. Let $P$ be the subset of all elements that can be written as the product $b
a b'$ for some $b, b' \in B,\ a \in A$. In particular, $P$ contains $A$ and $B$.  For, $x,y \in P$,
the product $xy$ is defined if $xy \in P$. Using the structure of $A \ast_C B$, one can prove that
$P$ is a pregroup. The universal group $U(P)$ is isomorphic to $A \ast_C B$, but observe that $P$
is not the same as in Example \ref{ex:pregr_2}.
\end{example}

\begin{example}
\label{ex:pregr_4} \cite{Stallings:1971}
Let $G$ be a group, $H$ a subgroup of $G$, and $\phi : H \to G$ a monomorphism. For $t \notin G$
construct four sets in one-to-one correspondence with $G$:
$$G,\ t^{-1} G,\ G t, t^{-1} G t.$$
Identify $h \in H \subseteq G$, with $t^{-1} \phi(h) t \in t^{-1} G t$. The multiplication is
naturally defined between $G$ and $G,\ G$ and $G t,\ t^{-1} G$ and $G,\ t^{-1} G$ and $G t,\ G t$
and $t^{-1} G,\ G t$ and $t^{-1} G t,\ t^{-1} G t$ and $t^{-1} G,\ t^{-1} G t$ and $t^{-1} G t$, by
cancelation of $t t^{-1}$ and multiplication in $G$. By the formulas $h t^{-1} = t^{-1} \phi(h)$
and $t h = \phi(h) t$, the multiplication is defined in all cases when one factor belongs to $H$.
Hence,
$$P = G \cup t^{-1} G \cup G t \cup t^{-1} G t$$
is a pregroup and $U(P) \simeq \langle G, t | t^{-1} h t = \phi(h),\ h \in H \rangle$.
\end{example}

From the examples above it follows that a group which splits into an amalgamated free product or an
HNN-extension has a non-trivial pregroup structure. The same holds in general: it was proved in
\cite[Theorem B]{Rimlinger:1987} that a group isomorphic to the fundamental group of a finite graph
of groups has a pregroup structure which arises from the graph of groups.
The converse, namely, that the universal group of a pregroup $P$ is isomorphic to the fundamental
group of a graph of groups also holds provided $P$ is of {\em finite height}. This property is
explained below.

For $x, y \in P$ we write $x \preceq y$ if and only if for any $z \in P$, $zx$ is defined whenever
$zy$ is defined. The relation ``$\preceq$'' is a {\em tree ordering} on $P$ (see \cite{Stallings:1971}),
that is, there exists a smallest element $1$ and
$$\forall\ x,y,z \in P: (x \preceq z\ {\rm and}\ y \preceq z)\ \Rightarrow\ (x \preceq y\
{\rm or}\ y \preceq x)$$
Elements $x, y \in P$ are comparable if either $x \preceq y$, or $y \preceq x$, or both. If both
$x \preceq y$ and $y \preceq x$ then we write $x \approx y$, and if $x \preceq y$ but not $y \preceq
x$ then we write $x \prec y$.

\begin{lemma}\cite[Lemma I.2.6]{Rimlinger:1987}
\label{le:pregr_order}
$x, y \in P$ are comparable if and only if $x^{-1} y \in P$.
\end{lemma}

An element $x \in P$ is of finite height $n \in \mathbb{N}$ if there exist $x_0, x_1, \ldots, x_n
\in P$ such that $1 = x_0 \prec x_1 \prec \cdots \prec x_n = x$ and for each $0 \leqslant i
\leqslant n - 1$, if $z \in P$ and $x_i \preceq z \preceq x_{i+1}$ then $z \approx x_i$ and $z
\approx x_{i+1}$. $P$ is of finite height if there exists a natural number $N$ which bounds heights
of all elements of $P$.

\smallskip

Now, if $P$ is of finite height then $U(P)$ is isomorphic to the fundamental group of a graph of
groups whose vertex and edge groups can be obtained as the universal groups of certain subpregroups
of $P$ (see \cite[Theorem A]{Rimlinger:1987}).

\subsection{Connection with length functions}
\label{subse:pregr_length}

Following the notation from \cite{Chiswell:1987} for a pregroup $P$ define
$$B = \{a \in P \mid za\ {\rm and}\ az\ {\rm are\ defined\ for\ all}\ z \in P\}.$$
Obviously, $B$ is a subgroup of $P$. Furthermore, if a  reduced $P$-product  contains an element
from $B$ then it  consists of a single element.

Below we use the following notation:
\begin{enumerate}
\item[(1)] if $x, y \in P$ then we write $x y = x \circ y$ if $xy$ is not defined,
\item[(2)] if a $P$-product $u = u_1 \cdots u_n$ is reduced then we put $|u| = n$. Notice, that the
function $u \to |u|$ induces a well-defined function on $U(P)$.
\end{enumerate}

\begin{theorem} \cite{Chiswell:1987}
Let $P$ be a pregroup and let $| \cdot |: U(P) \to \mathbb{Z}$ be defined by $u \to |u|$ for each
$u \in U(P)$. Then ``$| \cdot |$'' is a Lyndon length function (see Section \ref{sec:length_func})
if and only if $P$ satisfies an additional axiom (P6):
\begin{enumerate}
\item[(P6)] for any $x,y \in P$, if $x y$ is not defined but $x a$ and $a^{-1}y$ are both defined
for some $a \in P$ then $a \in B$.
\end{enumerate}
\end{theorem}

It is known (Theorem 2.7 in \cite{Nesayef:1983}) that the axiom (P6) is equivalent in a pregroup
$P$ to the following one:
\begin{enumerate}
\label{pr:P6}
\item[(P6')] for any $x, y \in P$, if $x y$ is not defined and $(a x) y$ is defined for some $a \in
P$ then $a x \in B$.
\end{enumerate}

\begin{remark}
Suppose $P$ satisfies (P6). If a reduced $P$-product $v_1 \cdots v_n$ is obtained from $u_1 \cdots
u_n$ by interleaving $v_i = a_{i-1}^{-1} u_i a_i$ for $i \in [1,n]$, where $a_0 = a_n = 1$ then
$a_1, \ldots, a_{n-1} \in B$.
\end{remark}

More on the connection of pregroups with length functions and Bass-Serre theory can be found in
\cite{Chiswell:1987}, \cite{Hoare:1988} and \cite{Rimlinger:1987}.

\section{$\Lambda$-trees}
\label{sec:lambda}

The theory of $\Lambda$-trees (where $\Lambda = \mathbb{R}$) has its origins in the papers by
I. Chiswell \cite{Chiswell:1976} and J. Tits \cite{Tits:1977}. The first paper contains a construction
of a $\mathbb{Z}$-tree starting from a Lyndon length function on a group (see Section
\ref{sec:length_func}), an idea considered earlier by R. Lyndon in \cite{Lyndon:1963}.

\smallskip

Later, in their very influential paper \cite{Morgan_Shalen:1991} J. Morgan and P. Shalen linked group
actions on $\mathbb{R}$-trees with topology and generalized parts of Thurston's Geometrization
Theorem. Next, they introduced $\Lambda$-trees for an arbitrary ordered abelian group $\Lambda$ and
the general form of Chiswell's construction. Thus, it became clear that abstract length functions with
values in $\Lambda$ and group actions on $\Lambda$-trees are just two equivalent approaches to the
same realm of group theory questions (more on this equivalence can be found in Section
\ref{sec:equiv}). The unified theory was further developed in the important paper by R. Alperin and
H. Bass \cite{Alperin_Bass:1987}, where authors state a fundamental problem in the theory of group
actions on $\Lambda$-trees: find the group theoretic information carried by a $\Lambda$-tree action
(analogous to Bass-Serre theory), in particular, describe finitely generated groups acting freely
on $\Lambda$-trees ($\Lambda$-free groups).

\smallskip

Here we introduce basics of the theory of $\Lambda$-trees, which can be found in more detail in
\cite{Alperin_Bass:1987} and \cite{Chiswell:2001}.

\subsection{Ordered abelian groups}
\label{subs:ord_ab}

In this section some well-known results on ordered abelian groups are collected. For proofs and
details we refer to the books \cite{Glass:1999} and \cite{Kopytov_Medvedev:1996}.

A set $A$ equipped with addition ``$+$'' and a partial order ``$\leqslant$'' is called a {\em
partially ordered} abelian group if:
\begin{enumerate}
\item[(1)] $\langle A, + \rangle$ is an abelian group,
\item[(2)] $\langle A, \leqslant \rangle$ is a partially ordered set,
\item[(3)] for all $a,b,c \in A,\ a \leqslant b$ implies $a + c \leqslant b + c$.
\end{enumerate}

An abelian group $A$ is called {\em orderable} if there exists a linear order ``$\leqslant$'' on $A$,
satisfying the condition (3) above. In general, the ordering on $A$ is not unique.

Let $A$ and $B$ be ordered abelian groups. Then the direct sum  $A \oplus B$ is orderable with
respect to the {\em right lexicographic order}, defined as follows:
$$(a_1,b_1) < (a_2,b_2) \Leftrightarrow b_1 < b_2 \ \mbox{or} \ b_1 = b_2 \ \mbox{and} \ a_1 < a_2.$$

Similarly, one can define the  right lexicographic order on finite direct sums of ordered abelian
groups or even on infinite direct sums if the set of indices is linearly ordered.

For elements $a,b$ of an ordered group $A$  the {\em closed segment} $[a,b]$ is defined by
$$[a,b] = \{c \in A \mid a \leqslant c \leqslant b\}.$$

A  subset $C \subset A$ is called {\em convex}, if for every $a, b \in C$ the set $C$ contains
$[a,b]$. In particular, a subgroup $B$ of $A$ is convex if $[0,b] \subset B$ for every positive $b
\in B$. In this event, the quotient $A / B$ is an ordered abelian group with respect to the order
induced from $A$.

A group $A$ is called {\em archimedean} if it has no non-trivial proper convex subgroups. It is
known  that $A$ is archimedean if and only if $A$ can be embedded into the ordered abelian group
of real numbers $\mathbb{R}_+$, or equivalently, for any $0 < a \in A$ and any $b \in A$ there
exists an integer $n$ such that $na > b$.

It is not hard to see that the set of convex subgroups of an ordered abelian group $A$ is linearly
ordered by inclusion (see, for example, \cite{Glass:1999}), it is called  {\em the complete chain of
convex subgroups} in $A$. Notice that
$$E_n = \{f(t) \in \Zt \mid {\rm deg}(f(t)) \leqslant n\}$$
is a convex subgroup of $\Zt$ (here ${\rm deg}(f(t))$ is the degree of $f(t)$) and
$$0 < E_0 < E_1 < \cdots < E_n < \cdots$$
is the complete chain of convex subgroups of $\Zt$.

If $A$ is finitely generated  then the complete chain of convex subgroups of $A$
$$0 = A_0 < A_1 < \cdots < A_n = A$$
is finite. The following result (see, for example, \cite{Chiswell:2001}) shows that this chain completely
determines the order on $A$, as well as the structure of $A$. Namely, the groups  $A_i / A_{i-1}$
are archimedean (with respect to the induced order) and $A$ is isomorphic (as an ordered group) to
the direct sum
\begin{equation}
\label{eq:order-convex}
A_1 \oplus A_2 / A_1 \oplus \cdots \oplus A_n / A_{n-1}
\end{equation}
with the right lexicographic order.

An ordered abelian group $A$ is called {\em discretely ordered} if $A$ has a non-trivial minimal
positive element (we denote it by $1_A$). In this event, for any $a \in A$ the following hold:

\begin{enumerate}
\item[(1)] $a + 1_A = \min\{b \mid b > a\}$,
\item[(2)] $a - 1_A = \max\{b \mid b < a\}$.
\end{enumerate}

For example, $A = \mathbb{Z}^n$ with the right lexicographic order is discretely ordered with
$1_{\mathbb{Z}^n} = (1, 0, \ldots, 0)$. The additive group of integer polynomials $\Zt$
is discretely ordered with $1_{\Zt} = 1$.

\begin{lemma} \cite{Myasnikov_Remeslennikov_Serbin:2005}
\label{le:discrete}
A finitely generated discretely ordered archimedean abelian group is infinite cyclic.
\end{lemma}

Recall that an ordered abelian group $A$ is {\em hereditary discrete} if for any convex subgroup
$E \leqslant A$ the quotient $A / E$ is discrete with respect to the induced order.

\begin{cor} \cite{Myasnikov_Remeslennikov_Serbin:2005}
\label{co:discrete}
Let $A$ be a finitely generated hereditary discrete ordered abelian group. Then $A$ is isomorphic
to the direct product of finitely many copies of $\mathbb{Z}$ with the lexicographic order.
\end{cor}

\subsection{$\Lambda$-metric spaces}
\label{subs:lambda-def}

Let $X$ be a non-empty set, $\Lambda$ an ordered abelian group. A {\em $\Lambda$-metric on $X$} is
a mapping $d: X \times X \longrightarrow \Lambda$ such that for all $x, y, z \in X$:
\begin{enumerate}
\item[(M1)] $d(x,y) \geqslant 0$,
\item[(M2)] $d(x,y) = 0$ if and only if $x = y$,
\item[(M3)] $d(x,y) = d(y,x)$,
\item[(M4)] $d(x,y) \leqslant d(x,z) + d(y,z)$.
\end{enumerate}

So a {\em $\Lambda$-metric space} is a pair $(X,d)$, where $X$ is a non-empty set and $d$ is a
$\Lambda$-metric on $X$. If $(X,d)$ and $(X',d')$ are $\Lambda$-metric spaces, an {\em isometry}
from $(X,d)$ to $(X',d')$ is a mapping $f: X \rightarrow X'$ such that $d(x,y) = d'(f(x),f(y))$ for
all $x, y \in X$.

A {\em segment} in a $\Lambda$-metric space is the image of an isometry $\alpha: [a,b]_\Lambda
\rightarrow X$ for some $a, b \in \Lambda$ and $[a,b]_\Lambda$ is a segment in $\Lambda$. The
endpoints of the segment are $\alpha(a), \alpha(b)$.

We call a $\Lambda$-metric space $(X,d)$ {\em geodesic} if for all $x, y \in X$, there is a segment
in $X$ with endpoints $x, y$ and $(X,d)$ is {\em geodesically linear} if for all $x, y \in X$,
there is a unique segment in $X$ whose set of endpoints is $\{x, y\}$.

It is not hard to see, for example, that $(\Lambda, d)$ is a geodesically linear $\Lambda$-metric
space, where $d(a,b) = |a - b|$, and the segment with endpoints $a,b$ is $[a,b]_\Lambda$.

Let $(X,d)$ be a $\Lambda$-metric space. Choose a point $v \in X$, and for $x,y \in X$, define
$$(x \cdot y)_v = \frac{1}{2} (d(x,v) + d(y,v) - d(x,y)).$$

Observe, that in general $(x \cdot y)_v \in \frac{1}{2} \Lambda$.

The following simple result follows immediately

\begin{lemma} \cite{Chiswell:2001}
If $(X,d)$ is a $\Lambda$-metric space then the following are equivalent:
\begin{enumerate}
\item for some $v \in X$ and all $x, y \in X, (x \cdot y)_v \in \Lambda$,
\item for all $v, x, y \in X, (x \cdot y)_v \in \Lambda$.
\end{enumerate}
\end{lemma}

Let $\delta \in \Lambda$ with $\delta \geqslant 0$. Then $(X,p)$ is {\em $\delta$-hyperbolic with
respect to $v$} if, for all $x, y, z \in X$,
$$(x \cdot y)_v \geqslant min \{(x \cdot z)_v, (z \cdot y)_v\} - \delta.$$

\begin{lemma}
\cite{Chiswell:2001}
If $(X, d)$ is $\delta$-hyperbolic with respect to $v$, and $t$ is any other point of $X$, then
$(X, d)$ is $2 \delta$-hyperbolic with respect to $t$.
\end{lemma}

A {\em $\Lambda$-tree} is a $\Lambda$-metric space $(X,d)$ such that:
\begin{enumerate}
\item[(T1)] $(X,d)$ is geodesic,
\item[(T2)] if two segments of $(X,d)$ intersect in a single point, which is an endpoint of both,
then their union is a segment,
\item[(T3)] the intersection of two segments with a common endpoint is also a segment.
\end{enumerate}

\begin{example}
$\Lambda$ together with the usual metric $d(a,b) = |a - b|$ is a $\Lambda$-tree. Moreover, any
convex set of $\Lambda$ is a $\Lambda$-tree.
\end{example}

\begin{example}
\label{examp:1}
A $\mathbb{Z}$-metric space $(X,d)$ is a $\mathbb{Z}$-tree if and only if there is a simplicial
tree $\Gamma$ such that $X = V(\Gamma)$ and $p$ is the path metric of $\Gamma$.
\end{example}

Observe  that in general a $\Lambda$-tree can not be viewed as a simplicial tree with the path
metric like in Example \ref{examp:1}.

\begin{lemma}
\cite{Chiswell:2001}
Let $(X,d)$ be $\Lambda$-tree. Then $(X,d)$ is $0$-hyperbolic, and for all $x, y, v \in X$ we have
$(x \cdot y)_v \in \Lambda$.
\end{lemma}

\smallskip

Eventually, we say that a group $G$ acts on a $\Lambda$-tree $X$ if any element $g \in G$ defines an
isometry $g : X \rightarrow X$. An action on $X$ is {\em non-trivial} if there is no point in $X$
fixed by all elements of $G$. Note, that every group has a {\em trivial action} on any $\Lambda$-tree,
when all group elements act as identity. An action of $G$ on $X$ is {\em minimal} if $X$ does not
contain a non-trivial $G$-invariant subtree $X_0$.

Let a group $G$ act as isometries on a $\Lambda$-tree $X$. $g \in G$ is called {\em elliptic} if
it has a fixed point. $g \in G$ is called an {\rm inversion} if it does not have a fixed point, but
$g^2$ does. If $g$ is not elliptic and not an inversion then it is called {\em hyperbolic}.

A group $G$  acts {\em freely} and {\em without inversions} on a $\Lambda$-tree $X$ if for all $1
\neq g \in G$, $g$ acts as a hyperbolic isometry. In this case we also say that $G$ is
{\em $\Lambda$-free}.

\subsection{Theory of a single isometry}
\label{subs:single_isom}

Let $(X,d)$ be a $\Lambda$-tree, where $\Lambda$ is an arbitrary ordered abelian group. Recall that
an isometry $g$ of $X$ is called {\em elliptic} if it has a fixed point.

\begin{lemma} \cite[Lemma 3.1.1]{Chiswell:2001}
\label{le:elliptic}
Let $g$ be an elliptic isometry of $X$ and let $X^g$ denote the set of fixed points of $g$. Then
$X^g$ is a closed non-empty $\langle g \rangle$-invariant subtree of $X$. If $x \in X$ and $[x,p]$
is the bridge between $x$ and $X^g$, then for any $a \in X^g,\ p = Y(x,a,g a)$ is the midpoint of
$[x, gx]$.
\end{lemma}

Next, an isometry $g$ of a $\Lambda$-tree $(X,d)$ is called an {\em inversion} if $g^2$ has a fixed
point, but $g$ does not.

\begin{lemma} \cite[Lemma 3.1.2]{Chiswell:2001}
\label{le:inversion}
Let $g$ be an isometry of a $\Lambda$-tree $(X,d)$. The following are equivalent:
\begin{enumerate}
\item $g$ is an inversion,
\item there is a segment of $X$ invariant under $g$, and the restriction of $g$ to this segment has
no fixed points,
\item there is a segment $[x,y]$ in $X$ such that $gx = y,\ gy = x$ and $d(x,y) \notin 2\Lambda$,
\item $g^2$ has a fixed point, and for all $x \in X,\ d(x,gx) \notin 2 \Lambda$.
\end{enumerate}
\end{lemma}

An isometry $g$ of a $\Lambda$-tree $(X,d)$ is called {\em hyperbolic} if it is not an inversion
and not elliptic. It follows that an isometry $g$ is hyperbolic if and only if $g^2$ has no fixed
point.

\begin{figure}[h]
\centering{\mbox{\psfig{figure=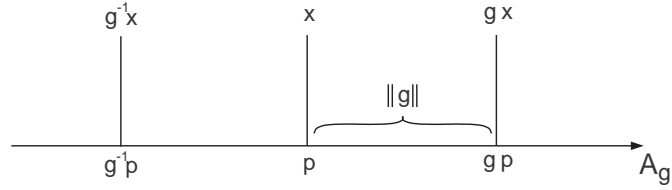,height=1.3in}}}
\caption{The characteristic set of $g$}
\label{axis}
\end{figure}

Suppose $g$ is an isometry of a $\Lambda$-tree $(X,d)$. The {\em characteristic set of $g$} is the
subset $A_g \subseteq X$ defined by
$$A_g = \{p \in X \mid [g^{-1}p,p] \cap [p,gp] = \{p\}\}.$$
By Lemma \ref{le:elliptic}, if $g$ is elliptic then $A_g = X^g$, and by Lemma \ref{le:inversion},
if $g$ is an inversion then $A_g = \emptyset$.

\begin{theorem} \cite[Theorem 3.1.4]{Chiswell:2001}
\label{le:hyperbolic}
Let $g$ be a hyperbolic isometry of a $\Lambda$-tree $(X,d)$. Then $A_g$ is a non-empty closed
$\langle g \rangle$-invariant subtree of $X$. Further, $A_g$ is a linear tree, and $g$ restricted
to $A_g$ is equivalent to a translation $a \rightarrow a + \|g\|$ for some $\|g\| \in \Lambda$ with
$\|g\| > 0$. If $x \in X$ and $[x,p]$ is the bridge between $x$ and $A_g$, then $p = Y(g^{-1} x, x,
gx),\ [x,gx] \cap A_g = [p,gp],\ [x,gx] = [x,p,gp,gx]$ and $d(x,gx) = \|g\| + 2d(x,p)$.
\end{theorem}

If $g$ is hyperbolic then $A_g$ is called the {\em axis} and $\|g\|$ the {\em translation length} of
$g$ which can be defined as follows
\[
\|g\| = \left\{ \begin{array}{ll}
\mbox{$\min\{d(x, g x) \mid x \in X\}$}  & \mbox{if $g$ is not an inversion} \\
\mbox{$0$ } & \mbox{otherwise}
\end{array}
\right.
\]
It can be shown that this minimum is always realized. If $g$ is elliptic or an inversion, then $\|g\|
= 0$.

\begin{cor} \cite[Corollary 3.1.5]{Chiswell:2001}
\label{co:axis}
Let $g$ be an isometry of a $\Lambda$-tree $(X,d)$. Then $g$ is an inversion if and only if $A_g =
\emptyset$. If $g$ is not an inversion then $\|g\| = \min \{ d(x,gx) \mid x \in X\}$ and $A_g = \{p
\in X \mid d(p,gp) = \|g\|\}$.
\end{cor}

\begin{cor} \cite[Corollary 3.1.6]{Chiswell:2001}
\label{co:axis_2}
Let $g$ be an isometry of a $\Lambda$-tree $(X,d)$ which is not an inversion. Then $A_g$ meets every
$\langle g \rangle$-invariant subtree of $X$, and $A_g$ is contained in every $\langle g
\rangle$-invariant subtree of $X$ with the property that it meets every $\langle g \rangle$-invariant
subtree of $X$.
\end{cor}

\begin{lemma} \cite[Lemma 3.1.7]{Chiswell:2001}
\label{le:axis}
If $g,h$ are both isometries of a $\Lambda$-tree $(X,d)$, then
\begin{enumerate}
\item $A_{hgh^{-1}} = h A_g$ and $\|hgh^{-1}\| = \|g\|$,
\item $A_{g^{-1}} = A_g$,
\item if $n \in \mathbb{Z}$ then $\|g^n\| = |n| \|g\|$, and $A_g \subseteq A_{g^n}$. If $\|g\| > 0$ and
$n \neq 0$ then $A_g = A_{g^n}$,
\item if $Y$ is a $\langle g \rangle$-invariant subtree of $X$, then $\|g\| = \|g \mid_Y\|$ and $A_{g
\mid_Y} = A_g \cap Y$.
\end{enumerate}
\end{lemma}

\subsection{$\Lambda$-free groups}
\label{subs:lambda-free}

Recall that a group $G$ is called {\em $\Lambda$-free} if for all $1 \neq g \in G$, $g$ acts as a
hyperbolic isometry. Here we list some known results about $\Lambda$-free groups for an arbitrary
ordered abelian group $\Lambda$. For all these results the reader can be referred to
\cite{Alperin_Bass:1987, Bass:1991, Chiswell:2001, Martino_Rourke:2005}

\begin{theorem}
\begin{enumerate}
\item[(a)] The class of $\Lambda$-free groups is closed under taking subgroups.
\item[(b)] If $G$ is $\Lambda$-free and $\Lambda$ embeds (as an ordered abelian group) in $\Lambda'$
then $G$ is $\Lambda'$-free.
\item[(c)] Any $\Lambda$-free group is torsion-free.
\item[(d)] $\Lambda$-free groups have the CSA property. That is, every maximal abelian subgroup $A$
is malnormal: $A^g \cap A = 1$ for all $g \notin A$.
\item[(e)] Commutativity is a transitive relation on the set of non-trivial elements of a
$\Lambda$-free group.
\item[(f)] Solvable subgroups of $\Lambda$-free groups are abelian.
\item[(g)] If $G$ is $\Lambda$-free then any abelian subgroup of $G$ can be embedded in $\Lambda$.
\item[(h)] $\Lambda$-free groups cannot contain Baumslag-Solitar subgroups other than $\mathbb{Z}
\times \mathbb{Z}$. That is, no group of the form $\langle a, t \mid t^{-1} a^p t = a^q \rangle$
can be a subgroup of a $\Lambda$-free group unless $p = q = \pm 1$.
\item[(i)] Any two generator subgroup of a $\Lambda$-free group is either free, or free abelian.
\item[(j)] The class of $\Lambda$-free groups is closed under taking free products.
\end{enumerate}
\end{theorem}

The following result was originally proved in \cite{Harrison:1972} in the case of finitely many
factors and $\Lambda = \mathbb{R}$. A proof of the result in the general formulation given below
can be found in \cite[Proposition 5.1.1]{Chiswell:2001}.

\begin{theorem}
If $\{G_i \mid i \in I \}$ is a collection of $\Lambda$-free groups then the free product $\ast_{i
\in I} G_i$ is $\Lambda$-free.
\end{theorem}

The following result gives a lot of information about the group structure in the case when $\Lambda
= \mathbb{Z} \times \Lambda_0$ with the left lexicographic order.

\begin{theorem} \cite[Theorem 4.9]{Bass:1991}
Let a group $G$ act freely and without inversions on a $\Lambda$-tree, where $\Lambda = \mathbb{Z}
\times \Lambda_0$. Then there is a graph of groups $(\Gamma, Y^*)$ such that:
\begin{enumerate}
\item[(1)] $G = \pi_1(\Gamma, Y^*)$,
\item[(2)] for every vertex  $x^* \in Y^*$, a vertex group $\Gamma_{x^*}$ acts freely and without
inversions on a $\Lambda_0$-tree,
\item[(3)] for every edge  $e \in Y^*$ with an endpoint $x^*$ an edge group $\Gamma_e$ is either
maximal abelian subgroup in $\Gamma_{x^*}$ or is trivial and $\Gamma_{x^*}$ is not abelian,
\item[(4)] if $e_1,e_2,e_3 \in Y^*$  are edges with an endpoint $x^*$ then $\Gamma_{e_1},
\Gamma_{e_2}, \Gamma_{e_3}$ are not all conjugate in $\Gamma_{x^*}$.
\end{enumerate}

Conversely, from the existence of a graph $(\Gamma, Y^*)$ satisfying conditions (1)--(4) it follows
that $G$ acts freely and without inversions on a $\mathbb{Z} \times \Lambda_0$-tree in the
following cases: $Y^*$ is a tree, $\Lambda_0 \subset Q$ and either $\Lambda_0 = Q$ or $Y^*$ is finite.
\end{theorem}

\subsection{$\mathbb{R}$-trees}
\label{sec:r-tres}

The case when $\Lambda = \mathbb{R}$ in the theory of groups acting on $\Lambda$-trees appears to be
the most well-studied (other than $\Lambda = \mathbb{Z}$, of course). $\mathbb{R}$-trees are usual
metric spaces with nice properties which makes them very attractive from geometric point of view.
Lots of results were obtained in the last two decades about group actions on these objects. The most
celebrated one is Rips' Theorem about free actions and a more general result of M. Bestvina and
M. Feighn about stable actions on $\mathbb{R}$-trees (see \cite{GLP:1994, Bestvina_Feighn:1995}).
In particular, the main result of Bestvina and Feighn together with the idea of obtaining a stable
action on an $\mathbb{R}$-tree as a limit of actions on an infinite sequence of $\mathbb{Z}$-trees
gives a very powerful tool in obtaining non-trivial splittings of groups into fundamental groups of
graphs of groups which is known as {\em Rips machine}.

\smallskip

An $\mathbb{R}$-tree $(X,d)$ is a $\Lambda$-metric space which satisfies the axioms (T1) -- (T3)
listed in Subsection \ref{subs:lambda-def} for $\Lambda = \mathbb{R}$ with usual order. Hence, all
the definitions and notions given in Section \ref{sec:lambda} hold for $\mathbb{R}$-trees.

The definition of an $\mathbb{R}$-tree was first given by Tits in \cite{Tits:1977}.

\begin{prop} \cite[Proposition 2.2.3]{Chiswell:2001}
\label{pr:eq_def}
Let $(X,d)$ be an $\mathbb{R}$-metric space. Then the following are equivalent:
\begin{enumerate}
\item $(X,d)$ is an $\mathbb{R}$-tree,
\item given two point of $X$, there is a unique arc having them as endpoints, and it is a segment,
\item $(X,d)$ is geodesic and it contains no subspace homeomorphic to the circle.
\end{enumerate}
\end{prop}

\begin{example}
\label{examp:r_tree1}
Let $Y = \mathbb{R}^2$ be the plane, but with metric $p$ defined by
\[
p((x_1,y_1),(x_2,y_2)) = \left\{ \begin{array}{ll}
\mbox{$|y_1| + |y_2| + |x_1 - x_2|$}  & \mbox{if $x_1 \neq x_2$} \\
\mbox{$|y_1 - y_2|$ } & \mbox{if $x_1 = x_2$}
\end{array}
\right.
\]
That is, to measure the distance between two points not on the same vertical line, we take their
projections onto the horizontal axis, and add their distances to these projections and the distance
between the projections (distance in the usual Euclidean sense).
\end{example}

\begin{example} \cite[Proposition 2.2.5]{Chiswell:2001}
\label{examp:r_tree2}
Given a simplicial tree $\Gamma$, one can construct its realization $real(\Gamma)$ by identifying
each non-oriented edge of $\Gamma$ with the unit interval. The metric on $real(\Gamma)$ is induced
from $\Gamma$.
\end{example}

\begin{example}
\label{exam:asympt_cone}
Let $G$ be a $\delta$-hyperbolic group. Then its Cayley graph with respect to any finite generating
set $S$ is a $\delta$-hyperbolic metric space $(X, d)$ (where $d$ is a word metric) on which $G$ acts
by isometries. Now, the asymptotic cone $Cone_\omega(X)$ of $G$ is a real tree (see \cite{Gr,
VanDerDries_Wilkie:1984,Drutu_Sapir:2005}) on which $G$ acts by isometries.
\end{example}

An $\mathbb{R}$-tree is called {\em polyhedral} if the set of all branch points and endpoints is
closed and discrete. Polyhedral $\mathbb{R}$-trees have strong connection with simplicial trees as
shown below.

\begin{theorem} \cite[Theorem 2.2.10]{Chiswell:2001}
\label{th:polyhedral}
An $\mathbb{R}$-tree $(X,d)$ is polyhedral if and only if it is homeomorphic to $real(\Gamma)$ (with
metric topology) for some simplicial tree $\Gamma$.
\end{theorem}

Now we briefly recall some known results related to group actions on $\mathbb{R}$-trees. The first
result shows that an action on a $\Lambda$-tree always implies an action on an $\mathbb{R}$-tree.

\begin{theorem} \cite[Theorem 4.1.2]{Chiswell:2001}
\label{th:lambda_r}
If a finitely generated group $G$ has a non-trivial action on a $\Lambda$-tree for some ordered
abelian group $\Lambda$ then it has a non-trivial action on some $\mathbb{R}$-tree.
\end{theorem}

Observe that in general nice properties of the action on a $\Lambda$-tree are not preserved when
passing to the corresponding action on an $\mathbb{R}$-tree above.

Next result was one of the first in the theory of group actions on $\mathbb{R}$-trees. Later it was
generalized to the case of an arbitrary $\Lambda$ in \cite{Urbanski_Zamboni:1993, Chiswell:1994}.

\begin{theorem} \cite{Harrison:1972}
Let $G$ be a group acting freely and  without inversions on an $\mathbb{R}$-tree $X$, and suppose
$g, h \in G\ \backslash \ \{1\}$. Then $\langle g, h \rangle$ is either free of rank two or abelian.
\end{theorem}

It is not hard to define an action of a free abelian group on an $\mathbb{R}$-tree.

\begin{example}
\label{exam:ab_act}
Let $A = \langle a,b \rangle$ be a free abelian group. Define an action of $A$ on $\mathbb{R}$
(which is an $\mathbb{R}$-tree) by embedding $A$ into $Isom(\mathbb{R})$ as follows
$$a \rightarrow t_1,\ \ b \rightarrow t_{\sqrt{2}},$$
where $t_\alpha(x) = x + \alpha$ is a translation. It is easy to see that
$$a^n b^m \rightarrow t_{n+m\sqrt{2}},$$
and since $1$ and $\sqrt{2}$ are rationally independent it follows that the action is free.
\end{example}

The following result was very important in the direction of classifying finitely generated
$\mathbb{R}$-free groups.

\begin{theorem} \cite{Morgan_Shalen:1991}
\label{th:Morgan_Shalen}
The fundamental group of a closed surface is $\mathbb{R}$-free, except for the non-orientable
surfaces of genus $1,2$ and $3$.
\end{theorem}

Then, in 1991 E. Rips completely classified finitely generated $\mathbb{R}$-free groups. The ideas
outlined by Rips were further developed by Gaboriau, Levitt and Paulin who gave a complete proof
of this classification in \cite{GLP:1994}.

\begin{theorem}[Rips' Theorem]
\label{th:Rips_0}
Let $G$ be a finitely generated group acting freely and without inversions on an $\mathbb{R}$-tree.
Then $G$ can be written as a free product $G = G_1 \ast \cdots \ast G_n$ for some integer $n
\geqslant 1$, where each $G_i$ is either a finitely generated free abelian group, or the
fundamental group of a closed surface.
\end{theorem}

\subsection{Rips-Bestvina-Feighn machine}
\label{sec:rbf}

Suppose $G$ is a finitely presented group acting isometrically on an $\R$-tree $\Gamma$. We assume
the action to be non-trivial and minimal. Since $G$ is finitely presented there is a finite simplicial
complex $K$ of dimension at most $2$ such that $\pi(K) \simeq G$. Moreover, one can assume that
$K$ is a {\em band complex} with underlying {\em union of bands} which is a finite simplicial $\R$-tree $X$ with finitely many {\em bands}
of the type $[0,1] \times \alpha$, where $\alpha$ is an arc of the real line, glued to $X$ so that
$\{0\} \times \alpha$ and $\{1\} \times \alpha$ are identified with sub-arcs of edges of $X$.
Following \cite{Bestvina_Feighn:1995} (the construction originally appears in \cite{Morgan_Shalen:1988})
one can construct a transversely measured lamination $L$ on $K$ and an equivariant map $\phi : \widetilde{K} \to
\Gamma$, where $\widetilde{K}$ is the universal cover of $K$, which sends leaves of the induced
lamination on $\widetilde{K}$ to points in $\Gamma$. The  complex $K$ together with the
lamination $L$ is called a band complex with $\widetilde{K}$ resolving the action of $G$ on
$\Gamma$.

Now, {\em Rips-Bestvina-Feighn machine} is a procedure which given a band complex $K$, transforms
it into another band complex $K'$ (we still have $\pi(K') \simeq G$), whose lamination splits into a
disjoint union of finitely many sub-laminations of several types - simplicial, surface, toral, thin
- and these sub-laminations induce a splitting of $K'$ into sub-complexes containing them. $K'$ can
be thought of as the ``normal form'' of the band complex $K$. Analyzing the structure of $K'$ and
its sub-complexes one can obtain some information about the structure of the group $G$.

In particular, in the case when the original action of $G$ on $\Gamma$ is {\em stable} one can
obtain a splitting of $G$. Recall that a non-degenerate (that is, containing more than one point)
subtree $S$ of $\Gamma$ is {\em stable} if for every non-degenerate subtree $S'$ of $S$, we have
$Fix(S') = Fix(S)$ (here, $Fix(I) \leqslant G$ consists of all elements which fix $I$ point-wise). The action
of $G$ on $\Gamma$ is {\em stable} if every non-degenerate subtree of $T$ contains a stable subtree.

\begin{theorem} \cite[Theorem 9.5]{Bestvina_Feighn:1995}
\label{th:Best_Feighn}
Let $G$ be a finitely presented group with a nontrivial, stable, and minimal action on an
$\R$-tree $\Gamma$. Then either
\begin{enumerate}
\item[(1)] $G$ splits over an extension $E$-by-cyclic, where $E$ fixes on arc of $\Gamma$, or
\item[(2)] $\Gamma$ is a line. In this case, $G$ splits over an extension of the kernel of the
action by a finitely generated free abelian group.
\end{enumerate}
\end{theorem}

The key ingredient of the Rips-Bestvina-Feighn machine is a set of particular operations, called
{\em moves}, on band complexes applied in a certain order. These operations originate from
the work of Makanin \cite{Makanin:1982} and Razborov \cite{Razborov:1985} that ideas of
Rips are built upon.

Observe that the group $G$ in Theorem \ref{th:Best_Feighn} must be finitely presented. To obtain a
similar result about finitely generated groups acting on $\R$-trees one has to further restrict
the action. An action of a group $G$ on an $\R$-tree $\Gamma$ satisfies the {\em ascending
chain condition} if for every decreasing sequence
$$I_1 \supset I_2 \supset \cdots \supset I_n \supset \cdots$$
of arcs in $\Gamma$ which converge into a single point, the corresponding sequence
$$Fix(I_1) \subset Fix(I_2) \subset \cdots \subset Fix(I_n) \subset \cdots$$
stabilizes.

\begin{theorem} \cite{Guirardel:2008}
\label{th:Guirardel_0}
Let $G$ be a finitely generated group with a nontrivial minimal action on an $\R$-tree $\Gamma$.
If
\begin{enumerate}
\item[(1)] $\Gamma$ satisfies the ascending chain condition,
\item[(2)] for any unstable arc $J$ of $\Gamma$,
\begin{enumerate}
\item[(a)] $Fix(J)$ is finitely generated,
\item[(b)] $Fix(J)$ is not a proper subgroup of any conjugate of itself, that is, if $Fix(J)^g
\subset Fix(J)$ for some $g \in G$ then $Fix(J)^g = Fix(J)$.
\end{enumerate}
\end{enumerate}
Then either
\begin{enumerate}
\item[(1)] $G$ splits over a subgroup $H$ which is an extension of the stabilizer of an arc of
$\Gamma$ by a cyclic group, or
\item[(2)] $\Gamma$ is a line.
\end{enumerate}
\end{theorem}

Now, we would like to discuss some applications of the above results which are based on the
construction outlined in \cite{Bestvina:1988} and \cite{Paulin:1988} making possible to obtain
isometric group actions on $\R$-trees as Gromov-Hausdorff limits of actions on hyperbolic spaces.
All the details can be found in \cite{Bestvina:1999}.

Let $(X, d_X)$ be a metric space equipped with an isometric action of a group $G$ which can be viewed as
a homomorphism $\rho: G \to Isom(X)$. Assume that $X$ contains a point $\varepsilon$ which is not
fixed by $G$. In this case, we call the triple $(X, \varepsilon, \rho)$ a {\em based $G$-space}.

Observe that given a based $G$-space $(X, \varepsilon, \rho)$ one can define a pseudometric $d =
d_{(X, \varepsilon, \rho)}$ on $G$ as follows
$$d(g,h) = d_X(\rho(g) \cdot \varepsilon, \rho(h) \cdot \varepsilon).$$
Now, the set ${\cal D}$ of all non-trivial pseudometrics on $G$ taken up to rescaling by positive
reals forms a topological space and we say that a sequence $(X_i, \varepsilon_i, \rho_i),\ i \in \N$
of based $G$-spaces {\em converges} to the based $G$-space $(X, \varepsilon, \rho)$ and write
$$\lim_{i \to \infty} (X_i, \varepsilon_i, \rho_i) = (X, \varepsilon, \rho)$$
if $d_{(X_i, \varepsilon_i, \rho_i)} \to d_{(X, \varepsilon, \rho)}$ in ${\cal D}$. Now, the following
result is the main tool in obtaining isometric group actions on $\R$-trees from actions on
Gromov-hyperbolic spaces.

\begin{theorem} \cite[Theorem 3.3]{Bestvina:1999}
\label{th:Paulin}
Let $(X_i, \varepsilon_i, \rho_i), i \in \N$ be a convergent sequence of based $G$-spaces. Assume
that
\begin{enumerate}
\item[(1)] there exists $\delta > 0$ such that every $X_i$ is $\delta$-hyperbolic,
\item[(2)] there exists $g \in G$ such that the sequence $d_{X_i}(\varepsilon_i, \rho_i(g) \cdot
\varepsilon_i)$ is unbounded.
\end{enumerate}
Then there is a based $G$-space $(\Gamma, \varepsilon)$ which is an $\R$-tree and an isometric
action $\rho : G \to Isom(\Gamma)$ such that $(X_i, \varepsilon_i, \rho_i) \to (\Gamma, \varepsilon,
\rho)$.
\end{theorem}

In fact, the above theorem does not guarantee that the limiting action of $G$ on $\Gamma$ has no
global fixed points. But in the case when $G$ is finitely generated and each $X_i$ is proper (closed
metric balls are compact), it is possible to choose base-points in $\varepsilon_i \in X_i$ to make
the action of $G$ on $\Gamma$ non-trivial (see \cite[Proposition 3.8, Theorem 3.9]{Bestvina:1999}).
Moreover, one can retrieve some information about stabilizers of arcs in $\Gamma$ (see
\cite[Proposition 3.10]{Bestvina:1999}).

Note that Theorem \ref{th:Paulin} can also be interpreted in terms of asymptotic cones (see
\cite{Drutu_Sapir:2005, Drutu_Sapir:2008} for details).

\smallskip

The power of Theorem \ref{th:Paulin} becomes obvious in particular when a finitely generated group
$G$ has infinitely many pairwise non-conjugate homomorphisms $\phi_i : G \to H$ into a word-hyperbolic
group $H$. In this case, each $\phi_i$ defines an action of $G$ on the Cayley graph $X$ of $H$ with
respect to some finite generating set. Now, one can define $X_i$ to be $X$ with a word metric rescaled so that
the sequence of $(X_i, \varepsilon_i, \rho_i), i \in \N$ satisfies the requirements of Theorem \ref{th:Paulin} and
thus obtain a non-trivial isometric action of $G$ on an $\R$-tree. Many results about word-hyperbolic
groups were obtained according to this scheme, for example, the following classical result.

\begin{theorem} \cite{Paulin:1991}
\label{th:Paulin_2}
Let $G$ be a word-hyperbolic group such that the group of its outer automorphisms $Out(G)$ is infinite.
Then $G$ splits over a virtually cyclic group.
\end{theorem}

Combined with the {\em shortening argument} due to Rips and Sela \cite{RipsSela:1994} this scheme
gives many other results about word-hyperbolic groups, for example, the theorems below.

\begin{theorem} \cite{RipsSela:1994}
\label{th:RipsSela}
Let $G$ be a torsion-free freely indecomposable word-hyperbolic group. Then the internal automorphism
group $Inn(G)$ of $G$ has finite index in $Aut(G)$.
\end{theorem}

\begin{theorem} \cite{Gromov:1987, Sela:1997}
\label{th:GromovSela}
Let $G$ be a finitely presented torsion-free freely indecomposable group and let $H$ be a
word-hyperbolic group. Then there are only finitely many conjugacy classes of subgroups of $G$
isomorphic to $H$.
\end{theorem}

For more detailed account of applications of the Rips-Bestvina-Feighn machine please refer to
\cite{Bestvina:1999}.

\section{Lyndon length functions}
\label{sec:length_func}

In 1963, R. Lyndon (see \cite{Lyndon:1963}) introduced a notion of {\em length function on a group}
in an attempt to axiomatize cancelation arguments in free groups as well as free products with
amalgamation and HNN extensions, and to generalize them to a wider class of groups. The main idea
was to measure the amount of cancellation in passing to the reduced form of a product of reduced
words in a free group and free constructions, and it turned out that the cancelation process could
be described by rather simple axioms. Using simple combinatorial techniques Lyndon described groups
with {\em free} $\mathbb{Z}$-valued length functions and conjectured (see \cite{Lyndon_Schupp:2001})
that any finitely generated group with a free $\mathbb{R}$-valued length function can be embedded
into a free product of finitely many copies of $\mathbb{R}$. The conjectures eventually was proved
wrong (counterexamples were initially given in \cite{Alperin_Moss:1985} and \cite{Promislow:1985})
but the idea of using length functions became quite popular (see, for example, \cite{Harrison:1972,
Chiswell:1976, Hoare:1979}), and then it turned out that the language of length functions described
the same class of groups as the language of actions on trees (see Section \ref{sec:equiv} for more
details).

Below we give the axioms of (Lyndon) length function and recall the main results in this field.

\smallskip

Let $G$ be a group and $\Lambda$ be an ordered abelian group. Then a function $l: G \rightarrow
\Lambda$ is called a {\em (Lyndon) length function} on $G$ if the following conditions hold:
\begin{enumerate}
\item [(L1)] $\forall\ g \in G:\ l(g) \geqslant 0$ and $l(1) = 0$,
\item [(L2)] $\forall\ g \in G:\ l(g) = l(g^{-1})$,
\item [(L3)] $\forall\ f, g, h \in G:\ c(f,g) > c(f,h) \rightarrow c(f,h) = c(g,h)$,

\noindent where $c(f,g) = \frac{1}{2}(l(f) + l(g) - l(f^{-1}g))$.
\end{enumerate}

Observe that in general $c(f,g) \notin \Lambda$, but $c(f,g) \in \Lambda_{\mathbb{Q}} = \Lambda
\otimes_{\mathbb{Z}} \mathbb{Q}$, where $\mathbb{Q}$ is the additive group of rational numbers,
so, in the axiom (L3) we view $\Lambda$ as a subgroup of $\Lambda_{\mathbb{Q}}$. But in some cases
the requirement $c(f,g) \in \Lambda$ is crucial so we state it as a separate axiom
\begin{enumerate}
\item [(L4)] $\forall\ f, g \in G:\ c(f,g) \in \Lambda.$
\end{enumerate}

It is not difficult to derive the following two properties of Lyndon length functions from the
axioms (L1) -- (L3):
\begin{itemize}
\item $\forall\ f, g \in G:\ l(f g) \leqslant l(f) + l(g)$,
\item $\forall\ f, g \in G:\ 0 \leqslant c(f,g) \leqslant \min\{l(f),l(g)\}$.
\end{itemize}

The following examples motivated the whole theory of groups with length functions.

\begin{example}
\label{exam:free_gr}
Given a free group $F(X)$ on the set $X$ one can define a (Lyndon) length function on $F$ as follows
$$w(X) \rightarrow |w(X)|,$$
where $| \cdot |$ is the length of the reduced word in $X \cup X^{\pm 1}$ representing $w$.
\end{example}

\begin{example}
\label{exam:free_prod}
Given two groups $G_1$ and $G_2$ with length functions $L_1 : G_1 \rightarrow \Lambda$ and $L_2 :
G_2 \rightarrow \Lambda$ for some ordered abelian group $\Lambda$ one can construct a length
function on $G_1 \ast G_2$ as follows (see \cite[Proposition 5.1.1]{Chiswell:2001}). For any $g \in G_1
\ast G_2$ such that
$$g = f_1 g_1 \cdots f_k g_k f_{k+1},$$
where $f_i \in G_1,\ i \in [1, k+1],\ f_i \neq 1,\ i \in [2, k]$ and $1 \neq g_i \in G_2,\ i \in
[1, k]$, define
$$L(g) = \sum_{i=1}^{k+1} L_1(f_i) + \sum_{j=1}^k L_2(g_j) \in \Lambda.$$
\end{example}

A length function $l: G \rightarrow \Lambda$ is called {\em free}, if it satisfies
\begin{enumerate}
\item [(L5)] $\forall\ g \in G:\ g \neq 1 \rightarrow l(g^2) > l(g)$.
\end{enumerate}

Obviously, the $\mathbb{Z}$-valued length function constructed in Example \ref{exam:free_gr} is
free. The converse is shown below (see also \cite{Hoare:1979} for another proof of this result).

\begin{theorem} \cite{Lyndon:1963}
\label{th:lyndon}
Any group $G$ with a length function $L : G \rightarrow \mathbb{Z}$ can be embedded into a free
group $F$ of finite rank whose natural length function extends $L$.
\end{theorem}

\begin{example}
\label{exam:free_prod_2}
Given two groups $G_1$ and $G_2$ with free length functions $L_1 : G_1 \rightarrow \Lambda$ and $L_2 :
G_2 \rightarrow \Lambda$ for some ordered abelian group $\Lambda$, the length function on $G_1 \ast
G_2$ constructed in Example \ref{exam:free_prod} is free.
\end{example}

Observe that if a group $G$ acts on a $\Lambda$-tree $(X,d)$ then we can fix a point $x \in X$ and
consider a function $l_x : G \to \Lambda$ defined as $l_x(g) = d(x, g x)$. Such a function $l_x$ on
$G$ we call a {\em length function based at $x$}. It is easy to check that $l_x$ satisfies all the
axioms (L1) -- (L4) of Lyndon length function. Now if $\| \cdot \|$ is the translation length function
associated with the action of $G$ on $(X,d)$ then the following axioms show the connection between
$l_x$ and $\| \cdot \|$.
\begin{enumerate}
\item[(i)] $l_x(g) = \| g \| + 2 d(x, A_g)$ if $g$ is not an inversion.
\item[(ii)] $\| g \| = \max\{0, l_x(g^2) - l_x(g)\}$.
\end{enumerate}
Here, it should be noted that for points $x \notin A_g$, there is a unique closest point of $A_g$
to $x$. The distance between these points is the one referred to in (i). While $A_g = A_{g^n}$ for
all $n \neq 0$ in the case where $g$ is hyperbolic, if $g$ fixes a point, it is possible that
$A_g \subset A_{g^2}$. We may have $l_x(g^22) - l_x(g) < 0$ in this case. Free actions are
characterized, in the language of length functions, by the facts (a) $\| g \| > 0$ for all $g
\neq 1$, and (b) $l_x(g^2) > l_x(g)$ for all $g \neq 1$. The latter follows from the fact that
$\| g \| = n \| g \|$ for all $g$. We note that there are axioms for the translation length function
which were shown to essentially characterize actions on $\Lambda$-trees, up to equivariant isometry,
by W. Parry, \cite{Parry:1991}.

\section{Infinite words}
\label{sec:inf_words}

The notion of Lyndon length function provides a very nice tool to study properties of groups but
unfortunately its applications are quite limited due to very abstract axiomatic approach. Even in
the case of free length functions methods are close in a sense to free group cancelation techniques
but necessity to use axioms makes everything cumbersome. Introduction of {\em infinite words} was
first of all motivated by the idea that working with elements of group with Lyndon length function
should be exactly as in free group. It is not surprising that the first group which the method of
infinite words was applied to was Lyndon's free group $\FZt$ (see \cite{Myasnikov_Remeslennikov_Serbin:2005})
which shares many properties with free group. Later the infinite words techniques were extensively
applied in \cite{Myasnikov_Remeslennikov_Serbin:2006, Kharlampovich_Myasnikov_Remeslennikov_Serbin:2006,
Khan_Myasnikov_Serbin:2007, Nikolaev_Serbin:2011, KMS:2011, KMS:2011(2), KMS:2011(3), KMRS:2012,
Nikolaev_Serbin:2011(2), Malyutin_Nagnibeda_Serbin:2011}.

\smallskip

Below we follow the construction given in \cite{Myasnikov_Remeslennikov_Serbin:2005}.

\subsection{Definition and preliminaries}
\label{subs:words-def}

Let $\Lambda$ be a discretely ordered abelian group with the minimal positive element $1$. It is
going to be clear from the context if we are using $1$ as an element of $\Lambda$, or as an integer.
Let $X = \{x_i \mid i \in I\}$ be a set. Put $X^{-1} = \{x_i^{-1} \mid i \in I\}$ and $X^\pm = X
\cup X^{-1}$. A {\em $\Lambda$-word} is a function of the type
$$w: [1,\alpha_w] \to X^\pm,$$
where $\alpha_w \in \Lambda,\ \alpha_w \geqslant 0$. The element $\alpha_w$ is called the {\em
length} $|w|$ of $w$.

\smallskip

By $W(\Lambda, X)$ we denote the set of all $\Lambda$-words. Observe, that $W(\Lambda, X)$ contains
an empty $\Lambda$-word which we denote by $\varepsilon$.

Concatenation $u v$ of two $\Lambda$-words $u, v \in W(\Lambda, X)$ is an $\Lambda$-word of length
$|u| + |v|$ and such that:
\[
(uv)(a) = \left\{ \begin{array}{ll}
\mbox{$u(a)$}  & \mbox{if $1 \leqslant a \leqslant |u|$} \\
\mbox{$v(a - |u|)$ } & \mbox{if $|u|  < a \leqslant |u| + |v|$}
\end{array}
\right.
\]
Next, for any $\Lambda$-word $w$ we define an {\em inverse} $w^{-1}$ as an $\Lambda$-word of the
length $|w|$ and such that
$$w^{-1}(\beta) = w(|w| + 1 - \beta)^{-1} \ \ (\beta \in [1,|w|]).$$

A $\Lambda$-word $w$ is {\it reduced} if $w(\beta + 1) \neq w(\beta)^{-1}$ for each $1 \leqslant
\beta < |w|$. We denote by $R(\Lambda, X)$ the set of all reduced $\Lambda$-words. Clearly,
$\varepsilon \in R(\Lambda, X)$. If the concatenation $u v$ of two reduced $\Lambda$-words $u$ and
$v$ is also reduced then we write $u v = u \circ v$.

\smallskip

For $u \in W(\Lambda, X)$ and $\beta \in [1, \alpha_u]$ by $u_\beta$ we denote the restriction of
$u$ on $[1, \beta]$. If $u \in R(\Lambda, X)$ and $\beta \in [1, \alpha_u]$ then
$$u = u_\beta \circ {\tilde u}_\beta,$$
for some uniquely defined ${\tilde u}_\beta$.

An element ${\rm com}(u,v) \in R(\Lambda, X)$ is called the ({\emph{longest}) {\it common initial
segment} of $\Lambda$-words $u$ and $v$ if
$$u = {\rm com}(u,v) \circ \tilde{u}, \ \ v = {\rm com}(u,v) \circ \tilde{v}$$
for some (uniquely defined) $\Lambda$-words $\tilde{u}, \tilde{v}$ such that $\tilde{u}(1) \neq
\tilde{v}(1)$.

Now, we can define the product of two $\Lambda$-words. Let $u, v \in R(\Lambda, X)$. If
${\rm com}(u^{-1}, v)$ is defined then
$$u^{-1} = {\rm com}(u^{-1},v) \circ {\tilde u}, \ \ v = {\rm com} (u^{-1},v) \circ {\tilde v},$$
for some uniquely defined ${\tilde u}$ and ${\tilde v}$. In this event put
$$u \ast v = {\tilde u}^{-1} \circ {\tilde v}.$$
The  product ${\ast}$ is a partial binary operation on $R(\Lambda, X)$.

\begin{example}
\label{ex:2-torsion}
Let $\Lambda = \mathbb{Z}^2$ with the right lexicographic order (in this case $1 = (1,0)$). Put
\[
\mbox{$w(\beta)$} = \left\{ \begin{array}{ll}
\mbox{$x$}  & \mbox{if $\beta = (s, 0) \ and \ s \geqslant 1$} \\
\mbox{$x^{-1}$} & \mbox{if $\beta = (s, 1) \ and \ s \leqslant 0$}
\end{array} \right. \]
Then
$$w: [1, (0,1)] \rightarrow X^{\pm}$$
is a reduced $\Lambda$-word. Clearly, $w^{-1} = w$ so $w \ast w = \varepsilon$. In particular,
$R(\Lambda, X)$ has $2$-torsion with respect to $\ast$.
\end{example}

\begin{theorem} \cite[Theorem 3.4]{Myasnikov_Remeslennikov_Serbin:2005}
\label{th:pre}
Let $\Lambda$ be a discretely ordered abelian group and $X$ be a set. Then the set of reduced
$\Lambda$-words $R(\Lambda, X)$ with the partial binary operation $\ast$ satisfies the axioms
{\rm (P1) -- (P4)} of a pregroup.
\end{theorem}

An element $v \in R(\Lambda, X)$ is termed {\it cyclically reduced} if $v(1)^{-1} \neq v(|v|)$. We
say that an element $v \in R(\Lambda, X)$ admits a {\it cyclic decomposition} if $v = c^{-1} \circ
u \circ c$, where $c, u \in R(\Lambda, X)$ and $u$ is cyclically reduced. Observe that a cyclic
decomposition is unique (whenever it exists). We denote by $CR(\Lambda, X)$ the set of all cyclically
reduced words in $R(\Lambda, X)$ and by $CDR(\Lambda, X)$ the set of all words from $R(\Lambda, X)$
which admit a cyclic decomposition.

\smallskip

Below we refer to $\Lambda$-words as {\it infinite words} usually omitting $\Lambda$ whenever it
does not produce any ambiguity.

\smallskip

A subset $G \leqslant R(\Lambda, X)$ is called a {\em subgroup} of $R(\Lambda, X)$ if $G$ is a group
with respect to $\ast$. We say that a subset $Y \subset R(\Lambda, X)$ {\em generates a subgroup}
$\langle Y \rangle$ in $R(\Lambda, X)$ if the product $y_1 \ast \cdots \ast y_n$ is defined for any
finite sequence of elements $y_1, \ldots, y_n \in Y^{\pm 1}$.

\begin{example}
\label{ex:ex324}
Let $\Lambda$ be a direct sum of copies of $\mathbb{Z}$ with the right lexicographic order. Then the
set of all elements of finite length in $R(\Lambda, X)$ forms a subgroup which is isomorphic to a
free group with basis $X$.
\end{example}

\subsection{Commutation in infinite words}
\label{sec:commutation}

Let $G$ be a subgroup of $CDR(\Lambda ,X)$, where $\Lambda$ is a discretely ordered abelian group.
We fix $G$ for the rest of the subsection.

Let $I_\Lambda$ index the set of all convex subgroups of $\Lambda$. $I_\Lambda$ is linearly ordered
(see, for example, \cite{Chiswell:2001}): $i < j$ if and only if $\Lambda_i < \Lambda_j$, and
$$\Lambda = \bigcup_{i \in I_\Lambda} \Lambda_i.$$
We say that $g \in G$ {\em has the height} $i \in I_\Lambda$ and denote $ht(g) = i$ if $|g| \in
\Lambda_i$ and $|g| \notin \Lambda_j$ for any $j < i$. Observe that this definition depends only on
$G$ since the complete chain of convex subgroups of $\Lambda$ is unique.

It is easy to see that
$$ht(g_1 g_2) \leqslant \max\{ht(g_1), ht(g_2)\},$$
hence, if $G = \langle g_1, \ldots, g_k \rangle$ then we define
$$ht(G) = \max\{ht(g_1),\ldots,ht(g_k)\}.$$

Using the characteristics of elements of $G$ introduced above we prove several technical results
we are used, for example, in Sections \ref{sec:reg} and \ref{sec:Z^n}.

The first result is an analog of Harrison's Theorem (see \cite{Harrison:1972}) in the case of
cyclically reduced elements.

\begin{lemma} \cite[Lemma 5]{KMRS:2012}
\label{le:LS}
Let $f, h \in G$ be cyclically reduced. If\/ $c(f^m, h^n) \geqslant |f| + |h|$ for some $m, n > 0$
then $[f,h] = \varepsilon$.
\end{lemma}

The next lemma shows that commutation implies similar cyclic decompositions.

\begin{lemma} \cite[Lemma 6]{KMRS:2012}
\label{le:cycl}
For any two $g_1, g_2 \in G$ if $[g_1,g_2] = \varepsilon$ and $g_1 = c^{-1} \circ h_1 \circ c,\
g_2 = d^{-1} \circ h_2 \circ d$ are their cyclic decompositions then $c = d$.
\end{lemma}

In particular, it follows that if $g \in G$ is cyclically reduced then all elements of $C_G(g)$ are
cyclically reduced as well.

\begin{lemma} \cite[Lemma 7]{KMRS:2012}
\label{le:0}
Let $f, h \in G$ be such that $h$ is cyclically reduced and $ht(f) > ht(h)$. If $ht(f^{-1} \ast h
\ast f) < ht(f)$ then for every $n \in \mathbb{N}$ either $f = h^n \circ f_n$, or $f = h^{-n} \circ
f_n$ for some $f_n \in G$.
\end{lemma}

Here is another special case of Harrison's Theorem.

\begin{lemma} \cite[Lemma 8]{KMRS:2012}
\label{le:1}
Let $f, h_1, h_2 \in G$ be such that $ht(h_1), ht(h_2) < ht(f)$ and $ht(f^{-1} \ast h_1 \ast f),\
ht(f^{-1} \ast h_2 \ast f) < ht(f)$. Then $[h_1, h_2] = \varepsilon$.
\end{lemma}

\begin{lemma} \cite[Lemma 9]{KMRS:2012}
\label{le:2}
Let $f, h_1 {\not = \epsilon} \in G$ be such that $f$ is cyclically reduced and $ht(h_1) < ht(f)$.
If $ht(f^{-1} \ast h_1 \ast f) < ht(f)$ and $[h_1,h_2] = \varepsilon$, where $ht(h_2) < ht(f)$,
then $ht(f^{-1} \ast h_2 \ast f) < ht(f)$.
\end{lemma}

Using the lemmas above one can easily prove the following well=known result.

\begin{prop}
\label{pr:1} \cite{Alperin_Bass:1987}
If $G < CDR(\Lambda,X)$ then for any $g \in G$, its centralizer $C_G(g)$ is a subgroup of $A$. In
particular, if $\Lambda = \mathbb{Z}^n$ then $C_G(g)$ is a free abelian group of rank not more than
$n$.
\end{prop}

\section{Equivalence}
\label{sec:equiv}

Here we show that all three approaches, namely, free actions on $\Lambda$-trees (see Section
\ref{sec:lambda}), free Lyndon $\Lambda$-valued length functions (see Section \ref{sec:length_func}),
and $\Lambda$-words (see Section \ref{sec:inf_words}) describe the same class of groups which is
called {\em $\Lambda$-free} groups.

\subsection{From actions to length functions}
\label{subs:act-len}

The following theorem is one of the most important results in the theory of length functions.

\begin{theorem} \cite{Chiswell:1976}
Let $G$ be a group and $l: G \rightarrow \Lambda$ a Lyndon length function satisfying the following
condition:
\begin{enumerate}
\item [(L4)] $\forall\ f, g \in G:\ c(f,g) \in \Lambda$.
\end{enumerate}
Then there are a $\Lambda$-tree $(X, d)$, an action of $G$ on $X$, and a point $x \in X$ such that
$l = l_x$.
\end{theorem}

The proof is constructive, that is, one can define a $\Lambda$-metric space out of $G$ and $l$, and
then prove that this space is in fact a $\Lambda$-tree on which $G$ acts by isometries (see
\cite[Subsection 2.4]{Chiswell:2001} for details).

\subsection{From length functions to infinite words}
\label{subs:len-words}

The following results show the connection between groups with Lyndon length functions and groups of
infinite words.

\begin{theorem} \cite[Theorem 4.1]{Myasnikov_Remeslennikov_Serbin:2005}
\label{co:3.1}
Let $\Lambda$ be a discretely ordered abelian group and $X$ be a set. Then any subgroup $G$ of
$CDR(\Lambda, X)$ has a free Lyndon length function with values in $\Lambda$ -- the restriction
$| \cdot |_G$ on $G$ of the standard length function $| \cdot |$ on $CDR(\Lambda, X)$.
\end{theorem}

The converse of  Theorem \ref{co:3.1} was obtained by Chiswell in \cite{Chiswell:2005}.

\begin{theorem} \cite[Theorem 3.9]{Chiswell:2005}
\label{chis}
Let $G$ have a free Lyndon length function $L : G \rightarrow \Lambda$, where $\Lambda$ is a
discretely ordered abelian group. Then there exists a length preserving embedding $\phi : G
\rightarrow CDR(\Lambda, X)$, that is, $|\phi(g)| = L(g)$ for any $g \in G$.
\end{theorem}

\begin{cor} \cite[Corollary 3.10]{Chiswell:2005}
\label{chis-cor}
Let $G$ have a free Lyndon length function $L : G \rightarrow \Lambda$, where $\Lambda$ is an
arbitrary ordered abelian group. Then there exists an embedding $\phi : G \rightarrow CDR(\Lambda',
X)$, where $\Lambda' = \mathbb{Z} \oplus \Lambda$ is discretely ordered with respect to the right
lexicographic order and $X$ is some set, such that, $|\phi(g)| = (0, L(g))$ for any $g \in G$.
\end{cor}

\subsection{From infinite words to actions}
\label{subs:words-act}

Below we follow the construction given in \cite{KMS:2011}.

\subsubsection{Universal trees}
\label{sec:universal}

Let $G$ be a subgroup of $CDR(\Lambda, X)$ for some discretely ordered abelian group $\Lambda$ and a
set $X$. We assume $G,\ \Lambda$, and $X$ to be fixed for the rest of this section.

Every element $g \in G$ is a function
$$g: [1,|g|] \rightarrow X^{\pm},$$
with the domain $[1,|g|]$ which a closed segment in $\Lambda$. Since $\Lambda$ can be viewed as a
$\Lambda$-metric space then $[1,|g|]$ is a geodesic connecting $1$ and $|g|$, and every $\alpha \in
[1,|g|]$ we view as a pair $(\alpha, g)$. We would like to identify initial subsegments of the
geodesics corresponding to all elements of $G$ as follows.

\smallskip

Let
$$S_G = \{(\alpha,g) \mid g \in G, \alpha \in [0,|g|]\}.$$
Since for every $f,g \in G$ the word $com(f,g)$ is defined, we can introduce an equivalence relation
on $S_G$ as follows: $(\alpha,f) \sim (\beta,g)$ if and only if $\alpha = \beta \in [0, c(f,g)]$.
Obviously, it is symmetric and reflexive. For transitivity observe that if $(\alpha,f) \sim (\beta,
g)$ and $(\beta,g) \sim (\gamma,h)$ then $0 \leqslant \alpha = \beta = \gamma \leqslant c(f,g),
c(g,h)$. Since $c(f,h) \geqslant \min\{c(f,g), c(g,h)\}$ then $\alpha = \gamma \leqslant c(f,h)$.

Let $\Gamma_G = S_G / \sim$ and $\epsilon = \langle 0, 1 \rangle$, where $\langle \alpha, f \rangle$
is the equivalence class of $(\alpha,f)$.

\begin{prop} \cite{KMS:2011}
\label{pr:lambda_tree}
$\Gamma_G$ is a $\Lambda$-tree,
\end{prop}
\begin{proof} At first we show that $\Gamma_G$ is a $\Lambda$-metric space. Define the metric by
$$d(\langle \alpha, f \rangle, \langle \beta, g \rangle) = \alpha + \beta - 2\min\{\alpha, \beta,
c(f,g)\}.$$
Let us check if it is well-defined. Indeed, $c(f,g) \in \Lambda$ is defined for every $f,g \in G$.
Moreover, let $(\alpha, f) \sim (\gamma,u)$ and $(\beta, g) \sim (\delta,v)$, we want to prove
$$d(\langle \alpha, f \rangle, \langle \beta, g \rangle) = d(\langle \gamma, u \rangle, \langle
\delta, v \rangle)$$
which is equivalent to
$$\min\{\alpha,\beta,c(f,g)\} = \min\{\alpha,\beta,c(u,v)\}$$
since $\alpha = \gamma,\ \beta = \delta$. Consider the following cases.

\begin{enumerate}
\item[(a)] $\min\{\alpha,\beta\} \leqslant c(u,v)$

Hence, $\min\{\alpha,\beta,c(u,v)\} = \min\{\alpha,\beta\}$ and it is enough to prove $\min\{\alpha,
\beta\}$ $= \min\{\alpha,\beta,c(f,g)\}$. From length function axioms for $G$ we have
$$c(f,g) \geqslant \min\{c(u,f),c(u,g)\},\ \ c(u,g) \geqslant \min\{c(u,v),c(v,g)\}.$$
Hence,
$$c(f,g) \geqslant \min\{c(u,f),c(u,g)\} \geqslant  \min\{c(u,f), \min\{c(u,v),c(v,g)\} \}$$
$$ = \min\{c(u,f),c(u,v),c(v,g)\}.$$
Now, from $(\alpha, f) \sim (\gamma,u),\ (\beta, g) \sim (\delta,v)$ it follows that $\alpha \leqslant
c(u,f),\ \beta \leqslant c(v,g)$ and combining it with the assumption $\min\{\alpha,\beta\} \leqslant
c(u,v)$ we have
$$c(f,g) \geqslant \min\{c(u,f),c(u,v),c(v,g)\} \geqslant \min\{\alpha,\beta\},$$
or, in other words,
$$\min\{\alpha,\beta,c(f,g)\} = \min\{\alpha,\beta\}.$$

\item[(b)] $\min\{\alpha,\beta\} > c(u,v)$

Hence, $\min\{\alpha,\beta,c(u,v)\} = c(u,v)$ and it is enough to prove $c(f,g) = c(u,v)$.

Since
$$c(u,f) \geqslant \alpha > c(u,v),\ c(v,g) \geqslant \beta > c(u,v),$$
then $\min\{c(u,f),c(u,v),c(v,g)\} = c(u,v)$ and
$$c(f,g) \geqslant \min\{c(u,f),c(u,v),c(v,g)\} = c(u,v).$$
Now we prove that $c(f,g) \leqslant c(u,v)$. From length function axioms for $G$ we have
$$c(u,v) \geqslant \min\{c(v,g),c(u,g)\} = c(u,g) \geqslant \min\{c(v,g),c(u,v)\} = c(u,v),$$
that is, $c(u,v) = c(u,g)$. Now,
$$c(u,v) = c(u,g) \geqslant \min\{c(u,f),c(f,g)\},$$
where $\min\{c(u,f),c(f,g)\} = c(f,g)$ since otherwise we have $c(u,v) \geqslant c(u,f) \geqslant
\alpha$ - a contradiction. Hence, $c(u,v) \geqslant c(f,g)$ and we have $c(f,g) = c(u,v)$.
\end{enumerate}

By definition of $d$, for any $\langle \alpha, f \rangle,\ \langle \beta, g \rangle$ we have
$$d(\langle \alpha, f \rangle, \langle \beta, g \rangle) = d(\langle \beta, g \rangle, \langle
\alpha, f \rangle) \geqslant 0,$$
$$d(\langle \alpha, f \rangle, \langle \alpha, f \rangle) = 0.$$
If
$$d(\langle \alpha, f \rangle, \langle \beta, g \rangle) = \alpha + \beta - 2 \min\{\alpha, \beta,
c(f,g)\} = 0$$
then $\alpha + \beta = 2 \min\{\alpha,\beta,c(f,g)\}$. It is possible only if $\alpha = \beta \leqslant
c(f,g)$ which implies $\langle \alpha, f \rangle = \langle \beta, g \rangle$. Finally, we have to
prove the triangle inequality
$$d(\langle \alpha, f \rangle, \langle \beta, g \rangle) \leqslant d(\langle \alpha, f \rangle, \langle
\gamma, h \rangle) + d(\langle \beta, g \rangle, \langle \gamma, h \rangle)$$
for every $\langle \alpha, f \rangle,\ \langle \beta, g \rangle,\ \langle \gamma, h \rangle \in
\Gamma_G$. The inequality above is equivalent to
$$\alpha + \beta - 2\min\{\alpha, \beta, c(f,g)\} \leqslant \alpha + \gamma$$
$$ - 2 \min\{\alpha, \gamma, c(f,h) + \beta + \gamma - 2\min\{\beta,\gamma,c(g,h)\}\}$$
which comes down to
$$\min\{\alpha,\gamma,c(f,h)\} + \min\{\beta,\gamma,c(g,h)\} \leqslant  \min\{\alpha,\beta,c(f,g)\} +
\gamma.$$

\smallskip

First of all, observe that for any $\alpha,\beta,\gamma \in \Lambda$ the triple $(\min\{\alpha,
\beta\},$ $\min\{\alpha,\gamma\},$ $\min\{\beta,\gamma\})$ is isosceles. Hence, by Lemma 1.2.7(1)
\cite{Chiswell:2001}, the triple
$$(\min\{\alpha,\beta,c(f,g)\},\ \min\{\alpha,\gamma,c(f,h)\},\ \min\{\beta,\gamma,c(g,h)\})$$
is isosceles too. In particular,
$$\min\{\alpha,\beta,c(f,g)\} \geqslant \min\{ \min\{\alpha,\gamma,c(f,h)\},\ \min\{\beta,\gamma,c(g,h)\} \}$$
$$ = \min\{\alpha,\beta,\gamma,c(f,h),c(g,h)\}.$$
Now, if
$$\min\{\alpha,\beta,\gamma,c(f,h),c(g,h)\} = \min\{\alpha,\gamma,c(f,h)\}$$
then $\min\{\beta,\gamma,c(g,h)\} = \gamma$ and
$$\min\{\alpha,\gamma,c(f,h)\} + \min\{\beta,\gamma,c(g,h)\} \leqslant  \min\{\alpha,\beta,c(f,g)\} + \gamma$$
holds. If
$$\min\{\alpha,\beta,\gamma,c(f,h),c(g,h)\} = \min\{\beta,\gamma,c(g,h)\}$$
then $\min\{\alpha,\gamma,c(f,h)\} = \gamma$ and
$$\min\{\alpha,\gamma,c(f,h)\} + \min\{\beta,\gamma,c(g,h)\} \leqslant  \min\{\alpha,\beta,c(f,g)\} + \gamma$$
holds again. So, $d$ is a $\Lambda$-metric.

\smallskip

Finally, we want to prove that $\Gamma_G$ is $0$-hyperbolic with respect to $\epsilon = \langle 0,
1 \rangle$ (and, hence, with respect to any other point in $\Gamma_G$). It is enough to prove that
the triple
$$((\langle \alpha, f \rangle \cdot \langle \beta, g \rangle)_\epsilon,\ (\langle \alpha, f \rangle
\cdot \langle \gamma, h \rangle)_\epsilon,\ (\langle \beta, g \rangle \cdot \langle \gamma, h
\rangle)_\epsilon)$$
is isosceles for every $\langle \alpha, f \rangle,\ \langle \beta, g \rangle,\ \langle \gamma, h
\rangle \in \Gamma_G$. But by definition of $d$ the above triple is isosceles if and only if
$$(\min\{\alpha,\beta,c(f,g)\},\ \min\{ \alpha, \gamma, c(f,h) \},\ \min\{ \beta, \gamma, c(g,h) \})$$
is isosceles which holds.

\smallskip

So, $\Gamma_G$ is a $\Lambda$-tree.
\end{proof}

Since $G$ is a subset of $CDR(\Lambda,X)$ and every element $g \in G$ is a function defined on
$[1_A,|g|]$ with values in $X^\pm$ then we can define a function
$$\xi : (\Gamma_G - \{\epsilon\}) \rightarrow X^\pm,\ \ \xi(\langle \alpha, g \rangle) = g(\alpha).$$
It is easy to see that $\xi$ is well-defined. Indeed, if $(\alpha, g) \sim (\alpha_1,g_1)$
then $\alpha = \alpha_1 \leqslant c(g,g_1)$, so $g(\alpha) = g_1(\alpha_1)$. Moreover, since every $g \in G$
is reduced then $\xi(p) \neq \xi(q)^{-1}$ whenever $d(p,q) = 1$.

$\xi$ can be extended to a function
$$\Xi : geod(\Gamma_G)_\epsilon \to R(\Lambda,X),$$
where $geod(\Gamma_G)_\epsilon = \{ (\epsilon,p] \mid p \in \Gamma_G \}$, so that
$$\Xi(\ (\epsilon, \langle \alpha, g \rangle]\ )(t) = g(t),\ t \in [1_A,\alpha].$$
That is, $\Xi(\ (\epsilon, \langle \alpha, g \rangle]\ )$ is the initial subword of $g$ of length
$\alpha$, and
$$\Xi(\ (\epsilon, \langle |g|, g \rangle]\ ) = g.$$
On the other hand, if $g \in G$ and $\alpha \in [1_A,|g|]$ then
the initial subword of $g$ of length $\alpha$ uniquely corresponds to $\Xi(\ (\epsilon, \langle
\alpha, g \rangle]\ )$. If $(\alpha, g) \sim (\alpha_1,g_1)$ then $\alpha = \alpha_1 \leqslant c(g,g_1)$,
and since $g(t) = g_1(t)$ for any $t \in [1_A,c(g,g_1)]$ then
$$\Xi(\ (\epsilon, \langle \alpha, g \rangle]\ ) = \Xi(\ (\epsilon, \langle \alpha_1, g_1 \rangle]\ ).$$

\begin{lemma} \cite{KMS:2011}
\label{le:subword_prod}
Let $u,v \in R(\Lambda,X)$. If $u \ast v$ is defined then $u \ast a$ is also defined, where $v = a
\circ b$. Moreover, $u \ast a$ is an initial subword of either $u$ or $u \ast v$.
\end{lemma}
\begin{proof} The proof follows from Figure \ref{cancel_diag}.
\end{proof}

\begin{figure}[htbp]
\centering{\mbox{\psfig{figure=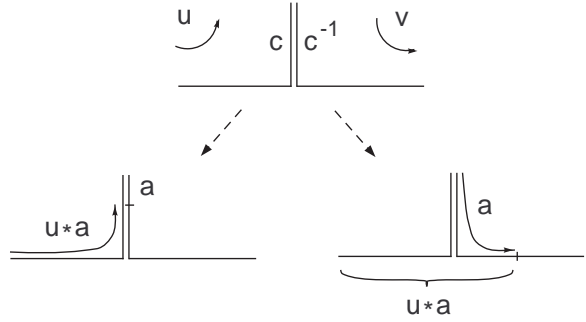,height=2in}}}
\caption{Possible cancelation diagrams in Lemma \ref{le:subword_prod}.}
\label{cancel_diag}
\end{figure}

Now, since for every $\langle \alpha, g \rangle \in \Gamma_G,\ \Xi(\ (\epsilon, \langle \alpha, g
\rangle]\ )$ is an initial subword of $g \in G$ then by Lemma \ref{le:subword_prod}, $f \ast
\Xi(\ (\epsilon, \langle \alpha, g \rangle]\ )$ is defined for any $f \in G$. Moreover, again by
Lemma \ref{le:subword_prod},  $f \ast \Xi(\ (\epsilon, \langle \alpha, g \rangle]\ )$ is an initial
subword of either $f$ or $f \ast g$. More precisely,
$$f \ast \Xi(\ (\epsilon, \langle \alpha, g \rangle]\ ) = \Xi(\ (\epsilon, \langle |f| - \alpha,
f \rangle]\ )$$
if $f \ast \Xi(\ (\epsilon, \langle \alpha, g \rangle]\ )$ is an initial subword of $f$, and
$$f \ast \Xi(\ (\epsilon, \langle \alpha, g \rangle]\ ) = \Xi(\ (\epsilon, \langle |f| + \alpha -
2 c(f^{-1},g), f\ast g \rangle]\ )$$
if $f \ast \Xi(\ (\epsilon, \langle \alpha, g \rangle]\ )$ is an initial subword of $f \ast g$.

\smallskip

Hence, we define a (left) action of $G$ on $\Gamma_G$ as follows:
$$f \cdot \langle \alpha, g \rangle = \langle |f| + \alpha - 2\min\{\alpha,c(f^{-1},g)\}, f \rangle$$
if $\alpha \leqslant c(f^{-1},g)$, and
$$f \cdot \langle \alpha, g \rangle = \langle |f| + \alpha - 2\min\{\alpha,c(f^{-1},g)\}, f\ast g
\rangle$$
if $\alpha > c(f^{-1},g)$.

\smallskip

The action is well-defined. Indeed, it is easy to see that $f \cdot \langle \alpha, g \rangle = f
\cdot \langle \alpha_1, g_1 \rangle$ whenever $(\alpha, g) \sim (\alpha_1,g_1)$.

\begin{lemma} \cite{KMS:2011}
\label{le:isometric}
The action of $G$ on $\Gamma_G$ defined above is isometric.
\end{lemma}
\begin{proof} Observe that it is enough to prove
$$d(\epsilon, \langle \alpha, g \rangle) = d(f \cdot \epsilon, f \cdot \langle \alpha, g \rangle)$$
for every $f,g \in G$. Indeed, from the statement above it is going to follow that the geodesic tripod
$(\epsilon,\langle |g|,g \rangle, \langle |h|,h \rangle)$ is isometrically mapped to the geodesic tripod
$(\langle |f|,f \rangle, f \cdot \langle |g|,g \rangle, f \cdot \langle |h|,h \rangle)$ and isometricity
follows.

We have
$$d(\epsilon, \langle \alpha, g \rangle) = d(\langle 0, 1 \rangle, \langle \alpha, g \rangle) = 0 +
\alpha - 2 \min\{ 0, \alpha, c(1,g) \} = \alpha,$$
$$d(f \cdot \epsilon, f \cdot \langle \alpha, g \rangle) = d(\langle |f|,f \rangle, f \cdot \langle
\alpha, g \rangle).$$
Consider two cases.

\begin{enumerate}
\item[(a)] $\alpha \leqslant c(f^{-1},g)$

Hence,
$$d(\langle |f|,f \rangle, f \cdot \langle \alpha, g \rangle) = d(\langle |f|,f \rangle, \langle |f|
- \alpha, f \rangle)$$
$$ = |f| + |f| - \alpha - 2 \min\{|f|,|f|-\alpha,c(f,f)\} = |f| + |f|-\alpha - 2(|f|-\alpha) = \alpha.$$

\item[(b)] $\alpha > c(f^{-1},g)$

Hence,
$$d(\langle |f|,f \rangle, f \cdot \langle \alpha, g \rangle) = d(\langle |f|,f \rangle, \langle |f|
+ \alpha - 2c(f^{-1},g), f\ast g \rangle)$$
$$ = |f| + |f| + \alpha - 2c(f^{-1},g) - 2\min\{|f|,|f| + \alpha - 2c(f^{-1},g),c(f,f\ast g\})$$
$$= 2|f| + \alpha - 2c(f^{-1},g) - 2\min\{|f| + \alpha - 2c(f^{-1},g),c(f,f\ast g)\}.$$
Let $f = f_1 \circ c^{-1},\ g = c \circ g_1,\ |c| = c(f^{-1},g)$. Then $|f| + \alpha - 2c(f^{-1},g)
= |f_1|+\alpha-c(f^{-1},g) > |f_1|$. At the same time, $c(f,f \ast g) = |f_1|$, so $\min\{|f| +
\alpha - 2c(f^{-1},g),c(f,f\ast g)\} = |f_1|$ and
$$d(\langle |f|,f \rangle, f \cdot \langle \alpha, g \rangle) = 2|f| + \alpha - 2c(f^{-1},g) - 2|f_1|
= 2|f| + \alpha - 2|c| - 2|f_1| = \alpha$$.
\end{enumerate}
\end{proof}

\begin{prop} \cite{KMS:2011}
\label{pr:action}
The action of $G$ on $\Gamma_G$ defined above is free and $L_\epsilon(g) = |g|$. Moreover, $\Gamma_G$
is minimal with respect to this action if and only if $G$ contains a cyclically reduced element $h
\in G$, that is, $|h^2| = 2|h|$.
\end{prop}
\begin{proof} {\bf Cialm 1.} The stabilizer of every $x \in \Gamma_G$ is trivial.

\smallskip

Next, suppose $f \cdot \langle \alpha, g \rangle = \langle \alpha, g \rangle$.
First of all, if $\alpha = 0$ then $|f| + \alpha - 2\min\{\alpha,c(f^{-1},g)\} = |f|$ then $|f| =
\alpha = 0$. Also, if $c(f^{-1},g) = 0$ then $|f| + \alpha - 2\min\{\alpha,c(f^{-1},g)\} = |f| +
\alpha$ which has to be equal to $\alpha$ form our assumption. In both cases $f = 1$ follows.

Assume $f \neq 1$ (which implies $\alpha,\ c(f^{-1},g) \neq 0$) and consider the following cases.

\begin{enumerate}
\item[(a)] $\alpha < c(f^{-1},g)$

Hence, from
$$\langle \alpha, g \rangle = \langle |f| - \alpha, f \rangle$$
we get $\alpha = |f| - \alpha \leqslant c(f,g)$. In particular, $|f| = 2 \alpha$.

Consider the product $f \ast g$. We have
$$f = f_1 \circ com(f^{-1},g)^{-1},\ g = com(f^{-1},g) \circ g_1.$$
Since $\alpha < c(f^{-1},g)$ then we have $com(f^{-1},g) = c_\alpha \circ c,\ |c_\alpha| = \alpha$.
Hence,
$$f = f_1 \circ c^{-1} \circ c_\alpha^{-1},\ g = c_\alpha \circ c \circ g_1.$$
On the other hand, from $|f| = 2\alpha$ we get $|f_1| + |c| = \alpha \leqslant c(f,g)$, so, $com(f,g)$
has $f_1 \circ c$ as initial subword. That is, $g = f_1 \circ c \circ g_2$, but now comparing two
representations of $g$ above we get $c_\alpha = f_1 \circ c^{-1}$ and $c_\alpha \ast c \neq c_\alpha
\circ c$ - a contradiction.

\item[(b)] $\alpha = c(f^{-1},g)$

We have $f = f_1 \circ c_\alpha^{-1},\ g = c_\alpha \circ g_1,\ |c_\alpha| = \alpha$. From  $\langle
\alpha, g \rangle = \langle |f| - \alpha, f \rangle$ we get $\alpha = |f| - \alpha \leqslant c(f,g)$, so
$|f| = 2\alpha$ and $|f_1| = \alpha$. Since $|f_1| = \alpha \leqslant c(f,g)$ then $g = f_1 \circ g_2$
from which it follows that $f_1 = c_\alpha$. But then $f_1 \ast c_\alpha^{-1} \neq f_1 \circ
c_\alpha^{-1}$ - contradiction.

\item[(c)] $\alpha > c(f^{-1},g)$

Hence, from
$$\langle \alpha, g \rangle = \langle |f| + \alpha - 2 c(f^{-1},g), f \ast g \rangle$$
we get $\alpha = |f| + \alpha - 2 c(f^{-1},g) \leqslant c(g,f \ast g)$. In particular, $|f| = 2 c(f^{-1},
g)$.

Consider the product $f \ast g$. We have
$$f = f_1 \circ c^{-1},\ g = c \circ g_1,$$
where $c = com(f^{-1},g)$. Hence, $|f_1| = |c| < \alpha \leqslant c(g,f \ast g) = c(g, f_1 \circ g_1)$.
It follows that $g = f_1 \circ g_2$ and, hence, $c = f_1$ which is impossible.

\end{enumerate}

\smallskip

{\bf Cialm 2.} $L_\epsilon(g) = |g|$

\smallskip

We have $L_\epsilon(g) = d(\epsilon, g \cdot \epsilon)$. Hence, by definition of $d$
$$d(\langle 0,1\rangle, g \cdot \langle 0,1\rangle) = d(\langle 0, 1 \rangle, \langle |g|, g \rangle)
= 0 + |g| - 2 \min\{ 0, |g|, c(1,g) \} = |g|.$$

\smallskip

{\bf Cialm 3.} $\Gamma_G$ is minimal with respect to the action if and only if $G$ contains a
cyclically reduced element $h \in G$, that is, $|h^2| = 2|h|$.

\smallskip

Suppose there exists a cyclically reduced element $h \in G$. Let $\Delta \subset \Gamma_G$ be a
$G$-invariant subtree.

First of all, observe that $\epsilon \notin \Delta$. Indeed, if $\epsilon \in \Delta$ then $f \cdot
\epsilon \in \Delta$ for every $f \in G$ and since $\Delta$ is a tree then $[\epsilon, f \cdot
\epsilon] \in \Delta$ for every $f \in G$. At the same time, $\Gamma_G$ is spanned by $[\epsilon, f
\cdot \epsilon],\ f \in G$, so, $\Delta = \Gamma_G$ - a contradiction.

Let $u \in \Delta$. By definition of $\Gamma_G$ there exists $g \in G$ such that $u \in [\epsilon,
g \cdot \epsilon]$. Observe that $A_g \subseteq \Delta$. Indeed, for example by Theorem 1.4
\cite{Chiswell:2001}, if $[u,p]$ is the bridge between $u$ and $A_g$ then $p = Y(g^{-1} \cdot u,\
u,\ g \cdot u)$. In particular, $p \in \Delta$ and since for every $v \in A_g$ there exist $g_1,
g_2 \in C_G(g)$ such that $v \in [g_1 \cdot p,\ g_2 \cdot p]$ then $A_g \subseteq \Delta$.

Observe that if $g$ is cyclically reduced then $\epsilon \in A_g$, that is, $\epsilon \in \Delta$ -
a contradiction. More generally, $\Delta \cap A_f = \emptyset$ for every cyclically reduced $f \in G$.
Hence, let $[p,q]$ be the bridge between $A_g$ and $A_h$ so that $p \in A_g,\ q \in A_h$. Then by
Lemma 2.2 \cite{Chiswell:2001}, $[p,q] \subset A_{gh}$, in particular, $p,q \in A_{gh}$. It follows
that $A_{gh} \subseteq \Delta,\ q \in A_{gh} \cap A_h$, and $\Delta \cap A_h \neq \emptyset$ - a
contradiction.

Hence, there can be no proper $G$-invariant subtree $\Delta$.

\smallskip

Now, suppose $G$ contains no cyclically reduced element. Hence, $\epsilon \notin A_f$ for every $f
\in G$. Let $\Delta$ be spanned by $A_f,\ f \in G$. Obviously, $\Delta$ is $G$-invariant. Indeed,
let $u \in [p,q]$, where $p \in A_f,\ q \in A_g$ for some $f,g \in G$. Then $h \cdot u \in [h \cdot p,
h \cdot q]$, where $h \cdot p \in h \cdot A_f = A_{hfh^{-1}},\ h \cdot q \in h \cdot A_g = A_{hgh^{-1}}$,
that is, $h \in \Delta$.

Finally, $\epsilon \in \Gamma_G - \Delta$.

\end{proof}

\begin{prop} \cite{KMS:2011}
\label{pr:universal}
If $(Z,d')$ is a $\Lambda$-tree on which $G$ acts freely as isometries, and $w \in Z$ is such that
$L_w(g) = |g|,\ g \in G$ then there is a unique $G$-equivariant isometry $\mu: \Gamma_G \to Z$
such that $\mu(\epsilon) = w$, whose image is the subtree of $Z$ spanned by the orbit $G \cdot w$
of $w$.
\end{prop}
\begin{proof} Define a mapping $\mu : \Gamma_G \rightarrow Z$ as follows
$$\mu(\langle \alpha, f \rangle) = x\ {\rm if}\ d'(w,x) = \alpha,\ d'(f \cdot w,x) = |f|-\alpha.$$
Observe that $\mu(\epsilon) = \mu(\langle 0,1\rangle) = w$

\smallskip

{\bf Claim 1.} $\mu$ is an isometry.

\smallskip

Let $\langle \alpha, f \rangle,\ \langle \beta, g \rangle \in \Gamma_G$. Then by definition of $d$
we have
$$d(\langle \alpha, f \rangle, \langle \beta, g \rangle) = \alpha + \beta - 2 \min\{ \alpha, \beta,
c(f,g)\}.$$
Let $x = \mu(\langle \alpha, f \rangle),\ y = \mu(\langle \beta, g \rangle)$. Then By Lemma 1.2
\cite{Chiswell:2001} in $(Z,d')$ we have
$$d'(x,y) = d(w,x) + d(w,y) - 2\min\{d(w,x),d(w,y),d(w,z)\},$$
where $z = Y(w,f\cdot w, g\cdot w)$. Observe that $d(w,x) = \alpha,\ d(w,y) = \beta$. At the same
time, since $L_w(g) = |g|,\ g \in G$ then
$$d(w,z) = \frac{1}{2}(d(w,f\cdot w) + d(w,g\cdot w) - d(f\cdot w,g\cdot w)) = \frac{1}{2}(|f| + |g|
- |f^{-1} g|) = c(f,g),$$
and
$$d(\mu(\langle \alpha, f \rangle), \mu(\langle \beta, g \rangle)) = d'(x,y) = \alpha + \beta -
2\min\{\alpha,\beta,c(f,g)\} = d(\langle \alpha, f \rangle, \langle \beta, g \rangle).$$

\smallskip

{\bf Claim 1.} $\mu$ is equivariant.

\smallskip

We have to prove
$$\mu(f\cdot \langle \alpha, g \rangle) = f \cdot \mu(\langle \alpha, g \rangle).$$
Let $x = \mu(\langle \alpha, g \rangle),\ y = \mu(f\cdot \langle \alpha, g \rangle)$.
By definition of $\mu$ we have $d'(w,x) = \alpha,\ d'(g \cdot w,x) = |g|-\alpha$.
\begin{enumerate}
\item[(a)] $\alpha \leqslant c(f^{-1},g)$

Hence,
$$f \cdot \langle \alpha, g \rangle = \langle |f| - \alpha, f \rangle.$$
and to prove $y = f\cdot x$ it is enough to show that $d'(w,f\cdot x) = |f|-\alpha$
and $d'(f\cdot w,f\cdot x) = \alpha$.

Observe that the latter equality holds since $d'(f\cdot w,f\cdot x) = d'(w,x) = \alpha$.
To prove the former one, by Lemma 1.2 \cite{Chiswell:2001} we have
$$d(w,f\cdot x) = d'(w,f\cdot w) + d'(f\cdot x, f\cdot w)$$
$$ - 2\min\{d'(w,f\cdot w),
d'(f\cdot x, f\cdot w), d'(f\cdot w,z)\},$$
where $z = Y(w,f\cdot w, (fg)\cdot w)$. Also,
$$d'(f\cdot w,z) = \frac{1}{2}(d'(f\cdot w,w) + d'(f\cdot w,(fg)\cdot w)- d'(w,(fg)\cdot w))$$
$$ = \frac{1}{2}(|f|+|g|-|f^{-1}g|) = c(f^{-1},g).$$
Since, $d'(w,f\cdot w) = |f|,\ d'(f\cdot x, f\cdot w) = \alpha$ then $\min\{d'(w,f\cdot w),
d'(f\cdot x, f\cdot w), d'(f\cdot w,z)\} = \alpha$, and
$$d'(w,f\cdot x) = |f|+\alpha - 2\alpha = |f|-\alpha.$$

\item[(b)] $\alpha > c(f^{-1},g)$

Hence,
$$f \cdot \langle \alpha, g \rangle = \langle |f| + \alpha - 2c(f^{-1},g), f\ast g \rangle.$$
and to prove $y = f\cdot x$ it is enough to show that $d'(w,f\cdot x) = |f| + \alpha - 2c(f^{-1},g)$
and $d'(f\cdot x,(fg) \cdot w) = |fg|-(|f| + \alpha - 2c(f^{-1},g))$.

Observe that $d'(f\cdot x,(fg) \cdot w) = d'(x,gw) = |g| - \alpha = |fg|-(|f| + \alpha - 2c(f^{-1},g))$,
so the latter equality holds.

By Lemma 1.2 \cite{Chiswell:2001} we have
$$d(w,f\cdot x) = d'(w,f\cdot w) + d'(f\cdot x, f\cdot w)$$
$$ - 2\min\{d'(w,f\cdot w),
d'(f\cdot x, f\cdot w), d'(f\cdot w,z)\},$$
where $z = Y(w,f\cdot w, (fg)\cdot w)$. Also,
$$d'(f\cdot w,z) = \frac{1}{2}(d'(f\cdot w,w) + d'(f\cdot w,(fg)\cdot w)- d'(w,(fg)\cdot w))$$
$$ = \frac{1}{2}(|f|+|g|-|f^{-1}g|) = c(f^{-1},g).$$
$d'(w,f\cdot w) = |f|,\ d'(f\cdot x, f\cdot w) = \alpha$, so $\min\{d'(w,f\cdot w),
d'(f\cdot x, f\cdot w), d'(f\cdot w,z)\} = d'(f\cdot w,z) = c(f^{-1},g)$, and
$$d(w,f\cdot x) = |f| + \alpha - 2c(f^{-1},g).$$

\smallskip

{\bf Claim 1.} $\mu$ is unique.

\smallskip

Observe that if $\mu' : \Gamma_G \rightarrow Z$ is another equivariant isometry such that
$\mu'(\epsilon) = w$ then for every $g \in G$ we have
$$\mu'(\langle |g|,g\rangle) = \mu'(g\cdot \langle 0,1\rangle) = g \cdot \mu'(\langle 0,1\rangle) =
g\cdot w.$$
That is, $\mu'$ agrees with $\mu$ on $G\cdot \epsilon$, hence $\mu = \mu'$ because isometries
preserve geodesic segments.

Thus, $\mu$ is unique. Moreover, $\mu(\Gamma_G)$ is the subtree of $Z$ spanned by $G \cdot w$.

\end{enumerate}

\end{proof}

The discussion above can be summarized in the following theorem.

\begin{theorem} \cite{KMS:2011}
Let $G$ be a subgroup of $CDR(\Lambda, X)$ for some discretely ordered abelian group $\Lambda$ and a
set $X$. Let $| \cdot | : G \to \Lambda$ be the length function on $G$ induced from $CDR(\Lambda, X)$.
Then there are a $\Lambda$-tree $(\Gamma_G, p)$, an action of $G$ on $\Gamma_G$ and a point $x \in
\Gamma_G$ such that $|g| = l_x(g)$ for any $g \in G$, where $l_x(g) = p(x, g \cdot x)$. Moreover, If
$(Y, d)$ is a $\Lambda$-tree on which $G$ acts freely by isometries, and $y \in Y$ is such that
$l_y(g) = |g|,\ g \in G$, then there is a unique $G$-equivariant isometry $\mu: \Gamma_G \to Y$ such
that $\mu(x) = y$, whose image is the subtree of $Y$ spanned by the orbit $G \cdot y$ of $y$.
\end{theorem}

\subsubsection{Examples}
\label{sec:examples}

Below we consider two examples of subgroups of $CDR(\Lambda, X)$, where $\Lambda = \mathbb{Z}^2$ and
$X$ an arbitrary alphabet, and explicitly construct the corresponding universal trees for these groups.

\begin{figure}[htbp]
\centering{\mbox{\psfig{figure=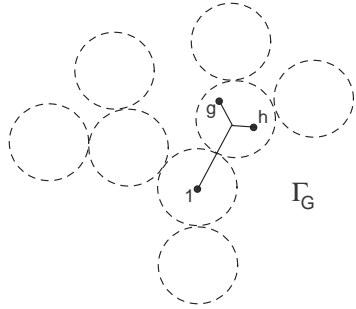,height=1.8in}}}
\caption{$\Gamma_G$ as a $\mathbb{Z}$-tree of $\mathbb{Z}$-trees.}
\label{pic2}
\end{figure}

\begin{example}
\label{ex:centr_ext_inf_words}
Let $F = F(X)$ be a free group with basis $X$ and the standard length function $|\cdot|$, and let
$u \in F$ a cyclically reduced element which is not a proper power. If we assume that $\mathbb{Z}^2
= \langle 1, t \rangle$ is the additive group of linear polynomials in $t$ ordered lexicographically
then the HNN-extension
$$G = \langle F, s \mid u^s = u \rangle$$
embeds into $CDR(\mathbb{Z}^2, X)$ under the following map $\phi$:
$$\phi(x) = x,\ \forall\ x \in X,$$
\[ \mbox{$\phi(s)(\beta)$} = \left\{ \begin{array}{ll}
\mbox{$u(\alpha)$,} & \mbox{if $\beta = m |u| + \alpha, m \geqslant 0, 1 \leqslant \alpha \leqslant |u|$,} \\
\mbox{$u(\alpha)$,} & \mbox{if $\beta = t - m |u| + \alpha, m > 0, 1 \leqslant \alpha \leqslant |u|$.}
\end{array}
\right.
\]

It is easy to see that $|\phi(s)| = t$ and $\phi(s)$ commutes with $u$ in $CDR(\mathbb{Z}^2, X)$. To simplify
the notation we identify $G$ with its image $\phi(G)$.

Every element $g$ of $G$ can be represented as the following reduced $\mathbb{Z}^2$-word
$$g = g_1 \circ s^{\delta_1} \circ g_2 \circ \cdots \circ g_k \circ s^{\delta_k} \circ g_{k+1},$$
where $[g_i, u] \neq 1$. Now, according to the construction described in Subsection \ref{sec:universal}, the
universal tree $\Gamma_G$ consists of the segments in $\mathbb{Z}^2$ labeled by elements from $G$ which
are glued together along their common initial subwords.

\begin{figure}[htbp]
\centering{\mbox{\psfig{figure=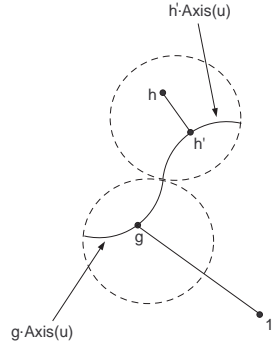,height=2in}}}
\caption{Adjacent $\mathbb{Z}$-subtrees in $\Gamma_G$.}
\label{pic3}
\end{figure}

Thus, $\Gamma_G$ can be viewed as a $\mathbb{Z}$-tree
of $\mathbb{Z}$-trees which are Cayley graphs of $F(X)$ and every vertex $\mathbb{Z}$-subtree can
be associated with a right representative in $G$ by $F$. The end-points of the segments $[1, |g|]$ and
$[1, |h|]$ labeled respectively by $g$ and $h$ belong to the same vertex $\mathbb{Z}$-subtree if and only if
$h^{-1} g \in F$ (see Figure \ref{pic2}).

In other words, $\Gamma_G$ is a ``more detailed'' version of the Bass-Serre tree $T$ for $G$,
in which every vertex is replaced by the Cayley graph of the base group $F$ and the adjacent $\mathbb{Z}$-subtrees
of $\Gamma_G$ corresponding to the representatives $g$ and $h$ are ``connected'' by means of $s^{\pm}$ which
extends $g \cdot Axis(u)$ to $h' \cdot Axis(u)$, where $h'^{-1} h \in F$ and $g^{-1} h \in s^{\pm} F$
(see Figure \ref{pic3}).
\end{example}

The following example is a generalization of the previous one.

\begin{figure}[htbp]
\centering{\mbox{\psfig{figure=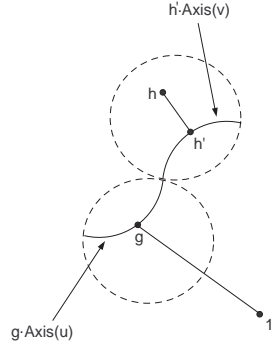,height=2in}}}
\caption{Adjacent $\mathbb{Z}$-subtrees in $\Gamma_H$.}
\label{pic4}
\end{figure}

\begin{example}
\label{ex:HNN_inf_words}
Let $F = F(X)$ be a free group with basis $X$ and the standard length function $|\cdot|$, and let
$u, v \in F$ be cyclically reduced elements which is not a proper powers and such that $|u| = |v|$.
The HNN-extension
$$H = \langle F, s \mid u^s = v \rangle$$
embeds into $CDR(\mathbb{Z}^2, X)$ under the following map $\psi$:
$$\psi(x) = x,\ \forall\ x \in X,$$
\[ \mbox{$\psi(s)(\beta)$} = \left\{\begin{array}{ll}
\mbox{$u(\alpha)$,} & \mbox{if $\beta = m |u| + \alpha, m \geqslant 0, 1 \leqslant \alpha \leqslant |u|$,} \\
\mbox{$v(\alpha)$,} & \mbox{if $\beta = t - m |v| + \alpha, m > 0, 1 \leqslant \alpha \leqslant |v|$.}
\end{array}
\right.
\]
It is easy to see that $|\psi(s)| = t$ and $u \circ \psi(s) = \psi(s) \circ v$ in $CDR(\mathbb{Z}^2, X)$.
Again, to simplify the notation we identify $H$ with its image $\psi(H)$.

The structure of $\Gamma_H$ is basically the same as the structure of $\Gamma_G$ in Example \ref{ex:centr_ext_inf_words}.
The only difference is that the adjacent $\mathbb{Z}$-subtrees of $\Gamma_H$ corresponding to the representatives
$g$ and $h$ are ``connected'' by means of $s^{\pm}$ which extends $g \cdot Axis(u)$ to $h' \cdot Axis(v)$, where
$h'^{-1} h \in F$ and $g^{-1} h \in s^{\pm} F$ (see Figure \ref{pic4}).
\end{example}

\subsubsection{Labeling of ``edges'' and ``paths'' in universal trees}
\label{subsubsec:label}

Let $G$ be a subgroup of $CDR(\Lambda, X)$ for some discretely ordered abelian group $\Lambda$ and a
set $X$, and let $\Gamma_G$ be its universal $\Lambda$-tree constructed in Subsection \ref{sec:universal}.
Recall that there exists a labeling function
$$\xi : (\Gamma_G - \{\epsilon\}) \rightarrow X^\pm,\ \ \xi(\langle \alpha, g \rangle) = g(\alpha)$$
on all points of $\Gamma_G$ except the base-point $\epsilon$.

It is easy to see that the labeling $\xi$ is not equivariant, that is, $\xi(v) \neq \xi(g \cdot v)$
in general (even if both $v$ and $g \cdot v$ are in $\Gamma_G - \{\varepsilon\}$, which is not stable
under the action of $G$). In the present paper we are going to introduce another labeling function
for $\Gamma_G$ defined not on vertices but on ``edges'', stable under the action of $G$. With this
new labeling $\Gamma_G$ becomes an extremely useful combinatorial object in the case $\Lambda =
\mathbb{Z}^n$, but in general such a labeling can be defined for every discretely ordered $\Lambda$.

First of all, for every $v_0, v_1 \in \Gamma_G$ such that $d(v_0, v_1) = 1$ we call the ordered pair
$(v_0, v_1)$ the {\em edge} from $v_0$ to $v_1$. Here, if $e = (v_0,v_1)$ then denote $v_0 = o(e),\
v_1 = t(e)$ which are respectively the {\em origin} and {\em terminus} of $e$. Now, if the vertex
$v_1 \in \Gamma_G - \{\varepsilon\}$ is fixed then, since $\Gamma_G$ is a $\Lambda$-tree, there is
exactly one point $v_0$ such that $d(\varepsilon, v_1) = d(\varepsilon, v_0) + 1$. Hence, there
exists a natural orientation, with respect to $\varepsilon$, of edges in $\Gamma_G$, where an edge
$(v_0,v_1)$ is {\em positive} if $d(\varepsilon, v_1) = d(\varepsilon, v_0) + 1$, and {\em negative}
otherwise. Denote by $E(\Gamma_G)$ the set of edges in $\Gamma_G$. If $e \in E(\Gamma_G)$ and $e =
(v_0,v_1)$ then the pair $(v_1,v_0)$ is also an edge and denote $e^{-1} = (v_1,v_0)$. Obviously,
$o(e) = t(e^{-1})$. Because of the orientation, we have a natural splitting
$$E(\Gamma_G) = E(\Gamma_G)^+ \cup E(\Gamma_G)^-,$$
where $E(\Gamma_G)^+$ and $E(\Gamma_G)^-$ denote respectively the sets of positive and negative
edges. Now, we can define a function $\mu : E(\Gamma_G)^+ \rightarrow X^\pm$ as follows: if $e =
(v_0,v_1) \in E(\Gamma_G)^+$ then $\mu(e) = \xi(v_1)$. Next, $\mu$ can be extended to $E(\Gamma_G)^-$
(and hence to $E(\Gamma_G)$) by setting $\mu(f) = \mu(f^{-1})^{-1}$ for every $f \in E(\Gamma_G)^-$.

\begin{example}
\label{ex:1}
Let $F = F(X)$ be a free group on $X$. Hence, $F$ embeds into (coincides with) $CDR(\mathbb{Z},X)$
and $\Gamma_F$ with the labeling $\mu$ defined above is just a Cayley graph of $F$ with respect to
$X$. That is, $\Gamma_F$ is a labeled simplicial tree.
\end{example}

The action of $G$ on $\Gamma_G$ induces the action on $E(\Gamma_G)$ as follows $g \cdot (v_0, v_1)
= (g \cdot v_0, g \cdot v_1)$ for each $g \in G$ and $(v_0, v_1) \in E(\Gamma_G)$. It is easy to see
that $E(\Gamma_G)^+$ is not closed under the action of $G$ but the labeling is equivariant as the
following lemma shows.

\begin{lemma}
\label{le:label_edges}
If $e,f \in E(\Gamma_G)$ belong to one $G$-orbit then $\mu(e) = \mu(f)$.
\end{lemma}
\begin{proof} Let $e = (v_0,v_1) \in E(\Gamma_G)^+$. Hence, there exists $g \in G$ such that $v_0 =
\langle \alpha, g \rangle,\ v_1 = \langle \alpha + 1, g \rangle$. Let $f \in G$ and consider the
following cases.

\smallskip

\noindent {\bf Case 1}. $c(f^{-1},g) = 0$

Then $f \ast g = f \circ g$. If $\alpha = 0$ then $f \cdot v_0 = \langle |f|, f \rangle = \langle
|f|, f \circ g \rangle$, and $f \cdot v_1 = \langle |f| + 1, f \circ g \rangle$. Hence, $f \cdot e
\in E(\Gamma_G)^+$ and $\mu(f \cdot e) = \xi(f \cdot v_1) = g(1) = \xi(v_1) = \mu(e)$.

\smallskip

\noindent {\bf Case 2}. $c(f^{-1},g) > 0$

\begin{enumerate}

\item[(a)] $\alpha + 1 \leqslant c(f^{-1},g)$

Then $f \cdot v_0 = \langle |f| + \alpha - 2\alpha, f \rangle = \langle |f| - \alpha, f \rangle$
and $f \cdot v_1 = \langle |f| - (\alpha + 1), f \rangle$. So, $d(\varepsilon, f \cdot v_1) <
d(\varepsilon, f \cdot v_0)$ and $f \cdot e \in E(\Gamma_G)^-$. Now,
$$\mu(f \cdot e) = \mu((f \cdot e)^{-1})^{-1} = \mu((f \cdot v_1, f \cdot v_0))^{-1} = \xi(f \cdot
v_0)^{-1} = f(|f|-\alpha)^{-1}$$
$$ = g(\alpha+1) = \xi(v_1) = \mu(e).$$

\item[(b)] $\alpha = c(f^{-1},g)$

We have $f \cdot v_0 = \langle |f| - \alpha, f \rangle$ and $f \cdot v_1 = \langle |f| + (\alpha +
1) - 2 c(f^{-1},g), f \ast g \rangle = \langle |f| - \alpha + 1, f \ast g \rangle$. It follows that
$f \cdot e \in E(\Gamma_G)^+$ and $\mu(f \cdot e) = \xi(f \cdot v_1) = (f \ast g)(|f| - \alpha +
1)$. At the same time, $f \ast g = f_1 \circ g_1$, where $|f_1| = |f| - c(f^{-1},g) = |f| - \alpha,
\ g = g_0 \circ g_1,\ |g_0| = \alpha$, so, $(f \ast g)(|f| - \alpha + 1) = g_1(1) = g(\alpha + 1)$
and $\mu(f \cdot e) = g(\alpha + 1) = \xi(\langle \alpha + 1, g \rangle) = \xi(v_1) = \mu(e)$.

\item[(c)] $\alpha > c(f^{-1},g)$

Hence, $f \cdot v_0 = \langle |f| + \alpha - 2c(f^{-1},g), f \ast g \rangle$ and $f \cdot v_1 =
\langle |f| + \alpha + 1 - 2c(f^{-1},g), f \ast g \rangle$. Obviously, $f \cdot e \in E(\Gamma_G)^+$
and
$$\mu(f \cdot e) = \xi(f \cdot v_1) = (f \ast g)(|f| + \alpha + 1 - 2c(f^{-1},g)) = g_1(\alpha + 1
- c(f^{-1},g))$$
$$ = g(\alpha + 1) = \xi(v_1) = \mu(e),$$
where $f \ast g = f_1 \circ g_1,\ |f_1| = |f| - c(f^{-1},g) = |f| - \alpha,\ g = g_0 \circ g_1,\
|g_0| = \alpha$.
\end{enumerate}

Thus, in all possible cases we got $\mu(f \cdot e) = \mu(e)$ and the required statement follows.
\end{proof}

Let $v,w$ be two points of $\Gamma_G$. Since $\Gamma_G$ is a $\Lambda$-tree there exists a
unique geodesic connecting $v$ to $w$, which can be viewed as a ``path'' is the following sense. A
{\em path from $v$ to $w$} is a sequence of edges $p = \{ e_\alpha \},\ \alpha \in [1, d(v,w)]$ such
that $o(e_1) = v,\ t(e_{d(v,w)}) = w$ and $t(e_\alpha) = o(e_{\alpha + 1})$ for every $\alpha \in
[1,d(v,w)-1]$. In other words, a path is an ``edge'' counter-part of a geodesic and usually, for the
path from $v$ to $w$ (which is unique since $\Gamma_G$ is a $\Lambda$-tree) we are going to use the
same notation as for the geodesic between these points, that is, $p = [v,w]$. In the case when $v =
w$ the path $p$ is empty. The {\em length} of $p$ we denote by $|p|$ and set $|p| = d(v,w)$. Now,
the {\em path label} $\mu(p)$ for a path $p = \{ e_\alpha \}$ is the function $\mu : \{ e_\alpha \}
\rightarrow X^\pm$, where $\mu(e_\alpha)$ is the label of the edge $e_\alpha$.

\begin{lemma}
\label{le:paths}
Let $v,w$ be points of $\Gamma_G$ and $p$ the path from $v$ to $w$. Then $\mu(p) \in R(\Lambda, X)$.
\end{lemma}
\begin{proof} From the definition of $\Gamma_G$ it follows that the statement is true when $v =
\varepsilon$. Let $v_0 = Y(\varepsilon, v, w)$ and let $p_v$ and $p_w$ be the paths from
$\varepsilon$ respectively to $v$ and $w$. Also, let $p_1$ and $p_2$ be the paths from $v_0$ to $v$
and $w$. Since $\mu(p_v), \mu(p_w) \in R(\Lambda,X)$ then $\mu(p_1), \mu(p_2) \in R(\Lambda,X)$ as
subwords. Hence, $\mu(p) \notin R(\Lambda, X)$ implies that the first edges $e_1$ and $e_2$
correspondingly of $p_1$ and $p_2$ have the same label. But this contradicts the definition of
$\Gamma_G$ because in this case $t(e_1) \sim t(e_2)$, but $t(e_1) \neq t(e_2)$.
\end{proof}

As usual, if $p$ is a path from $v$ to $w$ then its {\em inverse} denoted $p^{-1}$ is a path from
$w$ back to $v$. In this case, the label of $p^{-1}$ is $\mu(p)^{-1}$, which is again an element of
$R(\Lambda, X)$.

Define
$$V_G = \{ v \in \Gamma_G \mid\ \exists\ g \in G:\ v = \langle |g|,g \rangle \},$$
which is a subset of points in $\Gamma_G$ corresponding to the elements of $G$. Also, for every $v
\in \Gamma_G$ let
$$path_G(v) = \{ \mu(p) \mid\ p = [v,w]\ {\rm where}\ w \in V_G\}.$$
The following lemma follows immediately.

\begin{lemma}
\label{le:paths_2}
Let $v \in V_G$. Then $path_G(v) = G \subset CDR(\Lambda, X)$.
\end{lemma}

The action of $G$ on $E(\Gamma_G)$ extends to the action on all paths in $\Gamma_G$, hence, Lemma
\ref{le:label_edges} extends to the case when $e$ and $f$ are two $G$-equivalent paths in
$\Gamma_G$.

\section{Regular length functions and actions}
\label{sec:reg}

Regularity of Lyndon length function, or of the underlying action turns out to be an important
property in the theory of $\Lambda$-free groups. By imposing this restriction on the length function
or action one gains a lot of information about the group through inner combinatorics of group elements
viewed as $\Lambda$-words.

\subsection{Regular length functions}

In this section we define regular length functions and show some examples of groups with regular
length functions.

A length function $l: G \rightarrow \Lambda$ is  called {\it regular} if it satisfies the {\em
regularity} axiom:
\begin{enumerate}
\item[(L6)] $\forall\ g, f \in G,\ \exists\ u, g_1, f_1 \in G:$
$$g = u \circ g_1 \ \& \  f = u \circ f_1 \ \& \ l(u) = c(g,f).$$
\end{enumerate}

Observe that a regular length function does not have to be free, as well as, freeness does not imply
regularity.

Here are several examples of groups with regular free length functions.

\begin{example}
\label{ex:reg_len}
Let $F = F(X)$ be a free group on $X$. The length function
$$| \cdot | : F \rightarrow \mathbb{Z},$$
where $|f|$ is a the length  of $f \in F$ as a word in $X^{\pm 1}$, is regular.
\end{example}

The following is a more general example.

\begin{example} \cite{KMRS:2012}
\label{ex:1b}
Let $F = F(X)$ be a free group on $X$, $H$ a finitely generated subgroup of $H$, and $l_H$ the
restriction to $H$ of the length function in $F$ relative to $X$. Then $l_H$ is a regular length
function on $H$ if and only if there exists a basis $U$ of $H$ such that every two non-equal
elements from $U^{\pm 1}$ have different initial letters.
\end{example}

\begin{example}
\label{ex:2}  \cite{Myasnikov_Remeslennikov_Serbin:2005}
Lyndon's free $\Zt$-group $\FZt$ has a regular free length function with values in $\Zt$.
\end{example}

\begin{example} \cite{KMRS:2012}
\label{ex:3}
Let $F = F(X)$ be a free group with basis $X$, $| \cdot |$ the standard length function on $F$
relative to $X$, and  $u, v \in F$ such that $|u| = |v|$ and $u$ is not conjugate to $v^{-1}$. Then
the  HNN-extension
$$G = \langle F, s \mid u^s = v \rangle,$$
has  a regular free length function $l : G \rightarrow \mathbb{Z}^2$ which extends $| \cdot |$.
\end{example}

\begin{example} \cite{KMRS:2012}
\label{ex:4}
For any $n \geqslant 1$ the orientable surface group
$$G = \langle x_1, x_2, \ldots, x_{2n-1}, x_{2n} \mid [x_1,x_2] \cdots [x_{2n-1}, x_{2n}] = 1 \rangle$$
has a regular free length function $l : G \rightarrow \mathbb{Z}^2$.
\end{example}
\begin{proof}
It suffices to represent $G$ as an HNN extension from Example \ref{ex:3}. The word
$$R(X) = x_1 \cdots x_{2n} x_1^{-1} \cdots x_{2n}^{-1}.$$
is quadratic, so   there exists an automorphism $\phi$ of $F = F(x_1, \ldots, x_{2n})$ such that
$$R(X)^\phi = [x_1,x_2] \cdots [x_{2n-1}, x_{2n}]$$
(see, for example,   Proposition 7.6 \cite{Lyndon_Schupp:2001}). It follows that $G$ is isomorphic
to
$$G^\prime = \langle x_1,\ldots, x_{2n} \mid x_1 \cdots x_{2n} x_1^{-1} \cdots x_{2n}^{-1} = 1
\rangle,$$
which can be represented as an HNN-extension of the required form
$$G^\prime = \langle F(x_2,\ldots, x_{2n}), x_1 \mid x_1 (x_2 \cdots x_{2n}) x_1^{-1} = x_{2n}
x_{2n-1} \cdots x_2 \rangle,$$
since $|x_2 \cdots x_{2n}| = |x_{2n} x_{2n-1} \cdots x_2|$.
\end{proof}

\begin{example} \cite{KMRS:2012}
\label{ex:5}
For any $n,\ n \geqslant 3$ the non-orientable surface group
$$G = \langle x_1, x_2, \ldots,  x_{n} \mid x_1^2 x_2^2 \ldots x_n^2 = 1 \rangle$$
has a regular free length function $l : G \rightarrow \mathbb{Z}^2$.
\end{example}
\begin{proof}
Again, it suffices to represent $G$ as an HNN extension from Example \ref{ex:3}. An argument similar
to the one in the proof of Example \ref{ex:4} shows that the group $G$ is isomorphic to
$$G^\prime = \langle x_1, x_2, \ldots, x_n \mid x_1 \ldots x_{n-1} x_n x_1^{-1} \cdots x_{n-1}^{-1}
x_n \rangle $$
and the result follows, since the presentation above can be written as
$$G^\prime = \langle x_1, x_2, \ldots, x_n \mid x_1 (x_2 \ldots x_{n-1} x_n) x_1^{-1} = x_n^{-1}
x_{n-1} \cdots x_2 \rangle$$
\end{proof}

\begin{example} \cite{KMRS:2012}
\label{ex:6}
A free abelian group of rank $n$ has a free regular length function in ${\mathbb Z}^n$.
\end{example}
\begin{proof}
Let $G = \mathbb{Z}^n$ be a free abelian group of rank $n$. Then $G$ is an ordered abelian group
relative to the right lexicographic order ``$\leqslant$''. The absolute value $|u|$ of an element
$u \in G$, defined as $|u| = \max \{u,-u\}$,  gives a free length function $l : G \rightarrow
\mathbb{Z}^n$. It is easy to see that $l$ is regular.
\end{proof}

\begin{example} \cite{KMRS:2012}
\label{ex:7}
Let $G_i,\ i = 1,2$ be a group  with a free regular length function $l_i : G_i \to {\mathbb Z}^n$.
Then the free product $G = G_1 \ast G_2$ has a free regular length function in $l: G \to {\mathbb Z}^n$
that extends the functions $l_1$ and $l_2$.
\end{example}
\begin{proof}
Let $g \in G$ given in the reduced form  $g = u_1 v_1 \ldots u_k v_k,$ where $u_1, \ldots, u_k \in
G_1$ and $v_1, \ldots, v_k \in G_2$. Define $l : G \to  {\mathbb Z}^n$ by
$$l(g)=\sum_{i = 1}^k (l_1(u_i) + l_2(v_i)).$$
$l$ is a free length function (see, for example, \cite{Chiswell:2001}). Also, note that $c(g,h) = 0$
if $g \in G_1$ and $h \in G_2$. Next, $l$ is regular. Indeed, let $g = g_1 \cdots g_k,\ h = h_1 \cdots
h_n \in G$, where the products $g_1 \cdots g_k,\ h_1 \cdots h_n$ are reduced and $k \leqslant n$. Let $m =
\max\{i \mid g_i = h_i\}$. If $m = k$ then $h = g \circ (h_{m+1} \cdots h_n)$ and $c(g,h) = l(g)$,
so the regularity axiom holds for the pair $g, h$. If $m = n$ then necessarily $m = k$ and $h = g$,
so the the axiom holds again. Now assume that $m < k$. If we denote $c = g_1 \ldots g_m$ then $g =
c \circ g_{m+1} \circ g',\ h = c \circ h_{m+1} \circ h'$, where $g' = g_{m+2} \cdots g_k,\ h' =
h_{m+2} \cdots h_n$ ($g'$ is trivial if $k = m + 1$, and $h'$ is trivial if $n = m + 1$). Hence,
$g^{-1} h = (g')^{-1} \circ (g_{m+1}^{-1} h_{m+1}) \circ h'$ and
$$c(g, h) = \frac{1}{2}(l(g) + l(h) - l(g^{-1} h)) = \frac{1}{2}(l(c) + l(g_{m+1}) + l(g') + l(c) +
l(h_{m+1}) + l(h')$$
$$ - (l(g') + l(g_{m+1}^{-1} h_{m+1}) + l(h'))) = l(c) + c(g_{m+1}, h_{m+1}).$$
Since both $g_{m + 1}$ and $h_{m + 1}$ belong to the same factor $G_i$ and the length function on
$G_i$ is regular, it follows that there exists $u \in G_i$ such that $l(u) = c(g_{m + 1}, h_{m + 1}),\
g_{m + 1} = u \circ v,\ h_{m + 1} = u \circ w$ for some $v, w \in G_i$. Hence, $g = (c\ u) \circ v
\circ g',\ h = (c\ u) \circ w \circ h'$, where $c(g,h) = l(c\ u)$ and $c\ u \in G$.
\end{proof}

\begin{example}
Let $G$ be a finitely generated $\mathbb{R}$-free group. Then $G$ has a free regular length function
in $\mathbb{Z}^n$, where $n$ is the maximal rank of free abelian subgroups (centralizers) of $G$.
\end{example}
\begin{proof}
By Rips' Theorem every finitely generated $\mathbb{R}$-free group is a free product of groups
described in Examples \ref{ex:4}, \ref{ex:5}, \ref{ex:6}, hence the result.
\end{proof}

\begin{theorem}
\label{chis-cor-1}
Let $G$ have a free regular Lyndon length function $L : G \rightarrow \Lambda$, where $\Lambda$ is
an arbitrary ordered abelian group. Then there exists an embedding $\phi : G \rightarrow
R(\Lambda', X)$, where $\Lambda'$ is a discretely ordered abelian group and $X$ is some set, such
that, the Lyndon length function on $\phi(G)$ induced from $R(\Lambda',X)$ is regular.
\end{theorem}
\begin{proof}
Since $L$ is regular then for any $g, f \in G$ there exist $u, g_1, f_1 \in G$ such that
$$g = u \circ g_1 \ \& \  f = u \circ f_1 \ \& \ L(u) = c(g,f)$$
By Corollary \ref{chis-cor} it follows that
$$|\phi(g)| = |\phi(u)| + |\phi(g_1)|,\ |\phi(f)| = |\phi(u)| + |\phi(f_1)|.$$
Indeed, if for example $|\phi(g)| < |\phi(u)| + |\phi(g_1)|$ then $L(g) < L(u) + L(g_1)$ -- a
contradiction. So, we have
$$2 c(\phi(g),\phi(f)) = |\phi(g)| + |\phi(f)| - |\phi(g^{-1} f)| = |\phi(g)| + |\phi(f)| -
|\phi(g_1^{-1} \circ f_1)| = 2|\phi(u)|.$$
Hence, $c(\phi(g),\phi(f)) = |\phi(u)|$ and the length function on $\phi(G)$ induced from $R(A, X)$
is regular.
\end{proof}

Notice that the converse of the theorem above is obviously true.

\subsection{Regular actions}
\label{subse:regact}

In this section we give a geometric characterization of group actions that come from regular
length functions.

\smallskip

Let $G$ act on a $\Lambda$-tree $\Gamma$. The action is {\em regular with respect to $x \in \Gamma$}
if for any $g,h \in G$ there exists $f \in G$ such that $[x, f x] = [x, g x] \cap [x, h x]$.

\smallskip

The next lemma shows that regular actions exactly correspond to regular length functions (hence the
term).
\begin{lemma} \cite{KMRS:2012}
Let $G$ act on a $\Lambda$-tree $\Gamma$. Then the action of $G$ is regular with respect to $x \in
\Gamma$ if and only if the length function $L_x: G \rightarrow \Lambda$ based at $x$ is regular.
\end{lemma}
\begin{proof}
Let $d$ be the $\Lambda$-metric on $\Gamma$. By definition, the length function $L_x$ is regular if
for every $g,h \in G$ there exists $f \in G$ such that $g = f g_1,\ h = f h_1$, where $L_x(f) =
c(g,h)$ and $L_x(g) = L_x(f) + L_x(g_1),\ L_x(h) = L_x(f) + L_x(h_1)$.

\smallskip

Suppose the action of $G$ is regular with respect to $x$. Then for $g,h \in G$ there exists $f \in G$
such that $[x,fx] = [x,gx] \cap [x,hx]$. We have $[x,gx] = [x,fx] \cup [fx,gx],\ [x,hx] = [x,fx] \cup
[fx,hx]$ and $d(fx,gx) = d(x,(f^{-1} g)x) = L_x(f^{-1} g),\ d(fx,hx) = d(x,(f^{-1} h)x) = L_x(f^{-1}
h)$. Taking $g_1 = f^{-1} g,\ h_1 = f^{-1} h$ we have $L_x(g) = L_x(f) + L_x(g_1),\ L_x(h) = L_x(f) +
L_x(h_1)$. Finally, since $c(g,h) = \frac{1}{2}(L_x(g)+L_x(h)-L_x(g^{-1} h))$ and $L_x(g^{-1} h) =
d(x, (g^{-1}h)x) = d(gx, hx) = d(fx, gx) + d(fx, hx)$ we get $L_x(f) = c(g,h)$.

\smallskip

Suppose that $L_x$ is regular. Then from $g = f g_1,\ h = f h_1$, where $L_x(g) = L_x(f) + L_x(g_1),\
L_x(h) = L_x(f) + L_x(h_1)$ it follows that $[x,gx] = [x,fx] \cup [fx,gx],\ [x,hx] = [x,fx] \cup [fx,hx]$.
Now, $L_x(f) = c(g,h) = \frac{1}{2}(L_x(g)+L_x(h)-L_x(g^{-1} h))$, so $2d(x,fx) = d(x,gx) + d(x,hx) -
d(x,(g^{-1} h)x) = d(x,gx) + d(x,hx) - d(gx,hx)$. In other words,
$$d(gx,hx) = d(x,gx) + d(x,hx) - 2d(x,fx) = (d(x,gx) - d(x,fx)) +$$
$$ (d(x,hx) - d(x,fx)) = d(fx,gx) + d(fx,hx)$$
which is equivalent to $[x,fx] = [x,gx] \cap [x,hx]$.
\end{proof}

\begin{lemma} \cite{KMRS:2012}
Let $G$ act minimally on a $\Lambda$-tree $\Gamma$. If the action of $G$ is regular with respect to
$x \in \Gamma$ then all branch points of $\Gamma$ are $G$-equivalent.
\end{lemma}
\begin{proof}
From minimality of the action it follows that $\Gamma$ is spanned by the set of points $Gx = \{gx
\mid g \in G\}$.

Now let $y$ be a branch point in $\Gamma$. It follows that there exist (not unique in general) $g,
h \in G$ such that $[x,y] = [x, g x] \cap [x, h x]$. From regularity of the action it follows that
there exists $f \in G$ such that $y = f x$. Hence, every branch point is $G$-equivalent to $x$ and
the statement of the lemma follows.
\end{proof}

\begin{lemma} \cite{KMRS:2012}
Let $G$ act on a $\Lambda$-tree $\Gamma$. If the action of $G$ is regular with respect to $x \in
\Gamma$ then it is regular with respect to any $y \in G x$.
\end{lemma}
\begin{proof}
We have to show that for every $g,h \in G$ there exists $f \in G$ such that $[y, f y] = [y, g y]
\cap [y, h y]$. Since $y = t x$ for some $t \in G$ then we have to prove that $[t x, (f t) x] =
[t x, (g t) x] \cap [t x, (h t) x]$. The latter equality is equivalent to $[x, (t^{-1} f t) x] =
[x, (t^{-1} g t) x] \cap [x, (t^{-1} h t) x]$ which follows from regularity of the action with
respect to $x$.
\end{proof}

\begin{lemma} \cite{KMRS:2012}
Let $G$ act freely on a $\Lambda$-tree $\Gamma$ so that all branch points of $\Gamma$ are
$G$-equivalent. Then the action of $G$ is regular with respect to any branch point in $\Gamma$.
\end{lemma}
\begin{proof}
Let $x$ be a branch point in $\Gamma$ and $g, h \in G$. If $g = h$ then $[x, g x] \cap [x, h x] =
[x, g x]$ and $g$ is the required element. Suppose $g \neq h$. Since the action is free then $g x
\neq h x$ and we consider the tripod formed by $x, g x, h x$. Hence, $y = Y(x, g x, h x)$ is a
branch point in $\Gamma$ and by the assumption there exists $f \in G$ such that $y = f x$.
\end{proof}

\begin{figure}[htbp]
\label{reg_action}
\centering{\mbox{\psfig{figure=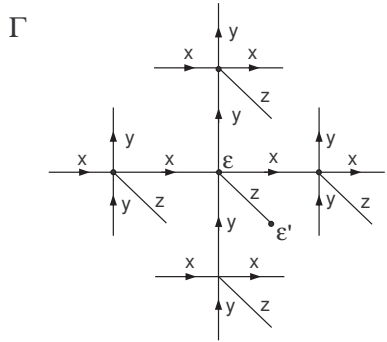,height=2in}}}
\caption{$\Gamma$ in Example \ref{example_non_reg}.}
\end{figure}

\begin{example}
\label{example_non_reg}
Let $\Gamma'$ be the Cayley graph of a free group $F(x,y)$ with the base-point $\varepsilon$. Let
$\Gamma$ be obtained from $\Gamma'$ by adding an edge labeled by $z \neq x^{\pm 1}, y^{\pm 1}$ at
every vertex of $\Gamma'$. $F(x,y)$ has a natural action on $\Gamma'$ which we can extend to the
action on $\Gamma$. The edge at $\varepsilon$ labeled by $z$ has an endpoint not equal to
$\varepsilon$ and we denote it by $\varepsilon'$. Observe that the action of $F(x,y)$ on $\Gamma$
is regular with respect to $\varepsilon$ but is not regular with respect to $\varepsilon'$.
\end{example}

\subsection{Merzlyakov's Theorem for $\Lambda$-free groups with regular actions}
\label{subs:merz}

Here are the results which show that groups with regular free Lyndon length functions
generalize free groups in the following sense.

\begin{theorem} \cite{Khan_Myasnikov_Serbin:2007}
Every finitely generated non-abelian group $G$ with a regular free Lyndon length function freely
lifts every positive sentence of the language $L_G$ which holds in $G$, that is, $G$ freely lifts
its positive theory $Th^+(G)$ (the set of all positive sentences that are true in $G$).
\end{theorem}

\begin{theorem} \cite{Khan_Myasnikov_Serbin:2007}
Let $G$ be a finitely generated non-abelian group with a regular free Lyndon length function. Then
every positive sentence in the language $L_G$ which holds in $G$ has term-definable Skolem functions
in $G$. Moreover, if $G$ has a decidable word problem then such Skolem functions can be found
effectively.
\end{theorem}

\section{Limit groups}
\label{sec:limit_gps}

These numerous characterizations of limit groups make them into a very robust tool linking  group
theory, topology and logic.

Limit groups play an important part in modern group theory. They appear in many different situations:
in combinatorial group theory as groups discriminated by $G$ ($\omega$-residually $G$-groups or
fully residually $G$-groups) \cite{Baumslag:1962, Baumslag:1967, Myasnikov_Remeslennikov:2000,
Baumslag_Miasnikov_Remeslennikov:2000, Baumslag_Miasnikov_Remeslennikov:2002}, in the algebraic
geometry over groups as the coordinate groups of irreducible varieties over $G$
\cite{Baumslag_Miasnikov_Remeslennikov:1999, Kharlampovich_Myasnikov:1998(1),
Kharlampovich_Myasnikov:1998(2), Kharlampovich_Myasnikov:2005(3), Sela:2001}, groups universally
equivalent to $G$ \cite{Remeslennikov:1989, GaglioneSpellman:1993, Myasnikov_Remeslennikov:2000}, limit groups of
$G$ in the Gromov-Hausdorff topology \cite{Champetier_Guirardel:2005}, in the theory of equations
in groups \cite{Lyndon:1960, Razborov:1982, Razborov:1995, Kharlampovich_Myasnikov:1998(1),
Kharlampovich_Myasnikov:1998(2), Kharlampovich_Myasnikov:2005(3), Groves:2005(2)}, in group actions
\cite{Bestvina_Feighn:1995, GLP:1994, Myasnikov_Remeslennikov_Serbin:2005, Groves:2005(1),
Guirardel:2004},  in the solutions of Tarski problems \cite{Kharlampovich_Myasnikov:2006, Sela:2006},
etc.

Recall, that   a group $G$ is called {\em fully residually free} if for any non-trivial $g_1, \ldots,
g_n \in G$ there exists a homomorphism $\phi$ of $G$ into a free group such that $\phi(g_1), \ldots,
\phi(g_n)$ are non-trivial.

It is a crucial result  that every limit group admits
a free action on a $\mathbb{Z}^n$-tree for an appropriate $n \in \mathbb{N}$, where $\mathbb{Z}^n$
is ordered lexicographically (see \cite{Kharlampovich_Myasnikov:1998(2)}). The proof comes in several steps.
The initial breakthrough is due to Lyndon, who introduced a construction of the  free  $\Zt$-completion $\FZt$ of a free group $F$ (nowadays it is called Lyndon's free
$\Zt$-group) and showed that this group, as well as all its subgroups,  is fully residually free \cite{Lyndon:1960}.
Much later Remeslennikov proved that every finitely generated fully residually free group has a free Lyndon length function with values in $\mathbb{Z}^n$ (but not necessarily  ordered lexicographically) \cite{Remeslennikov:1989}.  That was a first link between limit groups and free actions on $\mathbb{Z}^n$-trees. In 1995 Myasnikov and Remeslennikov showed that Lyndon free exponential group $\FZt$  has a free Lyndon length function with values in $\mathbb{Z}^n$ with  lexicographical ordering \cite{Myasnikov_Remeslennikov:1995} and stated a conjecture that every limit group embeds into $\FZt$. Finally, Kharlampovich and Myasnikov proved that every limit group $G$ embeds into $\FZt$  \cite{Kharlampovich_Myasnikov:1998(2)}.

Below, following \cite{Myasnikov_Remeslennikov_Serbin:2005} we construct a free $\Zt$-valued
length function on $\FZt$ which combined with the result of Kharlampovich and Myasnikov mentioned above
 gives a free $\mathbb{Z}^n$-valued length function on a given  limit group $G$. Then we discuss various algorithmic applications of
these results which are based on the technique of infinite words and Stallings foldings techniques for subgroups of $\FZt$ (see
\cite{Myasnikov_Remeslennikov_Serbin:2006, Kharlampovich_Myasnikov_Remeslennikov_Serbin:2006,
Nikolaev_Serbin:2011}).

We fix a set $X$, a free group $F = F(X)$, and consider the additive group of the polynomial ring $\Zt$ as a free abelian group (with basis  $1, t, t^2,\ldots $) ordered lexicographically.

\subsection{Lyndon's free group $\FZt$}

Let $A$ be an associative unitary ring. A group $G$ is termed an {\em $A$-group} if it is equipped
with a function ({\em exponentiation}) $G \times A \rightarrow G$:
$$(g,\alpha) \rightarrow  g^\alpha $$
satisfying the following conditions for arbitrary $g,h \in G$ and $\alpha, \beta \in A$:
\begin{enumerate}
\item[(Exp1)] $g^1 = g, \ \ \ g^{\alpha + \beta} = g^{\alpha} g^{\beta}, \ \ \ g^{\alpha \beta} =
(g^{\alpha})^{\beta}$,
\item[(Exp2)] $g^{-1} h^{\alpha} g = (g^{-1} h g)^{\alpha}$,
\item[(Exp3)] if $g$ and $h$ commute, then $(gh)^{\alpha} = g^{\alpha} h^{\alpha}$.
\end{enumerate}
The axioms (Exp1) and (Exp2) were introduced originally by R. Lyndon in \cite{Lyndon:1960}, the
axiom (Exp3) was added later in \cite{Myasnikov_Remeslennikov:1994}. A homomorphism $\phi: G
\rightarrow H$ between two $A$-groups is termed an $A$-homomorphism if $\phi(g^\alpha) = \phi(g)^\alpha$
for every $g \in G$ and $\alpha \in A$. It is not hard to prove (see, \cite{Myasnikov_Remeslennikov:1994})
that for every group $G$ there exists an $A$-group $H$ (which is unique up to an $A$-isomorphism)
and a homomorphism $\mu: G \longrightarrow H$ such that for every $A$-group $K$ and every $A$-homomorphism
$\theta: G \longrightarrow K$, there exists a unique $A$-homomorphism $\phi : H \longrightarrow K$
such that $\phi \mu = \theta$. We denote $H$ by $G^A$ and call it the {\em $A$-completion} of $G$.

In \cite{Myasnikov_Remeslennikov:1996} an effective  construction of $\FZt$ was given
in terms of extensions of centralizers. For a group $G$ let  $S = \{C_i \mid i \in I\}$ be a set of
representatives of conjugacy classes of proper cyclic centralizers in $G$, that is, every proper
cyclic centralizer in $G$ is conjugate to one from $S$, and no two centralizers from $S$ are conjugate.
Then the HNN-extension
$$H = \langle \ G, s_{i,j} \ \ (i \in I, j \in \mathbb{N}) \ \mid \ [s_{i,j}, u_i] = [s_{i,j},s_{i,k}] =
1\, (u_i \in C_i, i \in I, j,k \in \mathbb{N})\ \rangle,$$
is termed an {\em extension of cyclic centralizers} in $G$. Now the group $\FZt$ is
isomorphic to the direct limit of the following infinite chain of groups:
\begin{equation}
\label{eq:FZt}
F = G_0 < G_1 < \cdots < G_n < \cdots < \cdots,
\end{equation}
where $G_{i+1}$ is obtained from $G_i$ by  extension of all cyclic centralizers in $G_i$.

\subsection{$\Zt$-exponentiation on $CR(\Zt, X)$}
\label{subs:cr}

Define $t$-exponentiation on $CR(\Zt, X)$ as follows.

\begin{enumerate}
\item[(1)] Let $u \in CR(\Zt, X)$ be not a proper power and such that
$$|u| = f(t) = a_n t^n + a_{n-1} t^{n-1} + \cdots a_1 t + a_0.$$
Observe that for every $\beta \in [1, t^{n+1}]$ there exists $m \geqslant 0$ such that either $\beta
\in [m |u| + 1, (m+1) |u|]$, or $\beta \in [t |u| - (m + 1) |u| + 1, t |u| - m |u|]$. Hence, define
$u^t$ to be an element of $CR(\Zt, X)$ with length $t^{n+1}$ as follows
\[ \mbox{$u^t(\beta)$} = \left\{ \begin{array}{ll}
\mbox{$u(\alpha)$,} & \mbox{if $\beta = m |u| + \alpha, m \geqslant 0, 1 \leqslant \alpha \leqslant
|u|$,} \\
\mbox{$u(\alpha)$,} & \mbox{if $\beta = t |u| - m |u| + \alpha, m > 0, 1 \leqslant \alpha \leqslant
|u|$.}
\end{array}
\right.
\]

\item[(2)] If $v \in CR(\Zt, X)$ is such that $v = u^k$ for some $u \in CR(\Zt,X)$ then we set $v^t
= (u^t)^k$.

Thus we have defined an exponent $v^t$ for a given $v \in CR(\Zt, X)$. Notice that it follows from
the construction that $v^t$ starts with $v$ and ends with $v$. In particular, $v^t \in CR(\Zt,X)$.
It follows that $v^t \ast v = v^t \circ v = v \circ v^t = v \ast v^t$, hence, $[v^t, v] =
\varepsilon$.

\item[(3)]  Now for $v \in CR(\Zt, X)$ we define exponents $v^{t^k}$ by induction. Since $v^t \in
CR(\Zt, X)$ one can repeat the construction from (1) and define
$$v^{t^{k+1}} = (v^{t^k})^t.$$

\item[(4)] Now we define $v^{f(t)}$, where $f(t) \in \Zt$, by linearity, that is, if $f(t) = m_0 +
m_1 t + \ldots + m_k t^k$ then
$$v^{f(t)} = v^{m_0} \ast (v^t)^{m_1} \ast \cdots \ast (v^{t^k})^{m_k}.$$

Observe that the product above is defined because ${v^t}^{m+1}$ is cyclically reduced, and starts
and ends with ${v^t}^m$.
\end{enumerate}

\begin{lemma}
\label{le:cr_exp_1}
Let $v \in CR(\Zt, X), f(t) \in \Zt$. Then $v^{f(t)} \in CR(\Zt, X)$ and $[v^{f(t)}, v] =
\varepsilon,\ v^{-f(t)} = (v^{-1})^{f(t)}$.
\end{lemma}
\begin{proof} Follows directly from the definition of $\Zt$-exponentiation.
\end{proof}

\begin{remark}
\label{le:cr_exp_2}
Observe that under the above definition we loose the property $|u^{f(t)}| = |u| |f(t)|$ for $u \in
\FZt, f(t) \in \Zt$. But, at the same time, we obtain computational advantages which will be clear
later.
\end{remark}

\begin{lemma}
\label{le:cr_exp_3}
Let $u, v \in CR(\Zt, X)$ and $u = c^{-1} \ast v \ast c$ for some $c \in R(\Zt, X)$. Then for every
$f(t) \in \Zt$ we have
$$u^{f(t)} = c^{-1} \ast v^{f(t)} \ast c.$$
\end{lemma}
\begin{proof} Since $u$ and $v$ are cyclically reduced and $u = c^{-1} \ast v \ast c$ then $v = v_1
\circ v_2, u = v_2 \circ v_1, c = v_1$.

In view of (4) in the definition of $\Zt$-exponentiation above it suffices to  prove the lemma
for $f(t) = t^n$. For $f(t) = t$ we immediately get
$$(v_2 \circ v_1)^t = v_1^{-1} \ast (v_1 \circ v_2)^t \ast v_1$$
from the definition. This implies that $v^t$ and $u^t$ are cyclic permutations of each other and both
belong to $CR(\Zt, X)$, therefore one can apply the induction on $\deg{f(t)}$ and the lemma follows.
\end{proof}

\begin{lemma}
\label{le:cr_exp_4}
Let $u, v \in CR(\Zt, X)$ and $f(t), g(t) \in \Zt$ be such that $u^{f(t)} = v^{g(t)}$. Then $[u,v]$
is defined and is equal to $\varepsilon$.
\end{lemma}
\begin{proof} Since $[u, u^{f(t)}] = \varepsilon$ and $[v, v^{g(t)}] = \varepsilon$ then $[u,
v^{g(t)}] = \varepsilon$ and $[v,u^{f(t)}] = \varepsilon$. Now we are going to derive the required
statement from these equalities.

Observe that if $|u| = |v|$ then it follows automatically that $u = v^{\pm 1}$. Indeed, by the
definition of exponents $u^{f(t)}$ and $v^{g(t)}$ have correspondingly $u^{\pm 1}$ and $v^{\pm 1}$
as initial segments. Since $u^{f(t)} = v^{g(t)}$ then initial segments of length $|u|$ in both
coincide.

We can assume $|u| < |v|$ and consider $[u, v^{g(t)}] = \varepsilon$ (if $|u| > |v|$ then we
consider $[v, u^{f(t)}] = \varepsilon$ and apply the same arguments). Also, $g(t) > 1$, otherwise
we have nothing to prove.

Thus we have $u \ast v^{g(t)} = v^{g(t)} \ast u$. Since $u$ and $v$ are cyclically reduced and
equal $\Zt$-words have equal initial and terminal segments of the same length then $[u,v]$ is
defined and we have two cases.

\begin{enumerate}
\item[(a)] Suppose $u \ast v = u \circ v$.

Thus, automatically we have $v \ast u = v \circ u$. Next, $u \circ v^{g(t)}$ and $v^{g(t)} \circ u$
have the same initial segment of length $2|v|$. So $v = u \circ v_1 = v_1 \circ v_2$ and $|u| =
|v_2|$. Comparing terminal segments of $u \circ v^{g(t)}$ and $v^{g(t)} \circ u$ of length $|u|$ we
have $u = v_2$ and from $u \circ v_1 = v_1 \circ u$ it follows that $[u,v] = \varepsilon$.

\item[(b)] Suppose there is a cancellation in $u \ast v$.

Then, from $u^{f(t)} = v^{g(t)}$ it follows that $v^{-1} = v_1^{-1} \circ u$ and so $v = u^{-1}
\circ v_1$. Using the same arguments as in (a) we obtain $v = u^{-1} \circ v_1 = v_1 \circ v_2$,
$|u| = |v_2|$ and $u^{-1} = v_2$. It follows immediately that $[u,v] = \varepsilon$.
\end{enumerate}

\end{proof}

\subsection{$\Zt$-exponentiation on $CDR(\Zt, X)$}
\label{subs:cdr}

Let $v \in CDR(\Zt, X)$ have a cyclic decomposition $v = c^{-1} \circ u \circ c$ and $f(t) \in \Zt$.
We define $v^{f(t)}$ as follows
\begin{equation}
\label{le:cdr_exp_1}
v^{f(t)} = c^{-1} \circ u^{f(t)} \circ c.
\end{equation}
Observe that the product above is well defined since $u^{f(t)}$ starts and ends on $u$ if $f(t) >
0$, and starts and ends on $u^{-1}$ if $f(t) < 0$.

Thus we have defined $\Zt$-exponentiation function
$$\exp : CDR(\Zt, X) \times \Zt \rightarrow CDR(\Zt, X)$$
on the whole set $CDR(\Zt, X)$.

There are other ways of defining $\Zt$-exponentiation on $CDR(\Zt, X)$ but from now on we fix the
exponentiation described above.

\begin{lemma}
\label{le:cdr_exp_2}
Let $u,v \in CDR(\Zt, X)$ be such that $h(u) = h(v)$ and $[u, v] = \varepsilon$. Then $[u^{f(t)}, v]
= \varepsilon$ for any $f(t) \in \Zt$ provided $[u^{f(t)}, v]$ is defined.
\end{lemma}
\begin{proof} We can assume that either $u$ or $v$ is cyclically reduced. This is always possible
because both elements belong to $CDR(\Zt, X)$. Suppose we have $v^{-1} \ast u \ast v = u$, where
$u$ is cyclically reduced.

\begin{enumerate}
\item[(a)] Suppose $|u| < |v|$.

Since $u$ is cyclically reduced either $v^{-1} \ast u = v^{-1} \circ u$ or $u \ast v = u \circ v$.
Assume the former. Then $v$ has to cancel completely in $v^{-1} \ast u \ast v$ because this product
is equal to $u$ which is cyclically reduced. So $v$ has the form $v = u^k \circ w$, where $k < 0$ is
the smallest possible and $w$ does not have $u$ as an initial segment. We have then
$$v^{-1} \ast u \ast v = w^{-1} \ast u \ast w = w^{-1} \ast (u \circ w) = u.$$
and $w^{-1}$ cancels completely. In this case the only possibility is that $|w| < |u|$ (otherwise
we have a contradiction with the choice of $k$) and $[u,w] = \varepsilon$. So now we reduced
everything to the case (b) because clearly $[u^{f(t)}, u^k] = \varepsilon$ for any $f(t) \in \Zt$.

\item[(b)] Suppose $|u| > |v|$.

We have $v^{-1} \ast u \ast v = u$. $u$ is cyclically reduced, moreover, $u$ is a cyclic permutation
of itself that is $v^{-1} \ast u \ast v = u$. Finally, since $[u^{f(t)},v]$ is defined then
$$v^{-1} \ast u^{f(t)} \ast v = u^{f(t)}$$
follows from Lemma \ref{le:cr_exp_3}.

\end{enumerate}
\end{proof}

We summarize the properties of the exponentiation $exp$ in the following theorem.

\begin{theorem}
\label{le:cdr_exp_3}
The $\Zt$-exponentiation function
$$\exp : (u,f(t)) \mapsto u^{f(t)}$$
defined in {\rm (\ref{le:cdr_exp_1})} satisfies the following axioms:
\begin{enumerate}
\item[(E1)] $u^1 = u, \ \ \ u^{f g} = (u^f)^g, \ \ \ u^{f + g} = u^f \ast u^g$,
\item[(E2)] $(v^{-1} \ast u \ast v)^f = v^{-1} \ast u^f \ast v$ \noindent provided $[u,v] =
\varepsilon$ and $h(u) = h(v)$, or $u = v \circ w$, or $u= w^\alpha$, $v = w^\beta$ for some $w \in
CDR(\Zt, X)$ and $\alpha, \beta \in \Zt$,
\item[(E3)] if $[u,v] = \varepsilon$ and $u = w^\alpha,\ v = w^\beta$ for some $w \in CDR(\Zt, X)$
and $\alpha, \beta \in \Zt$ then
$$(u \ast v)^f = u^f \ast v^f$$
\end{enumerate}
\end{theorem}
\begin{proof} Let $u \in CDR(\Zt, X)$ and $\alpha, \beta \in \Zt$.

\begin{enumerate}
\item[(E1)] The equalities $u^1 = u$ and $(u^f)^g = u^{f g}$ follow directly from the definition of
exponentiation. We need to prove only that $u^{f + g} = u^f \ast u^g$. Let
$$u = c^{-1} \circ u_1^k \circ c$$
be a cyclic decomposition of $u$. Then
$$u^f = c^{-1} \circ (u_1^f)^k \circ c, \ \ u^g = c^{-1} \circ (u_1^g)^k \circ c.$$
Now
$$u^{f + g} = c^{-1} \circ (u_1^{f + g})^k \circ c = (c^{-1} \circ (u_1^f)^k \circ c) \ast (c^{-1}
\circ (u_1^g)^k \circ c),$$
as required.

\item[(E2)] If $u = w^\alpha, v = w^\beta$ for some $w \in CDR(\Zt, X)$ and $\alpha, \beta \in \Zt$,
then the result follows from the definition of exponentiation. If $[u,v] = \varepsilon$ and $h(u) =
h(v)$ then result follows from Lemma \ref{le:cdr_exp_2}. If $u = v \circ w$ then the result follows
from Lemma \ref{le:cr_exp_2}.

\item[(E3)] We have $(u \ast v)^f = (w^{\alpha + \beta})^f = w^{(\alpha + \beta)f} = w^{\alpha f}
\ast w^{\beta f} = (w^\alpha)^f \ast (w^\beta)^f = u^f \ast v^f$.
\end{enumerate}
\end{proof}

\subsection{Extension of centralizers in $CDR(\Zt, X)$}
\label{subs:ext_centr}

Below we recall some of the definitions given in \cite{Myasnikov_Remeslennikov_Serbin:2005} and
state a few lemmas and theorems needed for extension of centralizers in $CDR(\Zt, X)$. The proofs
are exactly the same as in the case of $\Zt$-exponentiation introduced in
\cite{Myasnikov_Remeslennikov_Serbin:2005}, so we only give references.

\smallskip

$u,v \in CDR(\Zt, X)$ are called {\em separated} if $u^m \ast v^n$ is defined for any $n,m \in
\mathbb{N}$ and there exists $r = r(u,v) \in \mathbb{N}$ such that for all $m,n > r$
$$u^m \ast v^n = u^{m-r} \circ_\delta (u^r \ast v^r) \circ_\delta v^{n-r}.$$

A subset $M \subseteq CDR(\Zt, X)$ is called an {\em S-set} if any two non-commuting elements of
$M$ with cyclic centralizers are separated. For example, it is easy to see that the free group $F$
is an $S$-subgroup of $CDR(\Zt, X)$.

Next, let $M \subseteq CDR(\Zt, X)$. A subset $R_M \subseteq CR(\Zt, X)$ is called a set of {\em
representatives of $M$} if $R_M$ satisfies the following conditions:
\begin{enumerate}
\item[(1)] $R_M$ does not contain proper powers,
\item[(2)] for any $u, v \in R_M$, $u \neq v^{-1}$,
\item [3)] for each $u \in M$ there exist $v \in R_M$, $k \in \Z$, $c \in R(\Zt, X)$, and a cyclic
permutation $\pi(v)$ of $v$ such that
$$u = c^{-1} \circ \pi(v)^k \circ c,$$
moreover, such $v, c, k, \pi(v)$ are unique.
\end{enumerate}

In \cite{Myasnikov_Remeslennikov_Serbin:2005} it was shown that a set of representatives $R_M$
exists for any $M \subseteq CDR(\Zt, X)$. Observe that $R_M$ does not have to be a subset of $M$.

The next definition we also borrow from \cite{Myasnikov_Remeslennikov_Serbin:2005} but here we
restate it with respect to the $\Zt$-exponentiation we introduced in Subsection \ref{subs:cdr}.

Let $G$ be a subgroup of $CDR(\Zt, X)$ and let
$$K(G) = \{ v \in G \mid C_G(v) = \langle v \rangle \}.$$
A {\em Lyndon's set} of $G$ is a set $R = R_{K(G)}$ of representatives of $K(G)$ which satisfies
the following conditions:
\begin{enumerate}
\item[(1)] $R \subset G$,
\item[(2)] for any $g \in G, u \in R$, and $\alpha \in \Zt$ the inner product $c(u^\alpha, g)$
exists and $c(u^\alpha, g) < k|u|$ for some $k \in \mathbb{N}$,
\item[(3)] no word from $G$ contains a subword $u^\alpha$, where $u \in R$ and $\alpha \in \Zt$ with
${\rm deg}(\alpha) > 0$.
\end{enumerate}

\begin{lemma}
\label{le:ext_centr_1}
Let $G$ be an $S$-subgroup of $CDR(\Zt, X)$ and let $R$ be a Lyndon's set of $G$. If $u, v \in
R^{\pm 1}$ and $g \in G$ are such that either $[u,v] \neq \varepsilon$ or $[u,g] \neq \varepsilon$
then there exists $r \in \mathbb{N}$ such that for all $m, n \geqslant r$ the following holds:
$$u^m \ast g \ast v^n = u^{m-r} \circ (u^r \ast g \ast v^r) \circ v^{n-r}.$$
\end{lemma}
\begin{proof} See the proof of Lemma 6.9 in \cite{Myasnikov_Remeslennikov_Serbin:2005}.
\end{proof}

\begin{lemma}
\label{le:ext_centr_2}
Let $G$ be an $S$-subgroup of $CDR(\Zt, X)$ and $R$ a Lyndon's set of $G$. If $u_1, \ldots, u_n \in
R^{\pm 1}$ and $g_1, \ldots, g_{n+1} \in G$ are such that for any $i = 2, \ldots, n$ either
$[u_{i-1}, u_i] \neq \varepsilon$, or $[u_i, g_i] \neq \varepsilon$ then there exists $r \in
\mathbb{N}$ such that
\begin{eqnarray*}
&&g_1 \ast u_1^{m_1} \ast g_2 \ast \cdots \ast u_n^{m_n} \ast g_{n+1} \\
&&\qquad =(g_1 \ast u_1^r) \circ u_1^{m_1 - 2r} \circ (u_1^r \ast g_2 \ast u_2^r) \circ
u_2^{m_2 - 2r} \circ \cdots \circ u_n^{m_n - 2r} \circ (u_n^r \ast g_{n+1})
\end{eqnarray*}
for all $m_i \in \mathbb{N},\ m_i > 2r,\ i \in [1,n]$.
\end{lemma}
\begin{proof} See the proof of Lemma 6.10 in \cite{Myasnikov_Remeslennikov_Serbin:2005}.
\end{proof}

Let $G$ be a subgroup of $CDR(\Zt, X)$ with a Lyndon's set $R$. A sequence
\begin{equation}
\label{eq:ext_centr_1}
p = (g_1, u_1^{\alpha_1}, g_2, \dots, g_n, u_n^{\alpha_n}, g_{n+1}),
\end{equation}
where $g_i \in G,\ u_i \in R,\ \alpha_i \in \Zt,\ n \geqslant 1$ is called an {\em $R$-form} over
$G$. An $R$-form (\ref{eq:ext_centr_1}) is {\em reduced} if ${\rm deg}(\alpha_i) > 0, i \in
[1,n]$, and if $u_i = u_{i-1}$ then $[u_i, g_i] \neq \varepsilon$.

Denote by ${\mathcal P}(G,R)$ the set of all $R$-forms over $G$. We define a partial function $w:
{\mathcal P}(G,R) \rightarrow R(\Zt, X)$ as follows. If
$$p = (g_1, u_1^{\alpha_1}, g_2, \ldots, g_n, u_n^{\alpha_n}, g_{n+1})$$
then
$$w(p) = (\cdots (g_1 \ast u_1^{\alpha_1}) \ast g_2) \ast \cdots \ast g_n ) \ast u_n^{\alpha_n})
\ast g_{n+1}$$
if it is defined.

An $R$-form $p = (g_1, u_1^{\alpha_1}, g_2, \ldots, g_n, u_n^{\alpha_n}, g_{n+1})$ over $G$ is
called {\em normal} if it is reduced and the following conditions hold:
\begin{itemize}
\item[(1)] $w(p) = g_1 \circ u_1^{\alpha_1} \circ g_2 \circ \cdots \circ g_n \circ u_n^{\alpha_n}
\circ g_{n+1}$,
\item[(2)] $g_i$ does not have $u_i^{\pm 1}$ as a terminal segment for any $i \in [1,n]$ and $g_i
\circ u_i^{\alpha_i}$ does not have $u_{i-1}^{\pm 1}$ as an initial segment for any $i \in [2,n]$.
\end{itemize}

\begin{lemma}
\label{le:ext_centr_3}
Let $G$ be an $S$-subgroup of $CDR(\Zt, X)$ with a Lyndon's set $R$. Then for every $R$-form $p$
over $G$ the following holds:
\begin{enumerate}
\item[(1)] the product $w(p)$ is defined and it does not depend on the placement of parentheses,
\item[(2)] there exists a reduced $R$-form $q$ over $G$ such that $w(q) = w(p)$,
\item[(3)] there exists a unique normal $R$-form $q$ over $G$ such that $w(p) = w(q)$,
\item[(4)] $w(p) \in CDR(\Zt, X)$.
\end{enumerate}
\end{lemma}
\begin{proof} See the proof of Lemma 6.13 in \cite{Myasnikov_Remeslennikov_Serbin:2005}.
\end{proof}

Let $G$ and $R$ be as above. By Lemma \ref{le:ext_centr_3} for every $g, h \in G,\ u \in R,\
\alpha \in \Zt$ the product $g \ast u^\alpha \ast h$ is defined and belongs to $CDR(\Zt, X)$. Put
$$P = P(G,R) = \{ g \ast u^\alpha \ast h \mid g,h \in G, u \in R, \alpha \in \Zt \}.$$
Multiplication $\ast$ induces a partial multiplication (which we again denote by $\ast$) on $P$ so
that for $p, q \in P$ the product $p \ast q$ is defined in $P$ if and only if $p \ast q$ is defined
in $R(\Zt, X)$ and $p \ast q \in P$. Now we are ready to prove the main technical result of this
subsection.

\begin{lemma}
\label{le:ext_centr_4}
Let $G$ be an $S$-subgroup of $CDR(\Zt, X)$ and let $R$ be a Lyndon's set for $G$. Then the set $P$
forms a pregroup with respect to the multiplication $\ast$.
\end{lemma}
\begin{proof} See the proof of Proposition 6.14 in \cite{Myasnikov_Remeslennikov_Serbin:2005}.
\end{proof}

The next two results reveal the structure of the universal group $U(P)$ of $P$.

\begin{theorem}
\label{th:ext_centr_2}
$P$ generates a subgroup $\langle P \rangle$ in $CDR(\Zt, X)$, which is isomorphic to $U(P)$.
\end{theorem}
\begin{proof} See the proof of Theorem 6.15 in \cite{Myasnikov_Remeslennikov_Serbin:2005}.
\end{proof}

Let $R = \{c_i \mid i \in I\}$. Put $S = \{s_{i,j} \mid i \in I, j \in \mathbb{N}\}$. Then the
group
$$G(R,S) = \langle G, S \mid [c_i,s_{i,j}] = [s_{i,j},s_{k,j}] = 1, i \in I, j,k \in \mathbb{N}
\rangle$$
is an extension of all cyclic centralizers of $G$ by a direct sum of countably many copies of an
infinite cyclic group. Sometimes, we will refer to $G(R,S)$ as an extension of all cyclic
centralizers of $G$ by $\Zt$.

\begin{theorem}
\label{th:ext_centr_3}
$\langle P \rangle \simeq G(R,S)$.
\end{theorem}
\begin{proof} Define $\phi : P \rightarrow G(R,S)$ as follows. Let $g_i \ast c_i^\alpha \ast h_i
\in P$ and $\alpha = a_n t^n + a_{n-1} t^{n-1} + \cdots + a_1 t + a_0$. Put
$$g_i \ast c_i^\alpha \ast h_i \ \stackrel{\phi}{\rightarrow}\ g_i\ s_{i,n}^{a_n}\
s_{i,n-1}^{a_{n-1}}\ \cdots\ s_{i,1}^{a_1}\ c_i^{a_0}\ h_i.$$
It is easy to see that $\phi$ is a morphism of pregroups. Since $\langle P \rangle \simeq U(P)$,
the morphism $\phi$ extends to a unique homomorphism $\psi : \langle P \rangle \rightarrow G(R,S)$.
We claim that $\psi$ is bijective. Indeed, observe first that $G(R,S)$ is generated by $G \cup S$.
Now, since $\psi(c_i^{t^j}) = s_{i,j}$ and $\psi$ is identical on $G$, it follows that $\psi$ is
onto. To see that $\psi$ is one-to-one it suffices  to notice that if
$$y = (g_1 \ast c_1^{\alpha_1} \ast h_1, g_2 \ast c_2^{\alpha_2} \ast h_2, \ldots, g_m \ast
c_m^{\alpha_m} \ast h_m).$$
is a reduced $R$-form then $y^\psi \neq 1$ by Britton's Lemma (see, for example,
\cite{Magnus_Karras_Solitar:1977}). This proves that $\psi$ is an isomorphism, as required.
\end{proof}

\begin{lemma}
\label{le:ext_centr_5}
If $G$ is subwords-closed then so is $\langle P \rangle$.
\end{lemma}
\begin{proof} See the proof of Lemma 6.18 in \cite{Myasnikov_Remeslennikov_Serbin:2005}.
\end{proof}

\begin{lemma}
\label{le:ext_centr_6}
If $G$ is subwords-closed then $H$ is an $S$-subgroup.
\end{lemma}
\begin{proof} See the proof of Lemma 6.19 in \cite{Myasnikov_Remeslennikov_Serbin:2005}.
\end{proof}

\subsection{Embedding of $\FZt$ into $CDR(\Zt, X)$}
\label{subs:embedding}

Let $F$ be a free non-abelian group. Recall that one can view the group $\FZt$ as a union of the
following infinite chain of groups:
\begin{equation}
\label{eq:embedding_1}
F = G_0 < G_1 < G_2 < \cdots < G_n < \cdots,
\end{equation}
where $G_n$ is obtained from $G_{n-1}$ by extension of all cyclic centralizers of $G_{n-1}$.

For each $n \in \mathbb{N}$ we construct by induction an embedding
$$\psi_n : G_n \rightarrow  CDR(\Zt, X)$$
such that $\psi_{n-1}$ is the restriction of $\psi_n$ to $G_{n-1}$. To this end, let $H_0$ be the
set of all words of finite length in $CDR(\Zt, X)$. Clearly, $F = H_0$. We denote by $\psi_0 : F
\rightarrow H_0$ the identity isomorphism. It is obvious that $H_0$ is subwords-closed. Moreover,
$H_0$ is an $S$-subgroup and it has a Lyndon's set.

Suppose by induction that there exists an embedding
$$\psi_{n-1} : G_{n-1} \rightarrow CDR(\Zt, X)$$
such that the image $H_{n-1} = \psi_{n-1}(G_{n-1})$ is an $S$-subgroup, it is subwords-closed, and
there exists a Lyndon's set, say $R_{n-1}$, in $H_{n-1}$. Then by Proposition \ref{le:ext_centr_4}
and Theorem \ref{th:ext_centr_2}, there exists an embedding $\psi_n :G_n \rightarrow CDR(\Zt,X)$.
Moreover, in this case, the image $H_n = \psi_n(G_n)$ is the subgroup of $CDR(\Zt, X)$ generated
by the pregroup
$$P(H_{n-1},R_{n-1}) = \{f \ast u^\alpha \ast h \ \mid \ f,h \in H_{n-1}, \ u \in R_{n-1}, \ \alpha
\in \Zt \}.$$
Notice that by Lemma \ref{le:ext_centr_6}, the group $H_n$ is an $S$-subgroup of $CDR(\Zt, X)$, and
by Lemma \ref{le:ext_centr_5}, $H_n$ is subwords-closed. So to finish the proof one needs to show
that $H_n$ has a Lyndon's set.

\begin{lemma}
\label{le:embedding_1}
Let $H_{n-1}$ from the series {\rm (\ref{eq:embedding_1})} be a subwords-closed $S$-subgroup of
$CDR(\Zt, X)$ with a Lyndon's set $R_{n-1}$. Then there exists a Lyndon's set $R_n$ in $H_n$.
\end{lemma}
\begin{proof} Recall that $K = K(H_n) \subset H_n$ is the subset consisting of all elements $v \in
H_n$ such that $C_{H_n}(v) = \langle v \rangle$. Denote by $R$ a set of representatives for $K$.

Since  $H_n$ is subwords-closed then we may assume that $R \subset H_n$. The same argument shows
that an element $f \in H_n$ does not contain a subword $u^\alpha$, where $u \in R$ and $\alpha \in
\Zt$ is infinite. Indeed, in this case it would imply that $u^\alpha \in H_n$, hence, $[u^\alpha,
u] = \varepsilon$, so the centralizer of $u$ in $H_n$ is not cyclic -- a contradiction with $u \in
R$. Finally, let $u \in R,\ g \in H_n$. Observe that $u \notin H_{n-1}$, so $u$ has a unique normal
form
$$u = f_1 \circ u_1^{\alpha_1} \circ f_2 \circ \dots \circ u_k^{\alpha_k} \circ f_{k+1},$$
where $f_i \in H_{n-1},\ u_i \in R_{n-1}$, and $\alpha_i \in \Zt$ is infinite for any $i \in [1,k]$.
If $g \in H_{n-1}$ then
\begin{equation}
\label{eq:embedding_2}
(g \ast u^m) \ast u = (g \ast u^m) \circ u, \ \ \ u \ast (u^m \ast g) = u \circ (u^m \ast g)
\end{equation}
holds for $m = 1$, since $R_{n-1}$ is a Lyndon's set for $H_{n-1}$. If $g \notin H_{n-1}$ then
$$g = g_1 \circ v_1^{\beta_1} \circ g_2 \circ \dots \circ v_l^{\beta_l} \circ g_{p+1},$$
where $g_j \in H_{n-1},\ v_j \in R_{n-1}$ and $\beta_j \in \Zt$ is infinite for any $j \in [1,p]$.
In this case (\ref{eq:embedding_2}) holds for any $m > p$.

It follows that the set $R_n = R$ is a Lyndon's set for $H_n$.
\end{proof}

\subsection{Limit groups embed into $\FZt$}
\label{subs:limit_gps_embed}

The following results illustrate the connection of limit groups and finitely generated
subgroups of $\FZt$.

\begin{theorem} \cite{Kharlampovich_Myasnikov:1998(2)}
\label{th:embKM}
Given a finite presentation  of a finitely generated fully residually free group $G$ one can
effectively construct an embedding $\phi : G \rightarrow \FZt$ (by specifying the images of
the generators of $G$).
\end{theorem}

Combining Theorem \ref{th:embKM} with the result on the representation of $\FZt$ as a union of a
sequence of extensions of centralizers one can get the following theorem.

\begin{theorem} \cite{Kharlampovich_Myasnikov:2005(2)}
\label{th:embKMR}
Given a finite presentation of a finitely generated fully residually free group $G$ one can
effectively construct a finite sequence of extension of centralizers
$$F < G_1 < \cdots < G_n,$$
where $G_{i+1}$ is an extension of the  centralizer of some element $u_i \in G_i$ by an infinite
cyclic group $\Z$, and an embedding $\psi^\ast : G \rightarrow G_n$ (by specifying the images of
the generators of $G$).
\end{theorem}

Now Theorem \ref{th:embKMR} implies the following important corollaries.

\begin{cor} \cite{Kharlampovich_Myasnikov:2005(2)}
\label{co:embKMR_1}
For every freely indecomposable non-abelian finitely generated fully residually free group one can
effectively find a non-trivial splitting (as an amalgamated product or HNN extension) over a cyclic
subgroup.
\end{cor}

\begin{cor} \cite{Kharlampovich_Myasnikov:2005(2)}
\label{co:embKMR_2}
Every finitely generated fully residually free group is finitely presented. There is an algorithm
that, given a presentation of a finitely generated fully residually free group $G$ and generators
of the subgroup $H$, finds a finite presentation for $H$.
\end{cor}

\begin{cor} \cite{Kharlampovich_Myasnikov:2005(2)}
\label{co:embKMR_3}
Every finitely generated residually free group $G$ is a subgroup of a direct product of finitely
many fully residually free groups; hence, $G$ is embeddable into $\FZt \times \cdots \times \FZt$.
If $G$ is given as a coordinate group of a finite system of equations, then this embedding can be
found effectively.
\end{cor}

Let $K$ be an HNN-extension of a group $G$ with associated subgroups $A$ and $B$. $K$ is called a
separated HNN-extension if for any $g\in G,\ A^g \cap B = 1$.

\begin{cor} \cite{Kharlampovich_Myasnikov:2005(2)}
\label{co:embKMR_4}
Let a group $G$ be obtained from a free group $F$ by finitely many centralizer extensions. Then
every finitely generated subgroup $H$ of $G$ can be obtained from free abelian groups of finite
rank by finitely many operations of the following type: free products, free products with abelian
amalgamated subgroups at least one of which is a maximal abelian subgroup in its factor, free
extensions of centralizers, separated HNN-extensions with abelian associated subgroups at least
one of which is maximal.
\end{cor}

\begin{cor} \cite{Kharlampovich_Myasnikov:2005(2),Groves_Wilton:2009}
\label{co:embKMR_5}
One can enumerate all finite presentations of finitely generated fully residually free groups.
\end{cor}

\begin{cor} \cite{Kharlampovich_Myasnikov:1998(2)}
\label{co:embKMR_6}
Every finitely generated fully residually free group acts freely on some $\Z^n$-tree with
lexicographic order for a suitable $n$.
\end{cor}

Finally, combining Theorem \ref{th:embKM} with the result on the effective embedding of $\FZt$ into
$R(\Zt,X)$ obtained in \cite{Myasnikov_Remeslennikov_Serbin:2005} one can get the following theorem.

\begin{theorem}
\label{th:embMR}
Given a finite presentation  of a finitely generated fully residually free group $G$ one can
effectively construct an embedding $\psi : G \rightarrow R(\Zt, X)$ (by specifying the images of
the generators of $G$).
\end{theorem}

\subsection{Algorithmic problems for limit groups}
\label{subs:alg_prob_limit_gps}

Existence of a free length function on $\FZt$ becomes a very powerful tool in solving various
algorithmic problems for subgroups of $\FZt$ (which are precisely limit groups). Using the length
function one can introduce unique normal forms for elements of $\FZt$ and then work with them
pretty puch in the same way as in free groups. In the seminal paper \cite{Stallings:1983} J.
Stallings introduced an extremely useful notion of a folding of graphs and initiated the study of
subgroups (and automorphisms) of free groups via folded directed labeled graphs. This approach
turned out to be very influential and allowed researches to prove many new results and simplify
old proofs (see \cite{Kapovich_Myasnikov:2002}). In \cite{Myasnikov_Remeslennikov_Serbin:2006}
Stallings techniques were generalized in order to effectively solve the Membership Problem for
finitely generated subgroups of $\FZt$ and this result can be reformulated for limit groups.

\begin{theorem}\cite{Myasnikov_Remeslennikov_Serbin:2006}
Let $G$ be a limit group and $G \hookrightarrow \FZt$ the effective embedding. For any f.g. subgroup
$H \leq G$ one can effectively construct a finite labeled graph $\Gamma_H$ that in the group $\FZt$ accepts
precisely the normal forms of elements from $H$.
\end{theorem}

Now, that the graph $\Gamma_H$ is constructed for $H$, one can use it to solve a lot of algorithmic
problems almost as in free groups. Recall that limit groups are finitely presented.

\begin{theorem}
Let $G$ be a limit group given by a finite presentation with additional information that it is a limit group or given as a subgroup of $\FZt$, and $H, K$ f.g. subgroups of $G$ given by their generators. Then
\begin{itemize}
\item $H \cap K$ is f.g (Howson Property) and can be found effectively
(\cite{Kharlampovich_Myasnikov_Remeslennikov_Serbin:2006}),
\item up to conjugation by elements from $K$ there are only finitely many subgroups of
$G$ of the type $H^g \cap K$, where $g \in G$ (\cite{Kharlampovich_Myasnikov_Remeslennikov_Serbin:2006}),
\item it can be decided effectively if $H$ is malnormal in $G$ (\cite{Kharlampovich_Myasnikov_Remeslennikov_Serbin:2006}),
\item it can be decided if $H^g = K$ (and $H^g \leq K$) for some $g \in G$ and if yes such $g$ can
be found effectively (\cite{Kharlampovich_Myasnikov_Remeslennikov_Serbin:2006}),
\item for any $g \in G$, its centralizer $C_G(g)$ in $G$ can be found effectively
(\cite{Kharlampovich_Myasnikov_Remeslennikov_Serbin:2006}),
\item homological and cohomological dimensions of $G$ can be computed effectively (\cite{KMRS:2008}),
\item it can be decided if $|G : H| < \infty$ (\cite{Nikolaev_Serbin:2011}),
\item $Comm_G(H)$ can be found effectively, hence it is possible to find effectively $n(H) \in \mathbb{N}$
such that for any $P \leq G$ if $|P : H | < \infty$ then $|P : H | < n(H)$ (\cite{Nikolaev_Serbin:2011}).
\end{itemize}
\end{theorem}

The following theorems use the elimination process for limit groups. For a limit group $G$, two
homomorphisms $\phi_i : G \rightarrow F,\ i = 1,2$ are automorphically equivalent if $\phi_1$ is
obtained by pre-composing $\phi_2$ with an automorphism of $G$ and post-composing with conjugation.
The quotient group of $G$ over the intersection of the kernels of all homomorphisms minimal in their
equivalence classes is called the maximal standard quotient of $G$ (or the shortening quotient).

\begin{theorem} (\cite[Theorem 13.1]{Kharlampovich_Myasnikov:2005(2)},
\cite[Theorem 35]{Kharlampovich_Myasnikov:2006})
\label{modulo}
Let $G$ be a freely indecomposable limit group and $K_1, \ldots, K_m$ be finitely generated subgroups
of $G$. There exists an algorithm to obtain an abelian $JSJ$-decomposition of $G$ with subgroups $K_1,
\ldots, K_m$ being elliptic, and  to find the maximal standard quotient (or shortening quotient)
with respect to this decomposition.
\end{theorem}

\begin{theorem} \cite{Bumagin_Kharlampovich_Myasnikov:2007}
Let $G \cong \langle \mathcal{S}_G \mid \mathcal{R}_G \rangle$ and $H \cong \langle \mathcal{S}_H
\mid \mathcal{R}_H \rangle$ be finite presentations of limit groups. There exists an algorithm that
determines whether or not $G$ and $H$ are isomorphic. If the groups are isomorphic, then the
algorithm finds an isomorphism $G \rightarrow H$.
\end{theorem}

\begin{theorem} \cite{Kharlampovich_Myasnikov:2005(2)}
The universal theory of a limit group $G$ in the language with coefficients from $G$ is decidable
(in the language without coefficients the universal theory of $G$ is the same as the universal
theory of a free group \cite{Remeslennikov:1989}).
\end{theorem}

\section{$\mathbb{Z}^n$-free groups}
\label{sec:Z^n}

In this section we consider in detail the situation when $\Lambda = \mathbb{Z}^n$ with the right
lexicographic order. The structure of $\mathbb{Z}^n$-free groups is very clear and the machinery
of infinite words plays a significant role in all proofs.

\subsection{Complete $\mathbb{Z}^n$-free groups}
\label{subs:regular_Z^n}

In this section we fix a finitely generated group $G$ which has a free regular length function with
values in $\mathbb{Z}^n,\ n \in \mathbb{Z}$ (with the right lexicographic order). In other words,
$G$ is a {\it complete} $\mathbb{Z}^n$-free group (see \cite{KMRS:2012}). Due to Theorem
\ref{chis-cor-1} we may and will view $G$ as a subgroup of $CDR(\mathbb{Z}^n, X)$ for an appropriate
set $X$. Therefore, elements of $G$ are infinite words from $CDR(\mathbb{Z}^n,X)$, multiplication in
$G$ is the multiplication ``$\ast$'' of infinite words, and the regular length function is the
standard length $| \cdot |$ of infinite words. That is, $G$ is complete in the sense that it is
closed under the operation of taking common initial subwords of its elements.

In the ordered group $\mathbb{Z}^n = \langle a_1 \rangle \oplus \ldots \oplus \langle a_n \rangle$ with basis $a_1, \ldots, a_n$ the subgroups $E_k = \langle a_1,
\ldots,$ $a_k\rangle$ are convex, and every non-trivial convex subgroup is equal to $E_k$ for some
$k$, so
$$0 = E_0 < E_1 < \cdots < E_n,$$
is the complete chain of convex subgroups of $\mathbb{Z}^n$. Recall, that the height $ht(g)$ of a
word $g \in G$ is equal to $k$ if $|g| \in E_k - E_{k-1}$ (see \cite{Khan_Myasnikov_Serbin:2007}).
Since $|g \ast h| \leqslant |g| + |h|$ and $|g^{-1}| = |g|$ one has for any $f, g \in G$:
\begin{enumerate}
\item $ht(f \ast g) \leqslant \max\{ht(f),ht(g)\}$,
\item $ht(g) = ht(g^{-1})$.
\end{enumerate}
We will assume that there is an element $g \in G$ with $ht(g) = n$, otherwise, the length function
on $G$ has  values in $\Z^{n-1}$, in which case we  replace $\Z^n$ with $\Z^{n-1}$. For any $k
\in [1,n]$
$$G_k = \{g \in G \mid ht(g) \leqslant k\}$$
is a subgroup of $G$ and
$$1 = G_0 < G_1 < \cdots < G_n = G.$$
Observe, that if $l:G \to \Lambda$ is a Lyndon length function with values in some ordered abelian
group $\Lambda$ and $\mu:\Lambda \to \Lambda^\prime$ is a homomorphism of ordered abelian groups
then the composition $l^\prime = \mu \circ l$ gives a Lyndon length function $l^\prime : G \to
\Lambda^\prime$. In particular, since $E_{n-1}$ is a convex subgroup of $\Z^n$ then the canonical
projection $\pi_n:\Z^n \to \Z$ such that  $\pi_n(x_1, x_2, \ldots, x_n) = x_n$ is an ordered
homomorphism, so the composition $\pi_n \circ |\cdot|$ gives a Lyndon length function $\lambda:
G \to \Z$ such that $\lambda(g) = \pi_n(|g|)$. Notice also that if $u = g \circ h$ then $\lambda(u)
= \lambda(g) + \lambda(h)$ for any $g, h, u \in G$.

All the proofs and details can be found in \cite{KMRS:2012}.

\subsubsection{Elementary transformations of infinite words}
\label{subsec:elem-moves}

In this section we describe  an analog of Nielsen reduction in the group $G$. Since $G$ is complete
the Nielsen reduced sets have much stronger non-cancelation properties then usual and the
transformations are simpler. On the other hand, since $\Z^n$ is non-Archimedean (for $n > 1$) the
reduction process is more cumbersome, it goes in stages along the complete series of convex subgroups
in $\Z^n$.

For a finite subset  $Y$ of $G$ define its $\lambda$-length as
$$|Y|_\lambda = \sum_{g \in Y} \lambda(g).$$
If $Y$ is a generating set of $G$ then  $|Y|_\lambda >0$, otherwise $G = G_{n-1}$. It follows that
$$Y = Y_+ \cup Y_0,$$
where
$$Y_+ = \{g \in Y \mid \lambda(g) > 0\},\ Y_{0} = \{g \in Y \mid \lambda(g) = 0\}.$$
Obviously, $|Y|_\lambda = |Y_+|_\lambda$ and $\langle Y_{0} \rangle$ is a finitely generated subgroup
of $G_{n-1}$.

Let $Y$ be a finite generating set for $G$. Assuming $Y = Y^{-1}$ we define three types of elementary
transformations of $Y$.

\medskip {\bf Transformation $\mu$}. Let  $f,g \in Y_+,\ f \neq g$,
$h \in \langle Y_0 \rangle$,  $u = com(f,h\ast g)$, and $\lambda(u) >0$. Then
$f = u \circ w_1,\ h \ast g = u \circ w_2$ for some $u, w_1, w_2$ from $G$ (since $G$ is complete).
Put
$$\mu_{f,g,h}(Y) = (Y - \{f^{\pm 1}, g^{\pm 1}\}) \bigcup \{w_1^{\pm 1}, w_2^{\pm 1}, u^{\pm 1}\}$$
$$\bigcup \{(f^{-1} \ast h \ast g)^{\pm 1} \mid {\rm if}\ \lambda(f^{-1} \ast h \ast g) = 0\}$$
if $f \neq g^{-1}$, and
$$\mu_{f,g,h}(Y) = (Y - \{g^{\pm 1}\}) \bigcup \{u^{\pm 1}, (w_2 \ast u)^{\pm 1}\}$$
if $f = g^{-1}$.

\begin{lemma} \cite{KMRS:2012}
\label{le:mu}
In the notation above
\begin{enumerate}
\item[(1)] $\langle Y \rangle = \langle \mu_{f,g,h}(Y) \rangle$,
\item[(2)] $|\mu_{f,g,h}(Y)|_\lambda < |Y|_\lambda$.
\end{enumerate}
\end{lemma}

\medskip

{\bf Transformation $\eta$}. Let $f \in Y_+$ be
such that $\lambda(f) > \lambda(com(f,h\ast f)) > 0$ for some $h \in \langle Y_0
\rangle$. Then $f = u \circ f_1,\ h \ast f = u \circ f_2,\ \lambda(u) > 0$.
Define
$$\eta_{f,h}(Y) = (Y - \{f^{\pm 1}\}) \bigcup \{f_1^{\pm 1}, u^{\pm 1}, (u^{-1} \ast h \ast u)^{\pm 1}\}.$$
Notice, that $f = (h^{-1} \ast u) \circ f_2 = u \circ f_1$ hence $\lambda(f_2) = \lambda(f_1) > 0$.
On the other hand $f_2 = (u^{-1}\ast h\ast u) \ast f_1$ and it follows that
$\lambda(u^{-1} \ast h \ast u) = 0$.

\begin{lemma} \cite{KMRS:2012}
\label{le:eta}
In the notation above
\begin{enumerate}
\item[(1)] $\langle Y \rangle = \langle \eta_{f,h}(Y) \rangle$,
\item[(2)] either $|\eta_{f,h}(Y)|_\lambda < |Y|_\lambda$, or $|\eta_{f,h}(Y)|_\lambda = |Y|_\lambda$
but then $|{\eta_{f,h}(Y)}_+| > |Y_+|$.
\end{enumerate}
\end{lemma}

\medskip

{\bf Transformation $\nu$}. Let $f \in Y_+$ be not cyclically reduced. Then $f = c^{-1} \circ
\overline{f} \circ c$, where $c \neq 1$ and $\overline{f}$ is cyclically reduced. In this case
$c^{-1} = com(f,f^{-1})$, hence (since $G$ is complete) $c, \overline{f} \in G$.
Put
$$\nu_f(Y) = (Y - \{f^{\pm 1}\}) \bigcup \{c^{\pm 1}, \overline{f}^{\pm 1}\}.$$

\begin{lemma} \cite{KMRS:2012}
\label{le:nu}
In the notation above
\begin{enumerate}
\item[(1)] $\langle Y \rangle = \langle \nu_f(Y) \rangle$,
\item[(2)] either $|\nu_f(Y)|_\lambda < |Y|_\lambda$, or $|\nu_f(Y)|_\lambda = |Y|_\lambda$ but
then $|\nu_f(Y)_+| = |Y_+|$.
\end{enumerate}
\end{lemma}

We write $Y \to Y'$ ($Y \to^\ast Y'$) if $Y'$ is obtained from $Y$ by a single (finitely many)
elementary transformation, that is, $\to^\ast$ is the transitive closure of the relation $\to$.
We call a generating set $Y$ of $G$ {\it transformation-reduced}  if none of the transformations
$\mu, \eta, \nu$ can be applied to $Y$. Recall that the binary relation $\to^\ast$ is called {\it
terminating} if there is no an infinite sequence of finite subsets $Y_i, i\in \N$, of $G$ such that
$Y_i \to Y_{i+1}$ for every $i \in \N$, i.e., every rewriting system $Y_1 \to Y_2 \to \ldots$ is
finite. We say that $\to^\ast$ is {\em uniformly terminating} if for every finite set $Y$ of $G$
there is a  natural number $n_Y$ such that every rewriting system starting at $Y$ terminates in at
most $n_Y$ steps.

\begin{prop} \cite{KMRS:2012}
\label{pr:min}
The following hold:
\begin{enumerate}
\item [1)] The relation $\to^\ast$ is uniformly terminating. Moreover, for any
finite subset $Y$ of $G$ one has $n_Y \leqslant (|Y|_\lambda)^3$. In particular,
there exists a finite transformation-reduced $Z \subset G$ which can be
obtained from $Y$ in not more than $(|Y|_\lambda)^3$ steps.
\item [(2)] If $Z$ is a transformation-reduced finite subset of $G$ then:
\begin{itemize}
\item[(a)] all elements of $Z_+$ are cyclically reduced;
\item[(b)] if $f,g \in Z_+^{\pm 1},\ f \neq g$ then $\lambda(com(f,h\ast g)) = 0$
for any $h \in \langle Z_0 \rangle$;
\item[(c)] if $f \in Z_+^{\pm 1}$ and $\lambda(com(f,h\ast f)) > 0$ for some
$h \in \langle Z_0 \rangle$ then $\lambda(com(f,h\ast f)) = \lambda(f)$.
\end{itemize}
\item [(3)] If $Z$  is a transformation-reduced finite subset of $G$ then one can add to $Z$
finitely many elements $h_1, \ldots, h_m \in G_{n-1}$ such that $T = Z \cup \{h_1, \ldots, h_m\}$
is transformation-reduced and satisfies the following condition
\begin{itemize}
\item[(d)] if $f \in T_+^{\pm 1}$ and $\lambda(com(f,h\ast f)) > 0$ for some
$h \in \langle T_0 \rangle$ then $\lambda(com(f,h\ast f)) = \lambda(f)$ and
$f^{-1} \ast h \ast f \in \langle T_0 \rangle$.
\end{itemize}
\end{enumerate}
\end{prop}

A finite set $Y$ of $G$ is called {\em reduced} if it satisfies the conditions (a) - (d) from
Proposition \ref{pr:min}.

\subsubsection{Minimal sets of generators and pregroups}
\label{subsec:minimal}

Let $Z$ be a finite reduced generating set of $G$. Put
$$P_Z = \{g \ast f \ast h \mid f \in Z_+^{\pm 1},\ g,h \in \langle Z_0 \rangle \} \cup \langle Z_0
\rangle.$$
Multiplication $\ast$ induces a partial multiplication (which we again denote by $\ast$) on $P_Z$
so that for $p, q \in P_Z$ the product $p \ast q$ is defined in $P_Z$ if and only if $p \ast q \in
P_Z$. Notice, that $P_Z$ is closed under inversion.

\begin{lemma} \cite{KMRS:2012}
\label{le:claim}
Let $x = h_1(x) \ast f_x \ast h_2(x),\ y = h_1(y) \ast f_y \ast h_2(y) \in P_Z$, where $h_i(x),
h_i(y) \in \langle Z_0 \rangle,\ i = 1,2$ and $f_x, f_y \in Z_+^{\pm 1}$. Then $x \ast y \in P_Z$
if and only if $f_x = f_y^{-1}$ and $f_x \ast (h_2(x) \ast h_1(y)) \ast f_x^{-1} \in \langle Z_0
\rangle$.
\end{lemma}

Now we are ready to prove the main technical result of this section.

\begin{theorem} \cite{KMRS:2012}
\label{pr:pregroup}
Let $G$ be a finitely generated complete $\Z^n$-free group. Then:
\begin{enumerate}
\item[(1)] $P_Z$ forms a pregroup with respect to the multiplication $\ast$ and inversion;
\item[(2)] the inclusion $P_Z \to G$ extends to the group isomorphism $U(P_Z) \to G$, where $U(P)$
is the universal group of $P_Z$;
\item[(3)] if $(g_1,\ldots,g_k)$ is a reduced $P_Z$-sequence for an element $g \in G$ then
$$\lambda(g) = \sum^k_{i=1} \lambda(g_i).$$
\end{enumerate}
\end{theorem}
\begin{proof}
Observe that $P_Z = P_Z^{-1} \subset G$ generates $G$ and every $g \in G$ corresponds to a finite
reduced $P_Z$-sequence
$$(u_1, u_2, \ldots, u_k),$$
where $u_i \in P_Z,\ i \in [1,k],\ u_i \ast u_{i+1} \notin P_Z,\ i \in [1,k-1]$ and $g = u_1 \ast u_2
\ast \cdots \ast u_k$ in $G$. By Theorem 2, \cite{Rimlinger:1987(2)}, to prove that $P_Z$ is a
pregroup and the inclusion $P_Z \to G$ extends to the isomorphism $U(P_Z) \to G$ it is enough to show
that all reduced $P_Z$-sequences representing the same element have the same $P_Z$-length.

Suppose two reduced $P_Z$-sequences
$$(u_1, u_2, \ldots, u_k),\ (v_1, v_2, \ldots, v_n)$$
represent the same element $g \in G$. That is,
$$(u_1 \ast \cdots \ast u_k) \ast (v_1 \ast \cdots \ast v_n)^{-1} = \varepsilon.$$
We use induction on $k+n$ to show that $k = n$. If the $P_Z$-sequence
$$(u_1, \ldots, u_k, v_n^{-1}, \ldots, v_1^{-1})$$
is reduced then
$$u_1 \ast \ldots \ast u_k \ast v_n^{-1} \ast \ldots \ast v_1^{-1} \neq \varepsilon$$
because $Z$ is a reduced set. Hence,
$$(u_1, \ldots, u_k, v_n^{-1}, \ldots, v_1^{-1})$$
is not reduced and $u_k \ast v_n^{-1} \in P_Z$. If $u_k = h_1 \ast f_1 \ast g_1,\ v_n = h_2 \ast f_2
\ast g_2$, where $h_i, g_i \in \langle Z_0 \rangle$ and $f_i \in Z_+^{\pm 1},\ i = 1,2$ then by Lemma
\ref{le:claim} $f_1 = f_2$ and $f_1 \ast (g_1 \ast g_2^{-1}) \ast f_2^{-1} = c \in \langle Z_0 \rangle$.
It follows that
$$(u_1, u_2, \ldots, u_{k-1} \ast (h_1 \ast c \ast h_2^{-1})),\ (v_1, v_2, \ldots, v_{n-1})$$
represent the same element $g \ast v_n^{-1} \in G$ and the sum of their lengths is less than $k + n$,
so the result follows by induction. Hence, (1) and (2) follow.

\smallskip

Finally we prove (3).

If $g_i = h_1(g_i) \ast f_{g_i} \ast h_2(g_i),\ i \in [1,k]$ then $\lambda(g_i) = \lambda(f_{g_i})$
because $\lambda(h_1(g_i)) = \lambda(h_2(g_i)) = 0$. On the other hand, since $Z$ is reduced and
$(g_1,\ldots,g_k)$ is a reduced $P_Z$-sequence then $\lambda(com(g_i^{-1},g_{i+1})) = 0$ for $i \in
[1,k-1]$. In other words $\lambda(g_i \ast g_{i+1}) = \lambda(g_i) + \lambda(g_{i+1})$ and the result
follows.
\end{proof}

\begin{cor} \cite{KMRS:2012}
\label{co:1}
$G_{n-1} = \langle Z_0 \rangle$.
\end{cor}

\subsubsection{Algebraic structure of complete $\mathbb{Z}^n$-free groups}
\label{subsec:structure}

\begin{theorem} \cite{KMRS:2012}
\label{th:main}
Let $G$ be a finitely generated complete $\Z^n$-free group and let $Z$ be a reduced generating
set for $G$. Then $G$ has the following presentation
$$G = \langle H, Y \mid t_i^{-1} C_H(u_{t_i}) t_i = C_H(v_{t_i}),\ t_i \in Y^{\pm 1} \rangle,$$
where $Y = Z_+$ is finite, $H = G_{n-1} = \langle Z_0 \rangle$ is finitely generated and
$C_H(u_{t_i})$ $C_H(v_{t_i})$ are either trivial or finitely generated free abelian subgroups of
$H$. Moreover, $H$ has a regular free Lyndon length function in $\mathbb{Z}^{n-1}$.
\end{theorem}
\begin{proof} From Theorem \ref{pr:pregroup} it follows that $G = U(P_Z)$, where
$$P_Z = \{g \ast f \ast h \mid f \in Z_+^{\pm 1},\ g,h \in \langle Z_0 \rangle \} \cup \langle Z_0
\rangle.$$
It follows that every element $g$ of $G$ can be represented as a reduced $P_Z$-sequence $g = (g_1,
\ldots,g_k)$, where $g_i \in P_Z - \langle Z_0 \rangle$ for any $i \in [1,k]$ and $g_i \ast g_{i+1}
\notin P_Z$ for any $i \in [1,k-1]$ if $k > 1$ (if $k = 1$ then $g_1$ may be in $\langle Z_0 \rangle$).
In fact (see \cite{Rimlinger:1987}), we have
$$G = U(G) = \langle P_Z \mid x y = z,\ (x, y, z \in P_Z\ {\rm and}\ x \ast y = z) \rangle.$$
Denote $H = \langle Z_0 \rangle$ and $Y = Z_+$. By Corollary \ref{co:1} we have $H = G_{n-1}$.

At first observe that $P_Z$ is infinite but for each $p \in P_Z$ either $p \in H$ or $p = h_1(p)
\ast f_p \ast h_2(p)$, where $f_p \in Y^{\pm 1}$ and $h_i(p) \in H,\ i = 1,2$. Hence, every $p \in
P_Z$ can rewritten in terms of $Y^{\pm}$ and finitely many generators of $H$. On the other hand, if
$x,y,z \in P_Z$ and $x \ast y = z$ then one of the three of them is in $H$ and without loss of
generality we can assume $z \in H$. Hence, either $x,y$ are in $H$ too, or $x,y \notin H$ and
assuming $x = h_1(x) \ast f_x \ast h_2(x),\ y = h_1(y) \ast f_y \ast h_2(y)$, where $h_i(x), h_i(y)
\in H,\ i = 1,2,\ f_x, f_y \in Y^{\pm 1}$ by Lemma \ref{le:claim} we get $f_x = f_y^{-1}$ and $f_x
\ast (h_2(x) \ast h_1(y)) \ast f_x^{-1} \in H$. Hence, every relator $x y = z$, where $x, y, z \in
P_Z$ and $x \ast y = z$ can be rewritten as
$$f_x \ast u_{x,y} \ast f_x^{-1} = v_{x,y,z},$$
where $f_x \in Y^{\pm 1}$ and $u_{x,y}, v_{x,y,z} \in H$. By Lemma \ref{le:2}, for each $q \in
Y^{\pm 1}$ there exists $u_q \in H$ such that for each $u \in H$ we have $q \ast u \ast q^{-1} \in H$
if and only if $u \in C_H(u_q)$. Since $Z$ is reduced it follows that for each $q \in Y^{\pm 1}$
both $C_H(u_q)$ and $q \ast C_H(u_q) \ast q^{-1}$ are in $H$ and also note that $q \ast C_H(u_q)
\ast q^{-1}$ is a centralizer of some element in $H$. Hence, every
$$f_x \ast u_{x,y} \ast f_x^{-1} = v_{x,y,z},$$
is a consequence of
$$f_x \ast C_H(u_x) \ast f_x^{-1} = C_H(v_x),$$
where $u_x,v_x$ depend only on $f_x$. Thus,
$$G = \langle Y, H \mid t_i^{-1} C_H(u_{t_i}) t_i = C_H(v_{t_i}),\ t_i \in Y^{\pm 1} \rangle,$$
where $Y$ is finite, $H$ is finitely generated and $C_H(u_y), C_H(v_y)$ are finitely generated
abelian (see Proposition \ref{pr:1}).

Finally, we have to show that $H$ has a regular free Lyndon length function in $\mathbb{Z}^{n-1}$.
Indeed, since $H = G_{n-1} < G$ then the free Lyndon length function with values in $\mathbb{Z}^{n-1}$
is automatically induced on $H$. We just have to check if it is regular.

Take $g, h \in H$ and consider $com(g,h)$. Since the length function on $G$ is regular then $com(g,h)
\in G = U(P_Z)$ and $com(g,h)$ can be represented by the reduced $P_Z$-sequence $(g_1, \ldots, g_k)$.
By Theorem \ref{pr:pregroup}, (3) it follows that
$$\lambda(com(g,h)) = \sum^k_{i=1} \lambda(g_i).$$
But if $\lambda(com(g,h)) > 0$ then $\lambda(g),\lambda(h) > 0$ - contradiction with the choice of
$g$ and $h$. Hence, $\lambda(g_i) = 0,\ i \in [1,k]$ and it follows that $k = 1$. Thus, $com(g,h) =
g_1 \in H$. This completes the proof of the theorem.
\end{proof}

\begin{theorem} \cite{KMRS:2012}
\label{th:main2}
Let $G$ be a finitely generated complete $\Z^n$-free group. Then $G$ can be represented as a union
of a finite series of groups
$$G_1 < G_2 < \cdots < G_n = G,$$
where $G_1$ is a free group of finite rank, and
$$G_{i+1} = \langle G_i, s_{i,1},\ \ldots,\ s_{i,k_i} \mid s_{i,j}^{-1}\ C_{i,j}\ s_{i,j} =
\phi_{i,j}(C_{i,j}) \rangle,$$
where for each $j \in [1,k_i],\ C_{i,j}$ and $\phi_{i,j}(C_{i,j})$ are cyclically reduced centralizers
of $G_i$, $\phi_{i,j}$ is an isomorphism, and the following conditions are satisfied:
\begin{enumerate}
\item[(1)] $C_{i,j} = \langle c^{(i,j)}_1, \ldots, c^{(i,j)}_{m_{i,j}} \rangle,\ \phi_{i,j}(C_{i,j})
= \langle d^{(i,j)}_1, \ldots, d^{(i,j)}_{m_{i,j}} \rangle$, where $\phi_{i,j}(c^{(i,j)}_k) =
d^{(i,j)}_k,\ k \in [1,m_{i,j}]$ and
$$ht(c^{(i,j)}_k) = ht(d^{(i,j)}_k) < ht(d^{(i,j)}_{k+1}) = ht(c^{(i,j)}_{k+1}),\ k \in [1,m_{i,j}-1],$$
$$ht(s_{i,j}) > ht(c^{(i,j)}_k),$$
\item[(2)]  $|\phi_{i,j}(w)| = |w|$ for any $w \in C_{i,j}$,
\item[(3)] $w$ is not conjugate to $\phi_{i,j}(w)^{-1}$ in $G_i$ for any non-trivial $w \in C_{i,j}$,
\item[(4)] if $A, B \in \{C_{i,1}, \phi_{i,1}(C_{i,1}), \ldots, C_{i,k_i}, \phi_{i,k_i}(C_{i,k_i})\}$
then either $A = B$, or $A$ and $B$ are not conjugate in $G_i$,
\item[(5)] $C_{i,j}$ can appear in the list
$$\{ C_{i,k}, \phi_{i,k}(C_{i,k}) \mid k \neq j \}$$
not more than twice.
\end{enumerate}
\end{theorem}
\begin{proof} Existence of the series
$$G_1 < G_2 < \cdots < G_n = G,$$
where $G_{i+1},\ i \in [1,n-1]$ can be obtained from $G_i$ by finitely many HNN-extensions in which
associated subgroups are maximal abelian of finite rank follows by induction applying Theorem
\ref{th:main}. Also, observe that $G_1$ has a free length function with values in $\mathbb{Z}$, hence,
by the result of Lyndon \cite{Lyndon:1963} it follows that $G_1$ is a free group. Moreover, $G_1$ is of finite
rank by Theorem \ref{th:main}.

Now, consider $G_{i+1}$. By Theorem \ref{th:main} we can assume that
\begin{equation}
\label{eq:1}
G_{i+1} = \langle G_i, t_1, t_2, \ldots, t_p \mid t_j^{-1} C_{G_i}(u_{t_j}) t_j = C_{G_i}(v_{t_j})
\rangle,
\end{equation}
where
\begin{enumerate}
\item[(a)] all $t_j$ are cyclically reduced,
\item[(b)] $G_i = \langle Y \rangle,\ ht(t_j) > ht(G_i)$ and
$$Y \cup \{t_1, t_2, \ldots, t_p\}$$
is a reduced generating set for $G_{i+1}$.
\end{enumerate}
In particular, $Y \cup \{t_1, t_2, \ldots, t_p\}$ is reduced, that is, it has the properties listed in
Proposition \ref{pr:min}.

At first, we can assume that all $C_{G_i}(u_{t_j}),\ C_{G_i}(v_{t_j})$ are cyclically reduced.
Indeed, if not then by Lemma \ref{le:cycl} we have $C_{G_i}(u_{t_j}) = c^{-1} \circ B \circ c$,
where $B$ is cyclically reduced, $c \in G_i$ by regularity of the length function on $G_i$, and
$$(t_j^{-1} \ast c^{-1}) \ast B \ast (c \ast t_j) = C_{G_i}(v_{t_j}).$$
Thus, we can substitute $t_j$ by $c \ast t_j,\ C_{G_i}(u_{t_j})$ by $B$, and the same can be done
for $C_{G_i}(v_{t_j})$.

Observe that conjugation by $t_j$ induces an isomorphism between $C_{G_i}(u_{t_j})$ and
$C_{G_i}(v_{t_j})$, and since we can assume both centralizers to be cyclically reduced then from
$$t_j^{-1} \ast C_{G_i}(u_{t_j}) \ast t_j = C_{G_i}(v_{t_j})$$
it follows that for $a \in C_{G_i}(u_{t_j}),\ b \in C_{G_i}(v_{t_j})$ if $t_j^{-1} \ast a \ast t_j
= b$ then $|a| = |b|$. In particular, if
$$C_{G_i}(u_{t_j}) = \langle c^{(i,j)}_1, \ldots, c^{(i,j)}_{m_{i,j}} \rangle,$$
where we can assume $ht(c^{(i,j)}_k) < ht(c^{(i,j)}_{k+1})$ for $k \in [1, m_{i,j} - 1]$, then all
$d^{(i,j)}_k = t_j^{-1} \ast c^{(i,j)}_k \ast t_j$ generate $C_{G_i}(v_{t_j})$ and $|c^{(i,j)}_k|
= |d^{(i,j)}_k|$. This proves (1) and (2).

\smallskip

Suppose there exist $w_1 \in C_{G_i}(u_{t_j})$ and $g \in G_i$ such that $g^{-1} \ast w_1 \ast g =
w_2^{-1}$, where $w_2 = \phi_i(w_1) \in C_{G_i}(v_{t_j})$. Observe that either $ht(g) \leqslant ht(w_1)
= ht(w_2)$ and in this case $w_1$ is a cyclic permutation of $w_2^{-1}$, or $ht(g) > ht(w_1)$.
In the latter case, $g$ has any positive power of $w_1^\delta, \delta \in \{1,-1\}$ as an initial
subword and any positive power of $w_2^{-\delta}$ as a terminal subword. Without loss of generality
we can assume $\delta = 1$. Hence, $t_j \ast g^{-1} = t_j \circ g^{-1}$ and
$$(t_j \circ g^{-1})^{-1} \ast w_1 \ast (t_j \circ g^{-1}) = w_1^{-1}.$$
Consider $h = com(t_j \circ g^{-1},(t_j \circ g^{-1})^{-1})$. Observe that $h \in G_i$ and $|h^{-1}
\ast w_1 \ast h| = |w_1|$. Indeed, $w_1$ is cyclically reduced, so either $(t_j \circ g^{-1})^{-1}
\ast w_1 = (t_j \circ g^{-1})^{-1} \circ w_1$, or $w_1 \ast (t_j \circ g^{-1}) = w_1 \circ (t_j
\circ g^{-1})$. Assuming the latter (the other case is similar) we have
$$(t_j \circ g^{-1})^{-1} \ast w_1 \ast (t_j \circ g^{-1}) = (t_j \circ g^{-1})^{-1} \ast (w_1 \circ
(t_j \circ g^{-1}))$$
and from $|(t_j \circ g^{-1})^{-1} \ast w_1 \ast (t_j \circ g^{-1})| = |w_1|$ it follows that $(t_j
\circ g^{-1})^{-1}$ cancels completely in the product $(t_j \circ g^{-1})^{-1} \ast w_1 \ast (t_j
\circ g^{-1})$. Eventually, since $h$ is an initial subword of $t_j \circ g^{-1}$, it follows that
$h^{-1}$ cancels completely in the product $h^{-1} \ast (w_1 \circ h)$, so $|h^{-1} \ast w_1 \ast h|
= |w_1|$. Thus, if $w_3 = h^{-1} \ast w_1 \ast h$ then $h$ ends with any positive power of $w_3$.
We have $t_j \circ g^{-1} = h \circ f \circ h^{-1}$, where $f$ is cyclically reduced. But at the
same time we have $f^{-1} \ast w_3 \ast f = w_3^{-1}$ and this produces a contradiction. Indeed, if
$ht(f) \leqslant ht(w_3)$ then $w_3$ is a cyclic permutation of $w_3^{-1}$ which is impossible. On the
other hand, if $ht(f) > ht(w_3)$ then $f$ has any power of $w_3^\alpha, \alpha \in \{1,-1\}$ as an
initial subword, and any power of $w_3^{-\alpha}$ as a terminal subword - a contradiction with the
fact that $f$ is cyclically reduced. This proves (3).

\smallskip

To prove (4), assume that two centralizers from the list
$$C_{G_i}(u_{t_1}), \ldots, C_{G_i}(u_{t_p})$$
are conjugate in $G_i$. Denote $C_1 = C_{G_i}(u_{t_1}),\ C_2 = C_{G_i}(u_{t_2})$ and let $C_1 =
h^{-1} \ast C_2 \ast h$ for some $h \in G_i$. Hence, in (\ref{eq:1}) every entry of $C_1$ can be
substituted by $h^{-1} \ast C_2 \ast h$ and some of the elements $t_1, t_2, \ldots, t_p$ can be
changed accordingly.

\smallskip

Finally, assume that there exists $t_k \neq t_j$ such that $t_k^{-1} \ast C_{G_i}(u_{t_j}) \ast t_k
\leqslant G_i$. Suppose $c(t_j, t_k) > 0$ and denote $z = com(t_j, t_k)$. Observe that $ht(z) >
ht(C_{G_i}(u_{t_j}))$. If $ht(z) < ht(t_j)$ then $z$ conjugates $C_{G_i}(u_{t_j})$ into a cyclically
reduced centralizer $A$ of $G_i$ and $z$ has any positive power of some $a \in A,\ ht(a) = ht(A)$
as a terminal subword. But then $ht(z^{-1} \ast t_j) = ht(z^{-1} \ast t_k) = ht(t_j)$ and since both
$z^{-1} \ast t_j$ and $z^{-1} \ast t_k$ conjugate $A$ into a cyclically reduced centralizer of $G_i$
it follows that $z^{-1} \ast t_j$ and $z^{-1} \ast t_k$ have $a^{\pm 1}$ as an initial subword.
If $z^{-1} \ast t_j$ has $a$ as an initial subword and $z^{-1} \ast t_k$ has $a^{-1}$ as an initial
subword then $z \ast (z^{-1} \ast t_k) \neq z \circ (z^{-1} \ast t_k)$, and we have a contradiction.
If both $z^{-1} \ast t_j$ and $z^{-1} \ast t_k$ have $a$ as an initial subword then $z$ cannot be
$com(t_j,t_k)$ and again we have a contradiction. Thus, $ht(z) = ht(t_j)$, but it is possible only
if $t_j = t_k$ since $Y \cup \{t_1, t_2, \ldots, t_p\}$ is a minimal generating set, and again we
have a contradiction with our choice of $t_k$. It follows that $c(t_j, t_k) = 0$ and if $t_j$ begins
with $c \in C_{G_i}(u_{t_j})$ then $t_k$ begins with $c^{-1}$. It also follows that there can be
only one $t_k \neq t_j$ such that $t_k^{-1} \ast C_{G_i}(u_{t_j}) \ast t_k \leqslant G_i$.

This completes the proof of the theorem.
\end{proof}

\subsection{HNN-extensions of complete $\mathbb{Z}^n$-free groups}
\label{sec:embedding}

Let $H$ be a finitely generated complete $\mathbb{Z}^n$-free group. Observe that $\mathbb{Z}^n \simeq
\bigoplus_{i=0}^{n-1} \langle t^i \rangle$ which is a subgroup of $\Zt$, so we can always
assume that $H$ has a regular free length function with values in $\Zt$. On the other hand,
observe that for a finitely generated complete $\mathbb{Z}^n$-free group there exists $n \in
\mathbb{N}$ such that the length function takes values in $\mathbb{Z}^n$.

The main goal of this section is to prove the following result.

\begin{theorem} \cite{KMRS:2012}
\label{main4}
Let $H$ be a finitely generated complete $\mathbb{Z}^n$-free group. Let $A$ and $B$ be centralizers
in $H$ whose elements are cyclically reduced and such that there exists an isomorphism $\phi : A
\rightarrow B$ with the following properties
\begin{enumerate}
\item $a$ is not conjugate to $\phi(a)^{-1}$ in $H$ for any $a \in A$,
\item $|\phi(a)| = |a|$ for any $a \in A$.
\end{enumerate}
Then the group
\begin{equation}
\label{eq:G}
G = \langle H, z \mid z^{-1} A z = B \rangle,
\end{equation}
is a finitely generated complete $\Zt$-free group and the length function on $G$ extends
the one on $H$.
\end{theorem}

\subsubsection{Cyclically reduced centralizers and attached elements}

From Theorem \ref{th:main2}, $H$ is union of the chain
$$F(X) = H_1 < H_2 < \cdots < H_n = H,$$
where
$$H_{i+1} = \langle H_i, s_{i,1},\ldots,\ s_{i,k_i} \mid s_{i,j}^{-1} C_{i,j} s_{i,j} = D_{i,j}
\rangle,$$
$C_{i,j},\ D_{i,j}$ are maximal abelian subgroups of $H_i$, and $ht(s_{i,j}) > ht(H_i)$ for any
$i \in [1,n-1],\ j \in [1,k_i]$.

Let $K$ be a cyclically reduced centralizer in $H$. It is easy to see that either $ht(K) = ht(H) =
n$, or $ht(K) < ht(H)$ and $K$ is a centralizer from $H_{n-1}$.

For a cyclically reduced centralizer $K$ of $H$ we define
$${\cal C}(K) = \{C_{i,j},\ D_{i,j}\} \cap \{ {\rm cyclically\ reduced\
centralizers\ conjugate\ to}\ K\ {\rm in}\ H\}.$$

\begin{lemma} \cite{KMRS:2012}
\label{le:attached}
Let $K$ be a cyclically reduced centralizer of $H$, and let ${\cal C}(K)$ be empty. Let $a$ be a
generator of $K$ of maximal height. Then there is no element in $H$ which has any positive power of
$a^{\pm 1}$ as an initial subword.
\end{lemma}

If ${\cal C}(K) \neq \emptyset$ then for $C \in {\cal C}(K)$ we call $w$ from the list $s_{i,j},\ i
\in [1,n-1],\ j \in [1,k_i]$ {\em attached to $C$} if $ht(w) > ht(C)$ and $ht(w^{-1} \ast C \ast w)
= ht(C)$. Observe that by Theorem \ref{th:main2}, $C$ can have at most two attached elements of the
same height, and if $w_1,w_2,\ w_1 \neq w_2$ are attached to $C$ and $ht(w_1) = ht(w_2)$ then
$w_1^{-1} \ast w_2 = w_1^{-1} \circ w_2$.

Below we are going to distinguish attached elements in the following way. Suppose $C \in {\cal C}(K)$,
and let $w$ be an element attached to $C$. If $c$ is a generator of $C$ of maximal height then we
call $w$ {\em left-attached to $C$ with respect to $c$} if $c^{-1} \ast w = c^{-1} \circ w$, and we
call $w$ {\em right-attached to $C$ with respect to $c$} if $c \ast w = c \circ w$.

\begin{lemma} \cite{KMRS:2012}
\label{le:attached2}
Let $K$ be a cyclically reduced centralizer of $H$, and let $C \in {\cal C}(K)$. Let $c$ be a
generator of $C$ of maximal height. If there exists a right(left)-attached to $C$ with respect to
$c$ element, then there exists $D \in {\cal C}(K)$ and its generator $d$ of maximal height such
that $c$ is conjugate to $d$ in $H$ and $D$ does not have right(left)-attached with respect to $d$
elements.
\end{lemma}

\begin{lemma} \cite{KMRS:2012}
\label{le:attached3}
Let $K$ be a cyclically reduced centralizer of $H$, and let $C \in {\cal C}(K)$. Let $c$ be a
generator of $C$ of maximal height. If there exists no right(left)-attached to $C$ with respect to
$c$ element, then there is no element $g \in H$ which has any positive power of $c$ as an initial
(terminal) subword.
\end{lemma}

\subsubsection{Connecting elements}
\label{subs:4.1}

We call a pair of elements $u, v \in CDR(\Zt,X)$ an {\em admissible pair} if
\begin{enumerate}
\item $u,v$ are cyclically reduced,
\item $u,v$ are not proper powers,
\item $|u| = |v|$,
\item $u$ is not conjugate to $v^{-1}$ (in particular, $u \neq v^{-1}$).
\end{enumerate}

For an admissible pair $\{u,v\}$ we define an infinite word $s_{u,v} \in R(\Zt,X)$, which
we call the {\em connecting element} for the pair $\{u,v\}$, in the following way
\[
\mbox{$s_{u,v}(\beta)$} = \left\{ \begin{array}{ll}
\mbox{$u(\alpha)$} & \mbox{if $\beta = (k |u| + \alpha, 0), k \geqslant 0, 1 \leqslant \alpha
\leqslant |u|$,} \\
\mbox{$v(\alpha)$} & \mbox{if $\beta = (-k|v|+\alpha, 1), k \geqslant 1, 1 \leqslant \alpha
\leqslant |v|$.}
\end{array}
\right.
\]
Since there exists $m > 0$ such that $u,v \in CDR(\mathbb{Z}^m,X) - CDR(\mathbb{Z}^{m-1},X)$ then
it is easy to see that $s_{u,v} \in R(\mathbb{Z}^{m+1},X) - R(\mathbb{Z}^m,X)$. Also $s_{u,v}^{-1}
= s_{v^{-1},u^{-1}}$ and $u \circ s_{u,v} = s_{u,v} \circ v$ - both follow directly from the
definition.

Notice that any two connecting elements $s_{u_1,v_1},\ s_{u_2,v_2}$ have the same length whenever
$u_1,\ v_1,\ u_2,\ v_2 \in CDR(\mathbb{Z}^m,X) - CDR(\mathbb{Z}^{m-1},X)$. In this event we have
$$|s_{u_1,v_1}| = |s_{u_2,v_2}| = (0,\ldots,0,1) \in \mathbb{Z}^{m+1}.$$

\begin{lemma} \cite{KMRS:2012}
\label{le:4.1}
Let $u,v$ be elements of a group $H \subset CDR(\Zt,X)$. If
the pair $\{u,v\}$ is admissible then $s_{u,v} \in CDR(\Zt,X)$.
\end{lemma}

\subsection{Main construction}
\label{subs:main_constr}

Now, let $A, B$ be cyclically reduced centralizers in $H$ such that there exists an isomorphism
$\phi : A \rightarrow B$ satisfying the following conditions
\begin{enumerate}
\item $a$ is not conjugate to $\phi_i(a)^{-1}$ in $H$ for any $a \in A$,
\item $|\phi(a)| = |a|$ for any $a \in A$.
\end{enumerate}

In particular, it follows that $ht(A) = ht(B)$.

\begin{remark}
\label{rem:4.2}
Observe that if $C$ is conjugate to $A$ and $D$ is conjugate to $B$ then
$$\langle H, z \mid z^{-1} A z = B \rangle \simeq \langle H, z' \mid z'^{-1} C z' = D \rangle.$$
Hence, it is always possible to consider $A$ and $B$ up to taking conjugates.
\end{remark}

Let $u$ be a generator of $A$ of maximal height, and let $v = \phi(u) \in B$. Then $v$ is a generator
of $B$ of maximal height and $|u| = |v|$. Observe that from the conditions imposed on $\phi$ it
follows that the pair $u,v$ is admissible. We fix $u$ and $v$ for the rest of the paper.

\begin{remark}
\label{rem:attached}
Observe that if ${\cal C}(A) = \emptyset$ then by Lemma \ref{le:attached}, $H$ does not contain an
element which has any positive power of $u^{\pm 1}$ as an initial subword (similar statement for $B$
and $v$ if ${\cal C}(B) = \emptyset$). If ${\cal C}(A) \neq \emptyset$ then by Lemma
\ref{le:attached2} we can assume $A$ to have no right-attached elements with respect to $u$. Hence,
by Lemma \ref{le:attached3}, $H$ does not contain an element which has any positive power of $u$ as
an initial subword. Similarly, if ${\cal C}(B) \neq \emptyset$ then by Lemma \ref{le:attached2} we
can assume $B$ to have no left-attached elements with respect to $v$. Again, by Lemma
\ref{le:attached3}, $H$ does not contain an element which has any positive power of $v$ as a
terminal subword.
\end{remark}

Now, we are in position to define $s \in R(\mathbb{Z}^{n+1},X)$ which is going to be an infinite
word representing $z$ from the presentation (\ref{eq:G}). Since $\Zt$-exponentiation is
defined on $CR(\Zt,X)$ (see \cite{Myasnikov_Remeslennikov_Serbin:2005} for details) then
for any $f(t) \in \Zt$ we can define $v^{f(t)}, u^{f(t)} \in CR(\Zt,X)$ so that
$$|v^{f(t)}| = |v| |f(t)|,\ |u^{f(t)}| = |u| |f(t)|$$
and
$$[v^{f(t)},v] = \varepsilon,\ [u^{f(t)},u] = \varepsilon.$$
Thus, if $\alpha = t^{n-ht(A)}$ then $|u^\alpha| = |v^\alpha| = |u| |\alpha|$ and $ht(u^\alpha) =
ht(v^\alpha) = ht(u) + (n-ht(A)) = n$. Hence, we define
$$s = s_{u^\alpha,v^\alpha} \in CDR(\mathbb{Z}^{n+1},X).$$
Observe that $ht(s) = n + 1 = ht(G) + 1$.

\begin{remark}
\label{rem:4.3}
It is easy to see that no element of $H$ has $s^{\pm 1}$ as a subword.
\end{remark}

\begin{lemma} \cite{KMRS:2012}
\label{le:4.2}
For any $h \in A$ we have $s^{-1} \ast h \ast s = \phi(h) \in B$.
\end{lemma}

Now, our goal is to prove that a pair $H, s$ generates a group in $CDR(\Zt, X)$.

\begin{lemma} \cite{KMRS:2012}
\label{le:stab0}
For any $g \in H$ there exists $N = N(g) > 0$ such that
$$g \ast u^k = (g \ast u^N) \circ u^{k-N},\ v^k \ast g = v^{k-N} \circ (v^N \ast g).$$
for any $k > N$.
\end{lemma}

\begin{lemma} \cite{KMRS:2012}
\label{le:stab1}
\begin{enumerate}
\item[{\bf (i)}] For any $g \in H - A$ there exists $N = N(g) > 0$ such that for any $k > N$
$$u^{-k} \ast g \ast u^k = u^{-k+N} \circ (u^{-N} \ast g \ast u^N) \circ u^{k-N}.$$
\item[{\bf (ii)}] For any $g \in H$ there exists $N = N(g) > 0$ such that for any $k > N$
$$v^k \ast g \ast u^k = v^{k-N} \circ (v^N \ast g \ast u^N) \circ u^{k-N}.$$
\item[{\bf (iii)}] For any $g \in H - B$ there exists $N = N(g) > 0$ such that for any $k > N$
$$v^k \ast g \ast v^{-k} = v^{k-N} \circ (v^N \ast g \ast v^{-N}) \circ u^{-k+N}.$$
\end{enumerate}
\end{lemma}

A sequence
\begin{equation}
\label{eq:p1}
p = (g_1, s^{\epsilon_1}, g_2, \dots, g_k, s^{\epsilon_k}, g_{k+1}),
\end{equation}
where $g_j \in H,\ \epsilon_j \in \{-1,1\}$, $k \geqslant 1$, is called an {\em $s$-form} over $H$.

An $s$-form (\ref{eq:p1}) is {\em reduced} if subsequences
$$\{s^{-1}, c, s\},\ \ \{s, d, s^{-1}\}$$
where $c \in A,\ d \in B$, do not occur in it.

Denote by ${\mathcal P}(H,s)$ the set of all $s$-forms over $H$. We define a partial function
$w: {\mathcal P}(H,s) \rightarrow R(\Zt,X)$ as follows. If
$$p = (g_1, s^{\epsilon_1}, g_2, \dots, g_k, s^{\epsilon_k}, g_{k+1})$$
then
$$w(p) = (\cdots (g_1 \ast s^{\epsilon_1}) \ast g_2) \ast \cdots \ast g_k) \ast s^{\epsilon_k})
\ast g_{k+1})$$
if it is defined.

\begin{lemma} \cite{KMRS:2012}
\label{le:p1}
Let $p = (g_1, s^{\epsilon_1}, g_2, \dots, g_k, s^{\epsilon_k}, g_{k+1})$ be an $s$-form over $H$.
Then the following hold.
\begin{enumerate}
\item[(1)] The product $w(p)$ is defined and it does not depend on the placement of parentheses.
\item[(2)] There exists a reduced $s$-form $q$ over $H$ such that $w(q) = w(p)$.
\item[(3)] If $p$ is reduced then there exists a unique representation for $w(p)$ of the following
type
$$w(p) = (g_1 \ast u_1^{N_1}) \circ (u_1^{-N_1} \ast s^{\epsilon_1} \ast v_1^{-M_1}) \circ (v_1^{M_1}
\ast g_2 \ast u_2^{N_2}) \circ \cdots $$
$$\cdots \circ (u_k^{-N_k} \ast s^{\epsilon_k} \ast v_k^{-M_k}) \circ (v_k^{M_k} \ast g_{k+1}),$$
where $N_j, M_j \geqslant 0,\ u_j = u, v_j = v$ if $\epsilon_j = 1$, and $N_j, M_j \leqslant 0,\ u_j = v, v_j
= u$ if $\epsilon_j = -1$ for $j \in [1,k]$. Moreover, $g_1 \ast u_1^{N_1}$ does not have $u_1^{\pm 1}$
as a terminal subword, $v_{j-1}^{M_{j-1}} \ast g_j \ast u_j^{N_j}$ does not have $u_j^{\pm 1}$ as a
terminal subword for every $j \in [2,k]$, and $v_{j-1}^{M_{j-1}} \ast g_j \ast s_i^{\epsilon_j}$ does
not have $v_{j-1}^{\pm 1}$ as an initial subword for every $j \in [2,k]$, $v_k^{M_k} \ast g_{k+1}$
does not have $v_k^{\pm 1}$ as an initial subword.
\item[(4)] $w(p) \in CDR(\Zt,X)$.
\end{enumerate}
\end{lemma}
\begin{proof} Let
$$p = (g_1, s^{\epsilon_1}, g_2, \dots, g_k, s^{\epsilon_k}, g_{k+1})$$
be an $s$-form over $H$.

We show first that (1) implies (2). Suppose that $w(p)$ is defined for every placement of parentheses
and all such products are equal. If $p$ is not reduced then there exists $j \in [2,k]$ such that
either $g_j \in A,\ \epsilon_{j-1} = -1,\ \epsilon_j = 1$, or $g_j \in B,\ \epsilon_{j-1} = 1,\
\epsilon_j = -1$. Without loss of generality we can assume the former. Thus, we have
$$s^{-1} \ast g_j \ast s = g'_j \in B \subseteq H$$
and we obtain a new $s$-form
$$p_1 = (g_1, s^{\epsilon_1}, g_2, \dots, g_{j-1} \ast g'_j \ast g_{j+1}, s^{\epsilon_{j+1}}, \dots,
g_k, s^{\epsilon_k}, g_{k+1})$$
which is shorter then $p$ and $w(p) = w(p_1)$. Proceeding this way (or by induction) in a finite
number of steps we obtain a reduced $s$-form
$$q = (f_1, s^{\delta_1}, f_2, \ldots, s^{\delta_l}, f_{l+1}),$$
such that $w(q) = w(p)$, as required.

\smallskip

Now we show that (1) implies (3). Assume that
$$p = (g_1, s^{\epsilon_1}, g_2, \dots, g_k, s^{\epsilon_k}, g_{k+1})$$
is reduced.

By Lemma \ref{le:stab0} and Lemma \ref{le:stab1} there exists $r \in \mathbb{N}$ such that for any
$\alpha > r$

\begin{enumerate}
\item[(a)] $g_1 \ast u^\alpha = (g_1 \ast u^r) \circ u^{\alpha-r},\ g_1 \ast v^{-\alpha} = (g_1
\ast u^{-r}) \circ u^{-\alpha+r}$,
\item[(b)] $v^\alpha \ast g_{k+1} = v^{\alpha-r} \circ (v^r \ast g_{k+1}),\ u^{-\alpha} \ast g_{k+1}
= u^{-\alpha+r} \circ (u^{-r} \ast g_{k+1})$,
\item[(c)] $u^{-\alpha} \ast g_j \ast u^\alpha = u^{-(\alpha-r)} \circ (u^{-r} \ast g_j \ast u^r)
\circ u^{\alpha-r}$ for all $j \in [2,k]$ such that $g_j \notin A$,
\item[(d)] $v^\alpha \ast g_j \ast u^\alpha = v^{\alpha-r} \circ (v^r \ast g_j \ast u^r) \circ
u^{\alpha-r}$ for all $j \in [2,k]$,
\item[(e)] $v^\alpha \ast g_j \ast v^{-\alpha} = v^{\alpha-r} \circ (v^r \ast g_j \ast v^{-r}) \circ
u^{-(\alpha-r)}$ for all $j \in [2,k]$ such that $g_j \notin B$.
\end{enumerate}

Since $p$ is reduced, that is, it does not contain neither a subsequence $\{s^{-1}, g_j,$ $ s\}$,
where $g_j \in A$, nor a subsequence $\{s, g_j, s^{-1}\}$, where $g_j \in B$, and $s$ has any power
of $u$ as an initial subword and any power of $v$ as a terminal subword then we have
$$w(p) = g_1 \ast s^{\epsilon_1} \ast g_2 \ast \cdots \ast g_k \ast s^{\epsilon_k} \ast g_{k+1} = $$
$$= (g_1 \ast u_1^r) \circ (u_1^{-r} \ast s^{\epsilon_1} \ast v_1^{-r}) \circ (v_1^r \ast g_2 \ast
u_2^r) \circ \cdots \circ (u_k^{-r} \ast s^{\epsilon_k} \ast v_k^{-r}) \circ (v_k^r \ast g_{k+1}),$$
where $u_j = u, v_j = v$ if $\epsilon_j = 1$ and $u_j = v^{-1}, v_j = u^{-1}$ if $\epsilon_j = -1$
for every $j \in [1,k]$.

Now, if $g_1 \ast u_1^r$ has $u_1^{\gamma_1},\ \gamma_1 \in \mathbb{Z}$ (with $\gamma_1$ maximal
possible) as a terminal subword then we denote $N_1 = r-\gamma_1$ and rewrite $w(p)$ as follows
$$w(p) = (g_1 \ast u_1^{N_1}) \circ (u_1^{-N_1} \ast s^{\epsilon_1} \ast v_1^{-r}) \circ (v_1^r \ast
g_2 \ast u_2^r) \circ \cdots$$
$$\cdots \circ (u_k^{-r} \ast s^{\epsilon_k} \ast v_k^{-r}) \circ (v_k^r \ast g_{k+1}).$$
Now, if $v_1^r \ast g_2 \ast u_2^r$ contains $v_1^{\delta_1},\ \delta_1 \in \mathbb{Z}$ (with
$\delta_1$ maximal possible) as an initial subword then we denote $M_1 = r-\delta_1$ and again
rewrite $w(p)$
$$w(p) = (g_1 \ast u_1^{N_1}) \circ (u_1^{-N_1} \ast s^{\epsilon_1} \ast v_1^{-M_1}) \circ (v_1^{M_1}
\ast g_2 \ast u_2^r) \circ \cdots$$
$$\cdots \circ (u_k^{-r} \ast s^{\epsilon_k} \ast v_k^{-r}) \circ (v_k^r \ast g_{k+1}).$$
In a finite number of steps we obtain the required result. Observe that by the choice of $N_i, M_i$
the representation of $w(p)$ is unique.

\smallskip

Now we prove (1) by induction on $k$. If $k = 1$ then by Lemma \ref{le:stab0} there exists $r \in
\mathbb{N}$ such that
$$g_1 \ast u^\alpha = (g_1 \ast u^r) \circ u^{\alpha-r},\ \ g_1 \ast v^{-\alpha} = (g_1 \ast v^{-r})
\circ v^{-\alpha+r},$$
$$v^\beta \ast g_2 = u^{\beta-r} \circ (u^r \ast g_2),\ \ u^{-\beta} \ast g_2 = u^{-\beta+r} \circ
(v^{-r} \ast g_2)$$
for any $\alpha, \beta > r$. Hence,
$$(g_1 \ast s^{\epsilon_1}) \ast g_2 = ((g_1 \ast u_1^r) \circ (u_1^{-r} \ast s^{\epsilon_1})) \ast
g_2 =$$
$$= ((g_1 \ast u_1^r) \circ (u_1^{-r} \ast s^{\epsilon_1} \ast v_1^{-r})) \circ (v_1^r \ast g_2),$$
where $u_1 = u, v_1 = v$ if $\epsilon_1 = 1$ and $u_1 = v^{-1}, v_1 = u^{-1}$ if $\epsilon_1 = -1$.
By Theorem 3.4 \cite{Myasnikov_Remeslennikov_Serbin:2005}, the product $(g_1 \ast s^{\epsilon_1})
\ast  g_2$ does not depend on the placement of parentheses. So (1) holds for $k = 1$.

Now, consider an initial $s$-subsequence of $p$
$$p_1 = (g_1, s^{\epsilon_1}, g_2, \dots, g_{k-1}, s^{\epsilon_{k-1}}, g_k).$$
By induction $w(p_1)$ is defined and it does not depend on the placement of parentheses. By the
argument above there exists a unique representation of $w(p_1)$
$$w(p_1) = (g_1 \ast u_1^{N_1}) \circ (u_1^{-N_1} \ast s^{\epsilon_1} \ast v_1^{-M_1}) \circ
(v_1^{M_1} \ast g_2 \ast u_2^{N_2}) \circ \cdots$$
$$\cdots \circ (u_{k-1}^{-N_{k-1}} \ast s^{\epsilon_{k-1}} \ast v_{k-1}^{-M_{k-1}}) \circ
(v_{k-1}^{M_{k-1}} \ast g_k),$$
where $N_j, M_j \geqslant 0,\ u_j = u, v_j = v$ if $\epsilon_j = 1$, and $N_j, M_j \leq 0,\ u_j = v, v_j
= u$ if $\epsilon_j = -1$ for $j \in [1,k-1]$. To prove that $p$ satisfies (1) it suffices to show
that
$$w(p_1) \ast (s^{\epsilon_k} \ast g_{k+1})$$
is defined and does not depend on the placement of parentheses.

Without loss of generality we assume $\epsilon_{k-1} = 1, \epsilon_k = 1$ - other combinations of
$\epsilon_{k-1}$ and $\epsilon_k$ are considered similarly.

By Lemma \ref{le:stab1}
$$((u^{-N_{k-1}} \ast s \ast v^{-M_{k-1}}) \circ (v^{M_{k-1}} \ast g_k)) \ast (s \ast g_{k+1}) =
(u^{-N_{k-1}} \ast s \ast v^{-M_{k-1}-r})$$
$$ \circ (v^{M_{k-1}+r} \ast g_k \ast u^{m_1}) \circ (u^{-m_1} \ast s \ast v^{-m_2}) \circ (v^{m_2}
\ast g_{k+1})$$
for some $m_1, m_2,r \in \mathbb{N}$. Thus $w(p_1) \ast (s \ast g_{k+1})$ is defined and does not
depend on the placement of parentheses.

\smallskip

Now we prove (4). By (3) there exists a unique representation of $w(p)$
$$w(p) = (g_1 \ast u_1^{N_1}) \circ (u_1^{-N_1} \ast s^{\epsilon_1} \ast v_1^{-M_1}) \circ (v_1^{M_1}
\ast g_2 \ast u_2^{N_2}) \circ \cdots$$
$$\cdots \circ (u_k^{-N_k} \ast s^{\epsilon_k} \ast v_k^{-M_k}) \circ (v_k^{M_k} \ast g_{k+1}),$$
where $N_j, M_j \geqslant 0,\ u_j = u, v_j = v$ if $\epsilon_j = 1$, and $N_j, M_j \leqslant 0,\ u_j = v,
v_j = u$ if $\epsilon_j = -1$ for $j \in [1,k]$. By Lemma 3.8 \cite{Myasnikov_Remeslennikov_Serbin:2005},
to prove that $w(p) \in CDR(\Zt,X)$ it suffices  to show that
$$g^{-1} \ast w(p) \ast g \in CDR(\Zt,X)$$
for some $g \in R(\Zt,X)$.

Without loss of generality we assume $\epsilon_k = 1$ and consider two cases.

\begin{enumerate}
\item[(i)] $g_{k+1} \ast g_1 \ast u_1^{N_1} \notin B$ or $g_{k+1} \ast g_1 \ast u_1^{N_1} \in B$ but
$\epsilon_1 = 1$.

Without loss of generality we assume the former and $\epsilon_1 = 1$. That is, $u_1 = u,\ v_1 = v$.
By Lemma \ref{le:stab1} there exists $N \in \mathbb{N}$ such that
$$(u^{-m} \ast (g_1 \ast u^{N_1})^{-1}) \ast w(p) \ast ((g_1 \ast u^{N_1}) \ast u^m) = (u^{-N_1-m}
\ast s \ast v^{-M_1})$$
$$\circ (v^{M_1} \ast g_2 \ast u_2^{N_2}) \circ \cdots \circ (u^{-N_k} \ast s \ast v^{-M_k-m}) \circ
(v^{M_k+m} \ast g_{k+1} \ast g_1 \ast u^{N_1+N}) \circ u^{m-N}$$
for any $m > N$. Thus,
$$(u^{-m} \ast (g_1 \ast u^{N_1})^{-1}) \ast w(p) \ast ((g_1 \ast u^{N_1}) \ast u^m) \in
CR(\Zt,X) \subset CDR(\Zt,X).$$

\item[(ii)] $g_{k+1} \ast g_1 \in B$ and $\epsilon_1 = -1$.

Thus, $a = s \ast (g_{k+1} \ast g_1) \ast s^{-1} \in A$ and we have
$$(g_1 \ast s)^{-1} \ast w(p) \ast (g_1 \ast s) = \left((g_2 \ast u_2^{N_2}) \circ (u_2^{-N_2} \ast
s^{\epsilon_2} \ast v_2^{-M_2}) \circ \cdots\right.$$
$$\left.\cdots \circ (u_{k-1}^{-N_{k-1}} \ast s^{\epsilon_{k-1}} \ast v_{k-1}^{-M_{k-1}}) \circ
(v_{k-1}^{M_{k-1}} \ast g_k)\right) \ast a = (g_2 \ast u_2^{N_2})$$
$$\circ (u_2^{-N_2} \ast s^{\epsilon_2} \ast
v_2^{-M_2}) \circ \cdots \circ (u_{k-1}^{-N_{k-1}} \ast s^{\epsilon_{k-1}} \ast v_{k-1}^{-M'_{k-1}})
\circ (v_{k-1}^{M'_{k-1}} \ast (g_k \ast a)),$$
where $M'_{k-1} \in \mathbb{Z}$ is the power which works for $g_k \ast a$. So the number of
$s^{\pm 1}$ is reduced by two and we can use induction.
\end{enumerate}
\end{proof}

Now we are ready to prove the main results of this subsection from which Theorem \ref{main4} follows.

\begin{theorem} \cite{KMRS:2012}
\label{th:U(P)}
Put
$$P = P(H,s) = \{ g \ast s^\epsilon \ast h \mid g,h \in H, \epsilon \in \{-1,0,1\} \} \subseteq
CDR(\Zt,X).$$
Then the following hold.
\begin{itemize}
\item [(1)] $P$ generates a subgroup $H^*$ in $CDR(\Zt,X)$.
\item [(2)] $P$, with the multiplication $\ast$ induced from $R(\Zt,X)$, is a pregroup
and $H^*$ is isomorphic to $U(P)$.
\item [(3)] $H^*$ is isomorphic to $G = \langle H, z \mid z^{-1} A z = B \rangle$.
\end{itemize}
\end{theorem}
\begin{proof} We need the following claims.

\medskip

{\bf Claim 1.} Let $g_j \ast s^{\epsilon_j} \ast h_j  \in P$ $j = 1,2$. If
$$g_1 \ast s^{\epsilon_1} \ast h_1 = g_2 \ast s^{\epsilon_2} \ast h_2$$
then $\epsilon_1 = \epsilon_2$ and $h_1 \ast h_2^{-1} \in A$ if $\epsilon_1 = -1$, and $h_1 \ast
h_2^{-1} \in B$ if $\epsilon_1 = 1$.

\smallskip

To prove the claim consider an $s$-form
$$a = (g_1, s^{\epsilon_1}, h_1 \ast h_2^{-1}, s^{-\epsilon_2},
g_2^{-1}).$$
By Lemma \ref{le:p1}, $w(a)$ is defined and
$$g_1 \ast s^{\epsilon_1} \ast h_1 \ast h_2^{-1} \ast s^{-\epsilon_2} \ast g_2^{-1} = \varepsilon.$$
Hence, $a$ is not reduced and the claim follows.

\medskip

For every $p \in P$ we fix now a representation $p = g_p \ast s^{\epsilon_p} \ast h_p$, where $g_p,
h_p \in H, \epsilon_p \in \{-1,0,1\}$.

\medskip

{\bf Claim 2.} Let $p = g_p \ast s^{\epsilon_p} \ast h_p, \ \ q = g_q \ast s^{\epsilon_q} \ast h_q$
be in $P$. If $p \ast q \in P$ then either $\epsilon_p \epsilon_q = 0$, or $\epsilon_p = - \epsilon_q
\neq 0$ and $h_p \ast g_q \in A$ if $\epsilon_p = -1$, and $h_p \ast g_q \in B$ if $\epsilon_q = 1$.

\smallskip

Let $x^{-1} = p \ast q \in P$ and $x = g_x \ast s^{\epsilon_x} \ast h_x$. Assume that $\epsilon_p
\epsilon_q \neq 0$.

\smallskip

\begin{enumerate}
\item[(a)] $\epsilon_x \neq 0$

\smallskip

Consider an $s$-form
$$a = (g_p, s^{\epsilon_p}, h_p \ast g_q, s^{\epsilon_q}, h_q \ast g_x, s^{\epsilon_x}, h_x).$$
By Lemma \ref{le:p1}, $w(a)$ is defined and
$$w(a) = g_p \ast s^{\epsilon_p} \ast h_p \ast g_q \ast s^{\epsilon_q} \ast h_q \ast g_x \ast
s^{\epsilon_x} \ast h_x = \varepsilon.$$
Hence, $a$ is not reduced and either a subsequence
$$\{s^{\epsilon_p}, h_p \ast g_q, s^{\epsilon_q}\},$$
or a subsequence
$$\{s^{\epsilon_q}, h_q \ast g_x, s^{\epsilon_x}\}$$
is reducible. In the former case we are done, so assume that $\{s^{\epsilon_q}, h_q \ast g_x,
s^{\epsilon_x}\}$ can be reduced. Without loss of generality we can assume that $\epsilon_q = -1,\
\epsilon_x = 1,\ h_q \ast g_x \in A$. Hence,
$$s^{\epsilon_q} \ast h_q \ast g_x \ast s^{\epsilon_x} = g \in B$$
and we have
$$w(a) = g_p \ast s^{\epsilon_p} \ast h_p \ast g_q \ast g \ast h_x = \varepsilon.$$
Now, it follows $\epsilon_p = 0$ - a contradiction with our assumption.

\smallskip

\item[(b)] $\epsilon_x = 0$

\smallskip

Hence, $x = g \in H$ and we consider an $s$-form
$$a = (g_p, s^{\epsilon_p}, h_p \ast g_q, s^{\epsilon_q}, h_q \ast g).$$
By Lemma \ref{le:p1}, $w(a)$ is defined and
$$w(a) = g_p \ast s^{\epsilon_p} \ast h_p \ast g_q \ast s^{\epsilon_q} \ast h_q \ast g = \varepsilon.$$
\end{enumerate}

Now, the claim follows automatically.

\medskip

Below we call a tuple $y = (y_1, \ldots, y_k) \in P^k$ a {\em reduced} $P$-sequence if
$y_j \ast y_{j+1} \notin P$ for $j \in [1,k-1]$. Observe, that if $y = (y_1, \ldots, y_k)$ is a
reduced $P$-sequence and $y_j = g_j \ast s_i^{\epsilon_j} \ast h_j$ then either $k \leqslant 1$ or
$y$ has the following properties which follow from Claim 2:
\begin{enumerate}
\item[(a)] $\epsilon_j \neq 0$ for all $j \in [1,k]$,
\item[(b)] if $\epsilon_j = -1,\ \epsilon_{j+1} = 1$ then $h_j \ast g_{j+1} \notin A$ for $j \in [1,k-1]$,
\item[(c)] if $\epsilon_j = 1,\ \epsilon_{j+1} = -1$ then $h_j \ast g_{j+1} \notin B$ for $j \in [1,k-1]$.
\end{enumerate}
In particular, the $s$-form over $H$
$$p_y = (g_1,\ s^{\epsilon_1},\ h_1 \ast g_2,\ s^{\epsilon_2},\
\ldots,\ h_{n-1} \ast g_n,\ s^{\epsilon_k},\ h_n),$$
is reduced.

\medskip

To prove (1) observe first that $P^{-1} = P$. Now if $y_1, \ldots, y_k \in
P$ then $y_1 \ast \cdots \ast y_k = w(p_y)$, where $y = (y_1, \ldots, y_k)$. Hence, by Lemma \ref{le:p1},
the product $y_1 \ast y_2 \ast \cdots \ast y_k$ is defined in $CDR(\Zt,X)$ and it belongs to
$CDR(\Zt,X)$. It follows that $H^* = \langle P \rangle$ is a subgroup of $CDR(\Zt,X)$
which consists of all words $w(p)$, where $p$ ranges through all possible $s$-forms over $H$. Hence, (1)
is proved.

\smallskip

Now we prove (2). By Theorem 2, \cite{Rimlinger:1987(2)}, to prove that $P$ is a pregroup and the
inclusion $P \to H^*$ extends to an isomorphism $U(P) \simeq H^*$ it is enough to show that all
reduced $P$-sequences representing the same element have the same $P$-length.

\smallskip

Suppose two reduced $P$-sequences
$$(u_1, u_2, \ldots, u_k),\ (v_1, v_2, \ldots, v_n)$$
represent the same element $g \in H^*$. That is,
$$(u_1 \ast \cdots \ast u_k) \ast (v_1 \ast \cdots \ast v_n)^{-1} = \varepsilon.$$
We use induction on $k+n$ to show that $k = n$. Observe that $k = 0$ implies $n = 0$, otherwise we
get a contradiction with Lemma \ref{le:p1} (3). Hence, we can assume $k, n > 0$, that is, $k + n
\geqslant 2$. If the $P$-sequence
$$a = (u_1, \ldots, u_k, v_n^{-1}, \ldots, v_1^{-1})$$
is reduced then the underlying $s$-form is reduced and hence, by Lemma \ref{le:p1} (3)
$$w(a) = u_1 \ast \ldots \ast u_k \ast v_n^{-1} \ast \ldots \ast v_1^{-1} \neq \varepsilon.$$
Hence,
$$(u_1, \ldots, u_k, v_n^{-1}, \ldots, v_1^{-1})$$
is not reduced and $u_k \ast v_n^{-1} \in P$. If $u_k = g_1 \ast s^{\epsilon_1} \ast h_1,\ v_n =
g_2 \ast s^{\epsilon_2} \ast h_2$, where $g_i, h_i \in H$ and $\epsilon_i \in \{-1,0,1\},\ i = 1, 2$
then by  Claim 2 either $\epsilon_1 \epsilon_2 = 0$, or $\epsilon_1 = \epsilon_2 \neq 0$ and $h_1
\ast h_2^{-1} \in A$ if $\epsilon_1 = -1$, and $h_1 \ast h_2^{-1} \in B$ if $\epsilon_1 = 1$. In the
former case, for example, if $\epsilon_2 = 0$ then $n = 1,\ v_n \in H$ and $b = (u_1, \ldots, u_k
\ast v_n^{-1})$ is a reduced $P$-sequence such that $w(b) = \varepsilon$ - a contradiction with
Lemma \ref{le:p1} (3) unless $k = 1,\ u_1 \in H$. In the latter case, $u_k \ast v_n^{-1} \in H$ and
it follows that
$$(u_1, u_2, \ldots, u_{k-1} \ast (u_k \ast v_n^{-1})),\ (v_1, v_2, \ldots, v_{n-1})$$
represent the same element in $H^*$ while the sum of their lengths is less than $k+n$, so the result
follows by induction.

\medskip

Finally, to prove (3) observe first that $H$ embeds into $G$. We denote this embedding by $\theta$.
Now we define a map $\phi : P \rightarrow G$ as follows. For $g \ast s^\epsilon \ast h \in P$ put
$$g \ast s^\epsilon \ast h \ \stackrel{\phi}{\rightarrow}\ \theta(g)\ z^\epsilon\ \theta(h).$$
It follows from Claim 2 that $\phi$ is a morphism of pregroups. Since $H^* \simeq U(P)$, the
morphism $\phi$ extends to a unique homomorphism $\psi : H^* \rightarrow G$. We claim that $\psi$
is bijective. Indeed, observe first that $G = \langle H, z \rangle$. Now, since $\psi(s^\epsilon)
= z^\epsilon$ and $\psi = \phi = \theta$ on $H$, it follows that $\psi$ is onto. To see that $\psi$
is one-to-one it suffices to notice that if
$$(g_1 \ast s^{\epsilon_1} \ast h_1, g_2 \ast s^{\epsilon_2} \ast h_2, \ldots, g_m \ast s^{\epsilon_m}
\ast h_m)$$
is a reduced $P$-sequence then
$$y = (g_1, s^{\epsilon_1}, h_1 \ast g_2, s^{\epsilon_2}, \ldots, s^{\epsilon_m}, h_m)$$
is a reduced $s$-form and $w(y)^\psi \neq 1$ by Britton's Lemma (see, for example,
\cite{Magnus_Karras_Solitar:1977}). This proves that $\psi$ is an isomorphism, as required.
\end{proof}

\begin{theorem} \cite{KMRS:2012}
\label{th:Greg} Let $G = \langle H, z \mid z^{-1} A z = B \rangle$. Then, in the notation above,
the free length function on $L : G \rightarrow \mathbb{Z}^{n+1}$ induced by the isomorphism $\psi :
H^* \rightarrow G$ is regular.
\end{theorem}
\begin{proof} Observe that is is enough to show that the length function induced on $H^* = \langle P
\rangle$ from $CDR(\Zt,X)$ is regular.

Let $g,h \in H^*$. Then $g$ and $h$ can be written in the unique normal forms
$$g = g_1 \circ (u_1^{-N_1} \ast s^{\epsilon_1} \ast v_1^{-M_1}) \circ g_2 \circ \cdots \circ
(u_k^{-N_k} \ast s^{\epsilon_k} \ast v_k^{-M_k}) \circ g_{k+1},$$
$$h = h_1 \circ (w_1^{-L_1} \ast s^{\delta_1} \ast x_1^{-P_1}) \circ h_2 \circ \cdots \circ
(w_m^{-L_m} \ast s^{\delta_m} \ast x_m^{-P_m}) \circ h_{m+1},$$
where $N_j, M_j \geqslant 0,\ u_j = u, v_j = v$ if $\epsilon_j = 1$, and $N_j, M_j \leqslant 0,\ u_j = v, v_j
= u$ if $\epsilon_j = -1$ for $j \in [1,k];\ L_i, P_i \geqslant 0,\ w_i = u, x_i = v$ if $\delta_i = 1$,
and $L_i, P_i \leqslant 0,\ w_i = v, x_i = u$ if $\delta_i = -1$ for $i \in [1,m]$. Moreover, $g_1$ does
not have $u_1^{\pm 1}$ as a terminal subword, $g_j$ does not have $u_j^{\pm 1}$ as a terminal
subword for every $j \in [2,k]$, and $g_j$ does not have $v_{j-1}^{\pm 1}$ as an initial subword
for every $j \in [2,k]$, $g_{k+1}$ does not have $v_k^{\pm 1}$ as an initial subword; $h_1$ does
not have $w_1^{\pm 1}$ as a terminal subword, $h_i$ does not have $w_i^{\pm 1}$ as a terminal
subword for every $i \in [2,m]$, and $h_i$ does not have $x_{j-1}^{\pm 1}$ as an initial subword
for every $i \in [2,m]$, $h_{m+1}$ does not have $x_m^{\pm 1}$ as an initial subword.

\smallskip

If there exist $k_1,k_2 > 0$ such that
$$c = c(g,h) \leqslant \min\{|g_1 \circ u_1^{k_1}|,|h_1 \circ w_1^{k_2}|\}$$
then $com(g,h) = com(g_1 \circ u_1^{k_1}, h_1 \circ w_1^{k_2}) \in H$. Now, assume that $r \in [1,k]$
is the minimal natural number such that $com(g,h)$ is an initial subword of
$$f_1 = g_1 \circ (u_1^{-N_1} \ast s^{\epsilon_1} \ast v_1^{-M_1}) \circ g_2 \circ \cdots \circ
(u_r^{-N_r} \ast s^{\epsilon_r} \ast v_r^{-M_r}) \circ g_{r+1} \circ u_r^{p_1},$$
where $p_1 \in \mathbb{Z}$. Similarly, assume that $q \in [1,m]$ is the minimal natural number such
that $com(g,h)$ is an initial subword of
$$f_2 = h_1 \circ (w_1^{-L_1} \ast s^{\delta_1} \ast x_1^{-P_1}) \circ h_2 \circ \cdots \circ
(w_q^{-L_q} \ast s^{\delta_q} \ast x_q^{-P_q}) \circ h_{q+1} \circ w_q^{p_2},$$
where $p_2 \in \mathbb{Z}$. From uniqueness of normal forms it follows that $r = q$ and we have $g_i
= h_i,\ u_i = w_i,\ v_i = x_i,\ N_i = L_i,\ \epsilon_i = \delta_i,\ i \in [1,r]$ and $M_i = P_i,\ i
\in [1,r-1]$.

\smallskip

Without loss of generality we can assume $\epsilon_r = 1$. Hence, $v_r = x_r = v$.

Observe that $com(g,h)$ can be represented as a concatenation $com(g,h) = c_1 \circ c_2$, where
$$c_1 = g_1 \circ (u_1^{-N_1} \ast s^{\epsilon_1} \ast v_1^{-M_1}) \circ g_2 \circ \cdots \circ
(u^{-N_r} \ast s \ast v^{-l})$$
and $l \geqslant \max\{M_r,P_r\}$, and
$$c_2 = com(v^{l-M_r} \circ g_{r+1} \circ u_r^{p_1}, v^{l-P_r} \circ h_{r+1} \circ w_r^{p_2}).$$
Obviously, $c_1 \in H^*$. Also, $c_2 \in H$ since $v^{l-M_r} \circ g_{r+1} \circ u_r^{p_1},\
v^{l-P_r} \circ h_{r+1} \circ w_r^{p_2} \in H$ and the length function on $H$ is regular.
Hence, $com(g,h) \in H^*$.
\end{proof}

\subsection{Completions of $\mathbb{Z}^n$-free groups}
\label{subs:completions_Z^n}

Let $G$ be a finitely generated subgroup of $CDR(\mathbb{Z}^n, X)$, where $\mathbb{Z}^n$ is ordered
with respect to the right lexicographic order. Here we do not assume $X$ to be finite. We are going
to construct a finite alphabet $Y$ and a finitely generated group $H$ which is subgroup of
$CDR(\mathbb{Z}^n, Y)$ such that the length function on $H$ induced from $CDR(\mathbb{Z}^n, Y)$ is
regular and $G$ embeds into $H$ so that the length is preserved by the embedding. In other words, we
are going to construct a finitely generated {\em $\mathbb{Z}^n$-completion} of $G$ (see \cite{KMS:2011}).
All the proofs and details can be found in \cite{KMS:2011(2)}.

\smallskip

Consider a finitely generated $\mathbb{Z}^n$-free group $G$, where $n \in \mathbb{N}$.
Suppose $n > 1$ and consider the $\mathbb{Z}^n$-tree $(\Gamma_G, d)$ which arises from the embedding
of $G$ into $CDR(\mathbb{Z}^n, X)$.

\smallskip

We say that $p,q \in \Gamma_G$ are {\em $\mathbb{Z}^{n-1}$-equivalent} ($p \sim q$) if $d(p,q) \in
\mathbb{Z}^{n-1}$, that is, $d(p,q) = (a_1, \ldots, a_n),\ a_n = 0$. From metric axioms it follows
that ``$\sim$'' is an equivalence relation and every equivalence class defines a {\em
$\mathbb{Z}^{n-1}$-subtree} of $\Gamma_G$.

\smallskip

Let $\Delta_G = \Gamma_G / \sim$ and denote by $\rho$ the mapping $\Gamma_G \rightarrow \Gamma_G /
\sim$. It is easy to see that $\Delta_G$ is a simplicial tree. Indeed, define $\widetilde{d} :
\Delta_G \rightarrow \mathbb{Z}$ as follows:
\begin{equation}
\label{eq:new_length}
\forall\ \widetilde{p},\ \widetilde{q} \in \Delta_G:\ \ \widetilde{d}(\widetilde{p},\widetilde{q}) =
k\ \ {\rm iff}\ \ d(p,q) = (a_1, \ldots, a_n)\ \ {\rm and}\ \ a_n = k.
\end{equation}
From metric properties of $d$ it follows that $\widetilde{d}$ is a well-defined metric.

Since $G$ acts on $\Gamma_G$ by isometries then $p \sim q$ implies $g \cdot p \sim g \cdot q$ for
every $g \in G$. Moreover, if $d(p,q) = (a_1, \ldots, a_n)$ then $d(g \cdot p, g \cdot q) = (a_1,
\ldots, a_n)$. Hence, $\widetilde{d}(g \cdot \widetilde{p}, g \cdot \widetilde{q}) = \widetilde{d}(
\widetilde{p}, \widetilde{q})$, that is, $G$ acts on $\Delta_G$ by isometries, but the action is not
free in general. From Bass-Serre theory it follows that $\Psi_G = \Delta_G / G$ is a graph in which
vertices and edges correspond to $G$-orbits of vertices and edges in $\Delta_G$.

\begin{lemma} \cite{KMS:2011(2)}
\label{le:psi}
$\Psi_G$ is a finite graph.
\end{lemma}

From Lemma \ref{le:psi} it follows that the number of $G$-orbits of $\mathbb{Z}^{n-1}$-subtrees in
$\Gamma_G$ is finite and equal to $|V(\Psi_G)|$. So, let $|V(\Psi_G)| = m$ and ${\cal T}_1, \ldots,
{\cal T}_m$ be these $G$-orbits.

Consider $\Psi_G$. The set of vertices and edges of $\Psi_G$ we denote correspondingly by
$V(\Psi_G)$ and $E(\Psi_G)$ so that
$$\sigma : E(\Psi_G) \rightarrow V(\Psi_G), \ \ \tau : E(\Psi_G) \rightarrow V(\Psi_G),\ \ ^- :
E(\Psi_G) \rightarrow E(\Psi_G)$$
satisfy the following conditions:
$$\sigma(\bar{e}) = \tau(e),\ \tau(\bar{e}) = \sigma(e),\ \bar{\bar{e}}= e,\ \bar{e} \neq e.$$
Let ${\cal T}$ be a maximal subtree of $\Psi_G$ and let $\pi : \Delta_G \rightarrow \Delta_G / G =
\Psi_G$ be the canonical projection of $\Delta_G$ onto its quotient, so $\pi(v) = Gv$ and $\pi(e) =
G e$ for every $v \in V(\Delta_G),\ e \in E(\Delta_G)$. There exists an injective morphism of
graphs $\eta : {\cal T} \rightarrow \Delta_G$ such that $\pi \circ \eta = id_{\cal T}$ (see Section
8.4 of \cite{Cohen}), in particular $\eta({\cal T})$ is a subtree of $\Delta_G$. One can extend
$\eta$ to a map (which we again denote by $\eta$) $\eta: \Psi_G \rightarrow \Delta_G$ such that
$\eta$ maps vertices to vertices, edges to edges, and so that $\pi \circ \eta = id_{\Psi_G}$.
Notice, that in general $\eta$ is not a graph morphism. To this end choose an orientation $O$ of the
graph $\Psi_G$. Let $e \in O - {\cal T}$. Then there exists an edge $e^\prime \in \Delta_G$ such
that $\pi(e^\prime) = e$. Clearly, $\sigma(e^\prime)$ and $\eta(\sigma(e))$ are in the same
$G$-orbit. Hence $g \cdot \sigma(e^\prime) = \eta(\sigma(e))$ for some $g \in G$. Define $\eta(e) =
g \cdot e^\prime$ and $\eta(\bar{e}) = \overline{\eta(e)}$. Notice that vertices $\eta(\tau(e))$ and
$\tau(\eta(e))$ are in the same $G$-orbit. Hence there exists an element $\gamma_e \in G$ such that
$\gamma_e \cdot \tau(\eta(e)) = \eta(\tau(e))$.

Put
$$G_v = Stab_G(\eta(v)),\ G_e = Stab_G(\eta(e))$$
and define boundary monomorphisms as inclusion maps $i_e : G_e \hookrightarrow G_{\sigma(e)}$ for
edges $e \in {\cal T} \cup O$ and as conjugations by $\gamma_{\bar{e}}$ for edges $e \notin {\cal T}
\cup O$, that is,
\[ i_e(g) = \left\{ \begin{array}{ll}
\mbox{$g$,} & \mbox{if $e \in {\cal T} \cup O$,} \\
\mbox{$\gamma_{\bar{e}} g \gamma_{\bar{e}}^{-1}$,} & \mbox{if $e \notin {\cal T} \cup O$.}
\end{array}
\right. \]
According to the Bass-Serre structure theorem we have
\begin{equation}
\label{eq:presentation}
G \simeq \pi({\cal G},\Psi_G,{\cal T}) = \langle G_v \ (v \in V(\Psi_G)),\ \gamma_e\ (e \in
E(\Psi_G)) \mid rel(G_v),
\end{equation}
$$\gamma_e i_e(g) \gamma_e^{-1} = i_{\bar{e}}(g)\ (g \in G_e),\ \gamma_e \gamma_{\bar{e}} = 1,\
\gamma_e = 1\ (e \in {\cal T})\rangle.$$

\smallskip

Let ${\cal K} = \rho^{-1}(\eta({\cal T})),\ \overline{\cal K} = \rho^{-1}(\eta(\Psi_G))$, hence,
${\cal K},\ \overline{\cal K}$ are subtrees of $\Gamma_G$ such that ${\cal K} \subseteq
\overline{\cal K}$. Obviously $T_0 \subseteq {\cal K}$. Moreover, both ${\cal K}$ and
$\overline{\cal K}$ contain finitely many $\mathbb{Z}^{n-1}$-subtrees, and meet every $G$-orbit of
$\mathbb{Z}^{n-1}$-subtrees of $\Gamma_G$.

For every $v \in \Psi_G$ we have $Stab_G(\eta(v)) = Stab_G(T_{\eta(v)})$, where $\eta(v) =
\rho(T_{\eta(v)})$. Denote by $T_0$ the $\mathbb{Z}^{n-1}$-subtree containing $\varepsilon$.
Obviously, $Stab_G(T_0)$ is a subgroup of $CDR(\mathbb{Z}^{n-1}, X)$.

\begin{lemma} \cite{KMS:2011(2)}
\label{le:stab}
Let $T$ be a $\mathbb{Z}^{n-1}$-subtree of ${\cal K}$. Then
$$Stab_G(T) = f_T \ast K_T \ast f_T^{-1},$$
where $K_T$ is a subgroup of $CDR(\mathbb{Z}^{n-1}, X)$ (possibly trivial) and $f_T =
\mu([\varepsilon, x_T]) \in CDR(\mathbb{Z}^n, X)$. Moreover, is $Stab_G(T)$ is not trivial then $x_T
\in Axis(g) \cap T$ for some $g \in Stab_G(T)$.
\end{lemma}

Let $e$ be an edge of $\Psi_G$ such that $e \in O,\ e \notin {\cal T}$. Let $v = \sigma(\eta(e)) =
\eta(\sigma(e)),\ w = \tau(\eta(e))$ and $u = \eta(\tau(e)) = \gamma_e \cdot w$. We have $u,v \in
\eta({\cal T}),\ w \notin \eta({\cal T})$. Hence,
$$\gamma_e\ Stab_G(w)\ \gamma_e^{-1} = Stab_G(u).$$
By definition we have $i_e(G_e) \subseteq G_v = Stab_G(T)$, where $T = \rho^{-1}(v)$ and
$i_{\bar{e}}(G_e) = \gamma_e G_e \gamma_e^{-1} \subseteq G_u = Stab_G(S)$, where $S = \rho^{-1}(u)$.
Thus, we have $i_e(G_e) = f_T \ast A \ast f_T^{-1},\ i_{\bar{e}}(G_e) = f_S \ast B \ast f_S^{-1}$,
where $A \leqslant K_T$ and $B \leqslant K_S$ are isomorphic abelian subgroups of
$CDR(\mathbb{Z}^{n-1}, X)$. So,
$$\gamma_e \ast (f_T \ast A \ast f_T^{-1}) \ast \gamma_e^{-1} = f_S \ast B \ast f_S^{-1}$$
and it follows that $f_S^{-1} \ast \gamma_e \ast f_T = r_e \in CDR(\mathbb{Z}^n, X)$ so that
$r_e \ast A \ast r_e^{-1} = B$. Thus, we have
$$\gamma_e = f_S \ast r_e \ast f_T^{-1}.$$
Observe that $r_e \in CDR(\mathbb{Z}^n, X) - CDR(\mathbb{Z}^{n-1}, X)$ because otherwise $\gamma_e
\cdot T = S$, that is, $u = v,\ S = T$ and thus $\gamma_e \in Stab_G(T)$ - a contradiction.

\subsubsection{Simplicial case}
\label{subs:n=1}

Let $G$ be a finitely generated subgroup of $CDR(\mathbb{Z},X)$. Hence, $\Gamma_G$ is a simplicial
tree and $\Delta = \Gamma_G / G$ is a folded $X$-labeled digraph (see \cite{Kapovich_Myasnikov:2002})
with labeling induced from $\Gamma_G$. $\Delta$ is finite which follows from the fact that $G$ is finitely
generated and from the construction of $\Gamma_G$. Moreover, $\Delta$ recognizes $G$ with respect
to some vertex $v$ (the image of $\varepsilon$) in the sense that $g \in CDR(\mathbb{Z},X)$ belongs
to $G$ if and only if there exists a loop in $\Delta$ at $v$ such that its label is exactly $g$.

The following lemma provides the required result.

\begin{lemma} \cite{KMS:2011(2)}
\label{le:n=1}
Let $G$ be a finitely generated subgroup of $CDR(\mathbb{Z},X)$. Then there exists a finite alphabet
$Y$ and an embedding $\phi : G \rightarrow H$, where $H = F(Y)$, inducing an embedding $\psi :
\Gamma_G \rightarrow \Gamma_H$ such that
\begin{enumerate}
\item[(i)] $|g|_G = |\phi(g)|_H$ for every $g \in G$,
\item[(ii)] if $A$ is a maximal abelian subgroup of $G$ then $\phi(A)$ is a maximal abelian subgroup
of $H$,
\item[(iii)] if $a$ and $b$ are non-$G$-equivalent ends of $\Gamma_G$ then $\psi(a)$ and $\psi(b)$
are non-$H$-equivalent ends of $\Gamma_H$,
\item[(iv)] if $A$ and $B$ are maximal abelian subgroups of $G$ which are not conjugate in $G$ then
$\phi(A)$ and $\phi(B)$ are not conjugate in $H$.
\end{enumerate}
\end{lemma}

Lemma \ref{le:n=1} can be generalized to the following result.

\begin{cor} \cite{KMS:2011(2)}
\label{co:n=1}
Let $G$ be a finitely generated subgroup of $CDR(\mathbb{Z},X)$. Assume that $\Gamma_G$ is embedded
into a $\mathbb{Z}$-tree $T$ whose edges are labeled by $X^{\pm}$ so that the action of $G$ on
$\Gamma_G$ extends to an action of $G$ on $T$, and there are only finitely many $G$-orbits of ends
of $T$ which belong to $T - \Gamma_G$. Then there exists a finite alphabet $Y$, a $\mathbb{Z}$-tree
$T'$ whose edges are labeled by $Y^{\pm}$, and a finitely generated subgroup $H \subseteq
CDR(\mathbb{Z},Y)$ such that $\Gamma_H$ is embedded into $T'$ so that the action of $H$ on
$\Gamma_H$ extends to a regular action of $H$ on $T'$. Also, there is an embedding $\theta: T \to
T'$, where $\theta(\Gamma_G) \subseteq \Gamma_H$, which indices an embedding $\phi : G \rightarrow
H$ such that
\begin{enumerate}
\item[(i)] $|g|_G = |\phi(g)|_H$ for every $g \in G$,
\item[(ii)] if $A$ ia a maximal abelian subgroup of $G$ then $\phi(A)$ is a maximal abelian
subgroup of $H$,
\item[(iii)] if $a$ and $b$ are non-$G$-equivalent ends of $T$ then $\theta(a)$ and $\theta(b)$ are
non-$H$-equivalent ends of $T'$.
\end{enumerate}
\end{cor}

\subsubsection{General case}
\label{subs:general}

Let $G$ be a finitely generated subgroup of $CDR(\mathbb{Z}^n, X)$ for some alphabet $X$. We are
going to use the notations introduced in Subsection \ref{subs:completions_Z^n}, that is, we assume
that ${\cal K},\ \Psi_G,\ \Delta_G$ etc. are defined for $G$ as well as the presentation
(\ref{eq:presentation}).

\smallskip

First of all, we relabel $\Gamma_G$ so that non-$G$-equivalent $\mathbb{Z}^{n-1}$-subtrees are
labeled by disjoint alphabets.

\smallskip

Recall that every edge $e$ in $\Gamma_G$ is labeled by a letter $\mu(e) \in X^\pm$. Let $T$ be a
$\mathbb{Z}^{n-1}$-subtree of ${\cal K}$ and $X_T$ a copy of $X$ (disjoint from $X$) so that we
have a bijection $\pi_T : X \rightarrow X_T$, where $\pi_T(x^{-1}) = \pi_T(x)^{-1}$ for every $x
\in X$. We assume $X_S \cap X_T = \emptyset$ for distinct $S, T \in {\cal K}$. Let $\Gamma'$ be a
copy of $\Gamma_G$ and $\nu : \Gamma' \rightarrow \Gamma_G$ a natural bijection (the bijection on
points naturally induces the bijection on edges). Denote $\varepsilon' = \nu^{-1}(\varepsilon)$.

Let $X' = \bigcup \{ X_T \mid T \in {\cal K} \}$. We introduce a labeling function $\mu' : E(\Gamma')
\rightarrow X'^{\pm}$ as follows: $\mu'(e) = \pi_T(\mu(\nu(e)))$ if $\nu(e) \in T$. $\mu'$ naturally
extends to the labeling of paths in $\Gamma'$. Now, if $V' = \nu^{-1}(V_G)$ then define
$$G' = \{ \mu'(p) \mid p = [\varepsilon', v']\ {\rm for\ some}\ v' \in V'\}.$$

\begin{lemma} \cite{KMS:2011(2)}
\label{le:relabel}
$G'$ is a subgroup of $CDR(\mathbb{Z}^n, X')$ which acts freely on $\Gamma'$ and there exists an
isomorphism $\phi : G \rightarrow G'$ such that $L_\varepsilon(g) = L_{\varepsilon'}(\phi(g))$.
\end{lemma}

According to Lemma \ref{le:relabel} we have $\Gamma' = \Gamma_{G'}$. Observe that the structure of
$\mathbb{Z}^{n-1}$-trees in $\Gamma_{G'}$ is the same as in $\Gamma_G$. Hence, if ``$\sim$'' is a
$\mathbb{Z}^{n-1}$-equivalence of points of $\Gamma_{G'}$ then $\Delta_{G'} = \Gamma_{G'} / \sim$
and $\Psi_{G'} = \Delta_{G'} / G'$ are naturally isomorphic respectively to $\Delta_G = \Gamma_G /
\sim$ and $\Psi_G = \Delta_G / G$. So, with a slight abuse of notation let $X = X',\ G = G'$.

\smallskip

Next, we would like to refine the labeling so as to make the alphabet $X$ finite. To do this we have
to analyze the structure of the $\mathbb{Z}^{n-1}$-subtrees of ${\cal K}$.

\begin{lemma} \cite{KMS:2011(2)}
\label{le:trivial_stab}
Let $T$ be a $\mathbb{Z}^{n-1}$-subtree of ${\cal K}$ such that $Stab_G(T)$ is trivial. Then $T$
contains only finitely many branch-points and each branch-point of $T$ is of the form
$Y(\varepsilon, x, y)$, where $x, y \in \{x_S\ (S \in {\cal K}),\ \gamma_e^{\pm 1} \cdot \varepsilon
\ (e \in \Psi_G)\}$.
\end{lemma}

In particular, from Lemma \ref{le:trivial_stab} it follows that every $\mathbb{Z}^{n-1}$-subtree $T$
of ${\cal K}$ with trivial stabilizer can be relabeled by a finite alphabet. Indeed, $T$ may be cut
at its branch-points into finitely many closed segments and half-open rays which do not contain any
branch-points. Then all these segments and rays can be labeled by different letters (all points in
each piece is labeled by one letter).

In the case of non-trivial stabilizer the situation is a little more complicated.

\begin{lemma} \cite{KMS:2011(2)}
\label{le:non_trivial_stab}
Let $T$ be a $\mathbb{Z}^{n-1}$-subtree of ${\cal K}$ such that $Stab_G(T) = f_T \ast K_T \ast
f_T^{-1}$ is non-trivial. Then $\Gamma_{K_T}$ embeds into $T$ (the base-point of $\Gamma_{K_T}$ is
identified with $x_T$), the action of $K_T$ on $\Gamma_{K_T}$ extends to the action of $K_T$ on
$T$ and the following hold
\begin{enumerate}
\item[(a)] every end of $T$ which does not belong to $\Gamma_{K_T}$ is $K_T$-equivalent to one of
the ends of a finite subtree which is the intersection of $T$ and the segments $[\varepsilon,
x_S],\ S \in {\cal K}$,
\item[(b)] every end $a$ of $T$ which does not belong to $\Gamma_{K_T}$ extends the axis of some
centralizer $C_a$ of $K_T$,
\item[(c)] there are only finitely many $K_T$-orbits of branch-points of $T$ which do not belong to
$\Gamma_{K_T}$,
\item[(d)] if $K_T \subset CDR(\mathbb{Z}^{n-1}, Y)$ for some finite alphabet $Y$ then the labeling
of $\Gamma_{K_T}$ by $Y$ can be $K_T$-equivariantly extended to a labeling of $T$ by a finite
extension $Y'$ of $Y$.
\end{enumerate}
\end{lemma}

\begin{cor} \cite{KMS:2011(2)}
\label{co:label}
If $G$ is a finitely generated subgroup of $CDR(\mathbb{Z}^n, X)$ then $X$ can be taken to be finite.
\end{cor}
\begin{proof} Follows from Lemma \ref{le:trivial_stab} and Lemma \ref{le:non_trivial_stab}.
\end{proof}

For a non-linear $\mathbb{Z}^{n-1}$-subtree $T$ of ${\cal K}$ with a non-trivial stabilizer let
${\cal B(T)}$ be the set of representatives of branch-points of $T - \Gamma_{K_T}$. By Lemma
\ref{le:non_trivial_stab}, ${\cal B(T)}$ is finite and every branch-point of $T$ which does not belong to
$\Gamma_{K_T}$ is $K_T$-equivalent to a branch-point from ${\cal B(T)}$. Let
$${\cal D}(T) = \{ \mu([x_T,y]) \mid y \in {\cal B(T)} \}.$$
Observe that ${\cal D}(T)$ is a finite subset of $CDR(\mathbb{Z}^{n-1},X)$.

\smallskip

Let $g \in G$. Hence, $[\varepsilon, g \cdot \varepsilon]$ meets finitely many
$\mathbb{Z}^{n-1}$-subtrees $T_0, T_1, \ldots, T_k$, where $T(g)_0 = T_0$ and $T_i$ is adjacent to
$T_{i-1}$ for each $i \in [1,k]$. Observe that $T_0$ is $\mathbb{Z}^{n-1}$-subtree of ${\cal K}$. We
have
$$[\varepsilon, g \cdot \varepsilon] \subseteq [x_{T_0}, x_{T_1}] \cup \cdots \cup [x_{T_{k-1}},
x_{T_k}] \cup [x_{T_k}, g \cdot \varepsilon].$$
Now, there exists $g_0 \in Stab_G(T_0)$ and a $\mathbb{Z}^{n-1}$-subtree $S_1$ of ${\cal K}$
adjacent to $T_0$ such that $T_1 = g_0 \cdot S_1$. Next, there exists $g_1 \in Stab_G(T_1)$ and a
$\mathbb{Z}^{n-1}$-subtree $S_2$ of ${\cal K}$ adjacent to $S_1$ such that $T_2 = (g_1 g_0) \cdot
S_2$, and so on. After $k$ steps we find a sequence of $\mathbb{Z}^{n-1}$-subtrees $S_0, S_1,
\ldots, S_k$ from ${\cal K}$, where $S_0 = T_0, S_i\ {\rm is\ adjacent\ to}\ S_{i-1},\ i \in [1,k]$
and $T_i = (g_{i-1} \cdots g_0) \cdot S_i$, where $g_i \in Stab_G(T_i)$. Hence,
$$[\varepsilon, g \cdot \varepsilon] \subseteq [x_{T_0}, g_0 \cdot x_{T_0}] \cup [g_0 \cdot x_{T_0},
g_0 \cdot x_{S_1}] \cup [g_0 \cdot x_{S_1}, x_{T_1}] \cup [x_{T_1}, (g_1 g_0) \cdot x_{S_1}]$$
$$\cup [(g_1 g_0) \cdot x_{S_1}, (g_1 g_0) \cdot x_{S_2}] \cup \cdots \cup [(g_{k-1} \cdots g_0)
\cdot x_{S_{k-1}}, (g_{k-1} \cdots g_0) \cdot x_{S_k}]$$
$$\cup [(g_{k-1} \cdots g_0) \cdot x_{S_k}, x_{T_k}] \cup [x_{T_k}, (g_k \cdots g_0) \cdot x_{S_k}],$$
where $(g_k \cdots g_0) \cdot x_{S_k} = g \cdot \varepsilon$.

Since
$$\mu([p,  q]) = \mu(g \cdot [p,  q]) = \mu([g \cdot p,  g \cdot q])$$
and
$$[(g_{i-1} \cdots g_0) \cdot x_{S_{i-1}}, (g_{i-1} \cdots g_0) \cdot x_{S_i}] = (g_{i-1} \cdots g_0)
\cdot [x_{S_{i-1}}, x_{S_i}]$$
for $i \in [1,k]$, then
$$\mu([(g_{i-1} \cdots g_0) \cdot x_{S_{i-1}}, (g_{i-1} \cdots g_0) \cdot x_{S_i}]) =
\mu([x_{S_{i-1}}, x_{S_i}]).$$
Also, observe that for any $i \in [1,k]$
$$[(g_{i-1} \cdots g_0) \cdot x_{S_i}, x_{T_i}] \cup [x_{T_i}, (g_i \cdots g_0) \cdot x_{S_i}]$$
is a path in $T_i$, where $(g_{i-1} \cdots g_0) \cdot x_{S_i}$ and $(g_i \cdots g_0) \cdot x_{S_i}$
are $Stab_G(T_i)$-equivalent to $x_{T_i}$. So, it follows that
$$\mu([x_{T_i}, (g_{i-1} \cdots g_0) \cdot x_{S_i}]) = f_i \in K_{T_i},\ \ \mu([x_{T_i},
(g_i \cdots g_0) \cdot x_{S_i}]) = h_i \in K_{T_i}.$$
Also, observe that $g_0 = \mu([x_{T_0}, g_0 \cdot x_{T_0}])$. Eventually, we have
$$g = g_0 \ast c_{S_0, S_1} \ast (f_1^{-1} \ast h_1) \ast c_{S_1, S_2} \ast \cdots \ast c_{S_{k-1},
S_k} \ast (f_k^{-1} \ast h_k),$$
where $c_{S_{i-1}, S_i}$ is the label of the path $[x_{S_{i-1}}, x_{S_i}]$ and the product on the
right-hand side is defined in $CDR(\mathbb{Z}^n,X)$.

\smallskip

Now we are ready to perform the inductive step.

\begin{theorem} \cite{KMS:2011(2)}
\label{th:completion}
Let $G$ be a finitely generated subgroup of $CDR(\mathbb{Z}^n, X)$ (assume that ${\cal K},\
\Psi_G,\ \Delta_G$ etc. are defined for $G$ as above). Suppose that for every non-linear
$\mathbb{Z}^{n-1}$-subtree $T$ of ${\cal K}$ with a non-trivial stabilizer there exists
\begin{enumerate}
\item[(a)] an alphabet $Y(T)$,
\item[(b)] a $\mathbb{Z}^{n-1}$-tree $T'$ whose edges are labeled by $Y(T)$,
\item[(c)] a finitely generated group $H_T \subset CDR(\mathbb{Z}^{n-1}, Y(T))$
\end{enumerate}
such that $\Gamma_{H_T}$ is embedded into $T'$ and the action of $H_T$ on $\Gamma_{H_T}$ extends to
a regular action of $H_T$ on $T'$. Moreover, assume that there is an embedding $\psi_T: T
\rightarrow T'$, where $\psi_T(\Gamma_{K_T}) \subseteq \Gamma_{H_T}$, which induces an embedding
$\phi_T : K_T \rightarrow H_T$, and such that if $a$ and $b$ are non-$K_T$-equivalent ends of $T$
then $\psi_T(a)$ and $\psi_T(b)$ are non-$H_T$-equivalent ends of $\psi_T(T)$.

Then there exists an embedding of ${\cal D}(T),\ T \in {\cal K}$ into $CDR(\mathbb{Z}^n, Y)$,
where $Y$ is a finite alphabet containing $\bigcup_{T \in {\cal K}} Y(T)$ such that
\begin{enumerate}
\item[(i)] $\{ H(T),\ {\cal D}(T), \{c_{x_T, x_S} \mid S\ {\rm is\ adjacent\ to}\ T\ {\rm in}\
{\cal K} \} \mid T \in {\cal K} \}$ generates a group $H$ of $CDR(\mathbb{Z}^n, Y)$ which acts
regularly on $\Gamma_H$ with respect to $\varepsilon_H$,
\item[(ii)] there exists an embedding $\psi: \Gamma_G \rightarrow \Gamma_H,\ \psi(\varepsilon_G) =
\varepsilon_H$ which induces an embedding $\phi : G \rightarrow H$, such that if $a$ and $b$ are
non-$G$-equivalent ends of $\Gamma_G$ then $\psi(a)$ and $\psi(b)$ are non-$H$-equivalent ends of
$\psi(\Gamma_G)$.
\end{enumerate}
\end{theorem}
\begin{proof} First of all, by Corollary \ref{co:label} we can assume $X$ to be finite. Hence, we
can assume that any two distinct $\mathbb{Z}^{n-1}$-subtrees $S$ and $T$ of ${\cal K}$ are labeled
distinct alphabets $X(S)$ and $X(T)$. Next, by Lemma \ref{le:trivial_stab}, in each
$\mathbb{Z}^{n-1}$-subtree $S$ of ${\cal K}$ with trivial stabilizer there are only finitely many
branch-points, so we can cut $S$ along these branch-points, obtain finitely many closed and
half-open segments, and relabel them by a finite alphabet. Thus we can assume all this to be done
already.

Let $T$ be a non-linear $\mathbb{Z}^{n-1}$-subtree of ${\cal K}$ with a non-trivial stabilizer.
Observe that by Lemma \ref{le:non_trivial_stab} every end $a$ of $T$ either is an end of $\Gamma_{K_T}$, or
$a = g \cdot b$, where $b$ is from a finite list of representatives of orbits of ends of $T -
\Gamma_{K_T}$.

By the assumption, $T$ embeds into $T'$ labeled by $Y(T)$, while $\Gamma_{K_T}$ embeds into
$\Gamma_{H(T)}$, where $H(T)$ acts regularly on $T'$. It follows that for every branch-point $b$ of
$T$ the label of $\psi_T([x_T,b])$ defines an element of $H(T)$. In particular, the label of
$\psi_T(d)$ belongs to $H(T)$ for every $d \in {\cal D}(T)$. Moreover, if $S_1, S_2$ are
$\mathbb{Z}^{n-1}$-subtrees of ${\cal K}$ adjacent to $T$ and $a_{S_1}, a_{S_2}$ are the
corresponding ends of $T$ then $a_{S_1}$ is not $H(T)$-equivalent to $a_{S_2}$. So, by the
assumption, $a_{S_1}$ is not $H(T)$-equivalent to $a_{S_2}$ and it follows that
$$(h_1 \cdot \theta([x_T,x_{S_1}] \cap T)) \cap (h_2 \cdot \theta([x_T,x_{S_2}] \cap T))$$
is a closed segment of $T'$, hence,
$$com(h_1 \ast c_{x_T, x_{S_1}}, h_2 \ast c_{x_T,x_{S_2}}),$$
is defined in $CDR(\mathbb{Z}^{n-1}, Y(T))$. Since $X(T) \cap X(S) = \emptyset$ then $h \ast
c^{-1}_{x_T,x_S} = h \circ c^{-1}_{x_T,x_S}$ for every $\mathbb{Z}^{n-1}$-subtree $S$ of
${\cal K}$ adjacent to $T$. Thus,
$$\{ H(T),\ {\cal D}(T), \{c_{x_T, x_S} \mid S\ {\rm is\ adjacent\ to}\ T\ {\rm in}\
{\cal K} \}\},$$
which is finite, generates a subgroup $H'(T)$ in $CDR(\mathbb{Z}^n, Q)$, where
$$Q = \bigcup_{T \in {\cal K}} Y(T),$$
so that $T$ embeds into $\Gamma_{H'(T)}$. Moreover, $H'(T)$ acts regularly on $\Gamma_{H'(T)}$.

Now, from the fact that alphabet $X(T)$ is disjoint from $X(S)$ if $T$ is not $G$-equivalent to
$S$ it follows that $\{H'(T) \mid T \in {\cal K} \}$ generates a subgroup $H$ of
$CDR(\mathbb{Z}^n,Y)$, where $Y$ is a finite alphabet containing $Q$. Observe that $\Gamma_{H'(T)}$
embeds into $\Gamma_H$ for each $T \in {\cal K}$. Moreover, for every $f, g \in H$ we have $w =
Y(\varepsilon_H,\ f \cdot \varepsilon_H,\ g \cdot \varepsilon_H)$ belongs to one of the subtrees
$\Gamma_{H'(T)}$, hence $[\varepsilon_H,w]$ defines an element of $H'(T) \subset H$. That is,
$H$ acts regularly on $\Gamma_H$.

Next, since
$$G \leqslant \langle K_T,\ \{ {\cal D}(T) \mid T \in {\cal K} \} \rangle \leqslant H$$
then $G$ embeds into $H$.

Finally, every end $a$ of $\Gamma_G$ uniquely corresponds to an end in $\Delta_G$. Every end of
$\Delta_G$ can be viewed as a reduced infinite path $p_a$ in $\Delta_G$ originating at $v \in
\Delta_G$ which is the image of $\varepsilon \in \Gamma_G$. Observe that two ends $a$ and $b$ of
$\Gamma_G$ are $G$-equivalent if and only if $\pi(p_a) = \pi(p_b)$ in $\Psi_G$.

Denote $\Delta_H = \Gamma_H / \sim$, where ``$\sim$'' is the equivalence of $\mathbb{Z}^{n-1}$-close
points. Since $\psi : \Gamma_G \rightarrow \Gamma_H$ is an embedding then $\Delta_G$ embeds into
$\Delta_H$ and with an abuse of notation we are going to denote this embedding by $\psi$. Let $w =
\psi(v)$.

Let $a$ and $b$ be non-$G$-equivalent ends of $\Gamma_G$ and let
$$p_a = v\ v_1\ v_2 \cdots,\ \ p_b = v\ u_1\ u_2 \cdots.$$
Assume that $\psi(a)$ and $\psi(b)$ are $H$-equivalent in $\Gamma_H$, that is, there exists $h \in
H$ such that $h \cdot p_{\psi(a)} = p_{\psi(b)}$. Since $p_{\psi(a)}$ and $p_{\psi(b)}$ have the
same origin $w$ then $h \cdot w = w$, that is, $h \in Stab_H(T'_0)$, where $T'_0$ is a
$\mathbb{Z}^{n-1}$-subtree of $\Gamma_H$ containing $\psi(T_0)$. Moreover, if $e_1 = (w,\psi(v_1)),\
f_1 = (w,\psi(u_1))$ then $h \cdot e_1 = f_1$ and it follows that $h \cdot a_1 = b_1$, where $a_1$
and $b_1$ are ends of $\psi(T_0)$ corresponding to $e_1$ and $f_1$. By the assumption of the theorem
there exists $\phi(g_1) \in Stab_{\phi(G)}(\psi(T_0))$ such that $\phi(g_1) \cdot a_1 = b_1$, so,
$\phi(g_1) \cdot \psi(v_1) = \psi(u_1)$. Since $\phi : G \rightarrow H$ and $\psi: \Gamma_G
\rightarrow \Gamma_H$ are embeddings, it follows that $g_1 \cdot v_1 = u_1$ and the images of
$\pi(u_1) = \pi(v_1)$ in $\Delta_G$.

Continuing in the same way we obtain $\pi(u_i) = \pi(v_i),\ i \geqslant 1$ in $\Delta_G$, so, $a$ and
$b$ are $G$-equivalent which gives a contradiction with the assumption that $\psi(a)$ and $\psi(b)$
are $H$-equivalent in $\Gamma_H$.
\end{proof}

\begin{theorem} \cite{KMS:2011(2)}
\label{co:main}
Let $G$ be a finitely generated subgroup of $CDR(\mathbb{Z}^n, X)$, where $X$ is arbitrary. Then
there exists a finite alphabet $Y$ and an embedding $\phi : G \rightarrow H$, where $H$ is a
finitely generated subgroup of $CDR(\mathbb{Z}^n, Y)$ with a regular length function, such that
\begin{enumerate}
\item[(a)] $|g|_G = |\phi(g)|_H$ for every $g \in G$,
\item[(b)] if $A$ ia a maximal abelian subgroup of $G$ then $\phi(A)$ is a maximal abelian
subgroup of $H$,
\item[(c)] if $A$ and $B$ are maximal abelian subgroups of $G$ which are non-conjugate in $G$
then $\phi(A)$ and $\phi(B)$ are non-conjugate in $H$.
\end{enumerate}
\end{theorem}
\begin{proof} We use the induction on $n$. If $n = 1$ then the result follows from Lemma \ref{le:n=1}.
Finally, the induction step follows from Theorem \ref{th:completion}.
\end{proof}

As a simple corollary of the above theorem we get the following result.

\begin{theorem} \cite{KMS:2011(2)}
\label{th:embed_complete}
Every finitely generated $\Z^n$-free  group $G$ has a length-preserving embedding into a
finitely generated complete $\Z^n$-free group $H$.
\end{theorem}

\subsection{Description of $\mathbb{Z}^n$-free groups}
\label{subs:description_Z^n}

Given two $\Zt$-free groups $G_1,\ G_2$ and maximal abelian subgroups $A \leqslant G_1,\ B \leqslant
G_2$ such that
\begin{enumerate}
\item[(a)] $A$ and $B$ are cyclically reduced with respect to the corresponding embeddings of $G_1$
and $G_2$ into infinite words,
\item[(b)] there exists an isomorphism $\phi : A \rightarrow B$ such that $|\phi(a)| = |a|$ for any $a \in A$.
\end{enumerate}
Then we call the amalgamated free product
$$\langle G_1, G_2 \mid A \stackrel{\phi}{=} B \rangle$$
the {\em length-preserving amalgam} of $G_1$ and $G_2$.

Given a $\Zt$-free group $H$ and non-conjugate maximal abelian subgroups $A, B \leqslant H$ such that
\begin{enumerate}
\item[(a)] $A$ and $B$ are cyclically reduced with respect to the embedding of $H$ into infinite words,
\item[(b)] there exists an isomorphism $\phi : A \rightarrow B$ such that $|\phi(a)| = |a|$ and
$a$ is not conjugate to $\phi(a)^{-1}$ in $H$ for any $a \in A$.
\end{enumerate}
Then we call the HNN extension
$$\langle H, t \mid t^{-1} A t = B \rangle$$
the {\em length-preserving separated HNN extension} of $H$.

As a corollary of Theorem \ref{th:main2} and Theorem \ref{main4} we get the description of
complete $\Z^n$-free groups in the following form.

\begin{theorem}
\label{th:complete_desrc}
A finitely generated group $G$ is complete $\Z^n$-free if and only if it can be obtained from free
groups by finitely many length-preserving separated HNN extensions and centralizer extensions.
\end{theorem}

Using the results of previous Subsection \ref{subs:completions_Z^n} we can prove a theorem similar
to Theorem \ref{th:complete_desrc} for the class of finitely generated $\Zt$-free groups (not
necessarily complete). For that we need the following theorems.

\begin{theorem}
\label{th:amalgam_Z^n}
Let $G_1$ and $G_2$ be finitely generated $\Z^n$-free groups. Then the length-preserving amalgam
$$G = \langle G_1, G_2 \mid A \stackrel{\phi}{=} B \rangle$$
of $G_1$ and $G_2$ is a finitely generated $\Z^{n^\prime}$-free group and the length function on $G$ 
extends the ones on $G_1$ and $G_2$.
\end{theorem}
\begin{proof} Both $G_1$ and $G_2$ a finitely generated $\Zt$-free. Hence, the free product $P =
G_1 \ast G_2$ is also finitely generated $\Zt$-free (see Example \ref{exam:free_prod} and
\cite[Proposition 5.1.1]{Chiswell:2001}) and canonical embeddings of $G_1$ and $G_2$ into $P$
preserve length. Note that $A$ and $B$ are non-conjugate maximal abelian subgroups of $P$.

Next, by Theorem \ref{co:main}, there exists a finitely generated complete $\Zt$-free group $H$ such
that $P$ embeds into $H$ and the embedding preserves length. Moreover, $A$ and $B$ stay
non-conjugate maximal abelian in $H$.

By Theorem \ref{main4}, the HNN extension $K = \langle H, t \mid t^{-1} A t = B \rangle$ is a finitely
generated complete $\Zt$-free group. Observe that $K$ contains a subgroup $K_0 = \langle t^{-1}
\widehat{G_1} t, \widehat{G_2} \rangle$, where $\widehat{G_i}$ denotes the copy of $G_i$ through the
embedding $G_i \hookrightarrow P \hookrightarrow K$. Finally, it is easy to see that $K_0$ is
isomorphic to $G$.
\end{proof}

\begin{theorem}
\label{th:HNN_Z^n}
Let $H$ be a finitely generated $\Z^n$-free group. Then the length-preserving separated HNN extension
$$G = \langle H, t \mid t^{-1} A t = B \rangle$$
of $H$ is a finitely generated $\Z^{n^\prime}$-free group and the length function on $G$ extends the 
one on $H$.
\end{theorem}
\begin{proof} The proof is very similar to the one of Theorem \ref{th:amalgam_Z^n}.

By Theorem \ref{co:main}, there exists a finitely generated complete $\Zt$-free group $P$ and an
embedding $\psi : H \to P$ such that $\psi$ preserves length and $\psi(A), \psi(B)$ are non-conjugate
maximal abelian in $P$. Observe that since $\psi(A)$ and $\psi(B)$ are isomorphic to the original subgroups
$A$ and $B$, and $\phi : A \to B$ is an isomorphism, there exists a natural isomorphism from
$\psi(A)$ to $\psi(B)$ which we also denote by $\phi$. Since $\psi$ preserves length, we have
$|\phi(a)| = |a|$ for any $a \in \psi(A)$. Next, if there exists $g \in P$ such that $g^{-1} a g = \phi(a)$
for some $a \in \psi(A)$ then from the CSA property of $P$ we get $g^{-1} \psi(A) g = \psi(B)$, that is,
$\psi(A)$ and $\psi(B)$ are conjugate - a contradiction. It follows that both conditions (a) and (b)
hold for $\psi(A)$ and $\psi(B)$.

By Theorem \ref{main4}, the HNN extension $K = \langle P, t \mid t^{-1} \psi(A) t = \psi(B) \rangle$
is a finitely generated complete $\Zt$-free group. Observe that $K$ contains a subgroup $K_0 =
\langle \psi(H), t \rangle$ which is isomorphic to $G$.
\end{proof}

\begin{theorem}
\label{th:centr_ext_Z^n}
Let $H$ be a finitely generated $\Z^n$-free group  and let $A$ be a maximal abelian subgroup of $H$.
Then the centralizer extension
$$G = \langle H, t \mid t^{-1} A t = A \rangle$$
is a finitely generated $\Z^{n^\prime}$-free group and the length function on $G$ extends the one on $H$.
\end{theorem}
\begin{proof} By Theorem \ref{co:main}, there exists a finitely generated complete $\Zt$-free group
$P$ and an embedding $\psi : H \to P$ such that $\psi$ preserves length and $\psi(A)$ is maximal
abelian in $P$.

By Theorem \ref{main4}, the HNN extension $K = \langle P, t \mid t^{-1} \psi(A) t = \psi(A) \rangle$
is a finitely generated complete $\Zt$-free group. Observe that $K$ contains a subgroup $K_0 =
\langle \psi(H), t \rangle$ which is isomorphic to $G$.
\end{proof}

\begin{theorem}
\label{th:Z^n_desrc}
A finitely generated group $G$ is $\Z^n$-free if and only if it can be obtained from free groups by 
a finite sequence of length-preserving amalgams, length-preserving separated HNN extensions, and 
centralizer extensions.
\end{theorem}
\begin{proof} Since every finitely generated $\Zt$-free group embeds into a finitely generated
complete $\Zt$-free group, from Bass-Serre Theory we get the required.
\end{proof}

\section{Elimination process over finitely presented $\Lambda$-free groups}
\label{sec:proc}

In this section  we will describe the Elimination process over finitely presented $\Lambda$-free
groups which we will use in Section \ref{sec:struct_th_lambda} to prove Theorems \ref{th:main1},
\ref{th:main3}, \ref{th:main2}. From now on we assume that $G = \langle X \mid R \rangle$ is a
finitely presented group which acts freely and regularly on a $\Lambda$-tree, where $\Lambda$ is a
discretely ordered abelian group, or, equivalently, $G$ can be represented by $\Lambda$-words over
some alphabet $Z$ and the length function on $G$ induced from $CDR(\Lambda, Z)$ is regular. Let us
fix the embedding $\xi : G \hookrightarrow CDR(\Lambda ,Z)$ for the rest of this section. For all the
details please refer to \cite{KMS:2011(3)}.

\subsection{The notion of a generalized equation}
\label{subs:notion_gen_eq}

\begin{definition}
A {\em combinatorial generalized equation} $\Omega$ (which is convenient to visualize as shown on
the picture below) consists of the following objects.

\begin{figure}[htbp]
\label{pic:example}
\centering{\mbox{\psfig{figure=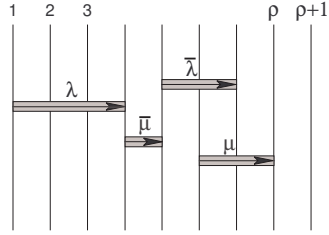}}}
\caption{A typical generalized equation.}
\end{figure}

\begin{enumerate}
\item A finite set of {\em bases} ${\mathcal M} = BS(\Omega)$. The set of bases ${\mathcal M}$
consists of $2n$ elements ${\mathcal M} = \{\mu_1, \ldots, \mu_{2n}\}$. The set ${\mathcal M}$ comes
equipped with two functions: a function $\varepsilon: {\mathcal M} \rightarrow \{1,-1\}$ and an
involution $\Delta : {\mathcal M} \rightarrow {\mathcal M}$ (that is, $\Delta$ is a bijection such
that $\Delta^2$ is an identity on ${\mathcal M}$). Bases $\mu$ and $\overline\mu = \Delta(\mu)$ are
called {\it dual bases}. We denote bases by letters $\mu, \lambda$, etc.

\item A set of {\em boundaries} $BD = BD(\Omega) = \{1, 2, \ldots, \rho + 1\}$, that is, integer
points of the interval $I = [1, \rho + 1]$. We use letters $i,j$, etc. for boundaries.

\item Two functions $\alpha : BS \rightarrow BD$ and $\beta : BS \rightarrow BD$. We call
$\alpha(\mu)$ and $\beta(\mu)$ the {\em initial and terminal boundaries} of the base $\mu$ (or
endpoints of $\mu$). These functions satisfy the following conditions: for every base $\mu \in BS$:
$\alpha(\mu) < \beta(\mu)$ if $\varepsilon(\mu) = 1$ and $\alpha(\mu) > \beta(\mu)$ if
$\varepsilon(\mu) = -1$.

\item A set of {\em boundary connections} $(p,\lambda , q),$ where $p$ is a boundary on $\lambda$
(that is a number between $\alpha(\lambda)$ and $\beta(\lambda)$) and $q$ on $\bar\lambda$. In this
case we say that $p$ and $q$ are $\lambda$-tied. If $(p,\lambda ,q)$ is a boundary connection then
$(q,\overline\lambda , p)$ is also a boundary connection. (The meaning of boundary connections will
be explained in the transformation (ET5)).
\end{enumerate}
\end{definition}

With a combinatorial generalized equation $\Omega$ one can canonically associate a system of
equations in {\em variables} $h = (h_1, \ldots, h_\rho)$ (variables $h_i$ are also called {\it
items}). This system is called a {\em generalized equation}, and (slightly abusing the terminology)
we denote it by the same symbol $\Omega$, or $\Omega(h)$ specifying the variables it depends on. The
generalized equation $\Omega$ consists of the following two types of equations.

\begin{enumerate}
\item Each pair of dual  bases $(\lambda, \overline\lambda)$ provides an equation
$$[h_{\alpha(\lambda)} h_{\alpha(\lambda) + 1} \ldots h_{\beta(\lambda ) - 1}]^{\varepsilon
(\lambda)} = [h_{\alpha(\overline\lambda)} h_{\alpha(\overline\lambda) + 1} \ldots h_{\beta(\overline
\lambda ) - 1}]^{\varepsilon(\overline\lambda)}.$$
These equations are called {\em basic equations}.

\item Every boundary connection $(p,\lambda,q)$ gives rise to a {\em boundary equation}
$$[h_{\alpha(\lambda)} h_{\alpha(\lambda) + 1} \cdots h_{p-1}] = [h_{\alpha(\overline\lambda)}
h_{\alpha(\overline\lambda ) + 1} \cdots h_{q-1}],$$
if $\varepsilon(\lambda) = \varepsilon(\overline\lambda)$ and
$$[h_{\alpha (\lambda )} h_{\alpha(\lambda ) + 1} \cdots h_{p-1}] = [h_{q} h_{q+1} \cdots h_{\alpha (
\overline\lambda)-1}]^{-1} ,$$
if $\varepsilon(\lambda)= -\varepsilon(\overline\lambda).$
\end{enumerate}

\begin{remark}
We  assume that every generalized equation comes from a combinatorial one.
\end{remark}

Given a generalized equation $\Omega(h)$ one can define the {\em group of $\Omega(h)$}
$$G_\Omega = \langle h \mid \Omega(h) \rangle.$$

\begin{definition}
Let $\Omega(h) = \{L_1(h) = R_1(h), \ldots, L_s(h) = R_s(h)\}$ be a generalized equation in
variables $h = (h_1, \ldots, h_\rho)$. A set $U = (u_1, \ldots, u_\rho) \subseteq R(\Lambda,Z)$
of nonempty $\Lambda$-words is called a {\em solution} of $\Omega$ if:
\begin{enumerate}
\item all words $L_i(U), R_i(U)$ are reduced,
\item $L_i(U) =  R_i(U),\ i \in [1,s]$.
\end{enumerate}
\end{definition}

Observe that a solution $U$ of $\Omega(h)$ defines a homomorphism $\xi_U : G_\Omega \to R(\Lambda,Z)$
induced by the mapping $h_i \to u_i,\ i \in [1,\rho]$ since after this substitution all the equations
of $\Omega(h)$ turn into identities in $R(\Lambda,Z)$.

If we specify a particular solution $U$ of a generalized equation $\Omega$ then we use a pair
$(\Omega, U)$.

\begin{definition}
A cancelation table $C(U)$ of a solution $U = (u_1, \ldots, u_\rho)$ is defined as follows
$$C(U) = \{ h_i^\epsilon h_j^\sigma \mid {\rm there\ is\ cancelation\ in\ the\ product}\ u_i^\epsilon
\ast u_j^\sigma,\ {\rm where}\  \epsilon, \sigma = \pm 1 \}.$$
\end{definition}

\begin{definition}
A solution $U^+$ of a generalized equation $\Omega$ is called {\em consistent} with a solution $U$
if $C(U^+) \subseteq C(U)$.
\end{definition}

\subsection{From a finitely presented group to a generalized equation}
\label{subs:constr_ge}

Recall that $G = \langle X \mid R \rangle$ is finitely presented and let $X = \{x_1, \ldots, x_n\}$
and $R = \{r_1(X),\ldots, r_m(X)\}$. Adding, if necessary, auxiliary generators, we can assume that
every relator involves at most three generators.

Since $\xi$ is a homomorphism it follows that after the substitution $x_i \rightarrow \xi(x_i),\ i
\in [1,n]$ all products $r_i(\xi(X)),\ i \in [1,m]$ cancel out. Hence, we have finitely many {\em
cancelation diagrams} over $CDR(\Lambda, Z)$, which give rise to a generalized equation $\Omega$
corresponding to the embedding $\xi : G \hookrightarrow CDR(\Lambda, Z)$.

The precise definition and all the details concerning cancelation diagrams over $CDR(\Lambda, Z)$
can be found in \cite{Khan_Myasnikov_Serbin:2007}. Briefly, a cancelation diagram for $r_i(\xi(X))$
can be viewed as a finite directed tree $T_i$ in which every positive edge $e$ has a label $\lambda_e$
so that every occurrence $x^\delta,\ \delta \in \{-1,1\}$ of $x \in X$ in $r_i$ corresponds to a
reduced path $e_1^{\epsilon_1} \cdots e_k^{\epsilon_k}$, where $\epsilon_i \in \{-1,1\}$, in $T_i$ and
$\xi(x^\delta) = \lambda_{e_1}^{\epsilon_1} \circ \ldots \circ \lambda_{e_k}^{\epsilon_k}$. In other
words, each $\lambda_e$ is a piece of some generator of $G$ viewed as a $\Lambda$-word. Moreover,
we assume that $|\lambda_e|$ is known (since we know the homomorphism $\xi$).

\begin{figure}[htbp]
\label{triangle}
\centering{\mbox{\psfig{figure=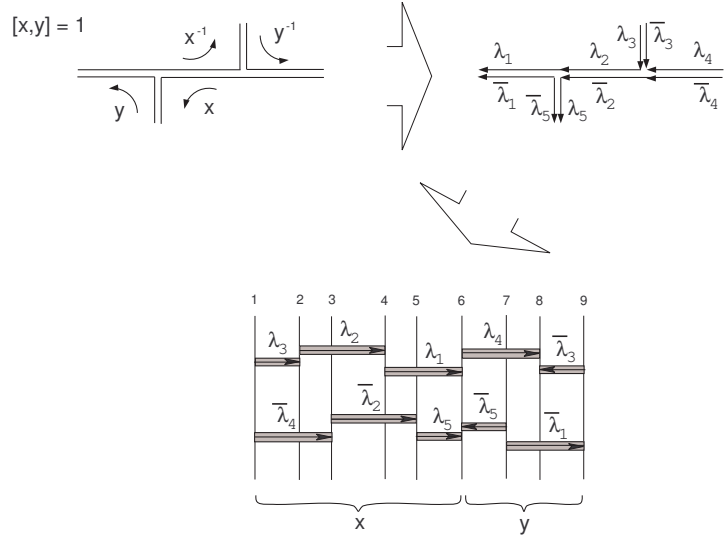,height=3in}}}
\caption{From the cancelation diagram for the relation $[x,y] = 1$ to the generalized equation.}
\end{figure}

Now we would like to construct a generalized equation $\Omega_i$ corresponding to $T_i$.
Denote by $X(T_i)$ all generators of $G$ which appear in $r_i$. Next, consider a segment $J$ in
$\Lambda$ of length
$$\sum_{x \in X(T_i)} |\xi(x)|$$
which is naturally divided by the lengths of $\xi(x),\ x \in X(T_i)$ into subsegments with respect to
any given order on $X(T_i)$. Since every $\xi(x),\ x \in X(T_i)$ splits into at least one reduced
product $\lambda_{e_1}^{\epsilon_1} \circ \ldots \circ \lambda_{e_k}^{\epsilon_k}$, every such
splitting gives a subdivision of the corresponding subsegment of $J$. Hence, we subdivide $J$
using all product representations of all $\xi(x),\ x \in X(T_i)$. As a result we obtain a
subdivision of $J$ into $\rho_i$ items whose endpoints become boundaries of $\Omega_i$.
Observe that each $\lambda_e$ appears exactly twice in the products representing some $\xi(x),\ x
\in X(T_i)$ and each such entry covers several adjacent items of $J$. This pair of entries defines
a pair of dual bases $(\lambda_e, \overline{\lambda_e})$.
Hence,
$${\mathcal M}_i = BS(\Omega_i) = \{ \lambda_e,\ \overline{\lambda_e} \mid e \in E(T_i)\}.$$
$\epsilon(\lambda_e)$ depends on the sign of $\lambda_e$ in the corresponding product representing
a variable from $X(T_i)$ (similarly for $\overline{\lambda_e}$).

In the same way one can construct $T_i$ and the corresponding $\Omega_i$ for each $r_i,\ i \in [1,m]$.
Combining all combinatorial generalized equations $\Omega_i,\ i \in [1,m]$ we obtain the equation
$\Omega$ with items $h_1,\ldots,h_\rho$ and bases ${\mathcal M} = \cup_i {\mathcal M}_i$. By definition
$$G_\Omega = \langle h_1, \ldots, h_\rho \mid \Omega(h_1,\ldots,h_\rho)\rangle.$$
At the same time, since each item can be obtained in the form
$$(\lambda_{i_1}^{\epsilon_1} \circ \ldots \circ \lambda_{i_k}^{\epsilon_k}) \ast
(\lambda_{j_1}^{\delta_1} \circ \ldots \circ \lambda_{j_l}^{\delta_l})^{-1},$$
it follows that $G_\Omega$ can be generated by ${\mathcal M}$ with the relators obtained by
rewriting $\Omega(h_1,\ldots,h_\rho)$ in terms of ${\mathcal M}$.

It is possible to transform the presentation $\langle h_1, \ldots, h_\rho \mid \Omega\rangle$ into
$\langle X \mid R \rangle$ using Tietze transformations as follows. From the cancelation diagrams
constructed for each relator in $R$ it follows that $x_i = w_i(h_1,\ldots,h_\rho) = w_i(\overline{h}
),\ i \in [1,n]$. Hence
$$\langle h_1, \ldots, h_\rho \mid \Omega\rangle \simeq \langle h_1, \ldots, h_\rho, X \mid \Omega
\cup \{x_i = w_i(\overline{h}),\ i \in [1,n]\} \rangle.$$
Next, from the cancelation diagrams it follows that $R$ is a set of consequences of $\Omega \cup
\{x_i = w_i(\overline{h}),\ i \in [1,n]\}$, hence,
$$\langle h_1, \ldots, h_\rho \mid \Omega\rangle \simeq \langle h_1, \ldots, h_\rho, X \mid \Omega
\cup \{x_i = w_i(\overline{h}),\ i \in [1,n]\} \cup R \rangle.$$
Finally, since the length function on $G$ is regular, for each $h_i$ there exists a word $u_i(X)$
such that $h_i = u_i(\xi(X))$ and all the equations in $\Omega \cup \{x_i = w_i(\overline{h}),\ i
\in [1,n]\}$ follow from $R$ after we substitute $h_i$ by $u_i(X)$ for each $i$. It follows that
$$\langle h_1, \ldots, h_\rho, X \mid \Omega \cup \{x_i = w_i(\overline{h}),\ i \in [1,n]\} \cup R
\rangle$$
$$\simeq \langle h_1, \ldots, h_\rho, X \mid \Omega \cup \{x_i = w_i(\overline{h}),\ i \in [1,n]\}
\cup R \cup \{h_j = u_j(X),\ j \in [1,\rho]\} \rangle$$
$$ \simeq \langle X \mid R \rangle.$$
It follows that $G \simeq G_\Omega$.

Let $\widetilde G$ be a finitely presented group with a free length function in $\Lambda$ (not
necessary regular). It can be embedded isometrically in the group $\widehat G$ with a free regular
length function in $\Lambda$ by \cite{Chiswell_Muller:2010}. That group can be embedded in 
$R(\Lambda', X)$. When we
make a generalized equation $\Omega$ for $\widetilde G$, we have to add only finite number of
elements from $\widehat G$. Let $G$ be a subgroup generated in $\widehat G$ by $\widetilde G$ and
these elements. Then $G$ is the quotient of $G_{\Omega}$ containing $\widetilde G$ as a subgroup.

\subsection{Elementary transformations}
\label{se:5.1}

In this subsection we describe {\em elementary transformations} of generalized equations. Let
$(\Omega, U)$ be a generalized equation together with a solution $U$. An elementary transformation
(ET) associates to a generalized equation $(\Omega, U)$ a generalized equation $(\Omega_1, U_1)$ and
an epimorphism $\pi : G_\Omega \rightarrow G_{\Omega_1}$ such that for the solution $U_1$ the
following diagram commutes

\[\begin{diagram}
\label{diag:1} \node{G_\Omega} \arrow{e,t}{\pi}
\arrow{s,l}{\xi_U} \node{G_{\Omega_1}}
\arrow{sw,r}{\xi_{U_1}}\\
\node{R(\Lambda,Z)}
\end{diagram}
\]

One can view (ET) as a mapping $ET : (\Omega,U) \rightarrow (\Omega_1, U_1)$.

\begin{enumerate}
\item[(ET1)] {\em (Cutting a base (see Fig. \ref{ET1}))}.
Let $\lambda$ be a base in $\Omega$ and $p$ an internal boundary of $\lambda$ (that is, $p \neq
\alpha(\lambda), \beta(\lambda)$) with a boundary connection $(p, \lambda, q)$. Then we cut the
base $\lambda$ at $p$ into two new bases $\lambda_1$ and $\lambda_2$, and cut $\overline\lambda$ at
$q$ into the bases $\overline\lambda_1$ and $\overline\lambda_2$.

\begin{figure}[htbp]
\centering{\mbox{\psfig{figure=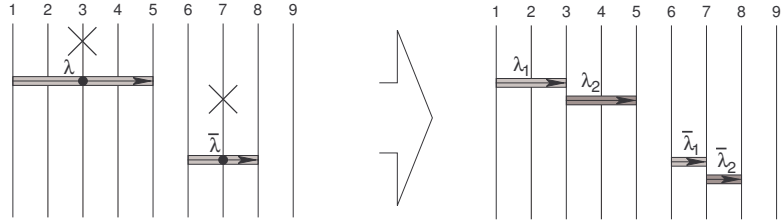,height=1.4in}}}
\caption{Elementary transformation (ET1).}
\label{ET1}
\end{figure}

\item[(ET2)] {\em (Transfering a base (see Fig. \ref{ET2}))}.
If a base $\lambda$ of $\Omega$ contains a base $\mu$ (that is, $\alpha(\lambda) \leqslant \alpha(\mu) <
\beta(\mu) \leqslant \beta(\lambda)$) and all boundaries on $\mu$ are $\lambda$-tied by boundary some
connections then we transfer $\mu$ from its location on the base $\lambda$ to the corresponding
location on the base $\overline\lambda$.

\begin{figure}[htbp]
\centering{\mbox{\psfig{figure=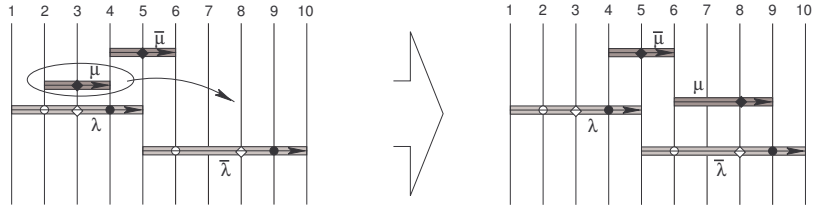,height=1.3in}}}
\caption{Elementary transformation (ET2).}
\label{ET2}
\end{figure}

\item[(ET3)] {\em (Removal of a pair of  matched bases (see Fig. \ref{ET3}))}.
If the bases $\lambda$ and $\overline\lambda$ are {\em  matched} (that is, $\alpha(\lambda) =
\alpha(\overline\lambda), \beta(\lambda) = \beta(\overline\lambda)$) then we remove $\lambda,
\overline\lambda$ from $\Omega$.

\begin{figure}[htbp]
\centering{\mbox{\psfig{figure=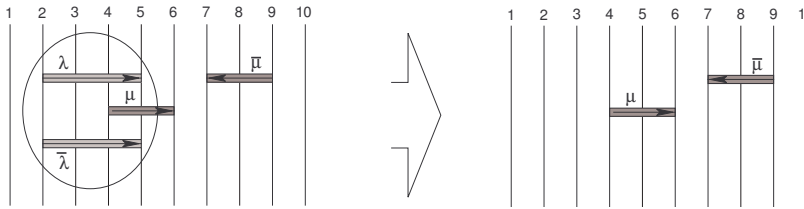,height=1.3in}}}
\caption{Elementary transformation (ET3).}
\label{ET3}
\end{figure}

\begin{remark}
Observe, that $\Omega$ and $\Omega_1$, where $\Omega_1 = ETi(\Omega)$ for $i \in \{1,2,3\}$ have the
same set of variables $h$ and the bijection $h_i \rightarrow h_i,\ i \in [1,\rho]$ induces an
isomorphism $G_\Omega \rightarrow G_{\Omega_1}$. Moreover, $U$ is a solution of $\Omega$ if and only
if $U$ is a solution of $\Omega_1$.
\end{remark}

\item[(ET4)] {\em (Removal of a lone base (see Fig. \ref{ET4}))}.
Suppose, a base $\lambda$ in $\Omega$ does not {\em intersect} any other base, that is, the items
$h_{\alpha(\lambda)}, \ldots, h_{\beta(\lambda) - 1}$ are contained only inside of the base
$\lambda$.

\begin{figure}[htbp]
\centering{\mbox{\psfig{figure=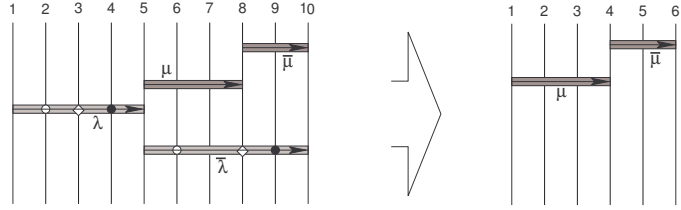,height=1.3in}}}
\caption{Elementary transformation (ET4).}
\label{ET4}
\end{figure}

Suppose also that all boundaries in $\lambda$ are $\lambda$-tied, that is, for every $i$
($\alpha(\lambda) < i \leqslant \beta(\lambda) - 1$) there exists a boundary $b(i)$ such that $(i,
\lambda, b(i))$ is a boundary connection in $\Omega$. Then we remove the pair of bases $\lambda$ and
$\overline\lambda$ together with all the boundaries $\alpha(\lambda) + 1, \ldots, \beta(\lambda) - 1$
(and rename the rest $\beta(\lambda) - \alpha(\lambda) - 1$ of the boundaries correspondingly).

We define the isomorphism $\pi: G_{\Omega} \rightarrow G_{\Omega_1}$ as follows:
$$\pi(h_j) = h_j\ {\rm if}\ j < \alpha(\lambda)\ {\rm or}\ j \geqslant \beta(\lambda)$$
\[\pi(h_i) = \left\{\begin{array}{ll}
h_{b(i)} \cdots h_{b(i) - 1}, & if\ \varepsilon(\lambda) = \varepsilon(\overline\lambda),\\
h_{b(i)} \cdots h_{b(i-1) - 1},& if\ \varepsilon(\lambda) = -\varepsilon(\overline\lambda)
\end{array}
\right.
\]
for $\alpha + 1 \leqslant i \leqslant \beta(\lambda) - 1$.

\begin{figure}[htbp]
\centering{\mbox{\psfig{figure=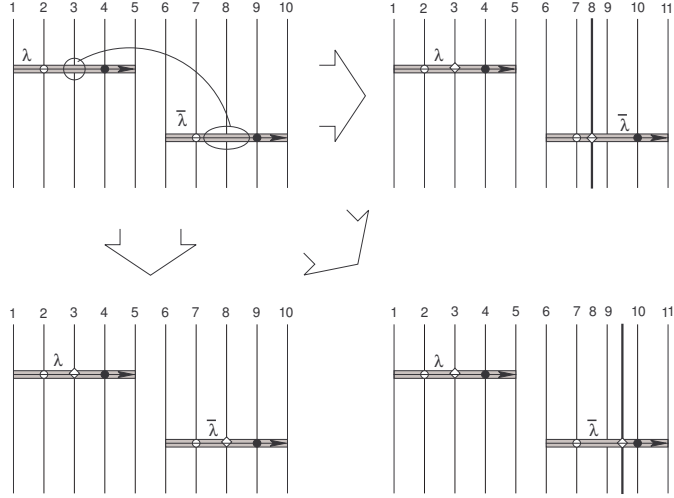,height=2.8in}}}
\caption{Elementary transformation (ET5).}
\label{ET5}
\end{figure}

\item[(ET5)] {\em (Introduction of a boundary (see Fig. \ref{ET5}))}. Suppose a point $p$ in a base
$\lambda$ is not $\lambda$-tied. The transformation (ET5) $\lambda$-ties it. To this end, denote by
$u_\lambda$ the element of $CDR(\Lambda ,Z)$ corresponding to $\lambda$ and let $u'_\lambda$ be the
beginning of this word ending at $p$. Then we perform one of the following two transformations
according to where the end of $u'_\lambda$ on $\overline\lambda$ is situated:

\begin{enumerate}
\item If the end of $u'_\lambda$ on $\overline\lambda$ is situated on the boundary $q$ then we
introduce the boundary connection $(p, \lambda ,q)$. In this case the corresponding isomorphism
$\pi : G_\Omega \rightarrow G_{\Omega_1}$ is induced by the bijection $h_i \rightarrow h_i,\ i \in
[1,\rho]$. (If we began with the group $\tilde G$ with non-regular length function,
 this is the only place where  $\pi : G_\Omega \rightarrow G_{\Omega_1}$ may be a proper
epimorphism, but its restriction on $\tilde G$ is still an isomorphism.)

\item If the end of $u'_\lambda$ on $\overline\lambda$ is situated between $q$ and $q+1$ then we
introduce a new boundary $q'$ between $q$ and $q+1$ (and rename all the boundaries), and also
introduce a new boundary connection $(p, \lambda, q')$. In this case the corresponding isomorphism
$\pi: G_\Omega \rightarrow G_{\Omega_1}$ is induced by the map $\pi(h) = h$, if $h \neq h_q$, and
$\pi(h_q) = h_{q'} h_{q'+1}$.
\end{enumerate}
\end{enumerate}

\subsection{Derived transformations and auxiliary transformations}
\label{se:5.2half}

In this section we define complexity of a generalized equation and describe several useful
``derived'' transformations of generalized equations. Some of them can be realized as finite
sequences of elementary transformations, others result in equivalent generalized equations but
cannot be realized by finite sequences of elementary moves.

A boundary is {\em open} if it is an internal boundary of some base, otherwise it is {\em closed}.
A section $\sigma = [i, \ldots, i + k]$ is said to be {\em closed} if the boundaries $i$ and $i+k$
are closed and all the boundaries between them are open.

Sometimes it will be convenient to subdivide all sections of $\Omega$ into {\em active} (denoted
$A\Sigma_\Omega$) and {\em non-active} sections. For an item $h$ denote by $\gamma(h)$ the number of
bases containing $h$. An item $h$ is called {\em free} is it meets no base, that is, if $\gamma(h) =
0$. Free variables are transported to the very end of the interval behind all items in $\Omega$ and
they become non-active.

\begin{enumerate}
\item[(D1)] {\em (Closing a section)}.
Let $\sigma$ be a section of $\Omega$. The transformation (D1) makes the section $\sigma$ closed.
Namely, (D1) cuts all bases in $\Omega$ through the end-points of $\sigma$.

\item[(D2)] {\em (Transporting a closed section)}.
Let $\sigma$ be a closed section of a generalized equation $\Omega$. We cut $\sigma$ out of the
interval $[1,\rho_\Omega]$ together with all the bases on $\sigma$ and put $\sigma$ at the end of the
interval or between any two consecutive closed sections of $\Omega$. After that we correspondingly
re-enumerate all the items and boundaries of the latter equation to bring it to the proper form.
Clearly, the original equation $\Omega$ and the new one $\Omega'$ have the same solution sets and
their coordinate groups are isomorphic

\item[(D3)] {\em (Moving free variables to the right)}.
Suppose that $\Omega$ contains a free variable $h_q$ in an active section. Here we close the section
$[q, q + 1]$ using (D1), transport it to the very end of the interval behind all items in $\Omega$
using (D2). In the resulting generalized equation $\Omega'$ the transported section becomes a
non-active section.

\item[(D4)] {\em (Deleting a complete base)}.
A base $\mu$ of $\Omega$ is called {\em complete} if there exists a closed section $\sigma$ in
$\Omega$ such that $\sigma = [\alpha(\mu), \beta(\mu)]$.

Suppose $\mu$ is an active  complete base of $\Omega$ and $\sigma$ is a closed section such that
$\sigma = [\alpha(\mu),\beta(\mu)]$. In this case using (ET5), we transfer all bases from $\mu$ to
$\overline\mu$, then using (ET4) we remove the lone base $\mu$ together with the section
$\sigma$.

\item[(D5)]  {\em (Linear elimination)}.

We first explain how to find {\em the kernel of the generalized equation}.
We will give a definition of eliminable base for an equation $\Omega$ that does not have any boundary
connections. An active base $\mu \in A\Sigma_\Omega$ is called {\em eliminable} if at least one of
the following holds

\begin{enumerate}
\item[(a)] $\mu$ contains an item $h_i$ with $\gamma(h_i) = 1$,

\item[(b)] at least one of the boundaries $\alpha(\mu), \beta(\mu)$ is different from $1, \rho +1$
and does not touch any other base (except for $\mu$).
\end{enumerate}

The process of finding the kernel works as follows. We cut the bases of $\Omega$ along all the
boundary connections thus obtaining the equation without boundary connections, then
consequently remove eliminable bases until no eliminable base is left in the equation. The
resulting generalized equation is called the {\em kernel} of $\Omega$ and we denote
it by $Ker(\Omega)$. One can show that it does not depend on a particular removal
process.  We say that an item $h_i$ {\it belongs to
the kernel} ($h_i \in Ker(\Omega)$), if $h_i$ belongs to at least one base in the kernel.
Notice that the kernel can be empty.

\begin{lemma}\cite{KMS:2011(3)}
\label{7-10}
If $\Omega $ is a generalized equation, then
$$G_\Omega \simeq G_{{Ker(\Omega)}} \ast F(K)$$
where $F(K)$ is a free group on $K$. The set $K$ can be empty.
\end{lemma}

Suppose that in $\Omega$ there is $h_i$ in an active section with $\gamma(h_i) = 1$ and such that
$|h_i|$ is comparable with the length of the active section. In this case we say that $\Omega$ is
{\em linear in $h_i$}.

If $\Omega$ is linear in $h_i$ in an active section such that both boundaries $i$ and $i+1$ are
closed then we remove the closed section $[i,i+1]$ together with the lone base using (ET4).

If there is no such $h_i$ but $\Omega$ is linear in some $h_i$ in an active section such that one
of the boundaries $i,i+1$ is open, say $i+1$, and the other is closed, then we perform (ET5) and
$\mu$-tie $i+1$ through the only base $\mu$ it intersects. Next, using (ET1) we cut $\mu$ in $i+1$
and then we delete the closed section $[i,i+1]$ by (ET4).

Suppose there is no $h_i$ as above but $\Omega$ is linear in some $h_i$ in an active section such
that both boundaries $i$ and $i+1$ are open. In addition, assume that there is a closed section
$\sigma$ containing exactly two (not matched) bases $\mu_1$ and $\mu_2$, such that $\sigma =
\sigma(\mu_1) = \sigma(\mu_2)$ and in the generalized equation $\widetilde{\Omega}$ (see the derived
transformation (D3)) all the bases obtained from $\mu_1,\mu_2$ by (ET1) in constructing
$\widetilde{\Omega}$ from $\Omega$, do not belong to the kernel of $\widetilde{\Omega}$. Here, using
(ET5), we $\mu_1$-tie all the boundaries inside of $\mu_1$, then using (ET2) we transfer $\mu_2$ onto
$\overline\mu_1$, and remove $\mu_1$ together with the closed section $\sigma$ using (ET4).

Suppose now that $\Omega$  satisfies the first assumption of the previous paragraph and does not
satisfy the second one. In this event we close the section $[i, i+1]$ using (D1) and remove it
using (ET4).

\begin{lemma} \cite{KMS:2011(3)}
\label{lin-el}
Suppose that the process of linear elimination continues infinitely and there is a corresponding
sequence of generalized equations
$$\Omega \rightarrow \Omega_1 \rightarrow \cdots \rightarrow \Omega_k \rightarrow \cdots.$$
Then
\begin{enumerate}
\item (\cite[Lemma 15]{Kharlampovich_Myasnikov:1998(2)}) The number of different generalized equations that appear in the
process is finite. Therefore some generalized equation appears in this process infinitely many
times.

\item (\cite[Lemma 15]{Kharlampovich_Myasnikov:1998(2)}) If $\Omega_j = \Omega_k,\ j < k$ then $\pi(j,k)$ is an isomorphism,
invariant with respect to the kernel, namely $\pi(j,k)(h_i) = h_i$ for any variable $h_i$ that
belongs to some base in $Ker(\Omega)$.

\item (\cite[Lemma 7]{KMS:2011(3)})The interval for the equation $\Omega_j$ can be divided into two disjoint parts, each being
the union of closed sections, such that one part is a generalized equation $Ker(\Omega)$ and the
other part is non-empty and corresponds to a generalized equation $\Omega'$, such that $G_{\Omega'}
= F(K)$ is a free group on variables $K$ and $G_\Omega = G_{Ker(\Omega)} \ast F(K)$.
\end{enumerate}
\end{lemma}

\item[(D6)] {\em (Tietze cleaning)}. Suppose that in $\Omega$ is linear in some $h_i$ in an active section  such that
$|h_i|$ is comparable with the length of the active section.
This transformation consists of four transformations performed consecutively
\begin{enumerate}
\item[(a)] linear elimination: if the process of linear elimination goes infinitely we replace the
equation by its kernel,
\item[(b)] deleting all pairs of matched bases,
\item[(c)] deleting all complete bases,
\item[(d)] moving all free variables to the right.
\end{enumerate}

\item[(D7)] {\em (Entire transformation)}.
We need a few definitions. A base $\mu$ of the equation $\Omega$ is called a {\em leading} base if
$\alpha(\mu) = 1$. A leading base is said to be {\it maximal} (or a {\em carrier base}) if
$\beta(\lambda) \leqslant \beta(\mu)$ for any other leading base $\lambda$. Let $\mu$ be a carrier base
of $\Omega$. Any active base $\lambda \neq \mu$ with $\beta(\lambda) \leqslant \beta(\mu)$ is called a
{\em transfer} base (with respect to $\mu$).

Suppose now that $\Omega$ is a generalized equation with $\gamma(h_i) \geqslant 2$ for each $h_i$ in
the active part of $\Omega$ and such that $|h_i|$ is comparable with the length of the active part.
{\em Entire transformation} is a sequence of elementary transformations which are performed as
follows. We fix a carrier base $\mu$ of $\Omega$. We transfer all transfer bases from $\mu$ onto
$\overline \mu$. Now, there exists some $i < \beta(\mu)$ such that $h_1, \ldots, h_i$ belong to only
one base $\mu$, while $h_{i+1}$ belongs to at least two bases. Applying (ET1) we cut $\mu$ along the
boundary $i+1$. Finally, applying (ET4) we delete the section $[1,i+1]$.
\end{enumerate}

\subsection{Complexity of a generalized equation and Delzant-Potyagailo complexity $c(G)$ of a group $G$}
\label{subs:complexity}

Denote by $\rho_A$ the number of variables $h_i$ in all active sections of $\Omega$, by $n_A =
n_A(\Omega)$ the number of bases in all active sections of $\Omega$, by $\nu'$ the number of open
boundaries in the active sections, and by $\sigma'$ the number of closed boundaries in the active
sections.

For a closed section $\sigma \in \Sigma_\Omega$ denote by $n(\sigma), \rho(\sigma)$ the number of
bases and, respectively, variables in $\sigma$.
$$\rho _A = \rho_A(\Omega) = \sum_{\sigma \in A\Sigma_\Omega} \rho(\sigma),$$
$$n_A = n_A(\Omega) = \sum_{\sigma \in A\Sigma_\Omega} n(\sigma).$$

The {\em complexity} of the  generalized  equation $\Omega $ is the number
$$\tau = \tau (\Omega) = \sum_{\sigma \in A\Sigma_\Omega} \max\{0, n(\sigma) - 2\}.$$

Notice that the entire transformation (D7) as well as the cleaning process (D4) do not increase
complexity of equations.

Below we recall Delzant-Potyagailo's result (see \cite{Delzant_Potyagailo:2001}). A family ${\cal C}$ of subgroups of a
torsion-free group $G$ is called {\em elementary} if
\begin{enumerate}
\item[(a)] ${\cal C}$ is closed under taking subgroups and conjugation,
\item[(b)] every $C \in {\cal C}$ is contained in a maximal subgroup $\overline{C} \in {\cal C}$,
\item[(c)] every $C \in {\cal C}$ is small (does not contain $F_2$ as a subgroup),
\item[(d)] all maximal subgroups from ${\cal C}$ are malnormal.
\end{enumerate}
$G$ admits a {\em hierarchy} over ${\cal C}$ if the process of decomposing $G$ into an amalgamated
product or an HNN-extension over a subgroup from ${\cal C}$, then decomposing factors of $G$ into
amalgamated products and/or HNN-extensions over a subgroup from ${\cal C}$ etc. eventually stops.

\begin{prop} (\cite{Delzant_Potyagailo:2001})
If $G$ is a finitely presented group without $2$-torsion and ${\cal C}$ is a family of elementary
subgroups of $G$ then $G$ admits a hierarchy over ${\cal C}$.
\end{prop}

\begin{cor} \cite{KMS:2011(3)}
If $G$ is a finitely presented $\Lambda$-free group then $G$ admits a hierarchy over the family of
all abelian subgroups.
\end{cor}

There is a notion of complexity of a group $G$ defined in \cite{Delzant_Potyagailo:2001} and denoted by $c(G)$. We will
only use the following statement that follows from there.

\begin{prop} (\cite{Delzant_Potyagailo:2001})
\label{DP} If $G$ is a non-trivial free product of finitely presented groups $G_1$ and $G_2$. Then $c(G_i)< c(G), i=1,2.$
Let $G$ be a finitely presented freely indecomposable $\Lambda$-free group (therefore CSA). Let $\Gamma$ be an abelian decomposition of $G$
as a fundamental group of a graph of groups with at least two vertices with non-cyclic vertex groups, maximal abelian subgroups being elliptic, and with each edge group being
maximal abelian at least in one of its vertex groups. Then for each vertex group $G_v$, $c(G_v) <
c(G)$.
\end{prop}

\subsection{Rewriting process for $\Omega$}
\label{se:5.2}

In this section we describe a rewriting process (elimination process) for a generalized equation
$\Omega$ and its solution corresponding to $G$ . Performing the elimination process we eventually detect a decomposition
of $G$ as a free product or (if it is freely indecomposable) as the fundamental group of a graph of
groups with vertex groups of three types: QH vertex groups,  abelian vertex groups (corresponding to
periodic structures, see below), non-QH, non-abelian vertex groups (we will call them {\em weakly
rigid} meaning that we do not split them in this particular decomposition). We also can detect
splitting of $G$ as an HNN-extension with stable letter infinitely longer than generators of the
abelian associated subgroups. After obtaining such a decomposition we continue the elimination
process with the generalized equation corresponding to free factors of $G$  or to weakly rigid
subgroups of $G$ (we will show that this generalized equation can be naturally obtained from the
generalized equation $\Omega$.) The Delzant-Potyagailo complexity of factors in a free decomposition
and complexity of weakly rigid subgroups is smaller than the complexity of $G$. In the case of an HNN extension we will show that the complexity $\tau$ of the generalized
equation corresponding to a weakly rigid subgroup is smaller that the complexity of $\Omega$.

We assume that $\Omega$ is in standard form, namely, that transformations (ET3), (D3) and (D4) have
been applied to $\Omega$ and that on each step we apply them to the generalized equation before
applying any other transformation.

Let $\Omega $ be a generalized equation. We construct a path $T(\Omega)$ (with associated structures),
as a directed path oriented from the root $v_0$, starting at $v_0$ and proceeding by induction on the
distance $n$ from the root.

We start with a general description of the path $T(\Omega)$. For each vertex $v$ in $T(\Omega)$ there
exists a unique generalized equation $\Omega_v$ associated with $v$. The initial equation $\Omega$ is
associated with the root $v_0$, $\Omega_{v_0} = \Omega$. In addition there is a homogeneous system of
linear equations $\Sigma_v$ with integer coefficients on the lengths of  variables of $\Omega _v$.
We take $\Sigma_{v_0}$ to be empty. For each edge $v \to v'$ (here $v$ and $v'$ are the origin and
the terminus of the edge) there exists an epimorphism $\pi(v,v') : G_{\Omega_v} \to G_{\Omega_v'}$
associated with $v \to v'$. If $\Omega$ was constructed for the group $G$ with regular free length
function, then  $\pi(v,v')$ is an isomorphism. If $\Omega$ was constructed for the group $\tilde G$
with free but not regular length function, then $\pi(v,v')$ is a monomorphism on $\tilde G$.

If
$$v_0 \rightarrow v_1 \rightarrow \cdots \rightarrow v_s \rightarrow u$$
is a subpath of $T(\Omega)$, then by $\pi(v,u)$ we denote composition of corresponding isomorphisms
$$\pi(v,u) = \pi(v,v_1) \circ \cdots \circ \pi(v_s,u).$$

If $v \to v'$ is an edge then there exists a finite sequence of elementary or derived transformations
from $\Omega_v$ to $\Omega_{v'}$ and the isomorphism $\pi(v,v')$ is a composition of the
isomorphisms corresponding to these transformations. We also assume that active (and non-active)
sections in $\Omega_{v'}$ are naturally inherited from $\Omega_v$, if not said otherwise. Recall that
initially all sections are active.

Suppose the path $T(\Omega)$ is constructed by induction up to  level $n$ and suppose $v$ is a
vertex at distance $n$ from the root $v_0$. We describe now how to extend the path from $v$. The
construction of the outgoing edge at $v$ depends on which case described below takes place at
vertex $v$. We say that two elements are comparable (or their lengths are comparable) if they have
the same height. There are three possible cases.

\begin{itemize}
\item {\bf Linear case:} there exists $h_i$ in the active part such that $|h_i|$ is comparable with
the length of the active part and $\gamma(h_i) = 1$.

\item {\bf Quadratic and almost quadratic case:} $\gamma(h_i) = 2$ for all $h_i$ in the active part
such that $|h_i|$ is comparable with the length of the active part.

\item {\bf General JSJ case:} $\gamma(h_i) \geqslant 2$ for all $h_i$ in the active part such that
$|h_i|$ is comparable with the length of the active part, and there exists such $h_i$ that $\gamma(h_i)
> 2$.
\end{itemize}

\subsubsection{Linear case}
\label{linear_case}

We apply Tietze cleaning (D6) at the vertex $v_n$ if it is possible. We re-write the system of linear
equations $\Sigma_{v_n}$  in new variables and obtain a new system $\Sigma_{v_{n+1}}$.

If $\Omega_{v_{n+1}}$ splits into two parts, $\Omega_{v_{n+1}}^{(1)} = Ker(\Omega_{v_n})$ and
$\Omega_{v_{n+1}}^{(2)}$ that corresponds to a free group $F(K)$, then when we put the free group section
$\Omega_{v_{n+1}}^{(2)}$ into a non-active part we decrease both  complexities $\tau$ and Delzant-Potyagailo's complexity $c$. It may happen that
the kernel is empty, then the process terminates.

If it is impossible to apply Tietze cleaning (that is $\gamma(h_i) \geqslant 2$ for any $h_i$ in the
active part of $\Omega_v$ comparable to the length of the active part), we apply the entire
transformation.

{\bf Termination condition:} $\Omega_v$ does not contain active sections. In this case the vertex
$v$ is called a {\it leaf} or an {\it end vertex}.

\subsubsection{Quadratic and almost quadratic case}
\label{quad_case}

Suppose that $\gamma_i = 2$ for each $h_i$ in the active part comparable with the length of the
active part of $\Omega_v$. First of all, we fill in all the $h_i's$ in the active part such that
$\gamma_i = 1$ by new (infinitely short) bases $\mu$ such that $\overline\mu$ covers a new variable
that we add to the non-active part.

We apply the entire transformation (D7), then apply Tietze cleaning (D6), if possible, then again
apply entire transformation, etc. In this process we, maybe, will remove some pairs of matching bases
decreasing the complexity $\tau$. Eventually we either end up with empty active part or the process
will continue infinitely, and the number of bases in the active part will be constant.

\begin{lemma} \cite{KMS:2011(3)}
\label{le:43}
If a closed section $\sigma$ has quadratic-coefficient bases, and the entire transformation goes
infinitely, then after a finite number of steps there will be quadratic bases belonging to $\sigma$
that have length infinitely larger than all participating quadratic coefficient bases on $\sigma$.
\end{lemma}

If $\sigma$ does not have quadratic-coefficient bases then $G_{R(\Omega_v)}$ splits as a free product
with one factor being a closed surface group or a free group.
We move $\sigma$ into a non-active part and thus decrease the complexity $\tau$.

We repeat the
described transformation until there is no quadratic base on the active part that has length comparable
with the length of the remaining active part. Then we  consider the remaining generalized equation
in the active part. We remove from the active part doubles of all quadratic coefficient bases that
belong to non-active part (doing this we may create new boundaries). We will remember the
relations corresponding to these pairs of bases. In this case the remaining generalized equation has
smaller complexity $\tau$. Relations corresponding to the quadratic sections that we made non-active show
that $G_{\Omega}$ is an HNN-extension of the subgroup generated by the variables in the active part
and (maybe) a free group. Removing these double bases we have to add equations to $\Sigma_{v_{i+1}}$
that guarantee that the associated cyclic subgroups are generated by elements of the same length.

\subsubsection{General JSJ-case}
\label{gen_jsj_case}

Generalized equation $\Omega_v$ satisfies the condition $\gamma_i \geqslant 2$ for each $h_i$ in the
active part such that $|h_i|$ is comparable with the length of the active part, and $\gamma_i > 2$
for at least one such $h_i$. First of all, we fill in all the $h_i's$ in the active part such that
$\gamma_i = 1$ by new (infinitely short) bases with doubles corresponding to free variables in the
non-active part. We apply the transformation (D1) to close the quadratic part and put it in front of
the interval.

\paragraph{(a) QH-subgroup case.}

Suppose that the entire transformation of the quadratic part (D7) goes infinitely. Then the quadratic
part of $\Omega_v$ (or the initial  section from the beginning of the  quadratic part until the first
base on the quadratic part that does not participate in the entire transformation) corresponds to a
QH-vertex or to the representation of $G_\Omega$ as an HNN-extension, and there is a quadratic base
(on this section) that is infinitely longer than all the quadratic coefficient bases (on this section).
We work with the quadratic part the same way as in the quadratic case until there is no quadratic
base satisfying the condition above. Then we make the quadratic section non-active, and consider the
remaining generalized equation where we remove doubles of all the quadratic coefficient bases. We
certainly have to remember that the bases that we removed express some variables in the quadratic
part (that became non-active) in the variables in the active part. We have to add an equation to
$\Sigma_{v_{i+1}}$ that guarantees that the associated cyclic subgroups are generated by elements of
the same length. In this case the subgroup of $G_\Omega$ that is isomorphic to the coordinate group
of the new generalized equation in active part is a vertex group in an abelian splitting of
$G_\Omega$ and has smaller Delzant--Potyagailo complexity.

\paragraph{(b) QH-shortening}

\begin{lemma} \cite{KMS:2011(3)}
Suppose that the quadratic part of $\Omega$ does not correspond to the HNN-splitting of $G_\Omega$
(there are only quadratic coefficient bases), or, we cannot apply the entire transformation to the
quadratic part infinitely. In this case either $G_{\Omega}$ is a non-trivial free product or, applying
the automorphism of $G_\Omega$, one can replace the words corresponding to the quadratic bases in
the quadratic part by their automorphic images such that in the new solution $H^+$ of $\Omega$ the length of the
quadratic part is bounded by some function $f_1(\Omega)$ times the length of the non-quadratic part.
Solution $H^+$ can be chosen consistent with $H$.
\end{lemma}
If there is a matching pair, we replace $G_{\Omega}$ by the group obtained by removing a cyclic free
factor corresponding to a matching pair. We also replace $\Omega$ by the generalized equation
obtained by removing the matching pair. The Delzant--Potyagailo complexity decreases.
\paragraph{(c) Abelian splitting: short shift.}

\begin{prop}
\label{PerSt}\cite{KMS:2011(3)}
Suppose $\Omega_v$ satisfies the following condition: the carrier base $\mu$ of the equation $\Omega_v$
intersects with its dual $\overline\mu$ (form an overlapping pair) and is at least twice longer than
$|\alpha(\overline\mu) - \alpha(\mu)|$. Then $G_{\Omega_v}$ either splits as a fundamental group of
a graph of groups that has a free abelian vertex group or splits as an HNN-extension with abelian
associated subgroups.
\end{prop}
The proof is given in \cite{KMS:2011(3)}, it uses the technique of so-called periodic structures
introduced by Razborov in his Ph.D thesis and almost repeats the proof given in
\cite{Kharlampovich_Myasnikov:1998(2)} to show that the coordinate group of a generalized equation
splits in this case  as a fundamental group of a graph of groups that has a free abelian vertex
group or splits as an HNN-extension with abelian associated subgroups. In the HNN-extension case, the base group is the coordinate group of the generalized equation obtained from the original by removing the corresponding stable letter. This reduces the complexity $\tau$ of the generalized equation.

\paragraph{(d) Abelian splitting: long shift.} If $\Omega$ does not satisfy the conditions of
(a)--(c), we perform QH-shortening, then apply the entire transformation and then, if possible, the
transformation (D6).

\begin{lemma} \cite{KMS:2011(3)}
\label{3.2}
Let
$$v_1 \rightarrow v_2 \rightarrow \cdots \rightarrow v_r \rightarrow \cdots$$
be an infinite path in $T(\Omega)$. Then there exists a natural number $N$ such that all the
generalized equations in vertices $v_n,\ n \geqslant N$ satisfy the general JSJ-case (d).
\end{lemma}
\begin{proof} Indeed, the Tietze cleaning either replaces the group by its proper free factor or
decreases the complexity. Every time when the case (a) holds we replace $G$ by some vertex group in
a non-trivial abelian splitting of $G$. This can be done only finitely many times
\cite{Delzant_Potyagailo:2001}. Every time when case (c) takes place, we decrease the complexity
$\tau$.
\end{proof}

\begin{prop}
\label{(d)} \cite{KMS:2011(3)}
The general JSJ case (d) cannot be repeated infinitely many times.
\end{prop}

\section{Structure theorems for $\Lambda$-free groups}
\label{sec:struct_th_lambda}

\subsection{Finitely generated $\mathbb{R}$-free groups (Rips' Theorem)}
\label{subs:rips_th}

R. Lyndon in \cite{Lyndon_Schupp:2001} conjectured that any group acting freely on an $\mathbb{R}$-tree
can be embedded into a free product of finitely many copies of $\mathbb{R}$. Counterexamples were
initially given in \cite{Alperin_Moss:1985} and \cite{Promislow:1985}.

Later, J. Morgan and P. Shalen in \cite{Morgan_Shalen:1991} showed that the fundamental groups of
closed surfaces (except non-orientable of genus $1,2$ and $3$) are $\mathbb{R}$-free. Since such
groups are not free products of subgroups of $\mathbb{R}$ this gives a wide class of counterexamples
to Lyndon's Conjecture.

In 1991 I. Rips came with an idea of a proof of the Morgan and Shalen conjecture about finitely generated $\mathbb{R}$-free groups. This result can be formulated
as follows.

\begin{theorem} \cite{GLP:1994},\cite{Bestvina_Feighn:1995} (Rips' Theorem)
\label{th:Rips}
Let $G$ be a finitely generated group acting freely and without inversions on an $\mathbb{R}$-tree.
Then $G$ can be written as a free product $G = G_1 \ast \cdots \ast G_n$ for some integer $n
\geqslant 1$, where each $G_i$ is either a finitely generated free abelian group, or the
fundamental group of a closed surface.
\end{theorem}
\begin{proof}
To prove the theorem we are going to use the techniques of Section \ref{sec:proc}. In this case
$\Lambda = \mathbb R$. Suppose $G$ is a finitely presented group with free Lyndon length function
$L$ in $\mathbb R$. By Theorem \ref{th:main4}, $G$ can be embedded into a finitely presented group
with a free regular length function in $\mathbb R$, so we assume from the beginning that $G$ has a
free regular length function in $\mathbb R$. By Corollary \ref{chis-cor} there exists an embedding
$\psi : G \rightarrow CDR(\mathbb Z \oplus \mathbb R, X)$ such that $|\psi(g)| = (0, L(g))$ for
any $g \in G$.

We construct the generalized  equation $\Omega$ for $G$ and apply the elimination process $\Omega$.
Linear case always splits off free factors of $G$. Quadratic case, almost quadratic case, general
JSJ case \ref{gen_jsj_case} (a) will produce closed surface groups factors. Indeed, in these cases,
by Lemma \ref{le:43}, the height of some quadratic bases  is higher that the height of quadratic
coefficient bases. Since all the bases have length $(0,r), r \in \mathbb R$, there are no quadratic
coefficient bases in these cases. General JSJ case (c) for $\Lambda$ produces abelian vertex groups
corresponding to periodic structures and HNN-extensions with stable letter infinitely longer than
the generators of associated abelian subgroups. Since $\Lambda = \mathbb R$, we do not have such
HNN-extensions. If the edge group were non-trivial, then applying automorphisms of $G$ we could
shorten generators of the abelian vertex group. Namely, there exists a number $N$ depending only on
$\Omega$ such that the carrier base $\mu$ of the current equation $\Omega_v$ intersects with its
double and in no longer than $N |\alpha(\overline\mu) - \alpha(\mu)|$. This situation is similar
to the general JSJ case (d), because the length of $\mu$ is bounded in terms of  $|\alpha(\overline
\mu) - \alpha(\mu)|$. One can similarly prove that it cannot be repeated infinitely many times.
Therefore, the edge group of the abelian vertex group is trivial.

We have shown that $G$ is a free product of free abelian groups and closed surface groups (notice
that a free group is also a free product of free abelian (cyclic) groups). Therefore any subgroup
of $G$ has the same structure. Since there is no proper infinite chain of quotients of $G$ that are
also free products of free abelian groups and closed surface groups, we can prove the theorem for a
finitely generated group $\bar G$ by adding relations of $G$ one-by-one and considering the chain of
finitely presented quotients.
\end{proof}

Rips' Theorem does not hold for infinitely generated groups. Counterexamples were given in
\cite{Dunwoody:1997} and \cite{Zastrow:1998}. In particular, Dunwoody showed that both groups
$$G_1 = \langle a_1, b_1, a_2, b_2, \ldots \mid b_1 = [a_2, b_2],\ b_2 = [a_3, b_3], \ldots \rangle$$
and
$$G_2 = \langle a_1, b_1, a_2, b_2, \ldots \mid b_1 = a_2^2 b_2^2,\ b_2 = a_3^2 b_3^2, \ldots
\rangle$$
are $\mathbb{R}$-free but cannot be decomposed  as free products of surface groups and subgroups of
$\mathbb{R}$.

Recently, Berestovskii and Plaut gave a new method to construct $\R$-free groups \cite{Berestovskii_Plaut:2009}. In particular, they provide new examples of $\R$-free groups that are not free products of free abelian and surface groups. In fact, some of these groups are locally free but not free (so, obviously, not subgroups of free products of free abelian and surface groups).

\subsection{Finitely generated $\mathbb{R}^n$-free groups}
\label{se:R^n-free}

In 2004 Guirardel proved the following result that reveals the structure of finitely generated
$\mathbb{R}^n$-free groups, which is reminiscent of the Bass' structural theorem for $\Z^n$-free
groups. This is not by chance, since every $\Z^n$-free group is also $\R^n$-free, and ordered
abelian groups $\Z^n$ and $\R^n$ have a similar convex subgroup structure.  However, it is worth to
point out that the original Bass argument for $\Lambda = \Z \oplus \Lambda_0$ does not work in the
case of $\Lambda = \R \oplus \Lambda_0$.

\begin{theorem} \cite{Guirardel:2004}
\label{th:Guirardel}
Let $G$ be a finitely generated, freely indecomposable $\mathbb{R}^n$-free group. Then $G$ can be
represented as the fundamental group of a finite graph of groups, where edge groups are cyclic and
each vertex group is a finitely generated $\mathbb{R}^{n-1}$-free.
\end{theorem}

In fact, there is a more detailed version  of this result,  Theorem 7.2 in \cite{Guirardel:2004},
which is rather technical, but gives more for applications. Observe also that neither Theorem
\ref{th:Guirardel} nor the more detailed version of it,  does not "characterize" finitely generated
$\R^n$-free groups, i.e. the converse of the theorem does not hold. Nevertheless, the result is
very powerful and gives several important corollaries.

\begin{cor} \cite{Guirardel:2004}
\label{co:Guirardel_1}
Every finitely generated $\mathbb{R}^n$-free group is finitely presented.
\end{cor}
This comes from Theorem \ref{th:Guirardel} and elementary properties of free constructions by
induction on $n$.

Theorem \ref{th:Guirardel} and the Combination Theorem for
relatively hyperbolic groups proved by F. Dahmani in  \cite{Dahmani:2003} imply the following.

\begin{cor}
\label{co:Guirardel_2}
Every finitely generated $\mathbb{R}^n$-free group is hyperbolic relative to its non-cyclic
abelian subgroups.
\end{cor}

A lot is known about groups which are hyperbolic relative to its maximal abelian subgroups ({\em toral
relatively hyperbolic groups}), so all of this applies to $\R^n$-free groups. We do not mention any of
these results here, because we discuss their much more general  versions in the next section in the
context of $\Lambda$-free groups for arbitrary $\Lambda$.

\subsection{Finitely presented $\Lambda$-free groups}
\label{subs:lambda-free-gps}

In this section we discuss recent results obtained on finitely presented $\Lambda$-free groups for an
arbitrary abelian ordered group $\Lambda$. Notice that finitely generated $\R^n$-free (or $\Z^n$-free) groups are
finitely presented, so all the results below apply to arbitrary finitely generated $\R^n$-free (or $\Z^n$-free) groups.

The elimination process (the $\Lambda$-Machine)  developed in  Section \ref{sec:proc} allows one to
prove the following
theorems.

\begin{theorem} [The Main Structure Theorem \cite{KMS:2011(3)}]
\label{th:main1}
Any finitely presented group $G$ with a regular free length function in an ordered abelian group
$\Lambda$ can be represented as a union of a finite series of groups
$$G_1 < G_2 < \cdots < G_n = G,$$
where
\begin{enumerate}
\item $G_1$ is a free group,
\item $G_{i+1}$ is obtained from $G_i$ by finitely many HNN-extensions in which associated subgroups
are maximal abelian, finitely generated, and length isomorphic as subgroups of $\Lambda$.
\end{enumerate}
\end{theorem}

\begin{theorem} \cite{KMS:2011(3)}
\label{th:main3}
Any finitely presented $\Lambda$-free groups is  $\R^n$-free.
\end{theorem}

In his book \cite{Chiswell:2001} Chiswell (see also \cite{Remeslennikov:1989})  asked the following 
principal question (Question 1, page 250): If $G$ is a finitely generated $\Lambda$-free group, is 
$G$ $\Lambda_0$-free for some finitely generated abelian ordered group $\Lambda_0$? The following 
result answers this question in  the affirmative in the strongest form. It comes from the proof of 
Theorem \ref{th:main3} (not the statement of the theorem itself).

\begin{theorem}
\label{co:main1}
Let $G$ be a finitely presented group with a free Lyndon length function $l : G \to \Lambda$. Then
the subgroup $\Lambda_0$ generated by $l(G)$ in $\Lambda$ is finitely generated.
\end{theorem}

\begin{theorem} \cite{KMS:2011(3)}
\label{th:main4}
Any finitely presented group $\widetilde G$ with a free length function in an ordered abelian group
$\Lambda$ can be isometrically embedded into a finitely presented group $G$ that has a free regular
length function in $\Lambda$.
\end{theorem}

The following result automatically follows from Theorem \ref{th:main1} and Theorem \ref{th:main4} by
simple application of Bass-Serre Theory.

\begin{theorem}
\label{co:main5}
Any finitely presented $\Lambda$-free group $G$ can be obtained from free groups by a finite sequence 
of amalgamated free products and HNN extensions along maximal abelian subgroups, which are free 
abelain groups of finite rank.
\end{theorem}

This theorem would have another proof provided Delzant-Potyagailo's proof of hierarchical accessibility
\cite{Delzant_Potyagailo:2001} were correct (it has a gap in the case of HNN extension). Indeed, a 
finitely presented group acting freely on a $\Lambda$-tree has a stable action on an $\R$-tree.
The proof is the same as the proof of Fact 5.1 in \cite{Guirardel:2004}. Therefore $G$ splits over a finitely generated 
abelian group \cite{Bestvina_Feighn:1995}. Then we could apply the result about hierarchical 
accessibility if it was available. Starting with this hierarhy one could obtain the one with edge groups 
maximal abelian.

The following result concerns with  abelian subgroups of $\Lambda$-free groups. For $\Lambda = \Z^n$ 
it follows from the main structural result for $\Z^n$-free groups and \cite{KMRS:2008}, for $\Lambda = 
\R^n$ it was proved in \cite{Guirardel:2004}. The statement 1) below answers to Question 2 (page 250) 
from \cite{Chiswell:2001} in the affirmative for finitely presented $\Lambda$-free groups.

\begin{theorem}
\label{co:main1b}
Let $G$ be a finitely presented $\Lambda$-free group. Then:
\begin{itemize}
\item [1)] every abelian subgroup of $G$ is a free abelian group of finite rank, which is uniformly 
bounded from above by the rank of the abelianization of $G$.
\item [2)] $G$ has only finitely many conjugacy classes of
maximal non-cyclic abelian subgroups,
\item [3)] $G$  has a finite classifying space and the cohomological dimension of $G$ is at most
$\max \{2, r\}$ where $r$  is the maximal rank of an abelian subgroup of $G$.
\end{itemize}
\end{theorem}
\begin{proof}
It comes from Theorem \ref{th:main1} by the standard properties of free product with amalgamation
and HNN-extensions. Another way to prove the theorem is to notice that finitely presented 
$\Lambda$-free groups are $\R^n$-free (Theorem \ref{th:main3}) and then apply the corresponding 
results for $\R^n$-free groups from \cite{Guirardel:2004}.
\end{proof}

\begin{theorem}
\label{co:Lambda_Guirardel_2}
Every finitely presented  $\Lambda$-free group is hyperbolic relative to its non-cyclic abelian
subgroups.
\end{theorem}
\begin{proof}
It follows from Theorem \ref{th:main3} and Corollary \ref{co:Guirardel_2} on $\R^n$-free groups, or
directly from the structural Theorem \ref{th:main1} and  the Combination Theorem for relatively
hyperbolic groups \cite{Dahmani:2003}.
\end{proof}

The following results answers affirmatively in the strongest form to the Problem (GO3) from the Magnus
list of open problems \cite{BaumMyasShpil:2002} in the case of finitely presented groups.

\begin{cor}
\label{co:Lambda_Guirardel_3}
Every finitely presented  $\Lambda$-free group is biautomatic.
\end{cor}
\begin{proof} It follows from Theorem \ref{co:Lambda_Guirardel_2} and Rebbechi's result
\cite{Rebbechi:2001}.
\end{proof}

\begin{theorem}
\label{th:Lambda_quasi-convex_hierarchy}
Every finitely presented $\Lambda$-free group $G$ has a quasi-convex hierarchy.
\end{theorem}
\begin{proof} By Theorem \ref{co:main5}, $G$ can be obtained by a finite sequence ${\cal S}$ of amalgamated
free products and HNN extensions along maximal abelian subgroups starting from free groups. Hence,
each group $H$ in this sequence is either an amalgamated free product of $H_1 \ast_{A=B} H_2$, or
an HNN extension $K \ast_{A^t = B}$, where the groups $H_1, H_2, K$ are constructed on previous steps.
To prove the corollary we assume that all groups involved in the construction of $H$ have quasi-convex
hierarchies and we have to show that $A$ is quasi-convex in $H$ both when $H = H_1 \ast_{A=B} H_2$
and $H = K \ast_{A^t = B}$.

Note that each group in the list $\{H_1, H_2, K\}$ is $\Lambda$-free, hence, biautomatic by Corollary
\ref{co:Lambda_Guirardel_3}, and from \cite{Bridson_Haefliger:1999} it follows that all their maximal
abelian subgroups are regular and quasi-convex in the corresponding groups.

Consider these cases.

\smallskip

{\bf Case I.} $H = \langle K, t \mid t^{-1} A t = B \rangle$, where $A$ and $B$ are maximal abelian
subgroups of $K$.

\smallskip

Suppose $A$ and $B$ are not conjugate in $K$. Then by \cite[Lemma 2]{GKM} both are maximal abelian
in $H$ and therefore quasi-convex there.

If $A^g = B$ for some $g \in K$, then we have
$$H \simeq \langle K, s \mid s^{-1} A s = A \rangle,$$
where $s$ represents $t g^{-1} \in H$ and $A$ embeds into a maximal free abelian subgroup $C =
\langle A, s\rangle$ as a direct factor (see \cite[Lemma 3]{KMRS:2008}). Hence, $A$ is quasi-convex
in $C$ and, hence, in $H$.

\smallskip

{\bf Case II.} $H = H_1 \ast_{A=B} H_2$, where $A$ and $B$ are maximal abelian subgroups respectively
of $H_1$ and $H_2$.

\smallskip

From \cite[Theorem 4.5]{Magnus_Karras_Solitar:1977} it follows that both $A$ and $B$ are maximal abelian
in $H$ and therefore quasi-convex there.

\end{proof}

Theorems \ref{th:Lambda_quasi-convex_hierarchy} and \ref{co:Lambda_Guirardel_2} imply  the following 
result (the argument is straightforward but rather technical and we omit it, see also Lemma 17.10 in  
\cite{Wise:2011}, though there is no proof there neither).

\begin{theorem}
\label{co:Lambda_quasi-convex_hierarchy_1}
Every finitely presented  $\Lambda$-free group $G$  is locally undistorted, that is, every finitely 
generated subgroup of $G$ is quasi-isometrically embedded into $G$.
\end{theorem}

Since a finitely generated $\R^n$-free group $G$ is hyperbolic relative to to its non-cyclic
abelian subgroups and $G$ admits a  quasi-convex hierarchy then  recent results of D. Wise
\cite{Wise:2011} imply the following.

\begin{cor}
\label{co:Lambda_special}
Every finitely presented $\Lambda$-free group $G$ is  virtually special, that is, some subgroup of
finite index in $G$ embeds into a right-angled Artin group.
\end{cor}

In his book \cite{Chiswell:2001} Chiswell posted Question 3 (page 250): Is every $\Lambda$-free group
orderable, or at least  right-orderable? The following result answers in the affirmative to Question 3
(page 250) from \cite{Chiswell:2001} in the case of finitely presented groups.

\begin{theorem}
\label{co:Lambda_special_0}
Every finitely presented $\Lambda$-free group is right orderable.
\end{theorem}
\begin{proof} I.Chiswell proved that every finitely generated $\R^n$-free group is right orderable
\cite{Chiswell:2013}, so the result now follows from Theorem \ref{th:main3}.
\end{proof}

The following addresses Chiswell's question whether $\Lambda$-free groups are orderable or not.
\begin{theorem}
\label{co:Lambda_special_1}
Every finitely presented $\Lambda$-free group is virtually orderable, that is, it contains an orderable
subgroup of finite index.
\end{theorem}
\begin{proof}
Indeed, in \cite{Duchamp_Krob:1992}, G. Duchamp and D. Krob show that right-angled Artin groups
are residually torsion free nilpotent. Hence, right-angled Artin groups are residually $p$-groups
for any $p$, so, they are orderable (see \cite{Rhemtulla:1973}). Now the result follows form
Corollary \ref{co:Lambda_special}.
\end{proof}

Note that one cannot remove ``virtually'' in the formulation of Theorem \ref{co:Lambda_special_1}
since S. Rourke recently showed that there are finitely presented $\Lambda$-free (even $\Z^n$-free)
groups which are not orderable (see \cite{Rourke:2012}).

Since right-angled Artin groups are linear (see \cite{Humphries:1994, Hsu_Wise:1999, DJ:2000} and
the class of linear groups is closed under finite extension we get the following

\begin{theorem}
\label{co:Lambda_quasi-convex_hierarchy_2}
Every finitely presented  $\Lambda$-free group is  linear.
\end{theorem}

Since every linear group is residually finite we get the following.

\begin{cor}
\label{co:Lambda_quasi-convex_hierarchy_3}
Every finitely presented  $\Lambda$-free group is residually finite.
\end{cor}

It is known that linear groups are equationally Noetherian (see
\cite{Baumslag_Miasnikov_Remeslennikov:1999} for discussion on equationally Noetherian groups),
therefore the following result holds.
\begin{cor}
\label{co:Lambda_quasi-convex_hierarchy_4}
Every finitely presented $\Lambda$-free group is equationally Noetherian.
\end{cor}

{\bf Hint of the proof of Theorem \ref{th:main1}.}  We perform the elimination process for the generalized
equation $\Omega = \Omega_{v_0}$ corresponding to $G$. If the process goes infinitely we obtain one
of the following:
\begin{itemize}
\item free splitting of $G$ with at least on non-trivial free factor \ref{linear_case} and, maybe, some surface
group factor \ref{quad_case},

\item a decomposition of $G$ as the fundamental group of a graph of groups with QH-vertex groups
corresponding to quadratic sections \ref{gen_jsj_case} (a),

\item decomposition of $G$ as the fundamental group of a graph of groups with abelian vertex groups
corresponding to periodic structures (see Proposition \ref{PerSt} in \ref{gen_jsj_case} (c) and HNN-extensions
with stable letters infinitely longer than the generators of the abelian associated subgroups
\ref{gen_jsj_case} (c).
\end{itemize}
Then we continue the elimination process with the generalized equation $\Omega_{v_1}$ where the
active part corresponds to weakly rigid subgroups. In the case of an HNN-extension, $\Omega_{v_1}$ is 
obtained from $\Omega_{v_0}$  by removing a pair of bases from the generalized equation corresponding 
to the periodic structure (and this equation has the same complexity as $\Omega_{v_0}$. This means that 
the complexity $\tau(\Omega_{v_1})$ is smaller than $\tau(\Omega_{v_0}).$ In the other cases the 
Delzant--Potyagailo complexity of weakly rigid subgroups is smaller by Proposition \ref{DP} (which follows 
from the correct part of the paper \ref{DP}. Therefore this procedure stops. At the end we obtain some 
generalized equation $\Omega_{v_{fin}}$ with all non-active sections. Continuing the elimination process 
till the end, we obtain, in addition, a complete system of linear equations with integer coefficients 
$\Sigma_{complete}$ on the lengths of  items $h_i$'s that is automatically satisfied and guarantees 
that the associated maximal abelian subgroups are length-isomorphic.

This completes the proof of Theorem \ref{th:main1}.
\begin{remark}
\label{nonreg1}
If we begin with the group $\widetilde G$  with free but not necessary
regular length function in $\Lambda$ then in the Elimination process we work with the generalized
equation $\Omega$ and add a finite number of elements from $\widehat G$ (see the end of Section 
\ref{subs:constr_ge}). Thus we got an embedding of $\widetilde G$ in a group that can be represented 
as a union of a finite series of groups
$$G_1 < G_2 < \cdots < G_n = G,$$
where
\begin{enumerate}
\item $G_1$ is a free group,
\item $G_{i+1}$ is obtained from $G_i$ by finitely many HNN-extensions in which associated subgroups
are maximal abelian, finitely generated, and length isomorphic as subgroups of $\Lambda$.
\end{enumerate}
\end{remark}
\begin{remark}
\label{fin}
As a result of the Elimination process, the equation $\Omega_{v_{fin}}$ (we will denote it
$\Omega_{fin}$) is defined on the multi-interval $I$, that is, a union of closed sections which have
a natural hierarchy: a section $\sigma_1$ is smaller than a section $\sigma_2$ if the largest base
on $\sigma_2$ is infinitely larger than the largest base on $\sigma_1$.

The lengths of bases satisfy the system of linear equations $\Sigma_{complete}$.
\end{remark}

{\bf Hint of the Proof of Theorem \ref{th:main3}.}
Suppose $G$, as above, is a finitely presented $\Lambda$-free group. By Theorem \ref{th:main4} we may assume that the action of $G$  is regular. Let $\Omega$ be a generalized
equation for $G$ corresponding to the union of closed sections $I$, $G = \langle \mathcal M \mid
\Omega({\mathcal M}) \rangle$.  Consider the Cayley graph $X = Cay(G, \cal M)$ of $G$ with respect to
the generators ${\cal M}$. Assign to edges of $Cay(G, \cal M)$ their lengths
in $\Lambda$ (and infinite words that represent the generators) and consider edges as closed intervals in $\Lambda$ of the corresponding length. For
each relation between bases of $\Omega$, $\lambda_{i_1} \cdots \lambda_{i_k} = \lambda_{j_1} \cdots
\lambda_{j_m}$ (without cancelation) there is a loop in $X$ labeled by this relation. Then the path
 labeled by $\lambda_{i_1} \cdots \lambda_{i_k}$ has the same length in
$\Lambda$ as the path labeled by $\lambda_{j_1} \cdots \lambda_{j_m}$. If $x$ is a point on the path
$\lambda_{i_1} \cdots \lambda_{i_k}$ at the distance $d \in \Lambda$ from the beginning of the path,
and $x'$ is a point on the path $\lambda_{j_1} \cdots \lambda_{j_m}$ at the distance $d$ from the
beginning, then we say that {\em $x$ and $x'$ are in the same leaf}. In other words, after we
substitute generators in $\mathcal M$ by their infinite word representations, we ``fold'' loops
into segments. We consider the equivalence
relation between points on the edges of $X$ generated by all such pairs $x \sim x'$. Equivalence classes of this
relation are called {\em leaves}. We also glue an arc isometric to a unit interval between each $x$
and $x'$. Let ${\cal F}$ be the foliation (the set of leaves). One can define a foliated complex
$\Sigma = \Sigma(X, {\cal F})$ associated to $X$ as a pair $(X, {\cal F})$. The paths in $\Sigma$ can travel vertically (along the leaves) and horizontally
(along the intervals in $\Lambda$). The length of a path $\gamma$ in $\Sigma (X, {\cal F})$ (denoted
$\|\gamma\|$) is the sum of the lengths of horizontal intervals in $\gamma$. Therefore  $\Sigma$  is a graph with points on horizontal intervals being vertices. We now identify all points of  $\Sigma$ that belong to the same leaf (identify vertices in the leaf and points on the vertical edges of this leaf in $\Sigma$). We denote by $T$ the obtained graph. A reduced path in $T$ is a path without backtracking. We call  removing of a backtracking in $T$ a reduction of a path.
\begin{lemma}\cite{KMS:2011(3)}\label{le:19.1} The image of a loop in $\Sigma$ can be reduced to a point in $T$.
\end{lemma}
Indeed, a loop in $\Sigma$ is a composition of vertical and horizontal subpaths.  If a vertical path of length one connects two points of $\Sigma$, then they are also connected in $\Sigma$ by a horizontal path that is mapped to a path of the form $rr^{-1}$ in $T$. Then a loop in $\Sigma$ can be transformed into a loop where all paths are horizontal. Therefore any simple loop in $\Sigma$ is mapped into a path with back-tracking in $T$, and if we reduce this loop in $T$ we get a point. This implies that there is a unique reduced path between any two points of $T$. A distance between two points in $T$ is the length of the reduced path between them.
 We extend naturally the
left action of $G$ on $X$  to the action on $T$. The following lemma holds.

\begin{lemma} \cite{KMS:2011(3)}
\label{le:19.2}
\begin{enumerate}
\item[(1)] $T$ is a $\Lambda$-tree.
\item[(2)] The left action of $G$ on $T$ is free.
\end{enumerate}
\end{lemma}
Similarly, we can construct a $\Lambda$-tree where $G$ acts freely beginning not with the original
generalized equation $\Omega$ but with a generalized equation obtained from $\Omega$ by the application
of the Elimination process. In this case we may have closed sections with some items $h_i$ with
$\gamma(h_i) = 1$. Then to construct the Cayley graph of $G$ in the new generators we cover items $h_i$ with
$\gamma(h_i)= 1$ by bases without doubles (these bases were previously removed at some stages of the
Elimination process as matching pairs).

We begin by considering the union of the  closed sections of the minimal height in the hierarchy
introduced in Remark \ref{fin}. Denote the union of these sections by $\sigma$. The group $H$ of the
generalized equation corresponding to their union is a free product of free groups, free abelian
groups and closed surface groups because for these groups we stop the Elimination process. Notice that for those sections for which $\gamma(h_i) = 1$ for some
maximal height items, these items $h_i$ are also products of some bases because initially every item
is a product of bases. We can assume the following:

\begin{enumerate}
\item For all closed sections of $\sigma$ such that $\gamma(h_i) = 2$ for all items of the maximal
height, the number of bases of maximal height cannot be decreased using entire transformation or
similar transformation applied from the right of the section ({\em right entire transformation}).

\item For all closed sections of $\sigma$ where we apply linear elimination, the number of bases
of maximal height cannot be decreased using the transformations as above as well as transformation
(E2) (transfer) preceded  by creation of necessary boundary connections as in (E5) (denote it
($E2_5$)), (E3) and first two types of linear elimination (D5) (denote it ($D5_{1,2}$)).
\end{enumerate}

Indeed, if we can decrease the number of bases of maximal height using the transformations described
above then we just do this and continue. Since these transformations do not change the total number
of bases we can also assume that they do not decrease the number of bases of second maximal height etc.
We now re-define the lengths of bases belonging to $\sigma$ in ${\mathbb R}^k$. We are going to show
that all components of the length of every base can be made zeros except for the components which
appear to be maximal in the lengths of bases from $\sigma$.

Let $\Omega_\sigma$ be the generalized equation corresponding to the sections from $\sigma$. Let $H$
be the group of the equation $H = \langle \mathcal M_\sigma \mid \Omega_\sigma(M) \rangle$.

Denote by $\Lambda_1$ the minimal convex subgroup of $\Lambda$ containing lengths of all bases in
$\sigma$, and by $\Lambda'$ a maximal convex subgroup of $\Lambda_1$ not containing lengths of maximal
height bases in $\sigma$ (it exists by Zorn's lemma). Then the quotient $\Lambda_1 / \Lambda'$ is a
subgroup of $\mathbb R$. Denote by $\hat\ell$ the length function in $\mathbb R$ on this quotient
induced from $\ell$. We consider elements of $\Lambda'$ as infinitesimals. Denote by $\widehat T$ the
$\mathbb R$-tree constructed from $T$ by identifying points at zero distance (see \cite{Chiswell:2001},
Theorem 2.4.7). Then $H$ acts on $\widehat T$. Denote by $\|\gamma\|_{\mathbb R}$ the induced length
of the path $\gamma$ in $\widehat T$, and by $d_{{\mathbb R}}(\bar x,\bar y)$ the induced distance.

However, the action of $H$ on $\widehat T$ is not free. The action is minimal, that is, there is no
non-empty proper invariant subtree. Notice that the canonical projection $f: T \rightarrow \widehat
T$ preserves alignment, and the pre-image of the convex set is convex. The pre-image of a point in
$\widehat T$ is an infinitesimal subtree of $T$.

\begin{lemma} \cite{KMS:2011(3)}
\label{le:62}
The action of $H$ on $\widehat T$ is superstable: for every non-degenerate arc $J \subset \widehat T$
with non-trivial fixator, and for every non degenerate subarc $S \subset J$, one has $Stab(S) =
Stab(J)$.
\end{lemma}

The proof is the same as the proof of Fact 5.1 in \cite{Guirardel:2004}.

\begin{prop} \cite{KMS:2011(3)}
\label{quad-main0}
One can define the lengths of bases in $\sigma$ in ${\mathbb R}^k$.
\end{prop}

We will show how to prove the proposition only for the case of a closed section $\sigma_1$ on the
lowest level corresponding to a closed surface group.

\begin{lemma} \cite{KMS:2011(3)}
\label{quad-main}
Let $\sigma_1$ be a closed section on the lowest level corresponding to a closed surface group and
the lengths of the bases satisfy some system of linear equations $\Sigma$. Then one can define the
lengths of bases in $\sigma_1$ in ${\mathbb R}^k$.
\end{lemma}
\begin{proof} Denote by ${\mathcal M}_{\sigma +}$ the set of bases (on all the steps of the process
of entire transformation applied to the lowest level) with non-zero oldest component and by
${\mathcal M}_{\sigma 0}$ the rest of the bases (infinitesimals). Denote by $\hat\ell$ the projection
of the length function $\ell$ to $\Lambda_1 / \Lambda'$ as before. Then $\lambda \in
{\mathcal M}_{\sigma +}$ if and only if $\hat\ell(\lambda) > 0$. We apply the entire transformation
to $\sigma_1$. If we obtain an overlapping pair or an infinitesimal section, where the process goes
infinitely, we declare it non-active and move to the right. This either decreases, or does not change
the number of bases in the active part. Therefore, we can assume that the process goes infinitely
and the number of bases in ${\mathcal M}_{\sigma +}$ never decreases. Therefore bases from
${\mathcal M}_{\sigma 0}$ are only used as transfer bases.

We will show that the stabilizer of a pre-image of a point, an infinitesimal subtree $T_0$ of $T$ is
generated by some elements in ${\mathcal M}_{\sigma 0}$. An element $h$ from $H$ belongs to such a
stabilizer if $\hat\ell(h) = 0$. If some product of bases not only from ${\mathcal M}_{\sigma 0}$
has infinitesimal length (denote this product by $g$), then by Lemma \ref{le:19.2}, the identity and 
$g$ belong to leaves at the infinitesimal distance $\delta$ in $\Sigma$. Therefore using elementary 
operations we can obtain a base of length $\delta$. Denote by $\Lambda''$ the minimal convex 
subgroup of $\Lambda$ containing all elements of the same height as bases from 
${\mathcal M}_{\sigma 0}$.

This implies the following lemma which we need to finish the proof of Lemma \ref{quad-main}.

\begin{lemma} \cite{KMS:2011(3)}
\label{quad}
Let $\overline\sigma_1$ be the projection of the quadratic section $\sigma_1$ to $\Lambda_1 /
(\Lambda'' \cap \Lambda_1)$. Suppose the process of entire transformation for $\sigma_1$ goes
infinitely and the number of bases in ${\mathcal M}_{\sigma +}$ never decreases. Then the process
of entire transformation for $\overline\sigma_1$ goes infinitely too and the number of bases in
$\mathcal{M}_{\sigma +}$ never decreases.

If $\ell(g) \in \Lambda'$ then $g$ is a product of bases in ${\mathcal M}_{\sigma 0}$ and $\ell(g)
\in \Lambda''$.
\end{lemma}

This lemma implies that there is no element in $H$ infinitely larger than all bases in
$\mathcal{M}_{\sigma 0}$, but infinitely smaller than all bases in $\mathcal{M}_{\sigma +}$.

Therefore, we can make all components of the lengths of bases in $\mathcal{M}$ zeros except for
those which are maximal components of some bases in $\mathcal M$. Since $\mathcal M$ is a finite
set, the number of such components is finite, and the length is defined in $\mathbb{R}^k$ for some
$k$ not larger than the number of pairs of bases.
\end{proof}

Similarly we can prove the Proposition for sections corresponding to the linear case and the case of
an abelian vertex group. Using induction on the number of levels obtained in the Elimination process,
we prove the statement of Theorem \ref{th:main3}.

The points of an $\mathbb{R}^n$-tree, where $G$ acts freely are the leaves in the foliation
corresponding to the new length of bases in $\mathbb{R}^n$. The new lengths of bases are exactly
their Lyndon lengths. $\Box$

{\bf Hint of the  Proof of Theorem \ref{th:main4}}. Notice that in the case when $\widetilde G$ is a
finitely presented group with a free length function in $\Lambda$ (not necessary regular) it can be
embedded in the group with a free regular length function in $\Lambda$. That group can
be embedded in $R(\Lambda', X)$. When we make a generalized equation for $\widetilde G$, we have to
add only a finite number of elements from $R(\Lambda', X)$. We run the elimination process for this
generalized equation as we did in the proof of Theorem \ref{th:main1} and obtain a group $G$ as in
Remark \ref{nonreg1}, where $\widetilde G$ is embedded, and then redefine the length of elements of
$G$ in $\mathbb{R}^n$ as above. Therefore $G$ acts freely and regularly on a $\mathbb{R}^n$-tree.
The theorem is proved.

Moreover, one shows by induction that the length function in $\mathbb{R}^n$ and in $\Lambda$ defined
on $G$ is regular. This proves Theorem \ref{th:main4}.

\subsection{Algorithmic problems for finitely presented $\Lambda$-free groups}
\label{subs:alg_prob_lambda}

The structural results of the previous section give solution to many algorithmic problems on finitely 
presented $\Lambda$-free groups.

\begin{theorem} \cite{KMS:2011(3)}
\label{th:word_conjugacy_lambda}
Let $G$ be a finitely presented $\Lambda$-free group. Then the following algorithmic problems
are decidable in $G$:
\begin{itemize}
\item the Word Problem;
\item the  Conjugacy Problems;
\item the Diophantine Problem (decidability of arbitrary equations in $G$).
\end{itemize}
\end{theorem}
\begin{proof}
By Theorem \ref{co:Lambda_Guirardel_2} a finitely presented $\Lambda$-free group $G$ is hyperbolic 
relative to its non-cyclic abelian subgroups. Decidability of the Conjugacy  Problem for such groups 
is known  - it was done by Bumagin \cite{Bumagin:2004} and Osin \cite{Osin:2006}. Decidability 
of the Diophantine Problems in such groups was proved by Dahmani \cite{Dahmani:2003}.
\end{proof}

Theorem \ref{co:Lambda_Guirardel_2} combined with results of Dahmani and Groves \cite{Dahmani_Groves:2008}
immediately  implies the following two corollaries.

\begin{cor}
\label{co:algorithm_lambda_1}
Let $G$ be a finitely presented  $\Lambda$-free group. Then:
\begin{itemize}
\item $G$ has a non-trivial abelian splitting and one can find such a splitting effectively,
\item $G$ has a non-trivial abelian JSJ-decomposition and one can find such a decomposition
effectively.
\end{itemize}
\end{cor}

\begin{cor}
\label{co:algorithm_lambda_2}
The Isomorphism Problem is decidable in the class of finitely presented groups that act  freely on
some $\Lambda$-tree.
\end{cor}

\begin{theorem}
\label{co:membership_lambda}
The Subgroup Membership Problem is decidable in every finitely presented $\Lambda$-free group.
\end{theorem}
\begin{proof} By Theorem \ref{co:Lambda_quasi-convex_hierarchy_1} every finitely generated subgroup 
of a finitely presented $\Lambda$-tree group $G$  is quasi-isometrically embedded into $G$. Obviously, 
the Membership Problem for every fixed quasi-isometrically embedded subgroup in a finitely generated 
group with decidable Word Problem is decidable.
\end{proof}

\bibliography{../main_bibliography}
\end{document}